\documentclass{article}
\usepackage[text={155mm,220mm},centering]{geometry}

\RequirePackage{amsthm,amsmath,amsfonts,amssymb}

\usepackage{mathrsfs} 
\usepackage{color}
\usepackage{enumitem}
\usepackage{mathtools}  
\usepackage{verbatim} 
\usepackage{bm} 
\usepackage[utf8]{inputenc}
\usepackage[T1]{fontenc}
\usepackage{hyperref} 
\hypersetup{	colorlinks=true,	linkcolor=blue,	filecolor=magenta,	urlcolor=blue,	allcolors=blue,}

\newtheorem{thm}{Theorem}[section]
\newtheorem{lem}[thm]{Lemma}
\newtheorem{prop}[thm]{Proposition}
\newtheorem{cor}[thm]{Corollary}

\theoremstyle{definition}
\newtheorem{rem}[thm]{Remark}

\newtheorem{exa}[thm]{Example}
\newtheorem{definition}[thm]{Definition}

\newcommand\numberthis{\addtocounter{equation}{1}\tag{\theequation}}

\def \P{P}
\def \E{\mathbb{E}}
\def \R{\mathbb{R}}
\def \Q{Q}

\def \myD{\ensuremath{D}}

\def \Ac{\mathcal{A}}
\def \Bc{\mathcal{B}}
\def \Ec{\mathcal{E}}

\def \Gc{\mathcal{G}}
\def \Hc{\mathcal{H}}

\def \Nc{\mathcal{N}}

\def \1{{\bf 1}}

\def \ive{\bm{i}}

\def \xve{x}
\def \yve{y}
\def \zve{z}
\def \zetave{\zeta}
\def \thetave{\theta}

\def \Xma{X}
\def \Yma{Y}
\def \Zma{Z}

\def \Hfun{\Psi}

\def \n{{m}}
\def \m{{N}}
\def \Nf{\mathfrak{N}}

\def \Ecal{\mathcal{E}}
\def \Ncal{\mathcal{N}}
\def \Acalcal{\mathcal{A}}

\def \varrho{\gamma}

\def \thetaeta{\theta} 

\def\i{\bm{i}}
\def\1ve{\bm{1}}
\def\Xfrak{\mathfrak{X}}
\def\dtheta{d_\Theta} 

\newcommand{\LN}[1]{{\color{red} \textit{LN: #1}}}
\newcommand{\YW}[1]{{\color{red}  #1}}

\newcommand{\myeoe}{\hfill $\Diamond$}

\DeclareMathOperator*{\argmax}{arg\,max} 

\allowdisplaybreaks

\begin{document}

	\title{Convergence of de Finetti's mixing measure in\\ latent structure models for observed exchangeable sequences}
	\author{Yun Wei and XuanLong Nguyen \\
	University of Michigan}
	\maketitle

\begin{abstract}
    Mixtures of product distributions are a powerful device for learning about heterogeneity within data populations. In this class of latent structure models, de Finetti's mixing measure plays the central role for describing the uncertainty about the latent parameters representing heterogeneity. In this paper posterior contraction theorems for de Finetti's mixing measure arising from finite mixtures of product distributions will be established, under the setting the number of exchangeable sequences of observed variables increases while sequence length(s) may be either fixed or varied. The role of both the number of sequences and the sequence lengths will be carefully examined. In order to obtain concrete rates of convergence, a first-order identifiability theory for finite mixture models and a family of sharp inverse bounds for mixtures of product distributions will be developed via a harmonic analysis of such latent structure models. This theory is applicable to broad classes of probability kernels composing the mixture model of product distributions for both continuous and discrete domain $\Xfrak$. Examples of interest include the case the probability kernel is only weakly identifiable in the sense of~\cite{ho2016convergence}, the case where the kernel is itself a mixture distribution as in hierarchical models, and the case the kernel may not have a density with respect to a dominating measure on an abstract domain $\Xfrak$ such as Dirichlet processes.
\end{abstract}

\tableofcontents

\paragraph{Acknowledgement} The authors are grateful for the support provided by NSF Grants DMS-1351362, DMS-2015361 and the Toyota Research Institute. We would like to sincerely thank Xianghong Chen, Danqing He and Qingtang Su for valuable discussions, and Judith Rousseau and Nhat Ho for helpful comments. We also thank the anonymous referees and associate editor for valuable comments and suggestions.

\section{Introduction}
\label{sec:introduction}

Latent structure models with many observed variables are among the most powerful and widely used tools in statistics for learning about heterogeneity within data population(s). An important canonical example of such models is the mixture of product distributions, which may be motivated by de Finetti's celebrated theorem for exchangeable sequences of random variables~\cite{aldous1985exchangeability,kallenberg2006probabilistic}. 
The theorem of de Finetti states roughly that if $X_1,X_2,\ldots$ is an infinite exchangeable sequence of random variables defined in a measure space $(\Xfrak,\Acalcal)$, then there exists a random variable $\theta$ in some space $\Theta$, where $\theta$ is distributed according to a probability measure $G$, such that $X_1,X_2,\ldots$ are conditionally i.i.d. given $\theta$. Denote by $P_\theta$ the conditional distribution of $X_i$ given $\theta$, we may express the joint distribution of a $\m$-sequence $X_{[\m]}:= (X_1,\ldots, X_\m)$, for any $\m\geq 1$, as a mixture of product distributions in the following sense: for any $A_1,\ldots, A_\m \subset \Acalcal$,
\[ 
\label{eqn:mixtureofiid}
P(X_1 \in A_1,\ldots, X_\m \in A_\m) = \int \prod_{n=1}^{\m} P_\theta(X_n \in A_n) G(d\theta).
\]
The probability measure $G$ is also known as de Finetti mixing measure for the exchangeable sequence. It captures the uncertainty about the latent variable $\theta$, which describes the mechanism according to which the sequence $(X_i)_{i}$ is generated via $P_\theta$. In other words, the de Finetti mixing measure $G$ can be seen as representing the heterogeneity within the data populations observed via sequences $X_{[\m]}$. A statistician typically makes some assumption about the family $\{P_\theta\}_{\theta\in \Theta}$, and proceeds to draw inference about the nature of heterogeneity represented by $G$ based on data samples $X_{[\m]}$.

In order to obtain an estimate of the mixing measure $G$, one needs multiple copies of the exchangeable sequences $X_{[\m]}$. As mentioned, some assumption will be required of the probability distributions $P_{\theta}$, as well as the mixing measure $G$.  Throughout this paper it is assumed that the map $\theta\mapsto P_{\theta}$ is injective. Moreover,
we will confine ourselves to the setting of exact-fitted finite mixtures, i.e., $G$ is assumed to be an element of $\Ec_{k}(\Theta)$, the space of discrete measures with $k$ distinct supporting atoms on $\Theta$, where $\Theta$ is a subset of $\R^q$. Accordingly, we may express $G = \sum_{j=1}^{k} p_{j} \delta_{\theta_j}$.  We may write the distribution for $X_{[\m]}$ in the following form, where we include the subscripts $G$ and $\m$ to signify their roles:
\begin{equation}
    \label{eqn:mixprod}
    P_{G,\m}(X_1 \in A_1,\ldots, X_{\m} \in A_{\m}) = \sum_{j=1}^{k} p_j \biggr \{\prod_{n=1}^{\m} P_{\theta_j}(X_{n} \in A_{n}) \biggr \}.
\end{equation}    
Note that when $\m=1$, we are reduced to a mixture distribution $P_G:=P_{G,1}= \sum_{j=1}^{k} p_j P_{\theta_j}$. Due to the role they play in the composition of the distribution $P_{G,\m}$, we also refer to $\{P_\theta\}_{\theta\in \Theta}$ as a family of  \emph{probability kernels} on $\Xfrak$.  
Given $\n$ \emph{independent} copies of exchangeable sequences $\{X_{[\m_i]}^i\}_{i=1}^{\n}$ each of which is respectively distributed according to $P_{G,\m_i}$ 
given in~\eqref{eqn:mixprod}, where $\m_i$ denotes the possibly variable length of the $i$-th sequence.
The primary question of interest in this paper is the efficiency of the estimation of the true mixing measure $G=G_0 \in \Ec_{k}(\Theta)$, for some known $k=k_0$, as sample size $(m, \m_1,\ldots, \m_\n)$ increases in a certain sense. 

Models described by Eq.~\eqref{eqn:mixtureofiid} are also known in the literature as mixtures of repeated measurements, or mixtures of grouped observations \cite{hettmansperger2000almost, elmore2004estimating,cruz2004semiparametric,Jochmans1996Nonparametric,vandermeulen2019operator,ritchie2020consistent}, with applications to domains such as psychological analysis, educational assessment, and topic modeling in machine learning. The random effects model described in Section 1.3.3 of \cite{lindsay1995mixture} in which the mixing measure is a discrete measure with finite number of atoms is also a special case of \eqref{eqn:mixtureofiid} with $P_\theta$ a normal distribution with mean $\theta$. While \cite{hettmansperger2000almost,cruz2004semiparametric} consider the case that the number of components $k$ is unknown,  \cite{elmore2004estimating,Jochmans1996Nonparametric,vandermeulen2019operator,ritchie2020consistent} focus on the case that $k$ is known, the same as our set up.  In many of the aforementioned works the models are nonparametric, i.e., no parametric forms for the probability kernels are assumed, and the focus is on the problem of density estimation due to the nonparametric setup. 
By contrast, in this paper we study mixture of product distributions \eqref{eqn:mixtureofiid} with the parametric form of component distribution imposed, since in practice prior knowledge on the component distribution $\P_{\theta}$ might be available. Moreover, we investigate the behavior of parameter estimates --- the convergence of the parameters $p_j$ and $\theta_j$, which are generally more challenging than density estimation in mixture models \cite{nguyen2013convergence, ho2016strong, ho2016convergence, ho2019singularity}.  

Before the efficiency question can be addressed, one must consider the issue of identifiability: under what conditions does the data distribution $P_{G,\m}$ uniquely identify the true mixing measure $G_0$? 
This question has occupied the interest of a number of authors ~\cite{teicher1967identifiability,elmore2005application,hall2005nonparametric}, with decisive results obtained recently by~\cite{allman2009identifiability} on finite mixture models for conditionally independent observations and by \cite{vandermeulen2019operator} on finite mixtures for conditionally i.i.d. observations (given by Eq.~\eqref{eqn:mixtureofiid}). Here, the condition is in the form of $\m \geq n_0$, for some natural constant $n_0 \geq 1$ possibly depending on $G_0$. We shall refer to $n_0$ as (minimal) \emph{zero-order identifiable length} 
or 0-identifiable length for short (a formal definition will be given later). For the conditionally i.i.d. case as in model \eqref{eqn:mixtureofiid}, \cite{vandermeulen2019operator} proves that as long as $N\geq 2k-1$, model \eqref{eqn:mixtureofiid} will be identifiable for any $G_0$. Note that $2k-1$ is only an upper bound of $n_0$. For a given parametric form of $\{P_\theta\}_{\theta\in \Theta}$ and a given truth $G_0$, the $0$-identifiable length might be smaller than $2k-1$.

Drawing from existing identifiability results,  
it is quite apparent that the observed sequence length $\m$ (or more precisely, $\m_1,\ldots, \m_\n$, in case of variable length sequences) must play a crucial role in the estimation of mixing measure $G$, in addition to the number $\n$ of sequences. Moreover, it is also quite clear that in order to have a consistent estimate of $G=G_0$, the number of sequences $\n$ must tend to infinity, whereas $\m$ may be allowed to be fixed. It remains an open question as to the precise roles $\n$ and $\m$ play in estimating $G$ and on the different types of mixing parameters: the component parameters (atoms $\theta_j$) and mixing proportions (probability mass $p_j$), and the rates of convergence of a given estimation procedure. 

Partial answers to this question were obtained in several settings of mixtures of product distributions. ~\cite{hettmansperger2000almost} proposed to discretize data so that the model in consideration becomes a finite mixture of product of identical binomial or multinomial distributions. Restricting to this class of models, a maximum likelihood estimator was applied, and a standard asymptotic analysis establishes root-$\n$ rate for mixing proportion estimates.  ~\cite{hall2003nonparametric,hall2005nonparametric} investigated a number of nonparametric estimators for $G$, and obtained the root-$\n$ convergence rate for both mixing proportion and component parameters in the setting of $k=2$ mixture components under suitable identifiability conditions. It seems challenging to extend their method and theory to a more general setting, e.g., $k > 2$. Moreover, 
no result on the effect of $\m$ on parameter estimation efficiency seems to be available.  Recently,
~\cite{nguyen2016borrowing,nguyen2015posterior} studied the posterior contraction behavior of several classes of Bayesian hierarchical model where the sample is also specified by  $\n$ sequences of $\m$ observations. His approach requires that both $\n$ and $\m$ tend to infinity and thus cannot be applied to our present setting where $\m$ may be  fixed.

In this paper we shall present a parameter estimation theory for general classes of finite mixtures of product distributions. An application of this theory will be posterior contraction theorems established for a standard Bayesian estimation procedure, according to which the de Finetti's mixing measure $G$ tends toward the truth $G_0$, as $\n$ tends to infinity, under suitable conditions. In a standard Bayesian procedure, the statistician endows the space of parameters $\Ec_{k_0}(\Theta)$ with a prior distribution $\Pi$, which is assumed to have compact support in these theorems, 
and applies Bayes' rule to obtain the posterior distribution on $\Ec_{k_0}(\Theta)$, to be denoted by $\Pi(G | \{X_{[\m_i]}^i\}_{i=1}^{\n})$. %
To anticipate the distinct convergence behaviors for the atoms and probability mass parameters, for any
$G=\sum_{i=1}^k p_i\delta_{\theta_i}, G'=\sum_{i=1}^k p'_i\delta_{\theta'_i} \in \Ec_k(\Theta)$, define
 $$
 \myD_{\m}(G,G') =  \min_{\tau\in S_{k}} \sum_{i=1}^{k}(\sqrt{\m}\|\theta_{\tau(i)}-\theta'_i\|_2+|p_{\tau(i)}-p'_i|), 
 $$ 
 where $S_{k}$ denotes all the permutations on the set $[k] := \{1,2,\ldots,k\}$. 
(The suitability of $D_N$ over other choices of metric will be discussed in Section~\ref{sec:prelim}).

Given $\n$ independent exchangeable sequences denoted by $\{X^i_{[\m_i]}\}_{i=1}^{\n}$. We naturally require that $\min_{i} \m_i \geq n_0$, where $n_0$ is the zero-order identifiable length depending on $G_0$. Moreover, to obtain concrete rates of convergence, we need also $\min_{i}\m_i \geq n_1$ for some minimal natural number $n_1:=n_1(G_0) \geq 1$. We shall call $n_1$ the \emph{minimal first-order identifiable length depending on $G_0$}, or 1-identifiable length for short (a formal definition will be given later).
Assume that $\{\m_i\}_{i=1}^{\n}$ are uniformly bounded from above by an arbitrary unknown constant, in Theorem~\ref{thm:posconnotid} it is established that under suitable regularity conditions on $P_\theta$, the posterior contraction rate for the mixing proportions is bounded above by $\n^{-1/2}$, up to a logarithmic quantity. For mixture components' supporting atoms, the contraction rate is 
   $$O_P\biggr (\sqrt{\frac{\ln (\sum_{i=1}^{\n} \m_i)}{\sum_{i=1}^{\n} \m_i}} \biggr).$$ 
Note that $\sum_{i=1}^{\n} \m_i$ represents the full volume of the observed data set. More precisely, for suitable kernel families $P_\theta$, as long as $\min_{i}\m_i \geq \max\{n_0, n_1\}$ and $\sup_{i} \m_i < \infty$, there holds
	\begin{align*}
	&\Pi \biggr (G\in \Ec_{k_0}(\Theta): D_{\sum_{i=1}^{\n}\m_i/\n}(G,G_0) \leq C(G_0)\bar{M}_\n \sqrt{ \frac{\ln(\sum_{i=1}^{\n}\m_i)}{\n} } \biggr |X_{[\m_1]}^1, \ldots,X_{[\m_{\n}]}^{\n} \biggr ) \to  1
	\end{align*}
	in $\P_{G_0,\m_1}\otimes\cdots \otimes \P_{G_0,\m_{\n}}$-probability as $\n\to \infty$ for any sequence $\bar{M}_m\to\infty$. The point here is that constant $C(G_0)$ is independent of $\n$, sequence lengths $\{\m_i\}_{i=1}^{\n}$ and their supremum.
In plain terms, we may say that with finite mixtures of product distributions, the posterior inference of atoms of each individual mixture component receives the full benefit of "borrowing strength" across sampled sequences; while the mixing probabilities gain efficiency from only the number of such sequences. This appears to be the first work in which such a posterior contraction theorem is established for de Finetti mixing measure arising from finite mixtures of product distributions. 

The Bayesian learning rates established appear intuitive, given the parameter space $\Theta \in \R^q$ is of finite dimension. On the role of $\n$, they are somewhat compatible to the previous partial results~\cite{hettmansperger2000almost,hall2003nonparametric,hall2005nonparametric}. However, we wish to make several brief remarks at this juncture.
\begin{itemize}
\item First, even for exact-fitted parametric mixture models, 
"parametric-like" learning rates of the form root-$\n$ or root-$(\n\m)$ should not to be taken for granted, because they do not always hold~\cite{ho2016convergence,ho2019singularity}. This is due to the fact that the kernel family $\{P_{\theta}\}_{\theta\in \Theta}$ may easily violate assumptions of strong identifiability often required for the root-$\n$ rate to take place. In other words, the kernel family $\{\P_{\theta}\}$ may be only \emph{weakly identifiable}, resulting in poor learning rates for a standard mixture, i.e., when $\m=1$. 
\item Second, the fact that by increasing the observed exchangeable sequence's length $\m$ so that $\m \geq n_1 \vee n_0$, one may obtain parametric-like learning rates in terms of both $\m$ and $\n$ is a remarkable testament of how repeated measurements can help to completely overcome a latent variable model's potential pathologies: parameter non-identifiability is overcome by making $\m \geq n_0$, while inefficiency of parameter estimation inherent in weakly identifiable mixture models is overcome by $\m \geq n_1$. For a deeper appreciation of this issue, see Section~\ref{sec:background} for a background on the role of identifiability notions in parameter estimation. 
\end{itemize}

Although the posterior contraction theorems for finite mixtures of product distributions presented in this paper are new, such results do not adequately capture the rather complex behavior of the convergence of parameters for a finite mixture of $\m$-product distributions. In fact, the heart of the matter lies in the establishment of a collection of general \emph{inverse bounds}, i.e., inequalities of the form
\begin{equation}
    \label{eqn:inverse}
    D_{\m}(G,G_0) \leq C(G_0) V(P_{G,\m},P_{G_0,\m}),
\end{equation}
where $V(\cdot,\cdot)$ is the variational distance. Note that \eqref{eqn:inverse} provides an upper bound on distance $D_\m$ of mixing measures in terms of the variational distance between the corresponding mixture of $\m$-product distributions. 
Inequalities of this type allow one to transfer the convergence (and learning rates) of a data population's distribution into that of the corresponding distribution's parameters (therefore the term "inverse bounds"). Several points to highlight are:
\begin{itemize}
    \item The local nature of~\eqref{eqn:inverse}, which may hold only for $G$ residing in a suitably small $D_{\m}$-neighborhood of $G_0$ whose radius may also depend on $G_0$ and $\m$, while constant $C(G_0)>0$ depends on $G_0$ but is independent of $\m$. In addition, the bound holds only when $\m$ exceeds threshold $n_1\geq 1$, 
    unless further assumptions are imposed. For instance,  under a first-order identifiability condition of $P_\theta$, $n_1 = 1$, 
    so this bound holds  for all $\m \geq 1$ while remaining local in nature. Moreover, inequality~\eqref{eqn:inverse} is sharp: the quantity $\m$ in $D_\m$ cannot be improved by $D_{\psi(\m)}$ for any sequence $\psi(\m)$ such that $\psi(\m)/\m \rightarrow \infty$ (see Lemma \ref{lem:optimalsquaretootN}).

    \item 
    The inverse bounds of the form~\eqref{eqn:inverse} are established without any overt assumption of identifiability. However, they carry striking consequences on both first-order and classical identifiability, which can be deduced from~\eqref{eqn:inverse} under a compactness condition (see Proposition~\ref{lem:n0nbar}): using the notation $n_0(G,\cup_{k\leq k_0}\Ec_k(\Theta_1))$ and $n_1(G, \Ec_{2k_0}(\Theta_1))$ to denote explicitly the dependence of 0- and 1-identifiable lengths on $G$ in the first argument and its ambient space in the second argument, respectively, we have  
    $$
     \sup_{G \in \cup_{k\leq k_0} \Ec_{k}(\Theta_1)} n_0(G,\cup_{k\leq k_0}\Ec_k(\Theta_1)) 
    \leq \sup_{G \in  \Ec_{2k_0}(\Theta_1)} n_1(G, \Ec_{2k_0}(\Theta_1)) <  \infty.
    $$  
    Note that classical identifiability captured by $n_0(G,\cup_{k\leq k_0} \Ec_{k}(\Theta_1))$ describes a global property of the model family while first-order identifiability captured by $n_1(G, \Ec_{2k_0}(\Theta_1))$ 
    is local in nature. The connection between these two concepts is possible because when the number of exchangeable variables $\m$ gets large, the force of the central limit theorem for product distributions comes into effect to make the mixture model eventually become identifiable, either in the classical or the first-order sense, even if the model may be initially non-identifiable or weakly identifiable (when $\m =1$).

    \item These inverse bounds hold for very broad classes of probability kernels $\{P_\theta\}_{\theta \in \Theta}$. In particular, they are established under mild regularity assumptions on the family of probability kernel $P_{\theta}$ on $\Xfrak$, when either $\Xfrak=\R^d$, or $\Xfrak$ is a finite set, or $\Xfrak$ is an abstract space. A standard but non-trivial example of our theory is the case the kernels $P_{\theta}$ belong to the exponential families of distributions. A more unusual example is the case where $P_{\theta}$ is itself a mixture distribution on $\Xfrak$. Kernels of this type are rarely examined in theory, partly because when we set $\m=1$ a mixture model using such kernels typically would \emph{not} 
    be parameter-identifiable. However, such "mixture-distribution" kernels are frequently employed by practitioners of hierarchical models (i.e., mixtures of mixture distributions). As the inverse bounds entail, this makes sense since the parameters become more strongly identifiable and efficiently estimable with repeated exchangeable measurements. 
    \item More generally, inverse bounds hold when $P_\theta$ does not necessarily admit a density with respect to a dominating measure on $\Xfrak$. An example considered in the paper is the case $P_\theta$ represents probability distribution on the space of probability distributions, namely, $P_\theta$ represents (mixtures of) Dirichlet processes. As such, the general inverse bounds are expected to be useful for models with nonparametric mixture components represented by $P_\theta$, the kind of models that have attracted much recent attention, e.g.,~\cite{Teh-etal-06,Rodriguez-etal-08,camerlenghi2019distribution,camerlenghi2019latent}.
\end{itemize}
The above highlights should make clear the central roles of the inverse bounds obtained in Section~\ref{sec:firstidentifiable} and Section~\ref{sec:inversebounds}, which deepen our understanding of the questions of parameter identifiability and provide detailed information about the convergence behavior of parameter estimation. In addition to an asymptotic analysis of Bayesian estimation for mixtures of product distributions that will be carried out in this paper, such inverse bounds may also be useful for deriving rates of convergence for non-Bayesian parameter estimation procedures, including maximum likelihood estimation and distance based estimation methods. 

The rest of the paper will proceed as follows. Section~\ref{sec:background} presents related work in the literature and a high-level overview of our approach and techniques. Section~\ref{sec:prelim} prepares the reader with basic setups and several useful concepts of distances on space of mixing measures that arise in mixtures of product distributions. Section~\ref{sec:firstidentifiable} is a self-contained treatment of first-order identifiability theory for finite mixture models, leading to several new  results that are useful for subsequent developments. Section~\ref{sec:inversebounds} presents inverse bounds for broad classes of finite mixtures of product distributions, along with specific examples. An immediate application of these bounds are posterior contraction theorems for de Finetti's mixing measures, the main focus of Section~\ref{sec:poscon}. Particular examples of interest for the inverse bounds established in Section~\ref{sec:inversebounds} include the case the probability kernel $P_\theta$ is itself a mixture distribution on $\Xfrak = \R$, and the case $P_\theta$ is a mixture of Dirichlet processes. These examples require development of new tools and are deferred to Section~\ref{sec:mixofmix}. 
Section \ref{sec:minimax} gives several technical results
demonstrating the sharpness of the established inverse bounds, which is then used to derive minimax
lower bounds for estimation procedures of de Finetti’s mixing parameters.
Section \ref{sec:extensions} discusses extensions and several future directions. 
Finally, (most) proofs of all theorems and lemmas will be provided in the Appendix. 

	\noindent {\bf Notation}
    For any probability measure $\P$ and $\Q$ on measure space       $(\Xfrak,\Ac)$ with densities respectively $p$ and $q$  with respect to    some base measure $\mu$, the variational distance between them is
	$V(\P,\Q) = \sup_{A\in \Ac} |\P(A)-\Q(A)| = \int_{\Xfrak} \frac{1}{2} |p(x)-q(x)|d\mu$. The Hellinger distance is given by
	$ h(\P,\Q) = \left(\int_{\Xfrak} \frac{1}{2} |\sqrt{p(x)}-\sqrt{q(x)}|^2d\mu\right)^{\frac{1}{2}}.
	$ The Kullback-Leibler divergence of $Q$ from $P$ is $K(p,q)=\int_{\Xfrak}p(x)\ln\frac{p(x)}{q(x)} d\mu$.
	Write $P\otimes Q$ to be the product measure of $P$ and $Q$ and $\otimes^N P$ for the $N$-fold product of $P$. 
	Any vector $x\in \R^d$ is a column vector with its $i$-th coordinate denoted by $x^{(i)}$. 
The inner product between two vectors $a$ and $b$ is denoted by $a^\top b$ or $\langle a, b\rangle$.  
%
%
	Denote by $C(\cdot)$ or $c(\cdot)$  a positive finite constant depending only on its parameters and the probability kernel $\{P_\theta\}_{\theta\in \Theta}$. In the presentation of inequality bounds and proofs,  they may differ from line to line. Write $a \lesssim b$ if $a\leq  c  b$ for some universal constant $c$; write $a \lesssim_\xi b$ if $a\leq c(\xi) b$. Write $a \asymp b$ if $a \lesssim b$ and $b \lesssim a$; write $a {\asymp}_{\xi} b$ (or $a \overset{\xi}{\asymp} b$) if $a \lesssim_{\xi} b$ and $b \lesssim_{\xi} a$.  


\section{Background and overview}
\label{sec:background}
\subsection{First-order identifiability and inverse inequalities}
In order to shed light on the convergence behavior of model parameters as data sample size increases, stronger forms of identifiability conditions shall be required of the family of probability kernels $P_\theta$. For finite mixture models, such conditions are often stated in terms of a suitable derivative of the density of $P_\theta$ with respect to parameter $\theta$, and the linear independence of such derivatives as $\theta$ varies in $\Theta$. The impacts of such identifiability conditions, or the lack thereof, on the convergence of parameter estimation can be quite delicate. Specifically, let $\Xfrak=\R^d$ and fix $\m = 1$, so we have $P_G = \sum_{j=1}^{k} p_j P_{\theta_j}$.  Assume that $P_{\theta}$ admits a density function $f(\cdot|\theta)$ with respect to Lebesgue measure on $\R^d$, and for all $x\in \R^d$, $f(\cdot|\theta)$ is differentiable with respect to $\theta$; moreover the combined collection of functions $\{f(\cdot|\theta)\}_{\theta \in \Theta}$ and $\{\nabla f(\cdot|\theta)\}_{\theta \in \Theta}$ are linearly independent. This type of condition, which concerns linear independence of the first derivatives of the likelihood functions with respect to parameter $\theta$, shall be generically referred to as \emph{first-order} identifiability condition of the probability kernel family $\{P_\theta\}_{\theta\in \Theta}$. A version of such condition was investigated by~\cite{ho2016strong}, who showed that their condition will be sufficient for establishing an inverse bound of the form
\begin{equation}
\label{eqn:inverseliminf}
\liminf_{\substack{G\overset{W_1}{\to} G_0\\ G\in \Ec_{k_0}(\Theta)}} \frac{V(P_G,P_{G_0})}{W_1(G,G_0)}>0. 
\end{equation}
where $W_1$ denotes the first-order Wasserstein distance metric on $\Ec_{k_0}(\Theta)$. The infimum limit quantifier should help to clarify somewhat the local nature of the inverse bound~\eqref{eqn:inverse} mentioned earlier. The development of this local inverse bound and its variants plays the fundamental role in the analysis of parameter estimation with finite mixtures in a variety of settings in previous studies, where stronger forms of identifiability conditions based on higher order derivatives may be required~\cite{chen1995optimal,nguyen2013convergence,Rousseau-Mengersen-11,ho2016strong,ho2016convergence,heinrich2018strong,ho2019singularity}. In addition,~\cite{nguyen2013convergence,nguyen2016borrowing} studied inverse bounds of this type for infinite mixture and hierarchical models.

As noted by~\cite{ho2016strong}, for exact-fitted setting of mixtures, i.e., the number of mixture components $k=k_0$ is known, conditions based on only first-order derivatives of $P_\theta$ will suffice.  Under a suitable first-order identifiability condition based on linear independence of $\{f(\cdot|\theta), \nabla_\theta f(\cdot|\theta)\}_{\theta \in \Theta}$, 
along with several additional regularity conditions, the mixing measure $G=G_0$ may be estimated via $\n$-i.i.d. sample $(X_{[1]}^1,\ldots,X_{[1]}^{\n})$ at the parametric rate of convergence $\n^{-1/2}$, due to ~\eqref{eqn:inverseliminf} and the fact that the data population density $p_{G_0}$ is typically estimated at the same parametric rate. However, first-order identifiability may not be satisfied, as is the case of two-parameter gamma kernel, or three-parameter skewnormal kernel, following from the fact that these kernels are governed by certain partial differential equations. In such situations, not only does the resulting Fisher information matrix of the mixture model become singular, the singularity structure of the matrix can be extremely complex --- an in-depth treatment of weakly identifiable mixture models can be found in~\cite{ho2019singularity}. Briefly speaking, in such situations~\eqref{eqn:inverseliminf} may not hold and the rate $\n^{-1/2}$ may not be achieved~\cite{ho2016convergence,ho2019singularity}. In particular, in the case of skewnormal kernels, extremely slow rates of convergence for the component parameters $\theta_j$ (e.g., $\n^{-1/4}, \n^{-1/6}, \n^{-1/8}$ and so on) may be established depending on the actual parameter values of the true $G_0$ for a standard Bayesian estimation or maximum likelihood estimation procedure ~\cite{ho2019singularity}. It remains unknown whether it is possible to devise an estimation procedure to achieve the parametric rate of convergence $\n^{-1/2}$ when the finite mixture model is only weakly identifiable, i.e., when first-order identifiability condition fails. 

In Section~\ref{sec:firstidentifiable} we shall revisit the described first-order identifiability notions, and then present considerable improvements upon the existing theory and deliver several novel results. First, we identify a tightened set of conditions concerning linear independence of $f(x|\theta)$ and $\nabla_\theta f(x|\theta)$
according to which the inverse bound~\eqref{eqn:inverse} holds. This set of conditions turns out to be substantially weaker than the identifiability condition of~\cite{ho2016strong}, most notably by requiring $f(x|\theta)$ be differentiable with respect to $\theta$ only for $x$ in a subset of $\Xfrak$ with positive measure. This weaker notion of first-order identifiability allows us to broaden the scope of probability kernels for which the inverse bound~\eqref{eqn:inverseliminf} continues to apply (see Lemma~\ref{lem:firstidentifiable}). Second, in a precise sense we show that this notion is in fact necessary for~\eqref{eqn:inverseliminf} to hold (see Lemma~\ref{lem:nececondition}), giving us an arguably complete characterization of first-order identifiability and its relations to the parametric learning rate for model parameters. Among other new results, it is worth mentioning that when the kernel family $\{P_{\theta}\}_{\theta \in \Theta}$ belongs to an exponential family of distributions on $\Xfrak$, there is a remarkable equivalence among our notion of first-order identifiability condition and the inverse bound of the form~\eqref{eqn:inverseliminf}, and the inverse bound in which variational distance $V$ is replaced by Hellinger distance $h$ (see Lemma~\ref{cor:expD1equivalent}).

Turning our attention to finite mixtures of product distributions, a key question is on the effect of number $\m$ of repeated measurements in overcoming weak identifiability (e.g., the violation of first-order identifiability). One way to formally define the first-order identifiable length (1-identifiable length) $n_1= n_1(G_0)$ is to make it the minimal natural number such that the following inverse bound holds for any $\m \geq n_1$ 
    \begin{equation}
     \liminf_{\substack{G\overset{W_1}{\to} G_0\\ G\in \Ecal_{k_0}(\Theta) }} \frac{V(P_{G,\m },P_{G_0,\m })}{W_1(G,G_0)}>0.\label{eqn:VNW1}
    \end{equation}
The key question is whether (finite) $1$-identifiable length exists, and how can we characterize it. The significance of this concept is that one can achieve first-order identifiability by allowing at least $\m \geq n_1$ repeated measurements and obtain the $\n^{-1/2}$ learning rate for the mixing measure. In fact, the component parameters can be learned at the rate $(\n\m)^{-1/2}$, the square root of the full volume of exchangeable data (modulo a logarithmic term).
The resolution of the question of existence and characterization of $n_1$ leads us to establish a collection inverse bounds involving mixtures of product distributions that we will describe next. Moreover, such inverse bounds are essential in deriving learning rates for mixing measure $G$ from a collection of exchangeable sequences of observations.

\subsection{General approach and techniques}
\label{sec:overview}
For finite mixtures of $\m$-product distributions, for $\m\geq 1$, the precise expression for the inverse bound to be established takes the form: under certain conditions of the probability kernel $\{P_\theta\}_{\theta \in \Theta}$, for a given $G_0 \in \Ecal_{k_0}(\Theta^\circ)$, 
\begin{equation}
\label{eqn:inverseliminfliminf}
\liminf_{\m\to \infty}\liminf_{\substack{G\overset{W_1}{\to} G_0\\ G\in \Ecal_{k_0}(\Theta) }} \frac{V(\P_{G,\m},\P_{G_0,\m})}{\myD_{\m}(G,G_0)}>0. 
\end{equation}
Compared to inverse bound~\eqref{eqn:inverseliminf} for a standard finite mixture, the double infimum limits reveals the challenge for analyzing mixtures of $\m$-product distributions; they express the delicate nature of the inverse bound informally described via~\eqref{eqn:inverse}. Moreover, \eqref{eqn:inverseliminfliminf} entails that the finite 1-identifiable length $n_1$ defined by~\eqref{eqn:VNW1} exists. 

Inverse bound~\eqref{eqn:inverseliminfliminf} will be established for broad classes of kernel $P_\theta$ and it can be shown that this bound is sharp. Among the settings of kernel that the bound is applicable, there is a setting when $\P_\theta$ belongs to any regular exponential family of distributions. More generally, this includes the setting where $\Xfrak$ may be an abstract space; no parametric assumption on $\P_\theta$ will be required. 
Instead, we appeal to a set of mild regularity conditions on the characteristic function of a push-forward measure produced by a measurable map $T$ acting on the measure space $(\Xfrak, \Acalcal)$. Actually, a stronger bound is established relating to the positivity of a notion of curvature on the space of mixtures of product distributions (see~\eqref{eqn:curvatureprodbound}).
We will see that this collection of inverse bounds, which are presented in Section~\ref{sec:inversebounds}, enables the study for a very broad range of mixtures of product distributions for exchangeable sequences.

The theorems establishing~\eqref{eqn:inverseliminfliminf} and~\eqref{eqn:curvatureprodbound} represent the core of the paper. For simplicity, let us describe the gist of our proof techniques by considering the case kernel $P_\theta$ belongs to an exponential family of distribution on $\Xfrak$ (see Theorem~\ref{thm:expfam}). Suppose the kernel admits a density function $f(x|\theta)$ with respect to a dominating measure $\mu$ on $\Xfrak$. At a high-level, this is a proof of contradiction: if ~\eqref{eqn:inverseliminfliminf} does not hold, then there exists a strictly increasing subsequence $\{\m_\ell\}_{\ell=1}^{\infty} $ of natural numbers  according to which there exists a sequence of mixing measures $\{G_\ell\}_{\ell=1}^{\infty}\subset \Ec_{k_0}(\Theta)\backslash\{G_0\}$ such that $D_{\m_\ell}(G_\ell,G_0) \rightarrow 0$ as $\ell \rightarrow\infty$ and
the integral form 
\begin{equation}
\label{eqn:integrallim}
    \frac{V(P_{G_\ell,\m_\ell},P_{G_0,\m_\ell})}{D_{\m_\ell}(G_\ell,G_0)} 
    = \int_{\Xfrak^{\m_\ell}} \biggr |\frac{p_{G_\ell,\m_\ell}(x_1,\ldots,x_{\m_\ell}) - p_{G_0,\m_\ell}(x_1,\ldots,x_{\m_\ell})}
    {D_{\m_\ell}(G_\ell,G_0)}  \biggr | d\otimes^{\m_\ell} \mu(x_1,\ldots,x_{\m_\ell})
\end{equation}
tends to zero. One may be tempted to apply Fatou's lemma to deduce that the integrand must vanish as $\ell\rightarrow \infty$, and from that one may hope to derive a contradiction with specified hypothesis on the probability kernel $f(x|\theta)$ (e.g. first-order identifiability). This is basically the proof technique of Lemma~\ref{lem:firstidentifiable} for establishing inverse bound~\eqref{eqn:inverseliminf} for finite mixtures. But this would not work here, because the integration domain's dimensionality increases with $\ell$. Instead we can exploit the structure of the mixture of $\m_\ell$-product densities in $p_{G_\ell,\m_\ell}$, and rewrite the integral as an expectation with respect to a suitable random variable of fixed domain.  What comes to our rescue is the central limit theorem, which is applied to a $\R^q$-valued random variable $Z_\ell = \left(\sum_{n=1}^{\m_{\ell}} T(X_n) - \m_{\ell} \E_{\theta_\alpha^0}T(X_1)\right)/\sqrt{\m_{\ell}}$, where $\E_{\theta_\alpha^0}$ denotes the expectation taken with respect to the probability distribution $P_{\theta}$ for some suitable $\theta=\theta_\alpha^0$ chosen among the support of true mixing measure $G_0$. Here $T:\Xfrak \rightarrow \R^q$ denotes the sufficient statistic for the exponential family distribution $P_{\theta}(dx_n)$, for each $n=1,\ldots,\m_\ell$. 

Continuing with this plan, by a change of measure the integral in~\eqref{eqn:integrallim} may be expressed as the expectation of the form $\E |\Hfun_\ell(Z_\ell)|$ for some suitable function $\Hfun_\ell:\R^q \rightarrow \R$. By exploiting the structure of the exponential families dictating the form of $\Hfun_\ell$, it is possible to obtain that for any sequence $z_\ell \rightarrow z$, there holds $\Hfun_\ell(z_\ell) \rightarrow \Hfun(z)$ for a certain function $\Hfun:\R^q \rightarrow \R$.  Since $Z_\ell$ converges in distribution to $Z$ a non-degenerate zero-mean Gaussian random vector in $\R^q$, it entails that $\Hfun_\ell(Z_\ell)$ converges to $\Hfun(Z)$ in distribution by a generalized continuous mapping theorem~\cite{wellner2013weak}. Coupled with a generalized Fatou's lemma~\cite{billingsley2008probability}, we arrive at $\E_{\theta_\alpha} |\Hfun(Z)| = 0$, which can be verified as a contradiction.

For the general setting where $\{P_{\theta}\}_{\theta\in \Theta}$ is a family of probability on measure space $(\Xfrak,\Acalcal)$, the basic proof structure remains the same, but we can no longer exploit the (explicit) parametric assumption on the kernel family $P_{\theta}$ (see Theorem~\ref{thm:genthm}). Since the primary object of inference is parameter $\theta \in \Theta \subset \R^q$, the assumptions on the kernel $P_\theta$ will center on the existence of a measurable map  $T:(\Xfrak, \Acalcal) \to (\R^{s},\mathcal{B}(\R^s))$ for
some $s\geq q$, and regularity conditions on the push-forward measure on $\R^s$: $T_{\#}P_\theta  := P_\thetave \circ T^{-1}$. This measurable map plays the same role as that of sufficient statistic $T$ when $P_\theta$ belongs to the exponential family. The main challenge lies in the analysis of function $\Hfun_\ell$ described in the previous paragraph. It is here that the power of Fourier analysis is brought to bear on the analysis of $\Hfun_\ell$ and the expectation $\E_{\theta_\alpha^0} \Hfun_\ell(Z_\ell)$. By the Fourier inversion theorem, $\Hfun_\ell$ may be expressed entirely in terms of the characteristic function of the push-forward measure $T_{\#}P_\theta$. Provided regularity conditions on such characteristic function hold, one is able to establish the convergence of $\Hfun_\ell$ toward a certain function $\Hfun:\R^s\rightarrow \R$ as before.

We shall provide a variety of examples demonstrating the broad applicability of Theorem~\ref{thm:genthm}, focusing on the cases $P_\theta$ does not belong to an exponential family of distributions. In some cases, checking for the existence of map $T$ is straightforward. When $P_\theta$ is a complex object, in particular, when $P_\theta$ is itself a mixture distribution, this requires substantial work, as should be expected. In this example, the burden of checking the applicability of Theorem~\ref{thm:genthm} lies primarily in evaluating certain oscillatory integrals composed of the map $T$ in question. Tools from harmonic analysis of oscillatory integrals will be developed for such a purpose and presented in Section~\ref{sec:mixofmix}. We expect that the tools developed here present a useful stepping stone toward a more satisfactory theoretical treatment of complex hierarchical models (models that may be viewed as mixtures of mixtures of distributions, e.g.~\cite{Teh-etal-06,Rodriguez-etal-08,nguyen2016borrowing,camerlenghi2019distribution}), which have received broad and increasingly deepened attention in the literature.

\section{Preliminaries}
\label{sec:prelim}
 
We start by setting up basic notions required for the analysis of mixtures of product distributions. Given exchangeable data sequences denoted by $X^i_{[\m_i]}:= (X_1^i,\ldots,X_{\m_i}^i)$ for $i=1,\ldots,\n$, while $\m_i$ denotes the length of sequence $X_{[\m_i]}^i$. For ease of presentation, for now, we shall assume that $N_i=N$ for all $i$. Later we will allow variable length sequences. These sequences are composed of elements in a measurable space $(\Xfrak,\Acalcal)$. Examples include $\Xfrak = \R^d$, $\Xfrak$ is a discrete space, and $\Xfrak$ is a space of measures. 
Regardless, parameters of interest are always encapsulated by discrete mixing measures $G \in \Ec_{k}(\Theta)$, the space of discrete measures with $k$ distinct support atoms residing in $\Theta\subset \R^q$. 
 
The linkage between parameters of interest, i.e., the mixing measure $G$, and the observed data sequences   is achieved via the mixture of product distributions that we now define. 
Consider a family of probability distributions  $\{\P_{\thetave}\}_{\thetave\in \Theta}$ on measurable space $(\Xfrak,\Ac)$, where $\thetave$ is the parameter of the family and $\Theta\subset \R^q$ is the parameter space. Throughout this paper it is assumed that the map $\theta\mapsto P_{\theta}$ is injective. For $N\in \mathbb{N}$, the $\m$-product probability family is denoted by $\{\P_{\thetave,\m}:= \bigotimes^{\m}\P_{\thetave}\}_{\thetave\in \Theta}$ on $(\Xfrak^{\m},\Ac^{\m})$, where $\Ac^{\m}$ is the product sigma-algebra. Given a mixing measure $G=\sum_{i=1}^{k} p_i \delta_{\theta_i} \in \Ec_{k}(\Theta)$, the mixture of $\m$-product distributions induced by $G$ is given by 
\[P_{G,\m} = \sum_{i=1}^{k} p_i P_{\theta_i,\m}.\]
Each exchangeable sequence $X^i_{[N]} = (X_1^i,\ldots,X_{\m}^i)$, for $i=1,\ldots,\n$, is an independent sample distributed according to $P_{G,\m}$.  Due to the role they play in the composition of  distribution $P_{G,\m}$, we also refer to $\{P_\theta\}_{\theta\in \Theta}$ as a family of  \emph{probability kernels} on $(\Xfrak,\Ac)$.

In order to quantify the convergence of mixing measures arising in mixture models, an useful device is a suitably defined optimal transport distance~\cite{nguyen2013convergence,nguyen2011wasserstein}.
Consider the Wasserstein-$p$ distance w.r.t. distance $\dtheta$ on $\Theta$: $\forall G=\sum_{i=1}^{k}p_i\delta_{\theta_i}, G'=\sum_{i=1}^{k'}p'_i\delta_{\theta'_i} 
$, define
\begin{equation}
W_p(G,G';\dtheta) = \left(\min_{\bm{q}} \sum_{i=1}^{k}\sum_{j=1}^{k'} q_{ij}\dtheta^p(\theta_i,\theta'_j)\right)^{1/p}, \label{eqn:Wpdef}
\end{equation}
where the infimum is taken over all joint probability distributions $\bm{q}$ on $[k]\times [k']$ such that, when expressing $\bm{q}$ as a $k\times k'$ matrix, the marginal constraints hold: $\sum_{j=1}^{k'} q_{ij} = p_i$ and $\sum_{i=1}^{k} q_{ij} = p'_j$. 
For the special case when $\dtheta$ is the Euclidean distance, write simply $W_p(G,G')$ instead of $W_p(G,G';\dtheta)$. Write $G_{\ell}\overset{W_p}{\to}G$ if $G_{\ell}$ converges to $G$ under the $W_p$ distance w.r.t. the Euclidean distance on $\Theta$.

For mixing measures arising in mixtures of $\m$-product distributions, a more useful notion is the following. For any $G=\sum_{i=1}^kp_i\delta_{\theta_i} \in \Ec_k(\Theta)$ and $G'=\sum_{i=1}^kp'_i\delta_{\theta'_i} \in \Ec_k(\Theta)$, define
 $$
 \myD_{\m}(G,G') =  \min_{\tau\in S_{k}} \sum_{i=1}^{k}(\sqrt{\m}\|\theta_{\tau(i)}-\theta'_i\|_2+|p_{\tau(i)}-p'_i|) 
 $$ 
 where $S_{k}$ denote all the permutations on the set $[k]$.
 It is simple to verify that $\myD_{\m}(\cdot,\cdot)$ 
is a valid metric on $\Ecal_{k}(\Theta)$ for each $\m$ and relate it to a suitable optimal transport distance metric. Indeed, $G = \sum_{i=1}^{k}p_i \delta_{\theta_i} \in \Ec_k(\Theta)$, due to the permutations invariance of its atoms, can be identified as a set $\{ (\theta_i,p_i): 1\leq i \leq k\}$, which can further be identified as $\tilde{G} = \sum_{i=1}^{k}\frac{1}{k} \delta_{(\theta_i,p_i)}\in \Ec_{k}(  \Theta \times \R)$. Formally, we define a map $\Ec_k(\Theta)\to \Ec_k( \Theta\times \R) $ by
\begin{equation}
G = \sum_{i=1}^{k}p_i \delta_{\theta_i} \mapsto \tilde{G} = \sum_{i=1}^{k}\frac{1}{k} \delta_{(\theta_i,p_i)}\in \Ec_{k}( \Theta \times \R). \label{eqn:a}
\end{equation}
Now, endow $ \Theta \times \R$ with a metric $M_{\m}$ defined by $M_{\m}((\theta,p), (\theta',p')) = \sqrt{\m}\|\theta-\theta'\|_2+|p-p'|$ and note the following fact.

\begin{lem}\label{lem:rotationinvariance}
	For any $ G=\sum_{i=1}^k\frac{1}{k}\delta_{\bar{\theta}_i},G'=\sum_{i=1}^k\frac{1}{k}\delta_{\bar{\theta}'_i} \in \Ec_k(\bar{\Theta})$ and distance $d_{\bar{\Theta}}$ on $\bar{\Theta}$, 
	$$W_p^p(G,G'; d_{\bar{\Theta}}) =  \min_{\tau\in S_{k}} \frac{1}{k}\sum_{i=1}^k d_{\bar{\Theta}}^p(\theta_i, \theta'_{\tau(i)}).$$
\end{lem}

A proof of the preceding lemma is available as Proposition 2 in \cite{nguyen2011wasserstein}. By applying Lemma \ref{lem:rotationinvariance} with $\bar{\Theta}$, $d_{\bar{\Theta}}$ replaced respectively by $\Theta\times \R$ and  $M_\m$, then for any $ G,G' \in \Ec_k(\Theta)$, $W_1(\tilde{G},\tilde{G'}; M_{\m})= \frac{1}{k} D_{\m}(G,G') $, which validates that $D_{\m} $ is indeed a metric on $\Ec_k(\Theta)$, and moreover it does not depend on the specific representations of $G$ and $G'$. 

The next lemma establishes the relationship between $D_{\m}$ and $W_1$ on $\Ec_k(\Theta)$. 
\begin{lem}\label{lem:relW1D1}
The following statements hold.
	\begin{enumerate}[label=\alph*)]
		\item A sequence $G_n\in \Ec_k(\Theta)$ converges to $G_0\in \Ec_k(\Theta)$ under $W_p$ if and only if  $G_n$ converges to $G_0$ under $D_N$. That is, $W_p$ and $D_N$ generate the same topology.
		\item \label{item:relW1D1b}
		Let $\Theta$ be bounded.  Then $W_1(G,G')\leq \max\left\{1,\frac{\text{diam}(\Theta)}{2}\right\} D_1(G,G')$ for any $G,G'  \in \Ec_k(\Theta)$. \\ 
		More generally for any $G=\sum_{i=1}^kp_i\delta_{\theta_i}$ and $G'=\sum_{i=1}^kp'_i\delta_{\theta'_i}$,
		$$ W_p^p(G,G')\leq \max\left\{1,\frac{\text{diam}^p(\Theta)}{2}\right\} \min_{\tau\in S_k} \sum_{i=1}^{k}\left(\|\theta_{\tau(i)}-\theta'_{i}\|_2^p + |p_{\tau(i)}-p'_i|  \right). $$ 
		
		
		\item 
		\label{item:relW1D1d}  
		Fix $G_0\in \Ec_k(\Theta)$. Then $\liminf\limits_{\substack{G \overset{W_1}{\to} G_0\\G\in \Ec_k(\Theta)}} \frac{W_1(G,G_0)}{D_1(G,G_0)}>0$ and $\liminf\limits_{\substack{G \overset{W_1}{\to} G_0\\G\in \Ec_k(\Theta)}} \frac{D_1(G,G_0)}{W_1(G,G_0)}>0$. That is, in a neighborhood of $G_0$ in $\Ec_k(\Theta)$, $D_1(G,G_0)\asymp_{G_0} W_1(G,G_0) $.
		
		\item 
		\label{item:relW1D1e}	   
		Fix $G_0\in \Ec_k(\Theta)$ and suppose $\Theta$ is bounded. Then $W_1(G,G_0)\geq C(G_0,\text{diam}(\Theta))D_1(G,G_0)$ for any $G\in  \Ec_{k}(\Theta)$, where constant $C(G_0,\text{diam}(\Theta))>0$ depends on $G_0$ and $\text{diam}(\Theta)$.
	\end{enumerate}
	
\end{lem}

We see that $W_1$ and $D_1$ generate the same topology on $\Ec_k(\Theta)$, and they are equivalent 
while fixing one argument. The benefit of $W_p$ is that it is defined on $\bigcup_{k=1}^\infty \Ec_k(\Theta)$ while $D_\m$ is only defined on $\Ec_k(\Theta)$ for each $k$ since its definition requires the two arguments have the same number of atoms. $D_\m$ allows us to quantify the distinct convergence behavior for atoms and probability mass, by placing different factors 
on the atoms and the probability mass parameters, while $W_p$ on $\bigcup_{k=1}^\infty \Ec_{k}(\Theta)$ would fail to do so, because $W_p$ couples the atoms and probability mass parameters (see Example \ref{exa:wassersteinandDN} for such an attempt).

The factor $\sqrt{N}$ present in the definition of $D_N$ arises from the anticipation that when we have independent exchangeable sequences of length $N$, the dependence of the standard estimation rate on $N$ for component parameters $\theta$ will be of order $1/\sqrt{N}$. Indeed, given one single exchangeable sequence from some component parameter $\theta_i$, as the coordinates in this sequence are conditionally independent and identically distributed, the standard rate for estimating $\theta_i$ is $1/\sqrt{N}$. On the other hand, the mixing proportions parameters $p_i$ cannot be estimated from a single such sequence (i.e., if $m=1$). One expects that for such parameters the number of sequences coming from the $\theta_i$ among all exchangeable sequences plays a more important role.  
In summary, the distance $D_N$ will be used to capture precisely the distinct convergence behavior due to the length $N$ of observed exchangeable sequences.

	\section{First-order identifiability theory}\label{sec:firstidentifiable}
	
	Let $\m=1$, a finite mixture of $\m$-product distributions is reduced to a standard finite mixture of distributions. Mixture components are modeled by a family of probability kernels $\{\P_{\thetave}\}_{\thetave\in \Theta}$ on $\Xfrak$, where $\thetave$ is the parameter of the family and $\Theta\subset \R^q$ is the parameter space. As discussed in the introduction, throughout the paper we assume that the map $\theta \mapsto \P_\theta$ is injective;
	it is the nature of the map $G \mapsto P_G$ that we are after. 
	Within this section, we further assume that $\{\P_{\thetave}\}_{\thetave\in \Theta}$ has density $\{f(x|\theta)\}_{\thetave\in \Theta}$ w.r.t. a dominating measure $\mu$ on $(\Xfrak,\Ac)$. Combining multiple mixture components using a mixing measure $G$ on $\Theta$ results in the finite mixture distribution, which admits the following density with respect to $\mu$: 
	$p_G(x) = \int f(x|\theta) G(d\theta)$. The goal of this section is to provide a concise and self-contained treatment of identifiability of finite mixture models. We lay down basic foundations and present new results that will prove useful for the general theory of mixtures of product distributions to be developed in the subsequent sections.

\subsection{Basic theory}	
\label{subsec:basictheory}

The classical identifiability condition posits that $P_G$ uniquely identifies $G$ for all $G\in \Ec_{k_0}(\Theta)$. This condition is satisfied if the collection of density functions $\{f(x|\theta)\}_{\theta \in \Theta}$ are linearly independent. To obtain rates of convergence for the model parameters, it is natural to consider the following condition concerning the first-order derivative of $f$ with respect to $\theta$.
	
	\begin{definition}
	\label{def:firstident}
	The family $\{f(x|\theta)\}_{\thetave\in \Theta}$ is $(\{\theta_i\}_{i=1}^k, \Nc)$ \textbf{first-order identifiable}  if 
	\begin{enumerate}[label=(\roman*)]
	    \item \label{item:firstidenti}
	    for every $x$ in the $\mu$-positive subset $\Xfrak\backslash \Nc$ where $\Nc\in \mathcal{A}$, $f(x|\theta)$ is first-order differentiable w.r.t. $\theta$ at $\{\theta_i\}_{i=1}^k$; and 
	    \item  \label{item:firstidentii}
	    $\{\theta_i\}_{i=1}^k\subset \Theta^\circ$ is a set of $k$ distinct elements and the system of two equations with variable $(a_1,b_1,\ldots,a_k,b_k)$:
	\begin{subequations}
	\begin{align}
	\sum_{i=1}^k \left(a_i^\top  \nabla_{\theta} f(x|\theta_i)+b_i f(x|\theta_i) \right) &=0, \quad \mu-a.e.\  x\in \Xfrak \backslash \Nc,  \label{eqn:conlininda}\\
	  \sum_{i=1}^{k}b_i & =0  \label{eqn:conlinindb}
	\end{align} 
	\end{subequations}
	has only the zero solution: 
	$b_i =0 \in \R \text{ and } a_i =\bm{0} \in \R^q, \quad \forall 1\leq i\leq k$.
 	\end{enumerate}
 	\end{definition}
 	
	This definition specifies a condition that is weaker than the definition of \emph{identifiable in the first-order} in \cite{ho2016strong} since it only requires $f(x|\theta)$ to be differentiable at a finite number of points $\{\theta_i\}_{i=1}^k$ and it requires linear independence of functions at those points. 
	Moreover, it does \emph{not} require $f(x|\theta)$ as a function of $\theta$ to be differentiable for $\mu$-a.e. $x$.  Our definition requires only linear independence between the density and its derivative w.r.t. the parameter over the constraints of the coefficients specified by \eqref{eqn:conlinindb}. (Having said that, we are not aware of any simple example that differentiates the \eqref{eqn:conlininda} \eqref{eqn:conlinindb} from \eqref{eqn:conlininda}. Actually, it is established in Lemma \ref{lem:firstidentifiableelaborate} \ref{item:firstidentifiableelaborateb}  that: under some regularity condition, \eqref{eqn:conlininda} \eqref{eqn:conlinindb} has the same solution set as \eqref{eqn:conlininda}.) We will see shortly that in a precise sense that the conditions given Definition~\ref{def:firstident} are also necessary.

	The significance of first-order identifiability conditions is that they entail a collection of inverse bounds that relate the behavior of some form of distances on mixture densities $P_G, P_{G_0}$ to a distance between corresponding parameters described by $D_1(G,G_0)$, as $G$ tends toward $G_0$. Denote $\Theta^\circ$ the interior of $\Theta$. For any $G_0 \in \Ec_{k_0}(\Theta^\circ)$, define 
	\begin{equation}
	B_{W_1}(G_0,r) = \biggr \{ G\in \bigcup_{k=1}^{k_0}\Ec_{k}(\Theta) \biggr | W_1(G,G_0)<r \biggr \}. \label{eqn:BW1def}
	\end{equation}
	It is obvious that $B_{W_1}(G_0,r)\subset \Ec_{k_0}(\Theta)$ for small $r$.
	
	\begin{lem} [Consequence of first-order identifiability] \label{lem:firstidentifiable}
	Let $G_0=\sum_{i=1}^{k_0}p_i^0\delta_{\theta_i^0} \in \Ec_{k_0}(\Theta^\circ)$.
	Suppose that the family $\{f(x|\theta)\}_{\theta \in \Theta}$ is $(\{\theta_i^0\}_{i=1}^{k_0}, \Nc)$ first-order identifiable in the sense of Definition~\ref{def:firstident} for some  $\Nc\in \Ac$. 
	\begin{enumerate}[label=\alph*)]
	\item  \label{item:firstidentifiablea}
	Then 
	    \begin{equation}
	    \liminf_{\substack{G\overset{W_1}{\to} G_0\\ G\in \Ec_{k_0}(\Theta)}} \frac{V(P_G,P_{G_0})}{D_1(G,G_0)} > 0. \label{eqn:noproductlowbou} 
	    \end{equation}
	    \item \label{item:firstidentifiableb}
	    If in addition,  for every $x$ in $\Xfrak\backslash \Nc$  $f(x|\theta)$ is continuously differentiable w.r.t. $\theta$ in a neighborhood of  $\theta_i^0$ for $i\in [k_0]
	    :=\{1,2,\ldots,k_0\}$, then
	    \begin{equation}
	     \lim_{r\to 0}\ \inf_{\substack{G, H\in B_{W_1}(G_0,r)\\ G\not = H }} \frac{V(\P_{G},\P_{H})}{\myD_{1}(G,H)}>0. \label{eqn:curvaturebound} 
	    \end{equation}
	\end{enumerate}
	\end{lem}
To put the above claims in context, note that the following inequality holds generally for \emph{any} probability kernel family $\{P_\theta\}_{\theta \in \Theta}$ (even those without a density w.r.t. a dominating measure, see Lemma~\ref{lem:VD1liminfuppbou}):
	\begin{equation}
\sup_{G_0\in \Ec_{k_0}(\Theta)} \liminf_{\substack{G\overset{W_1}{\to} G_0\\ G\in \Ecal_{k_0}(\Theta) }} \frac{V(\P_{G},\P_{G_0})}{\myD_{1}(G,G_0)} \leq 1/2. \label{eqn:liminfVDNup}
\end{equation}
Note also that
\begin{equation}
\lim_{r\to 0}\ \inf_{\substack{G, H\in B_{W_1}(G_0,r)\\ G\not = H }} \frac{V(\P_{G},\P_{H})}{\myD_{1}(G,H)} \leq \liminf_{\substack{G\overset{W_1}{\to} G_0\\ G\in \Ec_{k_0}(\Theta)}} \frac{V(P_G,P_{G_0})}{D_1(G,G_0)}  \label{eqn:reltwoinversebounds}
\end{equation}
for any probability kernel $P_\theta$ and any $G_0\in \Ec_{k_0}(\Theta^\circ)$. Thus ~\eqref{eqn:curvaturebound} entails~\eqref{eqn:noproductlowbou}. However,~\eqref{eqn:noproductlowbou} is sufficient for translating a learning rate for estimating a population distribution $P_G$ into that of the corresponding mixing measure $G$. To be concrete, if we are given an $\n$-i.i.d. sample from a parametric model $P_{G_0}$, a standard estimation method yields root-$\n$ rate of convergence for density $p_G$, which means that the corresponding estimate of $G$ admits root-$\n$ rate as well. 
	\begin{rem}
	\label{rem:generalizationoffirstidentifiable}
	Lemma \ref{lem:firstidentifiable} \ref{item:firstidentifiablea}  is a generalization of the Theorem 3.1 in \cite{ho2016strong} in several features. Firstly, $(\{\theta_i^0\}_{i=1}^{k_0},\Nc)$ first-order identifiable assumption in Lemma \ref{lem:firstidentifiable} is weaker since identifiability in the first-order in the sense of \cite{ho2016strong} implies $(\{\theta_i^0\}_{i=1}^{k_0},\Nc)$ first-order identifiability with $\Nc=\emptyset$. Example~\ref{exa:locgamma} gives a specific instance which satisfies the notion of first-order identifiability specified by Definition~\ref{def:firstident} but not the condition specified by~\cite{ho2016strong}. Secondly, it turns out that uniform Lipschitz assumption in Theorem 3.1 in \cite{ho2016strong} is redundant and 
	Lemma \ref{lem:firstidentifiable} \ref{item:firstidentifiablea} does not require it. 
    Lemma \ref{lem:firstidentifiable} \ref{item:firstidentifiableb} is an extension of \cite[equation (20)]{heinrich2018strong} in a similar sense as the above. 
    Finally, given additional features of $f$, the first-order identifiable notion can be further simplified (see Section~\ref{sec:appendixfinercharacterizations}).  \myeoe
	\end{rem}

\begin{proof}[Proof of Lemma \ref{lem:firstidentifiable}] 

Suppose the lower bound of \eqref{eqn:noproductlowbou} is incorrect. Then there exist $G_\ell\in \Ec_{k_0}(\Theta)\backslash \{G_0\}$, $G_\ell \overset{W_1}{\to} G_0$ such that 
	    $$
	    \frac{V(p_{G_\ell},p_{G_0})}{D_1(G_\ell,G_0)} \to 0, \text{ as } \ell \to \infty.
	    $$
	    We may write $G_\ell= \sum_{i=1}^{k_0}p^\ell_i \delta_{\theta_i^\ell}$ such that $\theta_i^\ell \to \theta_i^0$ and $p_i^\ell \to p_i^0$ as $\ell\to \infty$. With subsequences argument if necessary, we may further require
	    \begin{equation}
	    \frac{\theta_i^\ell - \theta_i^0}{D_1(G_\ell,G_0)} \to a_i \in \R^q, \quad \frac{p_i^\ell - p_i^0}{D_1(G_\ell,G_0)} \to b_i\in \R,  \quad \forall 1\leq i \leq k_0, \label{eqn:thetaGlG0rel}
	    \end{equation}
	    where $b_i$ and the components of $a_i$ are in $[-1,1]$ and $\sum_{i=1}^{k_0}b_i =0$. Moreover, $D_1(G_\ell,G_0)=\sum_{i=1}^{k_0}\left(\|\theta^\ell_i-\theta_i^0\|_2 +|p^\ell_i-p_i^0|\right)$ for sufficiently large $\ell$, which implies 
	    \begin{equation*}
	    \sum_{i=1}^{k_0}\|a_i\|_2+\sum_{i=1}^{k_0}|b_i| = 1. 
	    \end{equation*} It also follows that at least one of $a_i$ is not $\bm{0}\in \R^q$ or one of $b_i$ is not $0$.   
	    On the other hand,
	    \begin{align*}
	    0 & =\lim_{\ell\to \infty}\frac{2V(P_{G_\ell},P_{G_0})}{D_1(G_\ell,G_0)} \\ 
	    &\geq \lim_{\ell\to \infty}\int_{\Xfrak\backslash \Nc}\left| \sum_{i=1}^{k_0} p_i^\ell \frac{f(x|\theta_i^\ell) - f(x|\theta_i^0)}{D_1(G_\ell,G_0)} + \sum_{i=1}^{k_0} f(x|\theta_i^0) \frac{p_i^\ell - p_i^0}{D_1(G_\ell,G_0)}   \right| \mu(dx)\\
	    & \geq \int_{\Xfrak\backslash \Nc} \liminf_{\ell\to \infty}\left| \sum_{i=1}^{k_0} p_i^\ell \frac{f(x|\theta_i^\ell) - f(x|\theta_i^0)}{D_1(G_\ell,G_0)} + \sum_{i=1}^{k_0} f(x|\theta_i^0) \frac{p_i^\ell - p_i^0}{D_1(G_\ell,G_0)}   \right| \mu(dx) \\
	    & = \int_{\Xfrak\backslash \Nc} \left| \sum_{i=1}^{k_0} p_i^0 a_i^\top  \nabla_{\theta}f(x|\theta_i^0) + \sum_{i=1}^{k_0} f(x|\theta_i^0) b_i   \right| \mu(dx) .
	    \end{align*}
	    where the second inequality follows from Fatou's Lemma. Then $\sum_{i=1}^{k_0} p_i^0 a_i^\top  \nabla_{\theta}f(x|\theta_i^0) + \sum_{i=1}^{k_0} f(x|\theta_i^0) b_i=0$ for $\mu-a.e. x\in \Xfrak\backslash \Nc$. Thus we find a nonzero solution to \eqref{eqn:conlininda}, \eqref{eqn:conlinindb} with $k,\theta_i$ replaced by $k_0,\theta_i^0$. 
	    However, the last statement contradicts with the definition of $(\{\theta_i^0\}_{i=1}^{k_0},\Nc)$ first-order identifiable.

	    Proof of part~\ref{item:firstidentifiableb} continues in the Appendix.
	\end{proof}

	Lemma \ref{lem:firstidentifiable} states that under \ref{item:firstidenti} in Definition~\ref{def:firstident}, the constrained linear independence between the density and its derivative w.r.t. the parameter (item \ref{item:firstidentii} in the definition) is sufficient for \eqref{eqn:noproductlowbou} and \eqref{eqn:curvaturebound}. For a converse result, the next lemma shows \ref{item:firstidentii} is also necessary provided that \ref{item:firstidenti} holds for some $\mu$-negligible $\Nc$ and  $f(x|\theta)$ satisfies some regularity condition.

	\begin{lem} [Lack of first-order identifiability] \label{lem:nececondition}
	    Fix $G_0=\sum_{i=1}^{k_0}p_i^0\delta_{\theta_i^0}\in \Ec_{k_0}(\Theta^\circ)$. Suppose 
	    \begin{enumerate}[label=\alph*)]
	    \item \label{item:nececonditiona} 
	    there exists $\Nc$ (that possibly depends on $G_0$) such that $\mu(\Nc)=0$ and for every $x\not \in \Nc$,  $f(x|\theta)$ is differentiable with respect to $\theta$ at $\{\theta_i^0\}_{i=1}^{k_0}$;
	    \item \label{item:nececonditionb} 
	    equation \eqref{eqn:conlininda} (or equivalently, system of equations~\eqref{eqn:conlininda} and~\eqref{eqn:conlinindb}) with $k,\theta_i$ replaced respectively by $k_0,\theta_i^0$ has a nonzero solution $(a_1, b_1,\ldots,a_{k_0},b_{k_0})$;
	    \item  \label{item:nececonditionc}
	    for each $i\in [k_0]$, there exists $\gamma(\theta_i^0,a_i)$ such that for any $0<\Delta\leq\gamma(\theta_i^0,a_i)$
	    $$
	    \left|\frac{f(x|\theta_i^0+a_i \Delta)-f(x|\theta_i^0)}{\Delta}\right| \leq \bar{f}(x|\theta_i^0,a_i),  \quad \mu-a.e.\  x\in \Xfrak,
	    $$
	    where $\bar{f}(x|\theta_i^0,a_i)$ is integrable with respect to the measure $\mu$.
	    \end{enumerate}
	    Then 
	    \begin{equation}
	   \lim_{r\to 0}\ \inf_{\substack{G, H\in B_{W_1}(G_0,r)\\ G\not = H }} \frac{V(\P_{G},\P_{H})}{\myD_{1}(G,H)} =   \liminf_{\substack{G\overset{W_1}{\to} G_0\\ G\in \Ec_{k_0}(\Theta)}} \frac{V(P_G,P_{G_0})}{D_1(G,G_0)}=0.  \label{eqn:liminfD10}
	    \end{equation}
	\end{lem}

Lemma \ref{lem:nececondition} presents the consequence of the violation of first-order identifiability. Indeed, the conclusion~\eqref{eqn:liminfD10} suggests that $D_1(G,G_0)$ may vanish at a much slower rate than $V(P_G,P_{G_0})$, i.e., the convergence of parameters representing $G$  may be much slower than the convergence of data distribution $P_G$.

\begin{rem} \label{rem:weakerprattslemma1} 
    Condition \ref{item:nececonditionc} in the Lemma \ref{lem:nececondition} is to guarantee the exchange of the order between the limit and the integral and one may replace it by any other similar condition. A byproduct of this condition is that it renders the constraint~\eqref{eqn:conlinindb} redundant (see Lemma \ref{lem:firstidentifiableelaborate} \ref{item:firstidentifiableelaborateb}). While condition \ref{item:nececonditionc} is tailored for an application of the dominated convergence theorem in the proof, one may tailored the following condition for Pratt's Lemma. 
    
    Condition c'): there exists $\gamma_0>0$ such that  $\forall\ 1\leq i\leq k_0$, $\forall\ 0<\Delta<\gamma_0$,
	    $$
	    \left|\frac{f(x|\theta_i^0+a_i \Delta)-f(x|\theta_i^0)}{\Delta}\right| \leq \bar{f}_{\Delta}(x),  \quad \mu-a.e.\  x\in \Xfrak \backslash \Nc
	    $$
	    where $\bar{f}_{\Delta}(x)$ satisfies $\lim_{\Delta\to 0^+}\int_{\Xfrak\backslash \Nc} \bar{f}_{\Delta}(x) d\mu = \int_{\Xfrak\backslash \Nc} \lim_{\Delta\to 0^+} \bar{f}_{\Delta}(x) d\mu $. 
	  
	  Condition c') is weaker than condition c) since the former reduces to the latter if one let $\bar{f}_{\Delta}(x)=\bar{f}(x)<\infty$. 
	  \myeoe
\end{rem}

Combining all the conditions in Lemma \ref{lem:firstidentifiable} and Lemma \ref{lem:nececondition}, one immediately obtains the following equivalence between \eqref{eqn:noproductlowbou},  \eqref{eqn:curvaturebound} and the first-order identifiable condition. 

\begin{cor}\label{cor:curvaturenoncurvequiv} 
    Fix $G_0=\sum_{i=1}^{k_0}p_i^0\delta_{\theta_i^0}\in \Ec_{k_0}(\Theta^\circ)$. Suppose for $\mu$-a.e. $x\in \Xfrak$, $f(x|\theta)$ as a function $\theta$ is continuously differentiable in a neighborhood of $\theta_i^0$ for each $i\in [k_0]$. Suppose that for any $a\in \R^q$ and for each $i\in [k_0]$ there exists $\gamma(\theta_i^0,a)>0$ such that for any $0<\Delta\leq\gamma(\theta_i^0,a)$,
\begin{equation}
\left|\frac{f(x|\theta_i^0+a\Delta)-f(x|\theta_i^0)}{\Delta}\right|\leq \bar{f}_\Delta(x|\theta_i^0,a) \quad \mu-a.e.\ \Xfrak \label{eqn:prattincor}
\end{equation}
where $\bar{f}_\Delta(x|\theta_i^0,a)$ satisfies $\lim_{\Delta\to 0^+}\int_{\Xfrak}\bar{f}_\Delta(x|\theta_i^0,a)d\mu=\int_{\Xfrak}\lim_{\Delta\to0^+}\bar{f}_\Delta(x|\theta_i^0,a)d\mu$. Here $\bar{f}_\Delta(x|\theta_i^0,a)$ possibly depends on $\theta_i^0$ and $a$. Then \eqref{eqn:curvaturebound} holds if and only if \eqref{eqn:noproductlowbou} holds if and only if  \eqref{eqn:conlininda} with $k,\theta_i$ replaced respectively by $k_0,\theta_i^0$ has only the zero solution.
\end{cor}

Next, we highlight the role of condition \ref{item:nececonditionc} of Lemma \ref{lem:nececondition} in establishing either inverse bound \eqref{eqn:noproductlowbou} or
\eqref{eqn:liminfD10} based on our notion of first-order identifiability. As mentioned, condition \ref{item:nececonditionc}  posits the existence of an integrable envelope function to ensure the exchange of the limit and integral. Without this condition, the conclusion \eqref{eqn:liminfD10} of Lemma \ref{lem:nececondition} might not hold. 
The following two examples demonstrate the role of \ref{item:nececonditionc}, and serve as examples which are not first-order identifiable but for which inverse bound \eqref{eqn:noproductlowbou} still holds.
	
\begin{exa}[Uniform probability kernel]\label{exa:uniform}  
Consider the uniform distribution family $f(x|\theta)=\frac{1}{\theta}\bm{1}_{(0,\theta)}(x)$ with parameter space $\Theta = (0,\infty)$. This family is defined on $\Xfrak = \R$ with the dominating measure $\mu$ to be the Lebesgue measure. It is easy to see $f(x|\theta)$ is differentiable w.r.t. $\theta$ at $\theta\not = x$ and 
$$
\frac{\partial}{\partial \theta}f(x|\theta) = -\frac{1}{\theta} f(x|\theta)  \quad \text{when } \theta\not =x.
$$
So $f(x|\theta)$ is not first-order identifiable by our definition. Note for any $G_0\in \Ec_{k_0}(\Theta)$ this family does not satisfy the assumption \ref{item:nececonditionc} in Lemma \ref{lem:nececondition} and hence Lemma \ref{lem:nececondition} is not applicable. Indeed, by Lemma \ref{lem:unifidentifiable} this family satisfies \eqref{eqn:noproductlowbou} and \eqref{eqn:curvaturebound} for any $k_0$ and $G_0\in \Ec_{k_0}(\Theta)$. \myeoe
\end{exa}
	
\begin{lem} \label{lem:unifidentifiable}
    Let $f(x|\theta)$ be the uniform distribution family defined in Example \ref{exa:uniform}. Then for any $G_0\in \Ec_{k_0}(\Theta)$, inverse bounds~\eqref{eqn:noproductlowbou} and \eqref{eqn:curvaturebound} hold. 
\end{lem}

\begin{exa}[Location-scale exponential distribution kernel]\label{exa:locscaexp} 
Consider the location-scale exponential distribution on $\Xfrak = \R$, with density with respect to Lebesgue measure $\mu$  given by $f(x|\xi,\sigma)=\frac{1}{\sigma}\exp\left(-\frac{x-\xi}{\sigma}\right)\bm{1}_{(\xi,\infty)}(x)$ with parameter $\theta=(\xi,\sigma)$ and parameter space $\Theta = \R\times (0,\infty)$. It is easy to see $f(x|\xi,\sigma)$ is differentiable w.r.t. $\xi$ at $\xi\not = x$ and 
$$
\frac{\partial}{\partial \xi}f(x|\xi,\sigma) = \frac{1}{\sigma} f(x|\xi,\sigma)  \quad \text{when } \xi \not =x.
$$
So $f(x|\xi,\sigma)$ is not first-order identifiable. Note for any $G_0\in \Ec_{k_0}(\Theta)$ this family does not satisfy the third assumption in Lemma \ref{lem:nececondition} and hence Lemma \ref{lem:nececondition} is not applicable.
Indeed by Lemma \ref{lem:exadisidentifiable} this family satisfies \eqref{eqn:noproductlowbou} for any $k_0$ and $G_0\in \Ec_{k_0}(\Theta)$.
This lemma also serves as a correction for an erroneous result (Prop. 5.3 of \cite{ho2016convergence}). The mistake in their proof may be attributed to failing to account for the envelope condition \ref{item:nececonditionc} that arises due to shifted support of mixture components with distinct $\xi$ values. Interestingly, Lemma \ref{lem:exadisidentifiable} also establishes that the stronger version of inverse bounds, namely, inequality \eqref{eqn:curvaturebound} does not hold for some $G_0$.
\myeoe
\end{exa}

\begin{lem} \label{lem:exadisidentifiable}
    Let $f(x|\xi,\sigma)$ be the location-scale exponential distribution defined in Example \ref{exa:locscaexp}. Then for any $G_0\in \Ec_{k_0}(\Theta)$, inverse bound~\eqref{eqn:noproductlowbou} holds. Moreover, for any $k_0\geq 1$, there exists a $G_0\in \Ec_{k_0}(\Theta)$, such that inverse bound \eqref{eqn:curvaturebound} does not hold.
\end{lem}

In some context it is of interest to establish inverse bounds for Hellinger distance rather than variational distance on mixture densities, e.g., in the derivation of minimax lower bounds. Since $\sqrt{2} h \geq V$, the inverse bound \eqref{eqn:noproductlowbou}, which holds under first-order identifiability, immediately entails that
 $$
	    \liminf\limits_{\substack{G\overset{W_1}{\to} G_0\\ G\in \Ec_{k_0}(\Theta)}} \frac{h(P_G,P_{G_0})}{D_1(G,G_0)}>0.  
$$
Similarly, \eqref{eqn:curvaturebound} entails that 
$$
\lim\limits_{r\to 0}\ \inf\limits_{\substack{G, H\in B_{W_1}(G_0,r)\\ G\not = H }} \frac{h(\P_{G},\P_{H})}{\myD_{1}(G,H)}>0.
$$
For a converse result, the following is the Hellinger counterpart of Lemma~\ref{lem:nececondition}.
	\begin{lem} \label{lem:nececonditionh}  
	    Fix $G_0=\sum_{i=1}^{k_0}p_i^0\delta_{\theta_i^0}\in \Ec_{k_0}(\Theta^\circ)$. Suppose that
	    \begin{enumerate}[label=\alph*)]
	    \item there exists $\Nc$ (that possibly depends on $G_0$) such that $\mu(\Nc)=0$ and for every $x\not \in \Nc$,  $f(x|\theta)$ is differentiable with respect to $\theta$ at $\{\theta_i^0\}_{i=1}^{k_0}$;
	    \item the density family has common support, i.e. $S= \{x\in \Xfrak| f(x|\theta)>0 \}$ does not depend on $\theta\in \Theta$; 

	    \item \eqref{eqn:conlininda} with $k,\theta_i$ replaced respectively by $k_0,\theta_i^0$ has a nonzero solution $(a_1, b_1,\ldots,a_{k_0},b_{k_0})$;
	    
	    \item  \label{item:nececonditionhd}
	    there exists $\gamma_0>0$ such that  $\forall\ 1\leq i\leq k_0$, $\forall\ 0<\Delta\leq\gamma_0$,
	    $$
	    \left|\frac{f(x|\theta_i^0+a_i \Delta)-f(x|\theta_i^0)}{\Delta \sqrt{f(x|\theta_i^0)}}\right| \leq \bar{f}(x),  \quad \mu-a.e.\  x\in S \backslash \Nc,
	    $$
	    where $\bar{f}(x)$ satisfies $\int_{S \backslash \Nc}\bar{f}^2(x)d\mu<\infty$.
	    \end{enumerate}
	    Then 
	    \begin{equation}
	    \lim_{r\to 0}\ \inf_{\substack{G, H\in B_{W_1}(G_0,r)\\ G\not = H }} \frac{h(\P_{G},\P_{H})}{\myD_{1}(G,H)}=\liminf_{\substack{G\overset{W_1}{\to} G_0\\ G\in \Ec_{k_0}(\Theta)}} \frac{h(P_G,P_{G_0})}{D_1(G,G_0)}=0 \label{eqn:liminfD10h}
	    \end{equation}
	\end{lem}
\begin{rem}\label{rem:weakerprattslemma2} 
Similar to Remark \ref{rem:weakerprattslemma1}, one may replace the condition \ref{item:nececonditionhd} in the preceding lemma by the following weaker condition: 

Condition d'): there exist $\gamma_0>0$ such that  $\forall\ 1\leq i\leq k_0$, $\forall\ 0<\Delta\leq\gamma_0$,
	    $$
	    \left|\frac{f(x|\theta_i^0+a_i \Delta)-f(x|\theta_i^0)}{\Delta \sqrt{f(x|\theta_i^0)}}\right| \leq \bar{f}_{\Delta}(x),  \quad \mu-a.e.\  x\in S \backslash \Nc,
	    $$
	    where $\bar{f}_{\Delta}(x)$ satisfies $\lim_{\Delta\to0^+}\int_{S \backslash \Nc}\bar{f}_{\Delta}^2(x)d\mu = \int_{S \backslash \Nc}\lim_{\Delta\to0^+}\bar{f}_{\Delta}^2(x)d\mu <\infty$. \myeoe
\end{rem}

\subsection{Finer characterizations}
\label{sec:appendixfinercharacterizations}


In order to verify if the first-order identifiability condition is satisfied for a given probability kernel family $\{f(x|\theta)|\theta\in \Theta\}$, according to Definition~\ref{def:firstident} one needs to check that system of equations~\eqref{eqn:conlininda} and~\eqref{eqn:conlinindb} does not have non-zero solutions.
For many common probability kernel families, the presence of normalizing constant can make this verification challenging, because the normalizing constant is a function of $\theta$, which has a complicated form or no closed form, and its derivative can also be complicated.  Fortunately, the following lemma shows that under a mild condition one only needs to check for the family of kernel $\{f(x|\theta)\}$ defined up to a function of $\theta$ that is constant in $x$. Moreover, under additional mild assumptions, the equation~\eqref{eqn:conlinindb} can also be dropped from the verification.  

	
\begin{lem}\label{lem:firstidentifiableelaborate}
Suppose for every $x$ in the $\mu$-positive subset $\Xfrak\backslash \Nc$ for some $\Nc\in \mathcal{A}$, $f(x|\theta)$ is differentiable with respect to $\theta$ at $\{\theta_i\}_{i=1}^k$. Let $g(\theta)$ be a positive differentiable function on $\Theta^\circ$ and define $\tilde{f}(x|\theta) = g(\theta)f(x|\theta)$. 

\begin{enumerate}[label=\alph*)]
\item \label{item:firstidentifiableelaboratea}
\eqref{eqn:conlininda} has only the zero solution if and only if \eqref{eqn:conlininda} with $f$ replaced by $\tilde{f}$ has only the zero solution.

\item \label{item:firstidentifiableelaborateb}
Suppose $\mu(\Nc)=0$. For a fixed set $\{a_i\}_{i=1}^k\subset \R^q$ and for each $i\in [k]$ there exists $\gamma(\theta_i,a_i)>0$ such that for any $0<\Delta\leq\gamma(\theta_i,a_i)$,
\begin{equation}
\left|\frac{f(x|\theta_i+a_i\Delta)-f(x|\theta_i)}{\Delta}\right|\leq \bar{f}(x|\theta_i,a_i), \quad \mu-a.e.\ \Xfrak, \label{eqn:dctthetaiai2}
\end{equation}
where $\bar{f}(x|\theta_i,a_i)$ is $\mu$-integrable. Here $\gamma(\theta_i,a_i)$ and $\bar{f}(x|\theta_i,a_i)$ depend on $\theta_i$ and $a_i$. Then $(a_1,b_1,\ldots,a_k,b_k)$ is a solution of \eqref{eqn:conlininda} if and only if it is a solution of the system of equations \eqref{eqn:conlininda}, \eqref{eqn:conlinindb}.  Moreover, \eqref{eqn:dctthetaiai2} holds for some $\mu$-integrable $\bar{f}$ if and only if the same inequality with $f$ on the left side replaced by $\tilde{f}$ holds for some $\mu$-integrable $\bar{f}_1$.

\item \label{item:firstidentifiableelaboratec}
Suppose the conditions in \ref{item:firstidentifiableelaborateb} (for $f$ or $\tilde{f}$) hold for any set $\{a_i\}_{i=1}^k$. Then \eqref{eqn:conlininda} has the same solutions as the system of equations \eqref{eqn:conlininda}, \eqref{eqn:conlinindb}. Hence, the family $\{f(x|\theta)\}_{\theta\in\Theta}$ is $(\{\theta_i\}_{i=1}^k,\Nc)$ first-order identifiable if and only if \eqref{eqn:conlininda} with $f$ replaced by $\tilde{f}$ has only the zero solution. 
\end{enumerate}
\end{lem}

Note a similar extension as in Remark \ref{rem:weakerprattslemma1} can be made in Lemma \ref{lem:firstidentifiableelaborate} \ref{item:firstidentifiableelaborateb} and \ref{item:firstidentifiableelaboratec}.

\begin{rem}
Part \ref{item:firstidentifiableelaborateb}, or Part \ref{item:firstidentifiableelaboratec}, of Lemma  \ref{lem:firstidentifiableelaborate} shows that under some differentiability condition (i.e. $\mu(\Nc)=0$) and some regularity condition on the density $f(x|\theta)$ to ensure the exchangeability of the limit and the integral, in the definition of $(\{\theta_i\}_{i=1}^k,\Nc)$ first identifiable,  \eqref{eqn:conlinindb} adds no additional constraint and is redundant. In this case when we verify the first-order identifiability, we can simply check whether \eqref{eqn:conlininda} has only zero solution or not. 
In addition when some $\tilde{f}$ is available and is simpler than $f$, according to Part \ref{item:firstidentifiableelaboratec} of Lemma  \ref{lem:firstidentifiableelaborate}, for first-order identifiability it is sufficient to check whether \eqref{eqn:conlininda} with $f$ replaced by $\tilde{f}$ has only zero solution or not, provided that the $\mu(\Nc)=0$ for $\Nc$ corresponds to $\tilde{f}$ and \eqref{eqn:dctthetaiai2} with $f$ on the left side replaced by $\tilde{f}$ hold. \myeoe
\end{rem}

Probability kernels in the exponential families of distribution are frequently employed in practice. For these kernels, there is a remarkable equivalence among the first-order identifiability and inverse bounds for both variational distance and the Hellinger distance.

\begin{lem} \label{cor:expD1equivalent}
Suppose that the probability kernel $P_{\thetaeta}$ has a density function $f$ in the full rank exponential family, given in its canonical form $f(x|\thetaeta)=\exp(\langle \thetaeta,T(x)\rangle -A(\thetaeta))h(x)$ with $\thetaeta\in \Theta$, the natural parameter space.  Then \eqref{eqn:conlininda} has the same solutions as the system of equations \eqref{eqn:conlininda}, \eqref{eqn:conlinindb}. Moreover for a fixed $G_0 = \sum_{i=1}^{k_0} p_i^0\delta_{\thetaeta_i^0}\in \Ec_{k_0}(\Theta^\circ)$ the following five statements are equivalent:
\begin{enumerate}[label=\alph*)]
\item \label{item:expD1equivalenta}
$\lim_{r\to 0}\ \inf_{\substack{G, H\in B_{W_1}(G_0,r)\\ G\not = H }} \frac{V(\P_{G},\P_{H})}{\myD_{1}(G,H)}>0;$
\item  \label{item:expD1equivalentb}
$
\liminf_{\substack{G\overset{W_1}{\to} G_0\\ G\in \Ec_{k_0}(\Theta)}} \frac{V(P_G,P_{G_0})}{D_1(G,G_0)}>0;
$
\item \label{item:expD1equivalentc}
$
\lim_{r\to 0}\ \inf_{\substack{G, H\in B_{W_1}(G_0,r)\\ G\not = H }} \frac{h(\P_{G},\P_{H})}{\myD_{1}(G,H)}>0;
$
\item \label{item:expD1equivalentd}
$
\liminf_{\substack{G\overset{W_1}{\to} G_0\\ G\in \Ec_{k_0}(\Theta)}} \frac{h(P_G,P_{G_0})}{D_1(G,G_0)}>0;
$
\item  \label{item:expD1equivalente}
With $k,\theta_i$ replaced respectively by $k_0,\theta_i^0$, equation \eqref{eqn:conlininda}
has only the zero solution.
\end{enumerate}
\end{lem}

Parts \ref{item:expD1equivalentc} and \ref{item:expD1equivalentd} in Lemma \ref{cor:expD1equivalent}
are not used in this paper beyond the current section, but they may be of independent interest beyond the scope of this paper. In the last result, the exponential family is in its canonical form. The same conclusions hold for the exponential family represented in general parametrizations. Recall a homeomorphism is a continuous function that has a continuous inverse.

\begin{lem}
\label{cor:expD1equivalenttheta}
Consider the probability kernel $P_\theta$ has a density function $f$ in the full rank exponential family,  $f(x|\theta)=\exp\left(\langle \eta(\theta),T(x)\rangle -B(\theta) \right)h(x)$. Suppose the map $\eta:\Theta \to \eta(\Theta)\subset \R^q$ is a homeomorphism.
Fix $G_0=\sum_{i=1}^{k_0}p_i^0\delta_{\theta_i^0}\in \Ec_{k_0}(\Theta^\circ)$. Suppose the Jacobian matrix of the function $\eta(\theta)$, denoted by $J_\eta(\theta):=(\frac{\partial \eta^{(i)}}{\partial \theta^{(j)}}(\theta))_{ij}$  exists and is full rank at $\theta_i^0$ for $i\in[k_0]$. Then with $k,\theta_i$ replaced respectively by $k_0$, $\theta_i^0$, \eqref{eqn:conlininda}  has the same solutions as the system of equations \eqref{eqn:conlininda}, \eqref{eqn:conlinindb}.
Moreover the \ref{item:expD1equivalentb}, \ref{item:expD1equivalentd} and \ref{item:expD1equivalente} as in Lemma \ref{cor:expD1equivalent} are equivalent. If in addition $J_\eta(\theta)$ exists and is continuous in a neighborhood of $\theta_i^0$ for each $i\in [k_0]$, then the equivalence relationships of all the five statements in Lemma \ref{cor:expD1equivalent} hold.   
\end{lem}

Despite the simplicity of kernels in the exponential families, classical and/or first-order identifiability is not always guaranteed. For instance, it is well-known and can be checked easily that the mixture of Bernoulli distributions is not identifiable in the classical sense. We will study the Bernoulli kernel in the context of mixtures of product distributions in Example~\ref{exa:bernoulli}.
The following example is somewhat less well-known.

\begin{exa}[Two-parameter gamma kernel] \label{exa:gamma}

Consider the gamma distribution 
$$f(x|\alpha,\beta) = \frac{\beta^\alpha}{\Gamma(\alpha)}x^{\alpha-1}e^{-\beta x}\1ve_{(0,\infty)}(x)$$ with $\theta=(\alpha,\beta)\in \Theta:=\{(\alpha,\beta)|\alpha>0,\beta>0\}$ and the dominating measure $\mu$ is the Lebesgue measure on $\Xfrak=\R$. This is a full rank exponential family. 
For $k_0\geq 2$ define $\mathcal{G}\subset \Ec_{k_0}(\Theta^{\circ})=\Ec_{k_0}(\Theta)$ as 
$$
\mathcal{G}:= \{G\in \Ec_{k_0}(\Theta)| G = \sum_{i=1}^{k_0}p_i\delta_{\theta_i} \text{ and there exist } i\neq j \text{ such that } \theta_j-\theta_i=(1,0) \}.
$$
For any $G_0 = \sum_{i=1}^{k_0}p_i\delta_{\theta_i^0} \in \mathcal{G}$, let $i_0\not =j_0$ be such that $\theta_{j_0}^0-\theta_{i_0}^0=(1,0)$, i.e. $\alpha_{j_0}^0=\alpha_{i_0}^0+1$ and $\beta_{j_0}^0=\beta_{i_0}^0$. 
Then observing 
$$
\frac{\partial}{\partial \beta}f(x|\alpha,\beta) = \frac{\alpha}{\beta}f(x|\alpha,\beta) - \frac{\alpha}{\beta}f(x|\alpha+1,\beta),
$$
$(a_1,b_1,\ldots,a_{k_0},b_{k_0})$ with $a_{i_0}=(0,\beta_{i_0}/\alpha_{i_0})$, $b_{i_0}=-1$, $b_{j_0}=1$ and the rest to be zero is a nonzero solution of the system of equations \eqref{eqn:conlininda}, \eqref{eqn:conlinindb} with $k,\theta_i$ replaced respectively by $k_0,\theta_i^0$. Write gamma distribution in exponential family as in Lemma \ref{cor:expD1equivalenttheta} with  $\eta(\theta)=(\alpha-1,\beta)$ and $T(x)=(\ln x, -x)$. Since $\eta(\theta)$ satisfies all the conditions in Lemma \ref{cor:expD1equivalenttheta}, hence
$$
\liminf_{\substack{G\overset{W_1}{\to} G_0\\ G\in \Ec_{k_0}(\Theta)}} \frac{h(P_G,P_{G_0})}{D_1(G,G_0)}=\liminf_{\substack{G\overset{W_1}{\to} G_0\\ G\in \Ec_{k_0}(\Theta)}} \frac{V(P_G,P_{G_0})}{D_1(G,G_0)}=0.
$$
This implies that even if $V(p_G,p_{G_0})$ vanishes at a fast rate, $D_1(G,G_0)$ may not. 

Finite mixtures of gamma were investigated by~\cite{ho2016convergence}, who called $\mathcal{G}$ is a \emph{pathological} set of parameter values to highlight the effects of \emph{weak identifiability} (more precisely, the violation of first-order identifiability conditions) on the convergence behavior of model parameters when the parameter values fall in $\mathcal{G}$. 
(On the other hand, for $G_0\in \Ec_{k_0}(\Theta^\circ)\backslash \mathcal{G}$, it is shown in the proof of Proposition 5.1 (a) in \cite{ho2016convergence} that \eqref{eqn:conlininda} with $k,\theta_i$ replaced respectively by $k_0,\theta_i^0$ has only the zero solution. Their original proof works under the stringent condition $\alpha\geq 1$ for the parameter space. But multiplying their (26) by $x$ should reach the same conclusion for the general case $\alpha>0$. A direct proof is also straightforward by using Lemma \ref{lem:polexplinind} \ref{item:polexplinindb} and is similar to Example \ref{exa:gamma2}.)
Thus by Lemma \ref{cor:expD1equivalenttheta}, 
$$
\sqrt{2} \liminf_{\substack{G\overset{W_1}{\to} G_0\\ G\in \Ec_{k_0}(\Theta)}} \frac{h(P_G,P_{G_0})}{D_1(G,G_0)}\geq \liminf_{\substack{G\overset{W_1}{\to} G_0\\ G\in \Ec_{k_0}(\Theta)}} \frac{V(P_G,P_{G_0})}{D_1(G,G_0)}> 0.
$$
Notice that \ref{item:expD1equivalenta} and \ref{item:expD1equivalentc} in Lemma \ref{cor:expD1equivalent} 
also hold but are omitted here. Thus, outside of pathological set $\mathcal{G}$ the convergence rate of mixture density $p_G$ towards $p_{G_0}$ is carried over to the convergence of $G$ toward $G_0$ under $D_1$. It is the uncertainty about whether the true mixing measure $G_0$ is pathological or not that makes parameter estimation highly inefficient. Given $\n$-i.i.d. from a finite mixture of gamma distributions, where the number of components $k_0$ is given,~\cite{ho2016convergence} established minimax bound for estimating $G$ that is slower than any polynomial rate $\n^{-r}$ for any $r\geq 1$ under $W_r$ metric. 
\myeoe
\end{exa}

We end this section with several remarks to highlight the concern for parameter estimation for mixture models under weak identifiability and to set the stage for the next section.

\begin{rem} 
\label{rem:discussions}
a) It may be of interest to devise an efficient parameter estimation method (by, perhaps, a clever regularization or reparametrization technique) 
that may help to overcome the lack of first-order identifiability. We are not aware of a general way to achieve this. Absent of such methods, a promising direction for the statistician to take is to simply collect more data: not only by increasing the number $\n$ of independent observations, but also by increasing the number of repeated measurements. Finite mixtures of product distributions usually arise in this practical context:  when one deals with a highly heterogeneous data population which is made up of many latent subpopulations carrying distinct patterns, it is often possible to collect observations presumably coming from the same subpopulation, even if one is uncertain about the mixture component that a subpopulation may be assigned to.  Thus, one may aim to collect $\n$ independent sequences of $\m$ exchangeable observations, and assume that they are sampled from a finite mixture of $\m$-product distributions denoted by $P_{G,\m}$. 
Such possibilities arise naturally in practice. As a concrete example, \cite{elmore2004estimating} applied a finite mixture model with repeated measurements 
to observations from an education assessment study. In this study, each child is presented with a sequence of two dimensional line drawings of a rectangular vessel, each drawn in a different tilted position. Then each child is asked to draw on these figures how the water in the vessel would appear if the vessel were half filled with water. Thus the observations from each child can be represented as a vector of exchangeable data and the experimenter can increase the length $N$ by presenting each child with more independent random rectangular vessels. Other examples and applications include psychological study \cite{hettmansperger2000almost,cruz2004semiparametric} and topic modeling \cite{ritchie2020consistent}.

b) One natural question to ask is, how does increasing the number $\m$ of repeated measurements (i.e., the length of exchangeable sequences) help to overcome the lack of strong identifiability such as our notion of first-order identifiability. This question can be made precise in light of Lemma~\ref{lem:firstidentifiable}: whether there exist a natural number $n_1 \geq 1$ such that the following inverse bound holds for any $\m \geq n_1$
    \begin{equation}
     \liminf_{\substack{G\overset{W_1}{\to} G_0\\ G\in \Ecal_{k_0}(\Theta) }} \frac{V(P_{G,\m },P_{G_0,\m })}{\myD_1(G,G_0)}>0. \label{eqn:VND1lowbou}
    \end{equation}
    Observe that since $V(P_{G,\m },P_{G_0,\m })$ increases in $\m$ while the denominator $D_1(G,G_0)$ is fixed in $\m$, if ~\eqref{eqn:VND1lowbou} holds for some $\m = n_1$, then it also holds for all $\m\geq n_1$. 
    Moreover, what can we say about the role of $\m$ in parameter estimation in presence of such inverse bounds?
In the sequel these questions will be addressed by establishing inverse bounds for mixtures of product distributions. Such theory will also be  used to derive tight learning rates for mixing measure $G$ from a collection of exchangeable sequences of observations. \myeoe

\end{rem}

\section{Inverse bounds for mixtures of product distributions}	
\label{sec:inversebounds}

Consider a family of probability distributions $\{\P_{\thetave}\}_{\thetave\in \Theta}$ on some measurable space $(\Xfrak,\Acalcal)$ where $\thetave$ is the parameter of the family and $\Theta\subset \R^q$ is the parameter space. This yields the $\m$-product probability kernel on $(\Xfrak^{\m},\Ac^N)$, which is denoted by $\{\P_{\thetave,\m}:= \bigotimes^{\m}\P_{\thetave}\}_{\thetave\in \Theta}$.
For any $G =\sum_{i=1}^{k} p_i \delta_{\thetave_i} \in \mathcal{E}_k(\Theta)$ as mixing measure, the resulting finite mixture for the $\m$-product families is a probability measure on $(\Xfrak^\m,\Ac^N)$, namely, $\P_{G,\m} = \sum_{i=1}^kp_i \P_{\thetave_i,\m}$.

The main results of this section are stated in Theorem \ref{thm:expfam} and Theorem \ref{thm:genthm}. These theorems establish the following inverse bound under certain conditions of probability kernel family $\{P_\theta\}_{\theta \in \Theta}$ and some time that of $G_0$: for a fixed $G_0 \in \Ecal_{k_0}(\Theta^\circ)$ there holds
\begin{equation}
\liminf_{\m\to \infty}\liminf_{\substack{G\overset{W_1}{\to} G_0\\ G\in \Ecal_{k_0}(\Theta) }} \frac{V(\P_{G,\m},\P_{G_0,\m})}{\myD_{\m}(G,G_0)} > 0. \label{eqn:genthmcon}
\end{equation}
By contrast, an easy upper bound on the left side of~\eqref{eqn:genthmcon} holds generally (cf. Lemma~\ref{lem:VD1liminfuppbou}):
\begin{equation}
\sup_{\m \geq 1}\liminf_{\substack{G\overset{W_1}{\to} G_0\\ G\in \Ecal_{k_0}(\Theta) }} \frac{V(\P_{G,\m},\P_{G_0,\m})}{\myD_{\m}(G,G_0)} \leq 1/2. \label{eqn:liminfliminfVDNup}
\end{equation}
In fact, a strong inverse bound can also be established:
\begin{equation}
\liminf_{\m\to\infty} \lim_{r\to 0}\ \inf_{\substack{G, H\in B_{W_1}(G_0,r)\\ G\not = H }} \frac{V(\P_{G,\m},\P_{H,\m})}{\myD_{\m}(G,H)}>0. \label{eqn:curvatureprodbound}
\end{equation}
These inverse bounds relate to the positivity of a suitable notion of curvature on the space of mixtures of product distributions, and will be shown to have powerful consequences. It's easy to see that \eqref{eqn:curvatureprodbound} implies \eqref{eqn:genthmcon}, which in turn entails~\eqref{eqn:VND1lowbou}.

Section~\ref{sec:inverse-exponential} is devoted to establishing these bounds for $P_\theta$ belonging to exponential families of distributions.
In Section~\ref{sec:inverse-general} 
the inverse bounds are established for very general probability kernel families, where $\Xfrak$ may be an abstract space and no parametric assumption on $\P_\theta$ will be required. Instead, we appeal to a set of mild regularity conditions on the characteristic function of a push-forward measure produced by a measurable map $T$ acting on the measure space $(\Xfrak, \Acalcal,\P_\theta)$. We will see that this general theory enables the study for a very broad range of mixtures of product distributions for exchangeable sequences.

\subsection{Implications on classical and first-order identifiability}
\label{sec:implication}

Before presenting the section's main theorems, let us explicate some immediate implications of their conclusions expressed by inequalities~\eqref{eqn:genthmcon} and~\eqref{eqn:curvatureprodbound}. These inequalities contain detailed information about convergence behavior of de Finetti's mixing measure $G$ toward $G_0$, an useful application of which will be demonstrated in Section~\ref{sec:poscon}. For now, we simply highlight striking implications on the basic notions of identifiability of mixtures of distributions investigated in Section~\ref{sec:firstidentifiable}. Note that no overt assumption on classical or first-order identifiability is required in the statement of the theorems establishing~\eqref{eqn:genthmcon} or~\eqref{eqn:curvatureprodbound}. In plain terms these inequalities entail that by increasing the number $\m$ of exchangeable measurements, the resulting mixture of $\m$-product distributions becomes identifiable in both classical and first-order sense, even if it is not initially so, i.e., when $\m=1$ or small.
%

Define, for any $G_0 \in \Ec_{k_0}(\Theta^\circ)$, $ H \in \cup_{k=1}^\infty \Ec_k(\Theta)$ and $\Hc\subset \cup_{k=1}^\infty \Ec_k(\Theta)$, 

     \begin{eqnarray}
      \label{eqn:defn0n1n2}
      n_0 & :=  n_0(H, \Hc ) & :=  \min \biggr \{n \geq 1 \biggr | \forall G \in \Hc \setminus \{H\}, P_{G,n} \neq P_{H,n} \biggr \} ,
    \\
    n_1 := n_1(G_0) & := n_1(G_0,\Ec_{k_0}(\Theta)) & :=  \min \biggr \{n \geq 1 \biggr | \liminf_{\substack{G\overset{W_1}{\to} G_0 \nonumber \\ G\in \Ecal_{k_0}(\Theta) }} \frac{V(P_{G,n },P_{G_0,n })}{\myD_1(G,G_0)}>0 \biggr \},  \nonumber \\
    n_2  := n_2(G_0) & := n_2(G_0,\Ec_{k_0}(\Theta)) & :=  \min \biggr \{n \geq 1 \biggr | \lim_{r\to 0}\ \inf_{\substack{G, H\in B_{W_1}(G_0,r)\\ G\not = H }} \frac{V(\P_{G,n},\P_{H,n})}{\myD_{1}(G,H)}>0 \biggr \}.
    \nonumber
     \end{eqnarray}
$n_0$ is called minimal zero-order identifiable length (with respect to $H$ and $\Hc$), or \emph{$0$-identifiable length} for short. $n_1$ is called minimal first-order identifiable length (with respect to $G_0$ and $\Ec_{k_0}(\Theta)$), or \emph{$1$-identifiable length} for short.
Since $W_1(G,G_0)\asymp_{G_0} D_1(G,G_0)$ in small neighborhood of $G_0$ (see Lemma \ref{lem:relW1D1} \ref{item:relW1D1d}),  the two metrics can be exchangeable in the denominator of the above definition for $n_1$ and $n_2$. 
Note that $n_0$, depending on a set $\Hc$ to be specified, describes a global property of classical identifiability, a notion determined mainly by the algebraic structure of the mixture model's kernel family and its parametrization. 
(This is also known as "strict identifiability", cf., e.g., ~\cite{allman2009identifiability}).  On the other hand, both $n_1$ and $n_2$ characterize a local behavior of mixture densities $p_{G,\m}$ near a certain $p_{G_0,\m}$, a notion that relies primarily on regularity conditions of the kernel, as we shall see in what follows. When it clear from the context, we may use $n_1$ or $n_1(G_0)$ for $n_1(G_0,\Ec_{k_0}(\Theta))$ for brevity. Similar rules apply to $n_0$ and $n_2$. 

The following proposition provides the link between classical identifiability and strong notions of local identifiability provided either~\eqref{eqn:genthmcon} or~\eqref{eqn:curvatureprodbound} holds. In a nutshell, as $\m$ gets large, the two types of identifiability can be connected by the force of the central limit theorem applied to product distributions, which is one of the key ingredients in the proof of the inverse bounds.
%
%
Define two related and useful quantities: for any $G_0\in \Ec_{k_0}(\Theta^\circ)$
\begin{align}
\label{eqn:Nbar}
\underbar{\m}_1 := \underbar{\m}_1(G_0) := \min \biggr \{n \geq 1 \biggr | \inf_{N\geq n} \liminf_{\substack{G\overset{W_1}{\to} G_0\\ G\in \Ecal_{k_0}(\Theta) }} \frac{V(P_{G,\m },P_{G_0,\m })}{\myD_{\m}(G,G_0)}>0 \biggr \} \\
\label{eqn:N2bar}
\underbar{\m}_2  := \underbar{\m}_2(G_0) :=   \min \biggr \{n \geq 1 \biggr | \inf_{\m\geq n} \lim_{r\to 0}\ \inf_{\substack{G, H\in B_{W_1}(G_0,r)   \\ G\not = H }} \frac{V(\P_{G,\m},\P_{H,\m})}{\myD_{\m}(G,H)}>0 \biggr \}.
\end{align}
Note that~\eqref{eqn:genthmcon} means $\underbar{\m}_1(G_0) < \infty$, while~\eqref{eqn:curvatureprodbound} means $\underbar{\m}_2(G_0) < \infty$. The following proposition explicates the connections among $n_0$, $n_1$, $n_2$, $\underbar{N}_1$ and $\underbar{N}_2$.

\begin{prop}
    \label{lem:n0nbar}
    There hold the following statements.
    \begin{enumerate} [label=\alph*)]
    \item \label{item:n1n2} Consider any $G_0\in\Ec_{k_0}(\Theta^\circ)$, then $n_1(G_0) \leq n_2(G_0)$.
     Moreover, there exists $r:= r(G_0)>0$ such that 
     $$
     \sup_{G\in B_{W_1}(G_0,r)} n_1(G) \leq n_2(G_0).
     $$
    
    \item \label{item:n1Nbar} 
      Consider any $G_0\in\Ec_{k_0}(\Theta^\circ)$. If $\underbar{\m}_1(G_0) < \infty$, then $n_1(G_0) = \underbar{\m}_1(G_0) <\infty$. \\
     If $\underbar{\m}_2(G_0) < \infty$ then $n_2(G_0) = \underbar{\m}_2(G_0) <\infty$. In particular, the first or the second conclusion holds if \eqref{eqn:genthmcon} or  \eqref{eqn:curvatureprodbound} holds respectively.

        \item  \label{item:n1n0} 
        There holds for any subset $\Theta_1\subset \Theta^\circ$ 
     $$\sup_{G \in  \bigcup_{k\leq k_0}\Ec_{k}(\Theta_1)} n_0(G, \cup_{k\leq k_0} \Ec_k(\Theta_1) ) 
     \leq \sup_{G \in  \Ec_{2k_0}(\Theta_1)} n_1(G).$$

      \item \label{item:n1n2Nbar}
      Suppose the kernel family $P_{\theta}$ admits density $f(\cdot|\theta)$ with respect to a dominating measure $\mu$ on $\Xfrak$. Fix $G_0=\sum_{i=1}^{k_0} p_i^0 \delta_{\theta_i^0} \in \Ec_{k_0}(\Theta^\circ)$. Suppose for $\mu$-a.e. $x\in \Xfrak$, $f(x|\theta)$ as a function $\theta$ is continuously differentiable in a neighborhood of $\theta_i^0$ for each $i\in [k_0]$. Moreover, assume that condition \ref{item:nececonditionc} of Lemma~\ref{lem:nececondition} holds for any $\{a_i\}_{i=1}^{k_0}\subset \R^q$. Then, $n_2(G_0) = n_1(G_0)$.

     \item \label{item:n2n0} Suppose that \eqref{eqn:genthmcon} holds for every $G_0 \in \cup_{k\leq 2k_0} \Ec_{k}(\Theta^\circ)$. Moreover, suppose all conditions in part \ref{item:n1n2Nbar} are satisfied for every $G_0 \in \cup_{k\leq 2k_0} \Ec_{k}(\Theta^\circ)$.
     Then for any compact $\Theta_1\subset \Theta^\circ$, 
     $$
     \sup_{G \in \cup_{k\leq k_0} \Ec_{k}(\Theta_1)} n_0(G,\cup_{k\leq k_0} \Ec_k(\Theta_1) ) 
    \leq \sup_{G \in  \Ec_{2k_0}(\Theta_1)} n_1(G) <  \infty.
    $$
     
          \item \label{item:n2n0new} Suppose that \eqref{eqn:curvatureprodbound} holds for every $G_0 \in \cup_{k\leq 2k_0} \Ec_{k}(\Theta^\circ)$. Then for any compact $\Theta_1\subset \Theta^\circ$, the conclusion of part \ref{item:n2n0} holds.
     
     \end{enumerate}
\end{prop}

\begin{rem} \label{rem:N1Ninfcon}
Part~\ref{item:n1n2} and part~\ref{item:n1Nbar} of Proposition \ref{lem:n0nbar} highlight an immediate significance of inverse bounds~\eqref{eqn:genthmcon} and~\eqref{eqn:curvatureprodbound}: they establish pointwise finiteness of 1-identifiable length $n_1(G_0)$. Moreover, under the additional condition on first-order identifiability, one can have the following strong result as an immediate consequence: 
     Consider any $G_0\in \Ec_{k_0}(\Theta^\circ)$. If \eqref{eqn:noproductlowbou} and \eqref{eqn:genthmcon} hold,
       then $n_1(G_0)=\underbar{\m}_1(G_0)=1$. If \eqref{eqn:curvaturebound} and \eqref{eqn:curvatureprodbound} hold, then $n_1(G_0)=\underbar{\m}_1(G_0) = n_2(G_0) = \underbar{\m}_2(G_0) = 1$.  \myeoe
\end{rem}

\begin{rem}
    Proposition \ref{lem:n0nbar} \ref{item:n1n0} 
    relates  $\sup_{G \in  \bigcup_{k\leq k_0}\Ec_{k}(\Theta_1)} n_0(G,\cup_{k\leq k_0} \Ec_k(\Theta_1))$, the uniform $0$-identifiable length, to the uniform $1$-identifiable length $\sup_{G\in \Ec_{2k_0}(\Theta_1)} n_1(G,\Ec_{2k_0}(\Theta_1))$.
    Combining this with parts \ref{item:n2n0} and \ref{item:n2n0new} and inverse bounds~\eqref{eqn:genthmcon} and~\eqref{eqn:curvatureprodbound}, one arrives at a rather remarkable consequence: for $\Theta_1$ a compact subset of $\Theta^\circ$, they yield the finiteness of both 0-identifiable length $n_0(G,\cup_{k\leq k_0} \Ec_k(\Theta_1))$ 
    and 1-identifiable length $n_1(G)$ \emph{uniformly} over subsets of mixing measures with finite number of support points. In particular, as long as \eqref{eqn:genthmcon} or \eqref{eqn:curvatureprodbound} (along with some regularity conditions in the former) holds for every $G_0 \in \cup_{k\leq 2k_0} \Ec_{k}(\Theta^\circ)$, then $P_{G,N}$ will be strictly identifiable and first-identifiable on $\cup_{k\leq k_0} \Ec_{k}(\Theta_1)$ for sufficiently large $N$. That is, taking product helps in making the kernel identifiable in a strong sense. As we shall see in the next subsection, \eqref{eqn:curvatureprodbound} holds for every $G_0\in \bigcup_{k=1}^{\infty}\Ec_k(\Theta^\circ)$ when $\{P_\theta\}$ belongs to full rank exponential families of distributions. This inverse bound also holds for a broad range of probability kernels beyond the exponential families. \myeoe
\end{rem}
\begin{rem}\label{rem:claytonresult}
    Concerning only 0-identifiable length $n_0$, 
    a remarkable upper bound $$\sup_{G\in \cup_{k\leq k_0}\Ec_k(\Theta)}n_0(G,\cup_{k\leq k_0} \Ec_k(\Theta))\leq 2k_0 -1$$ was established in a recent paper \cite{vandermeulen2019operator}. This bound actually applies to the nonparametric component distributions, and extends also to our parametric component distribution setting. However, in a parametric component distribution setting, the upper bound $2k_0-1$ is far from being tight (cf. Example \ref{exa:gamma2}).  \myeoe
\end{rem}

\begin{proof}[Proof of Proposition \ref{lem:n0nbar}]
 a) 
    It is sufficient to assume that $n_2=n_2(G_0) < \infty$. Then there exists $r_0 > 0$ such that
        $$
      \inf_{\substack{G, H\in B_{W_1}(G_0,r_0)\\ G\not = H }} \frac{V(\P_{G,n_2},\P_{H,n_2})}{\myD_{1}(G,H)}>0.
      $$
      Then fixing $G$ in the preceding display yields $n_1(G)\leq n_2(G_0)$ and the proof is complete since $G$ is arbitrary in $B_{W_1}(G_0,r_0)$.
    
b) 
        By the definition of $\underbar{\m}_1$,
        \begin{equation}
        \liminf_{\substack{G\overset{W_1}{\to}G_0\\ G\in \Ec_{k_0}(\Theta)}  } \frac{V(P_{G,\underbar{\m}_1}, P_{G_0,\underbar{\m}_1})}{D_{1}(G,G_0)} \geq  \liminf_{\substack{G\overset{W_1}{\to}G_0\\ G\in \Ec_{k_0}(\Theta)}  } \frac{V(P_{G,\underbar{\m}_1},P_{G_0,\underbar{\m}_1})}{D_{\underbar{\m}_1}(G,G_0)} >0, 
        \label{eqn:VmD1bouc}
        \end{equation}
        which entails that $n_1 \leq \underbar{\m}_1$. On the other hand, for any $\m \in [n_1,\underbar{\m}_1]$ we have
        \begin{multline*}
        \liminf_{\substack{G\overset{W_1}{\to} G_0\\ G\in \Ecal_{k_0}(\Theta) }} \frac{V(P_{G,\m},P_{G_0,\m})}{\myD_\m(G,G_0)}  \geq \liminf_{\substack{G\overset{W_1}{\to} G_0\\ G\in \Ecal_{k_0}(\Theta) }} \frac{1}{\sqrt{\m}}\frac{V(P_{G,n_1},P_{G_0,n_1})}{\myD_1(G,G_0)} \\ 
        \geq  \frac{1}{\sqrt{\underbar{\m}_1}} \liminf_{\substack{G\overset{W_1}{\to} G_0\\ G\in \Ecal_{k_0}(\Theta) }} \frac{V(P_{G,n_1},P_{G_0,n_1})}{\myD_1(G,G_0)} > 0,
        \end{multline*}
        which entails $\underbar{\m}_1 \leq n_1$. Thus $\underbar{\m}_1 = n_1$.
        The proof of $n_2 = \underbar{\m}_2 < \infty$ is similar.

        c) 
        If suffices to prove the case
        $\bar{n} := \sup_{G\in  \Ec_{2k_0}(\Theta_1)} n_1(G) < \infty$.
        Take any $G \in \Ec_{k}(\Theta_1)$ for $1\leq k\leq k_0$, and suppose that $n_0(G, \cup_{k\leq k_0} \Ec_k(\Theta_1) ) > \bar{n}$. Then there exists a $G_1 \in \Ec_{\bar{k}}(\Theta_1)$ for some $1\leq \bar{k} \leq k_0$ such that $P_{G,\bar{n}} = P_{G_1,\bar{n}}$ but $G_1\neq G$.
        Collecting the supporting atoms of $G$ and $G_1$, there are at most $2k_0$ of those, and denote them by $\theta^0_1,\ldots, \theta^0_{k'} \in \Theta_1$. Supplement these with a set of atoms  $\{\theta^0_{i}\}_{i=k'+1}^{2k_0}\subset \Theta_1$ to obtain a set of distinct $2k_0$ atoms denoted by 
        $\{\theta^0_{i}\}_{i=1}^{2k_0}$. Now take $G_0$ to be any discrete probability measure supported by $\theta^0_1,\ldots, \theta^0_{2k_0}$ in $\Ec_{2k_0}(\Theta_1)$. Since $P_{G,\bar{n}} = P_{G_1,\bar{n}}$, the condition of Lemma~\ref{lem:nececondition2} for $G_0$ is satisfied and thus
        $$
         \liminf_{\substack{H \overset{W_1}{\to} G_0 \\ H \in \Ec_{2k_0}(\Theta) }} \frac{V(P_{H,\bar{n}},P_{G_0,\bar{n}})}{D_1(H ,G_0)} =0.
         $$
        But this contradicts with the definition of $\bar{n}$.

d) 
        By part \ref{item:n1n2} it suffices to prove for the case $n_1=n_1(G_0)<\infty$. By Lemma \ref{lem:conditioncprod}, the product family $\prod_{j=1}^{n_1}f(x_j|\theta)$ satisfies all the conditions in Corollary \ref{cor:curvaturenoncurvequiv}. Thus by Corollary~\ref{cor:curvaturenoncurvequiv}, 
         \[\lim_{r\to 0}\ \inf_{\substack{G, H\in B_{W_1}(G_0,r)\\ G\not = H }} \frac{V(\P_{G,n_1},\P_{H,n_1})}{\myD_{1}(G,H)}>0\]
         It follows that $n_2(G_0) \leq n_1$, which implies that $n_2(G_0) = n_1(G_0)$ by part \ref{item:n1n2}. 

e) 
        By part~\ref{item:n1Nbar} and part~\ref{item:n1n2Nbar}, $n_2(G_0)<\infty$ for every $G_0 \in \cup_{k\leq 2k_0} \Ec_{k}(\Theta^\circ)$.  Associate each $G_0 \in \cup_{k\leq 2k_0} \Ec_{k}(\Theta^\circ)$ with a neighborhood $B_{W_1}(G_0,r(G_0))$ as in part \ref{item:n1n2} such that its conclusion holds. Here we want to emphasize that by definition of $B_{W_1}$ in \eqref{eqn:BW1def}, $B_{W_1}(G_0,r(G_0))\subset \cup_{k\leq \ell} \Ec_{k}(\Theta)$ when $G_0\in \Ec_{\ell}(\Theta^\circ)$. Due to the fact that $\cup_{k\leq 2k_0}\Ec_{k}(\Theta_1)$ is compact and part \ref{item:n1n2}, we deduce that $n_1(G)$ is uniformly bounded for $G\in \cup_{k\leq 2k_0}\Ec_{k}(\Theta_1)$. Combining this with part~\ref{item:n1n0} we conclude the proof.
                       
f) 
        By part~\ref{item:n1Nbar} $n_2(G_0)<\infty$ for every $G_0 \in \cup_{k\leq 2k_0} \Ec_{k}(\Theta^\circ)$. The remainder of the argument is the same as part~\ref{item:n2n0}.
%
\end{proof}
        
We can further unpack the double infimum limits in its expression of~\eqref{eqn:genthmcon} to develop results useful for subsequent convergence rate analysis in Section~\ref{sec:poscon}.
First,  it is simple to show that the limiting argument for $\m$ can be completely removed when $\m$ is suitably bounded. Denote by $C(\cdot)$ or $c(\cdot)$  a positive finite constant depending only on its parameters and the probability kernel $\{P_\theta\}_{\theta\in \Theta}$. In the presentation of inequality bounds and proofs,  they may differ from line to line.

\begin{lem} \label{cor:identifiabilityallm} 
    Fix $G_0\in\Ec_{k_0}(\Theta^\circ)$.
    Suppose \eqref{eqn:genthmcon} holds.
    Then for any $\m_0 \geq n_1(G_0)$, there exist $c(G_0,\m_0)$ and $C(G_0)$ such that for any $ G\in\Ec_{k_0}(\Theta)$ satisfying $W_1(G,G_0)<c(G_0,\m_0)$,
    $$
    V(P_{G,\m},P_{G_0,\m})\geq C(G_0)D_{\m}(G,G_0), \quad  \forall \m \in [n_1(G_0),\m_0].
    $$
\end{lem}

A key feature of the above claim is that the radius $c(G_0,\m_0)$ of the local $W_1$ ball centered at $G_0$ over which the inverse bound holds depends on $\m_0$, but the multiple constant $C(G_0)$ does not. Next, given additional conditions, most notably the compactness on the space of mixing measures, we may remove completely the second limiting argument involving $G$. In other words, we may extend the domain of $G$ on which the inverse bound of the form $V \gtrsim W_1 \gtrsim D_1$ continues to hold, where the multiple constants are suppressed here. 
    
    \begin{lem} \label{cor:VlowbouW1} 
    Fix $G_0\in\Ec_{k_0}(\Theta^\circ)$. Consider any  $\Theta_1$ a compact subset of $\Theta$ containing the supporting atoms of $G_0$.
    Suppose the map $\theta\mapsto P_\theta$ from $(\Theta_1,\|\cdot\|_2)$ to $(\{P_\theta\}_{\theta\in \Theta_1}, V(\cdot,\cdot))$ is continuous. 
    Then for any $G\in \bigcup_{k=1}^{k_0}\Ec_{k}(\Theta_1)$ and any $\m\geq n_1(G_0)\vee  n_0(G_0,\cup_{k\leq k_0}\Ec_k(\Theta_1))$, provided $n_1(G_0)\vee n_0(G_0,\cup_{k\leq k_0}\Ec_k(\Theta_1)) <\infty$,
    $$
     V(P_{G,\m},P_{G_0,\m}) \geq  C(G_0,\Theta_1) W_1(G,G_0).
     $$
     \end{lem} 

Finally, a simple and useful fact which allows one to transfer an inverse bound for one kernel family $P_\theta$ to another kernel family by means of homeomorphic transformation in the parameter space.
	If $g(\theta)=\eta$ for some homeomorphic function $g: \Theta\to \Xi\subset \R^q$, for any $G=\sum_{i=1}^k p_i \delta_{\theta_i}\in\Ec_k(\Theta)$, denote $G^{\eta}=\sum_{i=1}^k p_i \delta_{g(\theta_i)}\in \Ec_k(\Xi)$ . Given a probability kernel family $\{P_\theta\}_{\theta\in \Theta}$, under the new parameter $\eta$ define
    $$
        \tilde{\P}_\eta = \P_{g^{-1}(\eta)}, \quad \forall \eta\in \Xi.
    $$
    Let $G^{\eta}$ also denote a generic element in $\Ec_{k_0}(\Xi)$, and $\tilde{\P}_{G^{\eta},\m}$ be defined similarly as $P_{G,\m}$.

    \begin{lem}[Invariance under homeomorphic parametrization with local invertible Jacobian]\label{lem:invariance} 
	Suppose $g$ is a homeomorphism. For $G_0=\sum_{i=1}^{k_0}p_i^0\delta_{\theta_i^0}\in \Ec_{k_0}(\Theta^\circ)$,
	suppose the Jacobian matrix of the function $g(\theta)$, denoted by $J_g(\theta):=(\frac{\partial g^{(i)}}{\partial \theta^{(j)}}(\theta))_{ij}$ exists and is full rank at $\theta_i^0$ for $i\in[k_0]$. Then $\forall \m$
	\begin{equation}
	 \liminf_{\substack{G^{\eta}\overset{W_1}{\to} G_0^\eta\\ G^{\eta}\in \Ecal_{k_0}(\Xi) }} \frac{V(\tilde{\P}_{G^{\eta},\m},\tilde{\P}_{G_0^\eta,\m})}{\myD_{\m}(G^{\eta},G_0^\eta)} \overset{G_0}{\asymp}   \liminf_{\substack{G\overset{W_1}{\to} G_0\\ G\in \Ecal_{k_0}(\Theta) }} \frac{V(\P_{G,\m},\P_{G_0,\m})}{\myD_{\m}(G,G_0)}. \label{eqn:invarianceprod}
	\end{equation}
	Moreover, if in addition $J_g(\theta)$ exists and is continuous in a neighborhood of $\theta_i^0$ for each $i\in [k_0]$, then $\forall N$
	$$
	\lim_{r\to 0}\ \inf_{\substack{G^{\eta}, H^{\eta}\in B_{W_1}(G_0^{\eta},r)\\ G^\eta \not = H^\eta }} \frac{V(\tilde{\P}_{G^\eta,N},\tilde{\P}_{H^\eta,N})}{\myD_{1}(G^\eta,H^\eta)} \overset{G_0}{\asymp} \lim_{r\to 0}\ \inf_{\substack{G, H\in B_{W_1}(G_0,r)\\ G\not = H }} \frac{V(\P_{G,N},\P_{H,N})}{\myD_{1}(G,H)}.
	$$

\end{lem}

Lemma \ref{lem:invariance} shows that if an inverse bound \eqref{eqn:genthmcon} or \eqref{eqn:curvatureprodbound} under a particular parametrization is established, then the same inverse bound holds for all other parametrizations that are homeomorphic and that have local invertible Jacobian. This allows one to choose the most convenient parametrization; for instance, one may choose the canonical form for an exponential family or another more convenient parametrization, like the mean parametrization. 

\subsection{Probability kernels in regular exponential family}
\label{sec:inverse-exponential}

We now present inverse bounds for the mixture of products of exponential family distributions.
Suppose that $\{P_\theta\}_{\theta \in \Theta}$ is a full rank exponential family of distributions on $\Xfrak$. Adopting the notational convention for canonical parameters of exponential families, we assume $P_\theta$ admits a density function with respect to a dominating measure $\mu$, namely $f(x|\thetaeta)$ for $\thetaeta \in \Theta$.

\begin{thm} \label{thm:expfam} 
Suppose that the probability kernel $\{f(x|\thetaeta)\}_{\thetaeta \in \Theta}$ is in a full rank exponential family of distributions in canonical form as in Lemma \ref{cor:expD1equivalent}. For any $G_0\in \Ec_{k_0}(\Theta^\circ)$, \eqref{eqn:genthmcon} and \eqref{eqn:curvatureprodbound} hold.
\end{thm}

In the last theorem the exponential family is in its canonical form. The following corollary extends to exponential families in the general form under mild conditions. 

\begin{cor} \label{cor:expfamtheta}
Consider the probability kernel $P_\theta$ has a density function $f$ in the full rank exponential family,  $f(x|\theta)=\exp\left(\langle \eta(\theta),T(x)\rangle -B(\theta) \right)h(x)$, where the map $\eta:\Theta \to \eta(\Theta)\subset \R^q$ is a homeomorphism. Suppose that $\eta$ is continuously differentiable on $\Theta^\circ$ and its the Jacobian is of full rank on $\Theta^\circ$. Then, for any $G_0\in \Ec_{k_0}(\Theta^\circ)$, ~\eqref{eqn:genthmcon} and~\eqref{eqn:curvatureprodbound} hold.
\end{cor}

As a consequence of Corollary~\ref{cor:expfamtheta},  Proposition~\ref{lem:n0nbar} and Lemma \ref{cor:expD1equivalenttheta}, we immediately obtain the following interesting algebraic result for which a direct proof may be challenging. 

\begin{cor} \label{cor:expprodlinind}
Let the probability kernel $\{f(x|\thetaeta)\}_{\thetaeta \in \Theta}$ be in a full rank exponential family of distributions as in Corollary \ref{cor:expfamtheta} and suppose that all conditions there hold. Then for any $k_0\geq 1$ and for any $G_0=\sum_{i=1}^{k_0}p_i^0\delta_{\theta_i^0}\in \Ec_{k_0}(\Theta^\circ)$, $n_1(G_0)=n_2(G_0)=\underbar{N}_1(G_0)=\underbar{N}_2(G_0)$ are finite. Moreover,
\begin{equation}
\sum_{i=1}^{k_0} \left(a_i^\top  \nabla_{\theta}  \prod_{n=1}^{\m} f(x_n|\theta_i^0)+b_i \prod_{n=1}^{\m} f(x_n|\theta_i^0) \right) =0, \quad \bigotimes^N\mu-a.e.\  (x_1,\ldots, x_\m) \in \Xfrak^{\m} \label{eqn:pdfn0}
\end{equation}
has only the zero solution: 
$$
b_i =0 \in \R \text{ and } a_i =\bm{0} \in \R^q, \quad \forall 1\leq i\leq k_0
$$ 
if and only if $N\geq n_1(G_0)$.
\end{cor}

Corollary \ref{cor:expprodlinind} establishes that for full rank exponential families of distribution specified in Corollary \ref{cor:expfamtheta} with full rank Jacobian of $\eta(\theta)$, there is a finite phase transition behavior specified by $n_1(G_0)$ of the $\m$-product in \eqref{eqn:pdfn0}: the system of equations \eqref{eqn:pdfn0} has nonzero solution when $\m < n_1(G_0)$ and as soon as $\m \geq n_1(G_0)$, it has only the zero solution. This also gives another characterization of $n_1(G_0)$ defined in \eqref{eqn:defn0n1n2} for such exponential families, which also provides a way to compute $n_1(G_0)=\underbar{\m}_1(G_0) = n_2(G_0) = \underbar{\m}_2(G_0)$. A byproduct is that $n_1(G_0)$ does not depend on the $p_i^0$ of $G_0$ since \eqref{eqn:pdfn0} only depends on $\theta_i^0$. 

We next demonstrate two non-trivial examples of mixture models that are either non-identifiable or weakly identifiable, i.e., when $\m=1$, but become first-identifiable by taking products. We work out the details on calculating $n_0(G_0)$ and $n_1(G_0)$; these should serve as convincing examples to the discussion at the end of Section \ref{sec:firstidentifiable}.
 

\begin{exa}[Bernoulli kernel]  
\label{exa:bernoulli} 
Consider the Bernoulli distribution $f(x|\theta)=\theta^x(1-\theta)^{1-x}$ with parameter space $\Theta=(0,1)$. Here the family is defined on $\Xfrak=\R$ and the dominating measure is $\mu = \delta_{0}+\delta_{1}$. It can be written in exponential form as in Lemma \ref{cor:expD1equivalenttheta} with $\eta(\theta)=\ln \theta -\ln(1-\theta)$ and $T(x)=x$. It's easy to check that $\eta'(\theta)=\frac{1}{\theta(1-\theta)}>0$ and thus all conditions in Lemma \ref{cor:expD1equivalenttheta}, Corollary \ref{cor:expfamtheta} and Corollary \ref{cor:expprodlinind} are satisfied. Thus any of those three results can be applied. In particular we may use the characterization of $n_1(G_0)$ in Corollary \ref{cor:expprodlinind} to compute $n_1(G_0)$.

For the $n$-fold product, the density $f_n(x_1,x_2,\ldots,x_n|\theta) := \prod_{\ell=1}^n f(x_\ell|\theta) = \theta^{\sum_{\ell=1}^n x_\ell} (1-\theta)^{n-\sum_{\ell=1}^n x_\ell}$. Then $\frac{\partial}{\partial \theta}f_n(x_1,\dots,x_n|\theta)$ is 
$$ \left(\sum_{\ell=1}^n x_\ell\right)\theta^{\sum_{\ell=1}^n x_\ell-1} (1-\theta)^{n-\sum_{\ell=1}^n x_\ell} - \left(n-\sum_{\ell=1}^n x_\ell\right)\theta^{\sum_{\ell=1}^n x_\ell} (1-\theta)^{n-\sum_{\ell=1}^n x_\ell-1}.$$ 

We now compute $n_1(G)$ for any $G=\sum_{i=1}^k p_i \delta_{\theta_i}\in \Ec_k(\Theta)$. Notice the support of $f$
is $\{0,1\}$ and hence the support of $f_n$ is $\{0,1\}^n$. Thus \eqref{eqn:pdfn0} with $k_0$, $N$, and $\theta_i^0$ replaced respectively by $k$, $n$ and $\theta_i$ become a system of $n+1$ linear equations: $\forall j=\sum_{\ell=1}^n x_\ell\in\{0\}\cup[n]$
\begin{equation}
\sum_{i=1}^k a_i \left(j(\theta_i)^{j-1} (1-\theta_i)^{n-j} - (n-j)(\theta_i)^{j} (1-\theta_i)^{n-j-1}\right) + \sum_{i=1}^k b_i (\theta_i)^{j} (1-\theta_i)^{n-j}
=0. \label{eqn:syslinequbernoulli}
\end{equation}
As a system of $n+1$ linear equations with $2k$ unknown variables, it has nonzero solutions when $n+1< 2k$. Thus $n_1(G)\geq 2k-1$ for any $G\in \Ec_k(\Theta)$.

Let us now verify that $n_1(G)=2k-1$ for any $G\in \Ec_k(\Theta)$. Indeed, for any $G=\sum_{i=1}^k p_i\delta_{\theta_i}\in \Ec_k(\Theta)$, the system of linear equations \eqref{eqn:syslinequbernoulli} with $n=2k-1$ is $A^\top  z=0$ with $z=(b_1,a_1,\ldots,b_k,a_k)^\top $ and 
$$
A_{ij} = 
\begin{cases}
f_j(\theta_m)  & i=2m-1 \\
f'_j(\theta_m) & i=2m
\end{cases}
\text{ for } j\in [2k], m\in [k],
$$
where $f_j(\theta)=\theta^{j-1}(1-\theta)^{n-(j-1)}$ with $n=2k-1$.
By Lemma \ref{lem:determinant} \ref{item:determinantc}  $\text{det}(A)=\prod_{1\leq \alpha<\beta\leq k}(\theta_{\alpha}-\theta_{\beta})^4$, with the convention $1$ when $k=1$. Thus, $A$ is invertible and the system of linear equations \eqref{eqn:syslinequbernoulli} with $n=2k-1$ has only zero solution. Thus by Corollary \ref{cor:expprodlinind} $n_1(G)\leq 2k-1$. By the conclusion from last paragraph $n_1(G)= 2k-1$. 

In Section \ref{sec:identifiabilityBernoulli} we also prove that $n_0(G,\cup_{\ell\leq k}\Ec_\ell (\Theta) ) = 2k-1$ for any $G\in \Ec_k(\Theta)$.
\myeoe
\end{exa}

The next lemma is on the determinant of a type of generalized Vandermonde matrices. Its part \ref{item:determinantc} is used in the previous example on Bernoulli kernel. 

\begin{lem}\label{lem:determinant} 
Let $x,y\in \R$.
\begin{enumerate}[label=\alph*)]
\item \label{item:determinanta} 
Let $f(x)$ be a polynomial. Define $q^{(1)}(x,y)=\frac{f(x)-f(y)}{x-y}$, $q^{(2)}(x,y)=\frac{f'(x)-f'(y)}{x-y}$, $\bar{q}^{(2)}(x,y)=\frac{q^{(1)}(x,y)-f'(y)}{x-y}
$, and $\bar{q}^{(3)}(x,y)=\frac{\bar{q}^{(2)}(x,y)-\frac{1}{2} q^{(2)}(x,y)}{x-y}
$. Then $q^{(1)}(x,y)$, $q^{(2)}(x,y)$, $\bar{q}^{(2)}(x,y)$ and $\bar{q}^{(3)}(x,y)$ are all multivariate polynomials.
\item \label{item:determinantb}
Let $f_j(x)$ be a polynomial and $f'_j(x)$ its derivative for $j\in [2k]$ for a positive integer $k$. For $x_1,\ldots,x_k\in \R$ define $A^{(k)}(x_1,\dots,x_k)\in \R^{(2k)\times (2k)}$ by 
$$
A_{ij}^{(k)}(x_1,\dots,x_k) = 
\begin{cases}
f_j(x_m)  & i=2m-1 \\
f'_j(x_m) & i=2m
\end{cases}
\text{ for } j\in [2k], m\in [k].
$$
Then for any $k\geq 2$, $\text{det}(A^{(k)}(x_1,\dots,x_k))=g_k(x_1,\ldots,x_k)\prod_{1\leq \alpha<\beta\leq k}(x_{\alpha}-x_{\beta})^4$ , where $g_k$ is some multivariate polynomial.
\item \label{item:determinantd}
For the special case $f_j(x)=x^{j-1}$, 
$A^{(k)}(x_1,\ldots,x_k)$  has determinant $\text{det}(A^{(k)}(x_1,\dots,x_k))=\prod_{1\leq \alpha<\beta\leq k}(x_{\alpha}-x_{\beta})^4$, with the convention $1$ when $k=1$.
\item \label{item:determinantc}
For the special case $f_j(x)=f_j(x|k)=x^{j-1}(1-x)^{n-(j-1)}$ with $n=2k-1$, 
$A^{(k)}(x_1,\ldots,x_k)$ has determinant $\text{det}(A^{(k)}(x_1,\dots,x_k))=\prod_{1\leq \alpha<\beta\leq k}(x_{\alpha}-x_{\beta})^4$, with the convention $1$ when $k=1$.
\end{enumerate}
\end{lem}

\begin{exa}[Continuation on two-parameter gamma kernel] 
\label{exa:gamma2} 
Consider the gamma distribution $f(x|\alpha,\beta)$ discussed in Example \ref{exa:gamma}. Let $k_0\geq 2$ and by Example \ref{exa:gamma} for any $G_0\in \Ec_{k_0}(\Theta)\backslash \mathcal{G}$, $n_1(G_0)=1$ and for any $G_0 \in \mathcal{G}$, where
we recall that $\Gc$ denotes the pathological subset of the gamma mixture's parameter space,
$$
\liminf_{\substack{G\overset{W_1}{\to} G_0\\ G\in \Ec_{k_0}(\Theta)}} \frac{h(P_G,P_{G_0})}{D_1(G,G_0)}=\liminf_{\substack{G\overset{W_1}{\to} G_0\\ G\in \Ec_{k_0}(\Theta)}} \frac{V(P_G,P_{G_0})}{D_1(G,G_0)}=0.
$$
This means $n_1(G_0)\geq 2$ for $G_0\in \mathcal{G}$. 

We now show that for $G_0\in \Ec_{k_0}(\Theta)$ $n_1(G_0)\leq 2$ and hence $n_1(G_0)= 2$ for $G_0\in \mathcal{G}$. Let 
$$f_2(x_1,x_2|\alpha,\beta):=f(x_1|\alpha,\beta)f(x_2|\alpha,\beta)= \frac{\beta^{2\alpha} }{(\Gamma(\alpha))^2 } (x_1x_2)^{\alpha-1}e^{-\beta(x_1+x_2)}\1ve_{(0,\infty)^2}(x_1,x_2)$$ 
be the density of the $2$-fold product w.r.t. Lebesgue measure on $\R^2$. Let $g(\alpha,\beta)=(\Gamma(\alpha))^2 /\beta^{2\alpha}$, which is a differentiable function on $\Theta$ and let $\tilde{f}_2(x_1,x_2|\alpha,\beta):=g(\alpha,\beta)\times$
$\times f_2(x_1,x_2|\alpha,\beta)$ to be the density without normalization constant. Note that $\frac{\partial}{\partial \alpha}\tilde{f}_2(x_1,x_2|\alpha,\beta)$ $= \tilde{f}_2(x_1,x_2|\alpha,\beta)\ln(x_1 x_2)$ and $\frac{\partial}{\partial \beta}\tilde{f}_2(x_1,x_2|\alpha,\beta) = -(x_1+x_2)\tilde{f}_2(x_1,x_2|\alpha,\beta)$. Then \eqref{eqn:conlininda} with $f$ replaced by $\tilde{f}_2$ is 
\begin{equation}
\sum_{i=1}^{k_0} \left(a^{(\alpha)}_i \ln(x_1 x_2) - a^{(\beta)}_i (x_1+x_2) + b_i\right)(x_1 x_2)^{\alpha_i-1}e^{-\beta_i(x_1+x_2)}=0. \label{eqn:gamma2foldprodnonor}
\end{equation}
Let $\bigcup_{i=1}^{k}\{\beta_i\}=\{\beta'_1,\beta'_2,\cdots,\beta'_k\}$ with $\beta'_1<\beta'_2<\ldots<\beta'_{k'}$ where $k'$ is the number of distinct elements. Define $I(\beta')=\{i\in[k]| \beta_i=\beta'\}$. Then \eqref{eqn:gamma2foldprodnonor} become for $\mu$-a.e. $x_1,x_2\in (0,\infty)$
\begin{align*}
    0=&\sum_{j=1}^{k'}\left(\sum_{i\in I(\beta'_j)} \left(a^{(\alpha)}_i \ln(x_1 x_2) - a^{(\beta)}_i (x_1+x_2) + b_i\right)(x_1 x_2)^{\alpha_i-1}\right)e^{-\beta'_j(x_1+x_2)}\\
    =& \sum_{j=1}^{k'} e^{-\beta'_j x_2}e^{-\beta'_j x_1} \left( \sum_{i\in I(\beta'_j)} a^{(\alpha)}_i (x_1 x_2)^{\alpha_i-1}\ln(x_1 ) \right. \\
   & \left. +\sum_{i\in I(\beta'_j)} \left(a^{(\alpha)}_i\ln(x_2)- a^{(\beta)}_i (x_1+x_2) + b_i\right)(x_1 x_2)^{\alpha_i-1}\right)
\end{align*}

When fixing any $x_2$ such that in the $\mu$-a.e. set such that the preceding equation holds, by Lemma \ref{lem:polexplinind} \ref{item:polexplinindb} for any $j\in [k']$, $\sum_{i\in I(\beta'_j)} a^{(\alpha)}_i (x_1 x_2)^{\alpha_i-1} \equiv 0 $ for any $x_1\neq 0$. Since $\alpha_i$ are distinct for $i\in I(\beta'_j)$ and $x_2>0$, $a_i^{(\alpha)}=0$ for any $i\in I(\beta'_j)$ for any $j\in [k']$. That is $a_i^{(\alpha)}=0$ for any $i\in [k]$. Analogously fixing $x_1$  produces $a_i^{(\beta)}=0$ for any $i\in [k]$. Plug these back into the preceding display and one obtains for $\mu$-a.e. $x_1,x_2\in (0,\infty)$
$$
0=\sum_{j=1}^{k'} \left( \sum_{i\in I(\beta'_j)}  b_i(x_1 x_2)^{\alpha_i-1}\right)e^{-\beta'_j x_2}e^{-\beta'_j x_1}
$$
Fixing any $x_2$ such that in the $\mu$-a.e. set such that the preceding equation holds, and apply Lemma \ref{lem:polexplinind} \ref{item:polexplinindb} again to obtain $b_i=0$ for $i\in [k]$. Thus \eqref{eqn:gamma2foldprodnonor} for any $G\in \Ec_k(\Theta)$ has only the zero solution. By Lemma \ref{cor:expD1equivalenttheta}, for $G_0 \in \Ec_{k_0}(\Theta)$
$$
\liminf_{\substack{G\overset{W_1}{\to} G_0\\ G\in \Ec_{k_0}(\Theta)}} \frac{V(P_{G,2},P_{G_0,2})}{D_1(G,G_0)}>0.
$$ 
Thus $n_1(G_0)\leq 2$, and hence $n_1(G_0) = 2$ for any $G_0\in \Gc$. 

Following an analogous analysis, one can show that $\{f(x|\theta_i)\}_{i=1}^k$ are linear independent for any distinct $\theta_1,\ldots,\theta_k\in \Theta$ for any $k$. The linear independence immediately implies that $p_G$ is identifiable on $\bigcup_{j=1}^\infty \Ec_j(\Theta)$, i.e. for any $G\in \Ec_k(\Theta)$ and any $G'\in \Ec_{k'}(\Theta)$, $P_G\not = P_{G'}$. Thus, $n_0(G,\bigcup_{j=1}^\infty \Ec_j(\Theta))=1$ 
for any $G\in \bigcup_{j=1}^\infty \Ec_j(\Theta)$. \myeoe 
\end{exa}

		The above examples demonstrate the remarkable benefits of having repeated (exchangeable) measurements: via the $\m$-fold product kernel $\prod_{j=1}^{\m} f(x_j|\theta)$ for sufficiently large $\m$, one can completely erase the effect of parameter non-identifiability in Bernoulli mixtures, and the effect of weak-identifiability in the pathological subset of the parameter spaces in two-parameter gamma mixtures. We have also seen that it is challenging to determine the 0- or 1-identifiable lengths even for these simple examples of kernels. It is even more so, when we move to a broader class of probability kernels well beyond the exponential families.

\subsection{General probability kernels}
\label{sec:inverse-general}

Unlike Section~\ref{sec:inverse-exponential}, which specializes to the probability kernels that are in the exponential families, in this section no such assumption will be required. In fact, we shall \emph{not} require that the family of probability distributions $\{P_{\thetave}\}_{\theta \in \Theta}$ on $\Xfrak$ admit a density function. Since the primary object of inference is the parameter $\theta \in \Theta \subset \R^q$, the assumptions on the kernel $P_\theta$ will center on the existence of a measurable map  $T:(\Xfrak, \Acalcal) \to (\R^{s},\mathcal{B}(\R^s))$ for some $s\geq q$, and regularity conditions on the push-forward measure on $\R^s$: $T_{\#}P_\theta := P_\thetave \circ T^{-1}$. The use of the push-forward measure to prove \eqref{eqn:curvatureprodbound} stems from the observation that
the variational distance between two distributions is lower bounded by any push-forward operation, which is equivalent to considering a subclass of the Borel sets in the definition of the variational distance. 

\begin{definition}[Admissible transform] \label{def:admissible}
A Borel measurable map $T:\Xfrak\to \R^s$ is admissible with respect to a set $\Theta_1 \subset \Theta^\circ$  if for each $\theta_0\in \Theta_1$ there exists $\gamma>0$ and $r\geq 1$ such that $T$ satisfies the following three properties. 
\begin{enumerate}[label= ({A}\arabic*)]
	\item 
	\label{item:genthmd} 
	(Moment condition)
	For $\theta\in B(\theta_0,\gamma)\subset \Theta^\circ$, the open ball centered at $\theta_0$ with radius $\gamma$, suppose $\lambda(\theta)=\lambda_\theta=\E_{\thetave}T\Xma_1$ and $\Lambda_{\thetave}:= \E_{\thetave}(T\Xma_1-\E_{\thetave}T\Xma_1)(T\Xma_1-\E_{\thetave}T\Xma_1)^\top $ exist where $X_1\sim \P_\theta$. Moreover, $\Lambda_{\thetave}$ is positive definite on $B(\theta_0,\gamma)$ and is continuous at $\theta_0$.  
	
	\item 
	\label{item:genthme}
	(Exchangeability of partial derivatives of characteristic functions)
	Denote by $\phi_T(\zeta|\theta)$ the characteristic function of the pushforward probability measure $T_{\#}\P_{\theta}$ on $\R^s$, i.e., $\phi_T(\zeta|\theta) := \E_{\theta} e^{i \langle \zeta, TX_1 \rangle}$, where
    $X_1\sim P_\theta$.
	$\frac{\partial \phi_T(\zeta|\thetave)}{\partial \theta^{(i)}}$ exists in $B(\theta_0,\gamma)$ and as a function of $\zeta$ it is twice continuously differentiable on $\R^s$ with derivatives satisfying: $\forall \theta\in B(\theta_0,\gamma)$
	\begin{equation*}
	\frac{\partial^2 \phi_T(\zeta|\thetave)}{\partial \zeta^{(j)}\partial \theta^{(i)}} = \frac{\partial^2 \phi_T(\zeta|\thetave)}{\partial \theta^{(i)}\partial \zeta^{(j)}}, \ \frac{\partial^3 \phi_T(\zeta|\thetave)}{\partial \zeta^{(\ell)}\partial \zeta^{(j)}\partial \theta^{(i)}} = \frac{\partial^3 \phi_T(\zeta|\thetave)}{\partial \theta^{(i)} \partial \zeta^{(\ell)} \partial \zeta^{(j)}}, \forall \zeta\in \R^s,  j,\ell \in [d],  i\in [k_0]
	\end{equation*}
	where the right hand side of both equations exist.
	
	\item 
	\label{item:genthmg}
	(Continuity and integrability conditions of characteristic function)
	$ \phi_T(\zeta|\theta)$ as a function of $\theta$ is twice continuously differentiable in $B(\theta_0,\gamma)$. There hold: for any $i\in [q],j\in [s]$, 
	\begin{equation}
	\label{eqn:uniformboundednew}
	\sup_{\theta\in B(\theta_0,\gamma)}\max\left\{ \sup_{\zeta\in \R^s} \left|\frac{\partial \phi_T(\zeta|\thetave)}{\partial \theta^{(i)}}\right|,  \sup_{ \|\zeta\|_2 < 1 } \left|\frac{\partial^2 \phi_T(\zeta|\thetave)}{\partial \zeta^{(j)}\partial \theta^{(i)}}\right|, \sup_{\|\zeta\|_2< 1} \left|\frac{\partial^2 \phi_T(\zeta|\thetave)}{\partial \theta^{(j)}\partial \theta^{(i)}}\right| \right\}<\infty,
	\end{equation}
	and for any $i,j\in [q]$,
	\begin{equation}
	\label{eqn:integrabilitynew}
	\sup_{\theta\in B(\theta_0,\gamma)} \int_{\R^s}\left| \phi_T(\zeta|\theta)\right|^{r}\left(1+\left|\frac{\partial^2 \phi_T(\zeta|\thetave)}{\partial \theta^{(j)}\partial \theta^{(i)}}\right|\right)    d\zetave  <\infty.
	\end{equation}

\end{enumerate}
\end{definition}

\begin{rem} \label{rem:explanationstodefadmissibletransform}
  The above definition of the admissible transform $T$ contains relatively mild regularity conditions concerning continuity, differentiability and integrability. 
  In particular, \ref{item:genthmd} is to guarantee the first two moments of $TX_1$, which are required for the application of a central limit theorem as outlined in Section \ref{sec:overview}. \ref{item:genthme} and \ref{item:genthmg} are used in the essential technical lemma (Lemma \ref{lem:technical}) to guarantee the following statement in Section \ref{sec:overview}: for any sequence $z_\ell \rightarrow z$, there holds $\Hfun_\ell(z_\ell) \rightarrow \Hfun(z)$ for certain functions $\Hfun_\ell$ and $\Hfun:\R^q \rightarrow \R$. The inequality \eqref{eqn:integrabilitynew} is also used to obtain the existence of the Fourier inversion formula (more specifically, to imply existence of a density of $\sum_{j=1}^N TX_j$ with respect to Lebesgue measure for $N\geq r$). Since the characteristic function has modular less than $1$, the larger the $r$, the smaller the left hand side of \eqref{eqn:integrabilitynew}. Here we only require existence of some $r\geq 1$ in \eqref{eqn:integrabilitynew}, which is a mild condition''. For more discussions on the role of $r$, see Theorem 2 in Section 5, Chapter XV of   \cite{feller2008introduction}. The conditions of the admissible transform are typically straightforward to verify if a closed form formulae of $\phi_T(\zeta|\theta)$ is available; examples will be provided in the sequel. 
  \myeoe
\end{rem}

\begin{thm} \label{thm:genthm}
Fix $G_0=\sum_{i=1}^{k_0}p_i^0\delta_{\theta_i^0}\in \Ec_{k_0}(\Theta^\circ)$. 
Assume that for each $\theta_i^0$, there exists measurable transform
	$T_i:(\Xfrak, \Ac) \to (\R^{s_i},\mathcal{B}(\R^{s_i}))$ that is admissible with respect to $\{\theta_i^0\}_{i=1}^{k}$ with $s_i \geq q$ such that 1) the mean map $\lambda_i(\theta)$ of $T_i$ defined in \ref{item:genthmd} is identifiable at $\theta_i^0$ over the set $\{\theta_i^0\}_{i=1}^{k_0}$, i.e.,  $\lambda_i(\theta_j^0) \not= \lambda_i(\theta_i^0)$ for any $j\in [k_0]\backslash\{i\}$ and 2) the Jacobian matrix of  $\lambda_i$ is of full column rank at $\theta_i^0$.
	Then \eqref{eqn:genthmcon} and \eqref{eqn:curvatureprodbound} hold.
\end{thm}

Note that the condition that $s_i\geq q$ is necessary for the Jacobian matrix of $\lambda_i$, which is of dimension $s_i\times q$, to be of full column rank. The following corollary is useful, when the admissible maps $T_i$ are identical for all $i$, which are the case for many (if not most) examples.

\begin{cor}\label{cor:genthm2ndform}
    Fix $G_0=\sum_{i=1}^{k_0}p_i^0\delta_{\theta_i^0}\in \Ec_{k_0}(\Theta^\circ)$. If there exists one measurable transform
	$T:(\Xfrak, \Ac) \to (\R^{s},\mathcal{B}(\R^{s}))$ that is admissible with respect to $\{\theta_i^0\}_{i=1}^{k_0}$ with $s \geq q$ such that 1) the mean map $\lambda(\theta)$ of $T$ defined in \ref{item:genthmd} is identifiable over the set $\{\theta_i^0\}_{i=1}^{k_0}$, i.e., $\lambda(\theta_j^0) \not= \lambda(\theta_i^0)$ for any distinct $i,j\in [k_0]$ and 2) the Jacobian matrix of  $\lambda$ is of full column rank at $\theta_i^0$ for any $i\in [k_0]$. Then \eqref{eqn:genthmcon} and \eqref{eqn:curvatureprodbound} hold.
\end{cor}

The proofs of Theorem~\ref{thm:expfam} and Theorem~\ref{thm:genthm} contain a number of potentially useful techniques and are deferred to Section~\ref{sec:proofprodinvbou}. We make additional remarks.

\begin{rem}[Choices of admissible transform $T$]
\label{rem:intuitiontochooseT}
If the probability kernel $P_\theta$  has a smooth closed form expression for the characteristic function and $\Xfrak$ is of dimension exceeding the dimension of $\Theta$, one may take $T$ to be identity map (see Example \ref{exa:uniformcontinue} in the sequel). If $\Xfrak$ is of dimension less than the dimension of $\Theta$, then one may take $T$ to be a moment map (see Example \ref{exa:exponentialcontinuerevision} and Example \ref{exa:kernelmixture}). On the other hand, if the probability kernel does not have a smooth closed form expression for the characteristic function, then one may consider $T$ to be the composition of moment maps and indicator functions of suitable subsets of $\Xfrak$ (see Example \ref{exa:dirichlet}). Unlike the three previous examples, the chosen $T$ in Example \ref{exa:dirichlet} depends on atoms $\{\theta_i^0\}_{i=1}^{k_0}$ of $G_0$. All these examples were obtained by constructing a single admissible map $T$ following Corollary \ref{cor:genthm2ndform}. There might exist cases for which it is difficult to come up with a single admissible map $T$ that satisfies the conditions of Corollary \ref{cor:genthm2ndform};
For such cases Theorem \ref{thm:genthm} will be potentially more useful. 
\myeoe
\end{rem}

\begin{rem}[Comparisons between Theorem \ref{thm:genthm} and Theorem \ref{thm:expfam}] 
\label{rem:comparisonsbetweentwotheorems}
     While Theorem \ref{thm:genthm} appears more powerful than Theorem~\ref{thm:expfam}, the latter is significant in its own right. Indeed, Theorem \ref{thm:genthm} provides an inverse bound for a very broad range of probability kernels, but it seems not straightforward to apply it to non-degenerate discrete distributions on lattice points, like Poisson, Bernoulli, geometric distributions etc. The reason is that for non-degenerate discrete distributions on lattice points, its characteristic function is periodic (see Lemma 4 in Chapter XV, Section 1 of \cite{feller2008introduction}), which implies that the characteristic function is not in $L^r$ for any $r\geq 1$. Thus, it does not satisfy  \ref{item:genthmg} for $T$ the identity map in the definition of the admissible transform. In order to apply Theorem \ref{thm:genthm} to such distributions one has to come up with suitable measurable transforms $T$ which induce distributions over a countable support that is not lattice points. On the contrary, Theorem~\ref{thm:expfam} is readily applicable to discrete distributions that are in the exponential family, including Poisson, Bernoulli, geometric distributions, etc.  \myeoe
\end{rem}	


\subsection{Examples of non-standard probability kernels}
\label{sec:non-standard}
The power of Theorem~\ref{thm:genthm} lies in its applicability to classes of kernels that do not belong to the exponential families. 

\begin{exa}[Continuation on uniform probability kernel] \label{exa:uniformcontinue}
In Example \ref{exa:uniform} this example has been shown to satisfy inverse bound \eqref{eqn:noproductlowbou} and \eqref{eqn:curvaturebound} for any $G_0\in \Ec_{k_0}(\Theta)$. Note this family is not an exponential family and thus Theorem \ref{thm:expfam} or Corollary \ref{cor:expfamtheta} is not applicable. Take the $T$ in Corollary \ref{cor:genthm2ndform} to be the identity map. 
Then $\lambda(\theta)=\frac{\theta}{2}$, $\Lambda_\theta= \frac{\theta^2}{12}$. So condition \ref{item:genthmd} is satisfied. The characteristic function is
	$$
	\phi_T(\zeta|\theta)= \frac{e^{\ive\zeta\theta}-1}{\ive\zeta\theta} \bm{1}(\zeta \not = 0) + \1ve(\zeta =0).  
	$$
	One can then calculate
	\begin{align*}
	\frac{\partial}{\partial \theta}\phi_T(\zeta|\theta) =& \frac{e^{\ive \zeta\theta}(e^{-\ive\zeta\theta}-1-(-\ive\zeta\theta))}{\ive\zeta\theta^2} \bm{1}(\zeta\not= 0),\\
	\frac{\partial^2}{\partial \zeta \partial \theta}\phi_T(\zeta|\theta) = & \frac{-e^{\ive\zeta\theta}(e^{-\ive\zeta\theta}-1-(-\ive\zeta\theta)-(-\ive\zeta\theta)^2)}{\ive \zeta^2\theta^2}\bm{1}(\zeta\not=0) + \frac{\bm{i}}{2}\bm{1}(\zeta=0),\\
	\frac{\partial^2}{\partial  \theta^2}\phi_T(\zeta|\theta) = & \frac{-2e^{\ive\zeta\theta}(e^{-\ive\zeta\theta}-1-(-\ive\zeta\theta)-\frac{1}{2}(-\ive\zeta\theta)^2)}{\ive \zeta\theta^3}\bm{1}(\zeta\not=0),
	\end{align*}
	and verify the condition \ref{item:genthme}. To verify \ref{item:genthmg} the following inequality (see \cite[(9.5)]{resnick2014probability}) 
	$$
	\left|e^{\ive x} - \sum_{k=0}^{j}\frac{(\ive x)^k}{k!}\right|\leq 2\frac{|x|^j}{j!}
	$$
	comes handy. It then follows that
	\begin{align*}
	\left|\frac{\partial}{\partial \theta}\phi_T(\zeta|\theta)\right| \leq  \frac{2}{\theta},\quad \left|\frac{\partial^2}{\partial \zeta \partial \theta}\phi_T(\zeta|\theta)\right| \leq  \frac{3}{2},\quad \left|\frac{\partial^2}{\partial  \theta^2}\phi_T(\zeta|\theta)\right| \leq  \frac{2|\zeta|}{\theta}.
	\end{align*}
	Then 
	\eqref{eqn:uniformboundednew} holds. Finally take $r=3$ and one obtains
	$$
	|\phi_T(\zeta|\theta)|^3\left(1+\left|\frac{\partial^2}{\partial  \theta^2}\phi_T(\zeta|\theta)\right|\right)\leq \begin{cases} 1+\frac{2}{\theta} & |\zeta|\leq 1 \\ \frac{8}{|\zeta|^3\theta^3}\left(1+\frac{2|\zeta|}{\theta}\right) & |\zeta|>1 \end{cases}.
	$$
	Thus \eqref{eqn:integrabilitynew} holds. 
	We have then verified that the identity map $T$ is admissible on $\Theta$.
	
	It is easy to see that $\lambda(\theta)=\theta/2$ is injective and that its Jacobian  $J_\lambda(\theta)=\frac{1}{2}$ is full rank. Then by Corollary \ref{cor:genthm2ndform}, \eqref{eqn:genthmcon} and \eqref{eqn:curvatureprodbound} hold for any $G_0\in \Ec_{k_0}(\Theta)$ for any $k_0\geq 1$. Moreover, by Remark \ref{rem:N1Ninfcon}, $n_1(G_0)=\underbar{\m}_1(G_0) = n_2(G_0) = \underbar{\m}_2(G_0) = 1$ for any $G_0\in \Ec_{k_0}(\Theta)$ for any $k_0\geq 1$. 
	\myeoe
\end{exa}

\begin{exa}[Continuation on location-scale exponential kernel] \label{exa:exponentialcontinuerevision} 
	In Example \ref{exa:locscaexp} this example has been shown to satisfy \eqref{eqn:noproductlowbou} for any $G_0\in \Ec_{k_0}(\Theta)$ and does not satisfy \eqref{eqn:curvaturebound} for some $G_0\in \Ec_{k_0}(\Theta)$ for any $k_0\geq 2$. Note this family is not an exponential family and thus Theorem \ref{thm:expfam} or Corollary \ref{cor:expfamtheta} is not applicable. Take $T$ in Corollary \ref{cor:genthm2ndform} to be $Tx=(x,x^2)^\top $ as a map from $\R\to \R^2$. In Appendix \ref{sec:detailexponentialexa} we show that all conditions of Corollary \ref{cor:genthm2ndform} are satisfied and hence \eqref{eqn:genthmcon} and \eqref{eqn:curvatureprodbound} hold for any $G_0\in \Ec_{k_0}(\Theta)$ for any $k_0\geq 1$. 
	Moreover, by Remark \ref{rem:N1Ninfcon}, $n_1(G_0)=\underbar{\m}_1(G_0)=1$ for any $G_0\in \Ec_{k_0}(\Theta)$ for any $k_0\geq 1$. Regarding $n_2(G_0)$, for every $k_0\geq 2$, there exists some $G_0\in \Ec_{k_0}(\Theta)$ such that $1<n_2(G_0)<\infty$.
	\myeoe
\end{exa}  

\begin{exa}[$P_\theta$ is itself a mixture distribution]
\label{exa:kernelmixture}
We consider the situation where $P_\theta$ is a rather complex object: it is itself a mixture distribution. With this we are moving from a standard mixture of product distributions to hierarchical models (i.e., mixtures of mixture distributions). Such models are central tools in Bayesian statistics. Theorem \ref{thm:expfam} or Corollary \ref{cor:expfamtheta} is obviously not applicable in this example, which requires the full strength of Theorem \ref{thm:genthm} or Corollary \ref{cor:genthm2ndform}. The application, however, is non-trivial requiring the development of tools for evaluating oscillatory integrals of interest. Such tools also prove useful in other contexts (such as Example~\ref{exa:dirichlet}). A full treatment is deferred to Section~\ref{sec:mixofmix}. \myeoe
\end{exa}

\begin{exa}[$P_\theta$ is a mixture of Dirichlet processes]
\label{exa:dirichlet} 
This example illustrates the applicability of our theory to models using probability kernels defined in abstract spaces. Such kernels are commonly found in nonparametric Bayesian literature~\cite{hjort2010bayesian,ghosal2017fundamentals}. In particular, in our specification of mixture of product distributions we will employ Dirichlet processes as the basic building block~\cite{ferguson1973bayesian,antoniak1974mixtures}. Full details are presented in Section~\ref{sec:dirichlet}.
\myeoe
\end{exa}
	
\section{Posterior contraction of de Finetti mixing measures}\label{sec:poscon}
	
	\label{sec:variable-sized}
	The data are $\n$ independent sequences of exchangeable observations, $X_{[\m_i]}^i = (X_{i1},X_{i2},\cdots,X_{i\m_i}) \in \Xfrak^{\m_i}$ for $i\in[\n]$. Each sequence $X_{[\m_i]}^i$ is assumed to be a sample drawn from a mixture of $\m_i-$product distributions $P_{G,\m_i}$ for some "true" de Finetti mixing measure $G=G_0 \in \Ec_{k_0}(\Theta)$. 
	The problem is to estimate $G_0$ given the $m$ independent exchangeable sequences.
	A Bayesian statistician endows $(\Ec_{k_0}(\Theta), \Bc(\Ec_{k_0}(\Theta)))$ with a prior distribution $\Pi$ and obtains the posterior distribution $\Pi(dG| X^1_{[\m_1]},\ldots, X^{\n}_{[\m_m]})$ by Bayes' rule, where $\Bc(\Ec_{k_0}(\Theta))$ is the Borel sigma algebra w.r.t. $D_1$ distance. In this section we study the asymptotic behavior of this posterior distribution as the amount of data $\n\times\m$ tend to infinity. 
	
	 Suppose throughout this section,  $\{P_{\theta}\}_{\theta\in \Theta}$ has density $\{f(x|\theta)\}_{\theta\in \Theta}$ w.r.t. a $\sigma$-finite dominating measure $\mu$ on $\Xfrak$; 
	 then $P_{G,\m_i}$ for $G=\sum_{i=1}^{k_0}p_i\delta_{\theta_i}$ has density w.r.t. to $\mu$:
	\begin{equation}
	p_{G,\m_i}(\bar{x}) = \sum_{i=1}^{k_0} p_i\prod_{j=1}^{\m_i} f(x_j|\theta_i), \quad \text{for } \bar{x}=(x_1,x_2,\cdots,x_{\m_i})\in \Xfrak^{\m_i}.  \label{eqn:pGNdef}
   \end{equation}
	Then the density of $X_{[\m_i]}^i$ conditioned on $G$ is $p_{G,\m_i}(\cdot)$.
	Since $\Theta$ as a subset of $\R^q$ is separable, $\Ec_{k_0}(\Theta)$ is separable. Moreover, suppose the map $\theta\mapsto \P_\theta$ from $(\Theta,\|
	\cdot\|_2)$ to $(\{\P_\theta\}_{\theta\in \Theta}, h(\cdot,\cdot))$ is  continuous, where $h(\cdot,\cdot)$ is the Hellinger distance. Then the map $G\mapsto P_{G,N}$ from $ (\Ec_{k_0}(\Theta),D_1) \to (p_{G,\m},h(\cdot,\cdot))$ is also continuous by Lemma \ref{lem:hellingeruppbou}. Then by \cite[Lemma 4.51]{guide2006infinite}, $(x,G)\mapsto p_{G,\m}(x)$ is measurable for each $\m$. 
	Thus, the posterior distribution (a version of regular conditional distribution) is the random measure given by 
	\begin{equation}
	\Pi(B|X_{[\m_1]}^1,\ldots,X_{[\m_{\n}]}^{\n}) = \frac{\int_B\prod_{i=1}^\n p_{G,\m_i}(X_{[\m_i]}^i)d\Pi(G)}{\int_{\Ec_{k_0}(\Theta)}\prod_{i=1}^\n p_{G,\m_i}(X_{[\m_i]}^i)d\Pi(G)}, \label{eqn:bayesrule}
	\end{equation}
	for any Borel measurable subset of $B \subset \Ec_{k_0}(\Theta)$.	
	For further details of why the last quantity is a valid posterior distribution, we refer to Section 1.3 in \cite{ghosal2017fundamentals}. 
	It is customary to express the above model equivalently in the hierarchical Bayesian fashion: \begin{align*}
	&G \sim \Pi,  \quad \theta_1,\theta_2,\cdots,\theta_\n|G \overset{i.i.d.}{\sim} G\\
	&X_{i1},X_{i2},\cdots,X_{i\m_i} |\theta_i  \overset{i.i.d.}{\sim} f(x|\theta_i) \quad \text{for }i=1,\cdots,\n.
	\end{align*}
	As above, the $\n$ independent data sequences are denoted by $X_{[\m_i]}^i = (X_{i1},\cdots,X_{i\m_i}) \in \Xfrak^{\m_i }$ for $i\in[\n]$. 
The following assumptions are required for the main theorems of this section.
	
	\begin{enumerate}[label=(B\arabic*)]
	
	\item \label{item:prior} (Prior assumption)  There is a prior measure $\Pi_{\theta}$ on $\Theta_1\subset \Theta$ with its Borel sigma algebra possessing a density w.r.t. Lebesgue measure that is bounded away from zero and infinity, where $\Theta_1$ is a compact subset of $\Theta$. Define $k_0$-probability simplex $\Delta^{k_0}:=\{(p_1,\ldots,p_{k_0}) \in \R_+^{k_0} | 
	\sum_{i=1}^{k_0} = 1\}$, Suppose there is a prior measure $\Pi_p$ on the $k_0$-probability simplex possessing a density w.r.t. Lebesgue measure on $\R^{k_0-1}$ that is bounded away from zero and infinity. Then $\Pi_p\times \Pi_{\theta}^{k_0}$ is a measure on $\{((p_1,\theta_1),\ldots,(p_{k_0},\theta_{k_0}))|p_i\geq 0, \theta_i\in \Theta_1, \sum_{i=1}^{k_0}p_i=1\}$, which induces a probability measure on $\Ec_{k_0}(\Theta_1)$. 
	Here, the prior $\Pi$ is generated by via independent $\Pi_p$ and $\Pi_\theta$ and the support $\Theta_1$ of $\Pi_\theta$ is such that $G_0\in \Ec_{k_0}(\Theta_1)$.


	\item \label{item:kernel} (Kernel assumption) Suppose that for every $\theta_1,\theta_2\in \Theta_1$, $\theta_0\in \{\theta_i^0\}_{i=1}^{k_0}$ and some positive constants $\alpha_0,L_1, \beta_0, L_2$,
	 \begin{align}
	 K(f(x|\theta_0),f(x|\theta_2))\leq &L_1\|\theta_0-\theta_2\|_2^{\alpha_0}, \label{eqn:KLlipschitz} \\
	 h(f(x|\theta_1),f(x|\theta_2))\leq & L_2\|\theta_1-\theta_2\|_2^{\beta_0}. \label{eqn:hellingerlipschitz}
	 \end{align}
	\end{enumerate}
	
	\begin{rem} \label{rem:Lippowerbou}
	\ref{item:prior} on the compactness of the support $\Theta_1$ is a standard assumption so as to obtain parameter convergence rate in finite mixture models (see~\cite{chen1995optimal,nguyen2013convergence,ho2016convergence,ho2016strong,ghosal2017fundamentals,heinrich2018strong, wu2020optimal}).
	See Section \ref{sec:compactness} for more discussions on the compactness assumption and an relaxation to boundedness assumption. The unbounded setting seems challenging and is beyond the scope of this paper. In this paper, for simplicity we consider the prior on finite mixing measures being generated by independent priors on component parameters $\theta$ and an independent prior on mixing proportions $p_i$ \cite{Rousseau-Mengersen-11,nguyen2013convergence,guha2019posterior} for general probability kernel $\{f(x|\theta)\}$. It is not difficult to extend our theorem to more complex forms of prior specification when a specific kernel $\{f(x|\theta)\}$ is considered. 

	The condition \ref{item:kernel} is not uncommon in parameter estimation (e.g. Theorem 8.25 in \cite{ghosal2017fundamentals}). Note that the conditions in \ref{item:kernel} imply some implicit constraints on $\alpha_0$ and $\beta_0$. Specifically,
	if \eqref{eqn:noproductlowbou} holds for $G_0\in \Ec_{k_0}(\Theta^\circ)$ and \ref{item:kernel} holds, then $\beta_0\leq 1$ and $\alpha_0\leq 2$. Indeed, for any sequence $G_\ell=\sum_{i=2}^{k_0} p^0_i\delta_{\theta_i^0} + p^0_1\delta_{\theta^\ell_1}\in\Ec_{k_0}(\Theta)\backslash \{G_0\}$ converges to $G_0=\sum_{i=1}^{k_0}p^0_i\delta_{\theta_i^0}$, by \eqref{eqn:noproductlowbou}, Lemma \ref{lem:Vupperbou} with $\m=1$ and \ref{item:kernel}, for large $\ell$
	\begin{multline} 
	C(G_0)\|\theta_1^\ell-\theta_1^0\|_2 = C(G_0)D_1(G_\ell,G_0) \leq V(P_{G_\ell},P_{G_0})\\ \leq V(f(x|\theta^\ell_1),f(x|\theta^0_1)) \leq h(f(x|\theta^\ell_1),f(x|\theta^0_1))\leq  L_2\|\theta_1^\ell-\theta^0_1\|^{\beta_0}_2,\label{eqn:hbeta0bou}
	\end{multline}
	which implies $\beta_0\leq 1$ (by dividing both sides by $\|\theta_1^\ell-\theta_1^0\|_2$ and letting  $\ell\to\infty$). In the preceding display $C(G_0)=\frac{1}{2}\liminf_{\substack{G\overset{W_1}{\to} G_0\\ G\in \Ec_{k_0}(\Theta)}} \frac{V(p_G,p_{G_0})}{D_1(G,G_0)}>0$.
	By \eqref{eqn:hbeta0bou} and Pinsker's inequality, 
	$$
	C(G_0)\|\theta_1^\ell-\theta_1^0\|_2\leq V(f(x|\theta^\ell_1),f(x|\theta^0_1))\leq \sqrt{\frac{1}{2}K(f(x|\theta^0_1),f(x|\theta^\ell_1))} \leq \sqrt{\frac{1}{2}L_1 \|\theta_1^\ell-\theta^0_1\|^{\alpha_0}_2},
	$$ for large $\ell$,
	which implies $\alpha_0\leq 2$. 
	The same conclusion holds if one replaces \eqref{eqn:noproductlowbou} with \eqref{eqn:genthmcon} by
	an analogous argument. 
	\myeoe
	\end{rem}

An useful quantity is the average sequence length $\bar{\m}_{\n}=\frac{1}{\n}\sum_{i=1}^{\n}\m_i$. The posterior contraction theorem will be characterized in terms of distance $D_{\bar{\m}_{\n}}(\cdot, \cdot)$, which extends the original notion of distance $D_\m(\cdot,\cdot)$ by allowing real-valued weight $\bar{\m}_{\n}$. 

	\begin{thm}\label{thm:posconnotid} 
	Fix $G_0\in\Ec_{k_0}(\Theta^\circ)$. Suppose  \ref{item:prior}, 
		\ref{item:kernel} and additionally \eqref{eqn:genthmcon} hold. 
	\begin{enumerate}[label=\alph*)]
	  \item \label{item:posconnotida} 
	  There exists some constant $C(G_0)>0$ such that as long as $n_0(G_0,\cup_{k\leq k_0}\Ec_k(\Theta_1)) \vee n_1(G_0) \leq \min_{i} \m_i \leq \sup_{i} \m_i < \infty$, for every $\bar{M}_\n\to \infty$ there holds
	\begin{align*}
	&\Pi \biggr (G\in \Ec_{k_0}(\Theta_1): D_{\bar{\m}_{\n}}(G,G_0) \geq C(G_0)\bar{M}_\n \sqrt{ \frac{\ln(\n\bar{\m}_{\n})}{\n} } \biggr |X_{[\m_1]}^1, \ldots,X_{[\m_{\n}]}^{\n} \biggr ) \to  0
	\end{align*}
	in $\bigotimes_{i=1}^{m}\P_{G_0,\m_i}$-probability as $\n\to \infty$.
	
	\item \label{item:posconnotidb}
	If in addition, ~\eqref{eqn:noproductlowbou} is satisfied. Then the claim in part \ref{item:posconnotida} holds with $n_1(G_0) = 1$.
	\end{enumerate}
	\end{thm}

\begin{rem} 
As discussed in Remark \ref{rem:claytonresult}, \cite{vandermeulen2019operator} establishes that $ n_0(G_0,\cup_{k\leq k_0}\Ec_k(\Theta))\leq 2k_0 -1$ for any $G_0\in \Ec_{k_0}(\Theta)$. While the uniform upper bound $2k_0-1$ might not be tight, it does show  $n_0(G_0,\cup_{k\leq k_0}\Ec_k(\Theta_1))<\infty$. By Proposition~\ref{lem:n0nbar} \ref{item:n1Nbar}, $n_1(G_0) < \infty$ is a direct consequence of \eqref{eqn:genthmcon}. Hence $n_0(G_0,\cup_{k\leq k_0}\Ec_k(\Theta_1)) \vee n_1(G_0) <\infty$.
\myeoe
    
\end{rem}

\begin{rem} \label{rem:consequenceofposcon}
	\label{rem:posconnotiid}
	 (a) In the above statement, note that the constant $C(G_0)$ also depends on $\Theta_1$, $k_0$, $q$, upper and lower bounds of the densities of $\Pi_\theta$, $\Pi_p$ and the density family $f(x|\theta)$ (including $\alpha_0$, $\beta_0$, $L_1$, $L_2$ etc). All such dependence are suppressed for the sake of a clean presentation; it is the dependence on $G_0$ and the \emph{independence} of $\n,\{\m_i\}_{i\geq 1}$ and $\m_0:=\sup_{i}\m_i < \infty$, that we want to emphasize. In addition, although $C(G_0)$ and hence the vanishing radius of the ball characterized by $D_{\bar{\m}_\n}$ does \emph{not} depend on $\m_0$, the rate at which the posterior probability statement concerning this ball tending to zero may depend on it.
	
	(b) Roughly speaking, the theorem produces the following posterior contraction rates. The rate toward mixing probabilities $p_i^0$ is $O_P((\ln(\n\bar{\m}_\n)/\n)^{1/2})$. 
   Individual atoms $\theta_i^0$ have much faster contraction rate, which utilizes the full volume of the data set: 
   \begin{equation}
   O_P\left(\sqrt{\frac{\ln(\n\bar{\m}_\n)}{\n\bar{\m}_\n}}\right) = 
   O_P\biggr (\sqrt{\frac{\ln (\sum_{i=1}^{\n} \m_i)}{\sum_{i=1}^{\n} \m_i}} \biggr). \label{eqn:atomconvergencerate}
   \end{equation}
	Note that the condition that $\sup_i N_i<\infty$ implies that $\bar{N}_m$ remains finite when $m\to \infty$. Since constant $C(G_0)$ is independent of $\bar{N}_m$ and $m$, the theorem establishes that the larger the average length of observed sequences, the faster the posterior contraction as $m\to \infty$. 
	
	(c) The distinction between the two parts of the theorem highlights the role of first-order identifiability in mixtures of $\m$-product distributions. Under first-order identifiability, ~\eqref{eqn:noproductlowbou} is satisfied, so we can establish the aforementioned posterior contraction behavior for a full range of sequence length $\m_i$'s, as long as they are uniformly bounded by an arbitrary unknown constant. When first-order identifiability is not satisfied, so~\eqref{eqn:noproductlowbou} may fail to hold, the same posterior behavior can be ascertained when the sequence lengths exceed certain threshold depending on the true $G_0$, namely, $n_1(G_0)\vee n_0(G_0,\cup_{k\leq k_0}\Ec_k(\Theta_1))$.
	
	(d) The proof of Theorems~\ref{thm:posconnotid}  utilizes general techniques of Bayesian asymptotics (see  \cite[Chapter 8]{ghosal2017fundamentals}) to deduce posterior contraction rate on the   density $h(P_{G,N},P_{G_0,N})$. The main novelty lies in the 
	 application of 
	 inverse bounds for mixtures of product distributions of exchangeable sequences established in Section~\ref{sec:inversebounds}. These are lower bounds of distances  $h(P_{G_0,\m},P_{G,\m})$ between a pair of distributions ($P_{G_0,\m}, P_{G,\m}$) in terms of distance $D_{\m}(G_0,G)$  between the corresponding $(G_0,G)$. The distance $D_{\m}(G_0,G)$ brings out the role of the sample size $\m$ of exchangeable sequences, resulting in the rate $\m^{-1/2}$ (or $\bar{N}_m^{-1/2}$, modulo the logarithm). 
\myeoe
	\end{rem}

The gist of the proof of Theorem \ref{thm:posconnotid} lies in the following lemma where we consider the equal length data sequence to distill the essence. This lemma also illustrates the connection between the inverse bound \eqref{eqn:genthmcon} and the convergence rate for the mixing measure $G_0$.

\begin{lem} \label{lem:convergencerate}
Consider $N\geq n_1(G_0)$ be fixed and suppose \eqref{eqn:genthmcon} holds. Let $X^i_{[N]} \overset{\operatorname{i.i.d.}}{\sim} p_{G_0,N}$ and let $\Pi$ be a prior distribution on $\Ec_{k_0}(\Theta)$. Suppose the posterior contraction rate towards the true mixture density $p_{G_0,N}$ is $\epsilon_{m,N}$: for any $\bar{M}_m\to \infty$,  $\Pi(V(P_{G,N},P_{G_0,N})\geq \bar{M}_m \epsilon_{m,N}| X^1_{[N]},\ldots, X^m_{[N]} )\to 0$ in probability as $m\to \infty$. Suppose the posterior consistency at the true mixing measure $G_0$ w.r.t. the distance $W_1$ holds: for any $a>0$, $\Pi(W_1(G,G_0) \geq a | X^1_{[N]},\ldots, X^m_{[N]} )\to 0$ in probability as $m\to \infty$. Then the posterior contraction rate to $G_0$ w.r.t. distance $D_N$ is  $\epsilon_{m,N}$, i.e. $\Pi(D_N(G,G_0)\geq \bar{M}_m \epsilon_{m,N}| X^1_{[N]},\ldots, X^m_{[N]} )\to 0$ in probability as $m\to \infty$.
\end{lem}
\begin{proof} All probabilities presented in this proof are posterior probabilities conditioning on the data $X^1_{[N]},\ldots, X^m_{[N]}$, of which the conditioning notation are suppressed for brevity.   
    \begin{align*}
    &\Pi(D_N(G,G_0)\geq \bar{M}_m \epsilon_{m,N})\\
    \leq 
    &\Pi(D_N(G,G_0)\geq \bar{M}_m \epsilon_{m,N}, W_1(G,G_0)<c(G_0,N) ) + \Pi( W_1(G,G_0)\geq c(G_0,N)  ) \\
    \leq 
    & \Pi(V(P_{G,N},P_{G_0,N})\geq C(G_0)\bar{M}_m \epsilon_{m,N}) + \Pi( W_1(G,G_0)\geq c(G_0,N)  ), \numberthis \label{eqn:convergencerate1}
    \end{align*}
    where in the first inequality $c(G_0,N)$ is the radius in Lemma \ref{cor:identifiabilityallm} with $N_0=N$, and the second inequality follows by Lemma \ref{cor:identifiabilityallm}. The proof is completed by noticing that the quantity in \eqref{eqn:convergencerate1} 
    converges to $0$ in probability as $m\to\infty$ by the hypothesises for any $\bar{M}_m\to \infty$. 
\end{proof}

\begin{rem} \label{rem:convergencerate}
Roughly speaking, the hypothesis of posterior consistency guarantees that as $m\to\infty$, $G$ lies in a small ball around $G_0$ w.r.t. $W_1$ distance, and then Lemma \ref{cor:identifiabilityallm} transfers the convergence rate from mixture densities to mixing measures. No particular structures from posterior distributions are utilized and one can easily modify the above lemma for other estimators, the maximum likelihood estimator for instance.

Theorem \ref{thm:posconnotid} can be seen as some sufficient conditions on the prior $\Pi$ and the kernel $f$ such that the hypothesises in Lemma \ref{lem:convergencerate} are satisfied. The setup in Theorem \ref{thm:posconnotid} is slightly more general, in that each data might have different sequence length. \myeoe
\end{rem}

Finally, the conditions \ref{item:kernel} and \eqref{eqn:genthmcon} can be verified for full rank exponential families and hence we have the following corollary from Theorem \ref{thm:posconnotid}. 

\begin{cor}  \label{cor:posconexp} 
Consider a full rank exponential family for kernel $P_\theta$ specified as in Corollary \ref{cor:expfamtheta} and assume all the requirements there are met. Fix $G_0\in \Ec_{k_0}(\Theta^\circ)$. Suppose that $\ref{item:prior}$ holds with $\Theta_1\subset \Theta^\circ$. Then the conclusions \ref{item:posconnotida}, \ref{item:posconnotidb} of Theorem \ref{thm:posconnotid} hold.
\end{cor}

\begin{exa}[Posterior contraction for weakly identifiable kernels: Bernoulli and gamma] 
Fix $G_0\in \Ec_{k_0}(\Theta^\circ)$.  
For the Bernoulli kernel studied in Example \ref{exa:bernoulli}, $n_1(G_0)=n_0(G_0,\cup_{k\leq k_0}\Ec_k(\Theta))=2k_0-1$. Suppose that $\ref{item:prior}$ holds with compact $\Theta_1\subset \Theta^\circ=(0,1)$. Then by Corollary \ref{cor:posconexp}, the conclusion \ref{item:posconnotida} of Theorem \ref{thm:posconnotid} holds provided $\min_i N_i \geq 2k_0-1$. 
For the gamma kernel studied in Examples \ref{exa:gamma} and \ref{exa:gamma2}, $n_1(G_0) = 2$ when $G_0\in \Gc$ and $n_1(G_0)=1$ when $G_0\in \Ec_{k_0}(\Theta^\circ)\backslash \Gc$; $n_0(G_0,\cup_{k\leq k_0}\Ec_k(\Theta))=1$. Suppose that $\ref{item:prior}$ holds with compact $\Theta_1\subset \Theta^\circ=(0,\infty)\times (0,\infty)$. 
Then by Corollary \ref{cor:posconexp}, the conclusion \ref{item:posconnotida} of Theorem \ref{thm:posconnotid} holds provided $\min_i N_i\geq 2$. Moreover, no requirement on $\min_i N_i$ is needed if $G_0\not \in \Gc$ is given. \myeoe
\end{exa}

\begin{exa}[Posterior contraction for weakly identifiable kernels: beyond exponential family] \label{exa:posnonexponential}
Here we present the posterior contraction rate for the four examples studied in Section \ref{sec:non-standard}, while the verification details are in Appendix \ref{sec:detailsinposconratenonexp}. Assume that the prior distribution satisfies the \ref{item:prior} for each example below. For the uniform probability kernel studied in Example \ref{exa:uniformcontinue}, the conclusion of Theorem \ref{thm:posconnotid} holds for any $N\geq 1$. For the location-scale exponential kernel studied in Example \ref{exa:exponentialcontinuerevision}, the conclusion of Theorem \ref{thm:posconnotid} holds for any $N\geq 1$. For the case that kernel is location-mixture Gaussian in Example \ref{exa:kernelmixture}, the conclusion of Theorem \ref{thm:posconnotid} holds and the specific values of $n_0(G_0,\cup_{k\leq k_0}\Ec_k(\Theta_1))$ and $n_1(G_0)$ are left as exercises. The kernel in Example \ref{exa:dirichlet} does not possess a density, which is needed in \eqref{eqn:bayesrule}, and thus the results in this section on posterior contraction do not apply. \myeoe
\end{exa}

\section{Hierarchical model: kernel $P_\theta$ is itself a mixture distribution}
\label{sec:mixofmix}

In this section we apply Theorem~\ref{thm:genthm} to the cases where $\P_\theta$ itself is a rather complex object: a finite mixture of distributions. 
Combining this kernel with a discrete mixing measure $G\in \Ec_{k_0}(\Theta)$, the resulting $P_{G}$ represents a mixture of finite mixtures of distributions, while $P_{G,\m}$ becomes a $k_0$-mixture of $\m$-products of finite mixtures of distributions. These recursively defined objects represent a popular and formidable device in the statistical modeling world: the world of hierarchical models.
We shall illustrate Theorem~\ref{thm:genthm} on only two examples of such models. However, the tools required for these applications are quite general, chief among them are bounds on relevant oscillatory integrals for suitable statistical maps $T$. 
We shall first describe such tools in Section \ref{sec:fourieranalysis} and then address the case
$P_\theta$ is a $k$-component Gaussian location mixture (Example~\ref{exa:kernelmixture}) and the case $P_\theta$ is a mixture of Dirichlet processes (Example~\ref{exa:dirichlet}).

\subsection{Bounds on oscillatory integrals}
\label{sec:fourieranalysis}
A key condition in Theorem~\ref{thm:genthm}, namely condition~\ref{item:genthmg}, is reduced to the $L^r$ integrability of certain oscillatory integrals:
\begin{equation}
\left\|\int_{\Xfrak} e^{\i\zeta^\top  Tx }f(x)dx\right\|_{L^r(\R^s)}  \label{eqn:fouriertransformcurves}
\end{equation}
 for a broad class of functions $f:\Xfrak\rightarrow \R$ and multi-dimensional maps $T:\Xfrak \rightarrow \R^s$. When $\Xfrak = \R^d$, the oscillatory integral $\int_{\Xfrak} e^{\i\zeta^\top  Tx }f(x)dx $ is also known as the Fourier transform of measures supported on
curves or surfaces; bounds for such quantities are important topics in harmonic analysis and geometric analysis. We refer to \cite{brandolini2007average} and the textbook \cite[Chapter 8]{stein1993harmonic} for further details and broader contexts. Despite there are many existing results, such results are typically established when $f(x)$ is supported on a compact interval or is smooth, i.e. $f$ has derivative of arbitrary orders. We shall develop an upper bound on \eqref{eqn:fouriertransformcurves} for our purposes to verify the integrability condition in \ref{item:genthmg} for a broad class of $f$, which is usually satisfied for probability density functions.  

We start with the following bounds for oscillatory integrals of the form $\int e^{\i \lambda \phi(x)} \psi(x) dx$, where function $\phi$ is called the phase, and function $\psi$ the amplitude.

\begin{lem}[van der Corput's Lemma] \label{lem:corput}
	Suppose $\phi(x)\in C^{\infty}(a,b)$, and that $|\phi^{(k)}(x)|\geq 1$ for all $x\in(a,b)$. Let $\psi(x)$ be absolute continuous on $[a,b]$. Then 
	$$
	\left|\int_{[a,b]}e^{i\lambda\phi(x)} \psi(x) dx \right| \leq c_k \lambda^{-\frac{1}{k}}\left[|\psi(b)|+\int_{[a,b]}|\psi'(x)|dx \right]
	$$
	and
	$$
	\left|\int_{[a,b]}e^{i\lambda\phi(x)} \psi(x) dx \right| \leq c_k \lambda^{-\frac{1}{k}}\left[|\psi(a)|+\int_{[a,b]}|\psi'(x)|dx \right]
	$$
	hold when either i)
		$k\geq 2$, or ii)
		$k=1$ and $\phi'(x)$ is monotonic.
    The constant $c_k$ is independent of $\phi$, $\psi$, $\lambda$ and the interval $[a,b]$.
\end{lem}
\begin{proof}
	See \cite[the Corollary on Page 334]{stein1993harmonic} for the proof of the first display; even though in its original version in this reference, $\psi$ is assumed to be $C^\infty$ but its proof only needs $\psi$ to be absolute continuous on $[a,b]$. The second display follows by applying the first display to $\tilde{\psi}(x) = \psi(a+b-x)$.
\end{proof}

It can be observed from Lemma~\ref{lem:corput} the condition on derivatives of the phase function plays a crucial role. For our purpose the phase function will be supplied by use of monomial map $T$. Hence, the following technical lemma will be needed.
\begin{lem} \label{eqn:smallsigularvalueA(x)}
	Let $A(x)\in \R^{d\times d}$ with entries $A_{\alpha\beta}(x)=0$ for $\alpha>\beta$ and $A_{\alpha\beta}(x)=\frac{j_\beta!}{(j_\beta-j_\alpha)! }x^{j_\beta-j_\alpha}$ for $1\leq \alpha \leq \beta \leq d$, where $1\leq j_1 < \ldots < j_d$ are given. Let $S_\text{min}(A(x))$ be the smallest singular value of $A(x)$. Then $S_\text{min}(A(x))\geq c_3 \max\{1,|x|\}^{-(j_d-j_1)(d-1)} $, where $c_3$ is a constant that depends only on $d,j_1,\ldots,j_d$.
\end{lem}

The following lemma provides a crucial uniform bound on oscillatory integrals given by a phase given by monomial map $T$.

\begin{lem} \label{lem:radialupp}
	Let $T:\R \to \R^d $ defined by $Tx=(x^{j_1},x^{j_2},\ldots,x^{j_d})^\top $ with $1\leq j_1 <j_2 <\ldots j_d$. Consider a bounded non-negative function $f(x)$ that is differentiable on $\R \backslash \{b_i\}_{i=1}^\ell$, where $b_1<b_2<\ldots<b_\ell$ with $\ell$ a finite number. The derivative $f'(x)\in L^1(\R)$ and it is continuous when it exists. Moreover, $f(x)$ and $|x|^{{\alpha_1}}f(x)$ are both increasing when $x<-c_1$ and decreasing when $x>c_1$ for some $c_1\geq \max\{ |b_1|, |b_\ell| \} $, where $\alpha_1=(j_d-j_1)(d-1)/j_1$. Then for $\lambda>1$,
	\begin{align*} 
	&\sup_{w\in S^{d-1}}\left|\int_\R \exp(\bm{i}\lambda w^\top  Tx  ) f(x) dx\right| \\
	\leq & C_1 \lambda^{-\frac{1}{j_d}} (c_1+2)^{\alpha_1}\left( \left\||x|^{\alpha_1}f(x)\right\|_{L^1(\R)} +(\ell+1)\|f\|_{L^\infty(\R)} +\left\|\left(|x|^{\alpha_1}+1\right)f'(x)\right\|_{L^1(\R)}\right),
	\end{align*}
	 where $C_1$ is a positive constant that only depends on $d,j_1,j_2,\ldots,j_d$. 
\end{lem}

Applying Lemma~\ref{lem:radialupp} we obtain a bound for the oscillatory integral in question. 

\begin{lem} \label{lem:charintegrability}
	Let $T$ and $f$ satisfy the same conditions as in Lemma \ref{lem:radialupp}. Define
	$g(\zeta) = \int_{\R} e^{\i\zeta^\top  Tx }f(x)dx$ for $\zeta \in \R^d$.
	Then for $r>dj_d$,
	\begin{align*}
	& \|g(\zeta)\|_{L^r(\R^d)} \\
	\leq & C_2 (c_1+2)^{{\alpha_1 }}  (\||x|^{\alpha_1}f(x) \|_{L^1(\R)} +(\ell+1)\|f\|_{L^\infty(\R)} +\|(|x|^{\alpha_1}+1)f'(x)\|_{L^1(\R)}+\|f\|_{L^1(\R)}
	 )
	\end{align*}
	where $C_2$ is a positive constant that depends on $r,d,j_1,j_2,\ldots,j_d$.
\end{lem}

\subsection{Kernel $P_\theta$ is a location mixture of Gaussian distributions}
\label{sec:mixofgaussian}
We are now ready for an application of Theorem~\ref{thm:genthm} to case kernel $P_\theta$ is a mixture of $k$ Gaussian distributions. As discussed in Example \ref{exa:kernelmixture}, with this example we are moving from a standard mixture of product distributions to hierarchical models (i.e., mixtures of mixture distributions). Such models are central tools in Bayesian statistics.

Let 
\begin{equation}
\label{eqn:parameterspaceforlocationmixturegaussian} \Theta=\{\theta = (\pi_1,\ldots,\pi_{k-1},\mu_1,\ldots,\mu_k)\in \R^{2k-1}| 0<\pi_i<1,\ \forall i;\ \mu_i<\mu_j,\ \forall 1\leq i<j\leq k\}
\end{equation}
and $\P_\theta$ w.r.t. Lebesgue measure on $\R$ has probability density 
\begin{equation}
f(x|\theta)=\sum_{i=1}^k\pi_i f_{\Nc}(x|\mu_i,\sigma^2) 
\label{eqn:locationmixture}
\end{equation}
where $\pi_k = 1- \sum_{i=1}^{k-1}\pi_i$ and $f_{\Nc}(x|\mu,\sigma^2)$ is the density of $\Nc(\mu,\sigma^2)$ with $\sigma$ a known constant. For the eligibility of this parametrization, see Section \ref{sec:locationscalmixturegaussian}. It follows from the classical result \cite[Proposition 1]{teicher1963identifiability} that the map $\theta\mapsto f(x|\theta)$ is injective on $\Theta$. The mixture of product distributions $P_{G,N}$ admits the density $p_{G,N}$ given in \eqref{eqn:pGNdef} (with $N_i=N$)  w.r.t. Lebesgue measure on $\R^\m$. Fix $G_0=\sum_{i=1}^{k_0}p_i^0 \delta_{\theta_i^0}$ with $\theta_i^0=(\pi_{1i}^0,\ldots,\pi_{(k-1)i}^0, \mu_{1i}^0, \ldots,\mu_{ki}^0)$. 

Let us now verify that Corollary \ref{cor:genthm2ndform} with the map $Tx= (x,x^2,\ldots,x^{2k-1})^\top $ can be applied for this model. 
The mean of $TX_1$ is $\lambda(\theta)\in \R^{2k-1}$ with its $j$-th entry given by 
\begin{equation}
\lambda^{(j)}(\theta) = \E_{\theta} X_1^j = \sum_{i=1}^k \pi_i \E (\sigma Y+\mu_i)^j, \quad j=1,\ldots,2k-1  \label{eqn:Tmeangaussian}
\end{equation}
where $X_1$ has density \eqref{eqn:locationmixture} and $Y$ has the standard Gaussian distribution $\Nc(0, 1)$. The covariance matrix of $TX_1$ is $\Lambda(\theta)\in \R^{(2k-1)\times (2k-1)}$ with its $(j,\beta)$ entries given by 
$$
\Lambda_{j\beta}(\theta) =  \E_\theta X_1^{j+\beta}- \lambda^{(j)}(\theta)\lambda^{(\beta)}(\theta) = \sum_{i=1}^k \pi_i \E (\sigma Y+\mu_i)^{j+\beta} - \lambda^{(j)}(\theta)\lambda^{(\beta)}(\theta).
$$
It follows immediately from these formulae that $\lambda(\theta)$ and $\Lambda(\theta)$ are continuous on $\Theta$. That is, \ref{item:genthmd} in Definition \ref{def:admissible} is satisfied. The characteristic function of $TX_1$ is   
\begin{equation}
\phi_T(\zeta|\theta)   = \E_{\theta} \exp(\bm{i}\zeta^\top  TX_1) =  \sum_{i=1}^k \pi_i h(\zeta|\mu_i,\sigma)  \label{eqn:chrfunformula}
\end{equation}
where $h(\zeta|\mu,\sigma) = \E \exp(\bm{i}\zeta^\top  T (\sigma Y+\mu) )$. Denote by $f_{\Nc}(x|\mu,\sigma)$ the density of $\Nc(\mu, \sigma^2)$. The verification of \ref{item:genthme} in Definition \ref{def:admissible} is omitted since it is a straightforward application of the dominated convergence theorem. In Appendix \ref{sec:detailsgaussian} it is shown by some calculations that to verify condition \ref{item:genthmg} it remains to establish that there exists some $r\geq 1$ such that
$\int_{\R^{2k-1}}\left| \phi_T(\zeta|\theta)\right|^{r}  d\zetave$ on $\Theta$ is upper bounded by a finite continuous function of $\theta$. 

Note that $f_{\Nc}(x|\mu,\sigma)$ is differentiable everywhere and $\frac{\partial f_{\Nc}(x|\mu,\sigma)}{\partial x} \in L^1(\R)$. Moreover $\alpha_1$ in Lemma \ref{lem:charintegrability} for $T$ is $4(k-1)^2$ and $f_{\Nc}(x|\mu,\sigma)$, $x^{4(k-1)^2}f_{\Nc}(x|\mu,\sigma)$ are increasing on $\left(-\infty, -\frac{|\mu|+\sqrt{\mu^2+16(k-1)^2\sigma^2}}{2}\right)$ and decreasing on $\left(\frac{|\mu|+\sqrt{\mu^2+16(k-1)^2\sigma^2}}{2},\infty\right)$. By Lemma \ref{lem:charintegrability}, for $r>(2k-1)^2$, and for $Tx = (x,x^2,\cdots,x^{2k-1})$
\begin{align*}
&\left\| \int_{\R} e^{i\zeta^\top  Tx} f_{\Nc}(x|\mu,\sigma) dx \right\|_{L^r(\R^{2k-1})} \\
\leq &C(r) \left(\frac{|\mu|+\sqrt{\mu^2+16(k-1)^2\sigma^2}}{2}+2\right)^{4(k-1)^2}\\
&\left(\||x|^{4(k-1)^2}f_{\Nc}(x|\mu,\sigma)\|_{L^1(\R)}+\frac{1}{\sqrt{2\pi}\sigma}+\left\|(|x|^{4(k-1)^2}+1)\frac{\partial f_{\Nc}(x|\mu,\sigma)}{\partial x}\right\|_{L^1(\R)}+1\right)\\
:=& h_3(\mu,\sigma), 
\end{align*}
where $C(r)$ is a constant that depends only $r$. It can be verified easily by the dominated convergence theorem that $h_3(\mu,\sigma)$ is a continuous function of $\mu$. Then
\begin{align*}
\| \phi_T(\zeta|\theta) \|_{L^r(\R^{2k-1})} \leq  \sum_{i=1}^k \pi_i \left\| \int_{\R} e^{i\zeta^\top  Tx} f_{\Nc}(x|\mu_i,\sigma) dx \right\|_{L^r(\R^{2k-1})}  \leq  \sum_{i=1}^k \pi_i h_3(\mu_i,\sigma),
\end{align*}
which is a finite continuous function of $\theta = (\pi_1,\ldots,\pi_{k-1},\mu_1,\ldots,\mu_k)$. Thus \ref{item:genthmg} is verified. We have then verified that $T$ is admissible with respect to $\Theta$. 
That the mean map $\lambda(\theta)$ is injective is a classical result (e.g. \cite[Corollary 3.3]{gandhi2016moment}). To apply Corollary \ref{cor:genthm2ndform} it remains to check that the Jacobian matrix $J_\lambda(\theta)$ of $\lambda(\theta)$ is of full column rank. Such details are established in the Section \ref{sec:momentmapinjectivity}. 

In summary, we have shown that all conditions in Corollary \ref{cor:genthm2ndform} are satisfied and thus, for $P_\theta$ having the density in \eqref{eqn:locationmixture}, the inverse bounds \eqref{eqn:genthmcon} and  \eqref{eqn:curvatureprodbound} hold for any $G_0\in \Ec_{k_0}(\Theta)$.

\subsection{Moment map for location mixture of Gaussian distributions has full-rank Jacobian}
\label{sec:momentmapinjectivity}  

In this subsection we verify that the Jacobian $J_{\lambda}(\theta)$ for the moment map $\lambda(\theta)$ specified in Section \ref{sec:mixofgaussian} is of full rank. By \eqref{eqn:Tmeangaussian}, for any $j\in [2k-1]$:
\begin{multline}
\lambda^{(j)}(\theta)=\sum_{i=1}^k \pi_i \left(\mu_i^j + \sum_{\ell=1}^\ell \sigma^{\ell} \E Y^\ell \mu_i^{j-\ell} \right) = \sum_{i=1}^k \pi_i \left(\mu_i^j + \sum_{\substack{\ell=2\\ \ell \text{ even}}}^j \sigma^{\ell} (\ell-1)!! \mu_i^{j-\ell} \right)\\
=\sum_{i=1}^k \pi_i \mu_i^j + \sum_{\substack{\ell=2\\ \ell \text{ even}}}^j \sigma^{\ell} (\ell-1)!! \sum_{i=1}^k \pi_i \mu_i^{j-\ell}. \label{eqn:lambdatelescope}
\end{multline}
Denote $\bar{\lambda}^{(j)}(\theta)=\sum_{i=1}^k \pi_i \mu_i^j$ and $\bar{\lambda}(\theta)=(\bar{\lambda}^{(1)}(\theta),
\ldots,\bar{\lambda}^{(2k-1)}(\theta))\in \R^{2k-1}$. By \eqref{eqn:lambdatelescope}, $\lambda^{(j)}(\theta)=\bar{\lambda}^{(j)}(\theta)+\sum_{\substack{\ell=2\\ \ell \text{ even}}}^j \sigma^{\ell} (\ell-1)!! \bar{\lambda}^{(j-\ell)}(\theta)$, which implies 
$$
\nabla_{\theta}\lambda^{(j)}(\theta)=\nabla_{\theta}\bar{\lambda}^{(j)}(\theta)+\sum_{\substack{\ell=2\\ \ell \text{ even}}}^j \sigma^{\ell} (\ell-1)!! \nabla_{\theta}\bar{\lambda}^{(j-\ell)}(\theta).
$$
Since $\nabla_{\theta}\lambda^{(j)}(\theta)$ and $\nabla_{\theta}\bar{\lambda}^{(j)}(\theta)$ are respectively the $j$-th row of $J_\lambda(\theta)$ and $J_{\bar{\lambda}}(\theta)$, 
\begin{equation}
\text{det}(J_\lambda(\theta))=\text{det}(J_{\bar{\lambda}}(\theta)).
\label{eqn:detrelJlambdaandbar}
\end{equation}
Also, observe
\begin{align*}
& \text{det}(J_{\bar{\lambda}}(\theta)) \\
=& \left(\prod_{\ell=1}^k \pi_\ell\right) \text{det} \begin{pmatrix}
\mu_1-\mu_k, & \ldots & \mu_{k-1}-\mu_k, & 1, &\ldots & 1 \\ 
\mu_1^2-\mu_k^2, & \ldots & \mu_{k-1}^2-\mu_k^2, & 2\mu_1, &\ldots & 2\mu_k\\
\vdots& \vdots & \vdots & \vdots & \vdots &\vdots\\
\mu_1^{2k-1}-\mu_k^{2k-1}, & \ldots & \mu_{k-1}^{2k-1}-\mu_k^{2k-1}, & (2k-1)\mu_1^{2k-1}, &\ldots & (2k-1)\mu_k^{2k-1}
\end{pmatrix}\\
=& \left(\prod_{\ell=1}^k \pi_\ell\right) (-1)^{k+1} \text{det} \begin{pmatrix}
1, & \ldots& 1,& 1,&0,&\ldots&0\\
\mu_1, & \ldots & \mu_{k-1}, &\mu_k, & 1, &\ldots & 1 \\ 
\mu_1^2, & \ldots & \mu_{k-1}^2,& \mu_k^2 & 2\mu_1, &\ldots & 2\mu_k\\
\vdots& \vdots & \vdots & \vdots&\vdots & \vdots &\vdots\\
\mu_1^{2k-1}, & \ldots & \mu_{k-1}^{2k-1}, & \mu_k^{2k-1}& (2k-1)\mu_1^{2k-1}, &\ldots & (2k-1)\mu_k^{2k-1}
\end{pmatrix}\\
=& \left(\prod_{\ell=1}^k \pi_\ell\right) (-1)^{k+1} \left(\prod_{i=1}^k (-1)^{k+i-2i}\right) \prod_{1\leq \alpha<\beta\leq k}(\mu_{\alpha}-\mu_{\beta})^4 \numberthis \label{eqn:detJlambdabar}
\end{align*}
where the second equality holds since we may subtract the $k$-th column of the $2k\times 2k$ matrix from each of its first $k-1$ columns and then do Laplace expansion along its first row, and the last equality follows by observing that the $(k+i)$-th column of the 
$2k\times 2k$ matrix is the derivative of the $i$-th column and by applying Lemma \ref{lem:determinant} \ref{item:determinantd} after some column permutation. By \eqref{eqn:detrelJlambdaandbar} and \eqref{eqn:detJlambdabar}, $\text{det}(J_{\lambda}(\theta))\neq 0$ on $\Theta$. That is $J_{\lambda}(\theta)$ is of full column rank for any $\theta\in \Theta$.

\subsection{Kernel $P_\theta$ is mixture of Dirichlet processes}
\label{sec:dirichlet}

Now we tackle Example~\ref{exa:dirichlet}, which is motivated from modeling techniques in nonparametric Bayesian statistics. In particular, the kernel $P_\theta$ is given as a distribution on a space of measures: $P_\theta$ is a mixture of Dirichlet processes (DPs), so that $P_{G,\m}$ is a finite mixture of products of mixtures of DPs. 
This should not be confused with the use of DP as a prior for mixing measures arising in mixture models. Rather, this is more akin to the use of DPs as probability kernels that arise in the famous hierarchical Dirichlet processes~\cite{Teh-etal-06} (actually, this model uses DP both as a prior and kernels). The purpose of this example is to illustrate Theorem \ref{thm:genthm} when (mixtures of) Dirichlet processes are treated as kernels. 

Let $\Xfrak=\mathscr{P}(\mathfrak{Z})$ be the space of all probability measures on a Polish space $(\mathfrak{Z},\mathscr{Z})$. $\Xfrak$ is equipped with the weak topology and the corresponding Borel sigma algebra $\mathcal{A}$. Let $\mathscr{D}_{\alpha H}$ denote the Dirichlet distribution on $(\Xfrak,\mathcal{A})$, which is specified by two parameters, concentration parameter $\alpha \in (0,\infty)$ and base measure $H \in \Xfrak$. Formal definition and key properties of the Dirichlet distributions can be found in the original paper of~\cite{ferguson1973bayesian}, or a recent textbook~\cite{ghosal2017fundamentals}. In this example, we take the probability kernel $P_\theta$ to be a mixture of two Dirichlet distributions with different concentration parameters, while the base measure is fixed and known: $P_\theta = \pi_1 \mathscr{D}_{\alpha_1 H} + (1-\pi_1) \mathscr{D}_{\alpha_2 H}$. Thus, the parameter vector is three dimensional which shall be restricted by the following constraint:
$\theta:=(\pi_1,\alpha_1,\alpha_2)\in \Theta=\{(\pi_1,\alpha_1,\alpha_2)|0<\pi_1<1, 2<\alpha_1<\alpha_2 \}$. It can be easily verified that the map $\theta \to P_{\theta}$ is injective.
Kernel $P_\theta$ so defined is a simple instance of the so-called mixture of Dirichlet processes first studied by~\cite{antoniak1974mixtures}, but considerably more complex instances of model using Dirichlet as the building block have become a main staple in the lively literature of Bayesian nonparametrics \cite{hjort2010bayesian,Teh-etal-06,Rodriguez-etal-08,camerlenghi2019distribution}. For notational convenience in the following we also denote $Q_{\alpha}:=\mathscr{D}_{\alpha H}$ for $\alpha=\alpha_1$ and $\alpha=\alpha_2$, noting that $H$ is fixed, so we may write $P_\theta = \pi_1 Q_{\alpha_1} + (1-\pi_1) Q_{\alpha_2}$. 

Having specified the kernel $P_\theta$, let $G\in \Ec_k(\Theta)$. The mixture of product distributions $P_{G,N}$ is defined in the same way as before (see Eq.~\eqref{eqn:mixprod}). Now we show that for $G_0\in \Ec_{k_0}(\Theta^\circ)=\Ec_{k_0}(\Theta)$, 
\eqref{eqn:genthmcon} and \eqref{eqn:curvatureprodbound} hold by applying Corollary \ref{cor:genthm2ndform} via a suitable map $T$.  

Consider a map $T:\Xfrak\to \R^3$ defined by $T x= ((x(B))^2,(x(B))^3,(x(B))^4)^\top $ for some $B\in \mathscr{Z}$ to be specified later. The reason we restrict the domain of $\Theta$ is so that this particular choice of map will be shown to be admissible. Define $T_1:\Xfrak\to \R$ by $T_1 x=x(B)$  and $T_2:\R\to \R^3$ by $T_2 z=(z^2,z^3,z^4)^\top $. Then $T=T_2\circ T_1$. For $X\sim P_\theta$, $T_1X$ has distribution
$$
P_\theta \circ T_1^{-1} = \pi_1 \left(Q_{\alpha_1}\circ T_1^{-1}\right)  + \pi_2 \left(Q_{\alpha_2}\circ T_1^{-1}\right) .
$$
where $\pi_2=1-\pi_1$. 
By a standard property of Dirichlet distribution, as $Q_\alpha = \mathscr{D}_{\alpha H}$, we have $Q_{\alpha}\circ T_1^{-1}$ corresponds to $\text{Beta}(\alpha H(B),\alpha(1-H(B)))$, a Beta distribution. Thus with $\xi=H(B)$, $Q_{\alpha}\circ T_1^{-1}$ has density w.r.t. Lebesgue measure on $\R$
$$
g(z|\alpha,\xi) = \frac{1}{B(\alpha\xi,\alpha(1-\xi))} z^{\alpha\xi-1} (1-z)^{\alpha(1-\xi)-1}\1ve_{(0,1)}(z),
$$
where $B(\cdot,\cdot)$ is the beta function. Then $P_\theta \circ T_1^{-1}$ has density w.r.t. Lebesgue measure $\pi_1 g(z|\alpha_1,\xi)+\pi_2 g(z|\alpha_2,\xi)$. 

Now, the push-forward measure $P_\theta\circ T^{-1}=(P_\theta\circ T_1^{-1})\circ T_2^{-1}$ has mean $\lambda(\theta)\in \R^3$ with
$$
\lambda^{(j)}(\theta)= \sum_{i=1}^2 \pi_i \int_{\R} z^{j+1}  g(z|\alpha_i,\xi) dz  = \sum_{i=1}^2 \pi_i \prod_{\ell=0}^{j}\frac{\alpha_i\xi+\ell}{\alpha_i+\ell}  \quad \forall j=1,2,3
$$
and has covariance matrix $\Lambda$ with its $j\beta$ entry given by
$$
\Lambda_{j\beta}(\theta) = \sum_{i=1}^2 \pi_i \int_{\R} z^{j+\beta+2}  g(z|\alpha_i,\xi) dz - \lambda^{j}(\theta)\lambda^{\beta}(\theta)=\sum_{i=1}^2 \pi_i \prod_{\ell=0}^{j+\beta+1}\frac{\alpha_i\xi+\ell}{\alpha_i+\ell}
- \lambda^{j}(\theta)\lambda^{\beta}(\theta).$$
It follows immediately from these formula that $\lambda(\theta)$ and $\Lambda(\theta)$ are continuous on $\Theta$, i.e., \ref{item:genthmd} in Definition \ref{def:admissible} is satisfied. Furthermore, 
observe that $P_\theta\circ T^{-1}$ has characteristic function
$$
\phi_T(\zeta|\theta) = \pi_1 h(\zeta|\alpha_1,\xi) + \pi_2 h(\zeta|\alpha_2,\xi) \label{eqn:chrfunformuladirichlet}
$$
where $h(\zeta|\alpha,\xi) = \int_{\R} \exp(\bm{i}\sum_{j=1}^3 \zeta^{(j)} z^{j}) g(z|\alpha,\xi)dz$. The verification of \ref{item:genthme} in Definition \ref{def:admissible} is omitted since it is a straightforward application of the dominated convergence theorem. In Appendix \ref{sec:detailsdirichlet} 
we provide detailed calculations to verify partially condition \ref{item:genthmg} so that it remains to establish there exists some $r\geq 1$ such that
$\int_{\R^{2k-1}}\left| \phi_T(\zeta|\theta)\right|^{r}  d\zetave$ on $\Theta$ is upper bounded by a finite continuous function of $\theta$. So far we have verified \ref{item:genthmd}, \ref{item:genthme} and some parts of \ref{item:genthmg} for the chosen $T$ for every $B$.

To continue the verification of \ref{item:genthmg} we now specify $B$. For $G_0=\sum_{i=1}^{k_0} p_i^0\delta_{\theta_i^0}$ with $\theta_i^0=(\pi_{1i}^0,\alpha_{1i}^0,\alpha_{2i}^0)\in \Theta$, let $B$ be such that $\xi=H(B) \in (1/\min_{i\in [k_0]} \alpha_{1i}^0,1-1/\min_{i\in [k_0]} \alpha_{1i}^0)$. Notice that since $\alpha_{1i}^0> 2$,  $(1/\min_{i\in [k_0]} \alpha_{1i}^0,1-1/\min_{i\in [k_0]} \alpha_{1i}^0)$ is not empty. Hence to verify the condition \ref{item:genthmg} in Definition \ref{def:admissible} w.r.t. $\{\theta_i^0\}_{i=1}^{k_0}$ for $T$ with the $B$ specified it suffices to establish there exists some $r\geq 1$ such that $\int_{\R^3}\left| \phi_T(\zeta|\theta)\right|^{r} d\zetave$ in a small neighborhood of $\theta_0$ is upper bounded by a finite continuous function of $\theta$ for each $\theta_0\in \{\theta_i^0\}_{i=1}^{k_0}$. 
 
Since $g(z|\alpha,\xi)$ is differentiable w.r.t. to $z$ on $\R\backslash \{0,1\}$ and when $z\neq 0,1$, $\frac{\partial g(z|\alpha,\xi)}{\partial z}$ is
$$
  \frac{\1ve_{(0,1)}(z)}{B(\alpha\xi,\alpha(1-\xi))} \left((\alpha\xi-1)z^{\alpha\xi-2} (1-z)^{\alpha(1-\xi)-1}-(\alpha(1-\xi)-1)z^{\alpha\xi-1} (1-z)^{\alpha(1-\xi)-2} \right),
$$
which is in $L^1$ when $\alpha\geq \min_{i\in [k_0]} \alpha_{1i}^0 - \gamma$ such that $\alpha\xi>1$ and $\alpha (1-\xi)>1$, where $\gamma$ depends on $T$ through $\xi$. Moreover, $g(z|\alpha,\xi)$ and $z^2g(z|\alpha,\xi)$ are both increasing on $(-\infty, -1)$ and decreasing on $(1,\infty)$. Now, by appealing to Lemma \ref{lem:charintegrability}, for $r>12$, and for $\alpha\geq \min_{i\in [k_0]} \alpha_{1i}^0 - \gamma$
\begin{align*}
&\left\| h(\zeta|\alpha,\xi)\right\|_{L^r(\R^3)} \\
\leq &C(r) (1+2)^2 \left(\|z^2g(z|\alpha,\xi)\|_{L^1}+3\|g(z|\alpha,\xi)\|_{L^\infty}+\left\|(z^2+1)\frac{\partial g(z|\alpha,\xi)}{\partial z}\right\|_{L^1}+1\right)\\
:=& h_5(\alpha,\xi),
\end{align*}
where $C(r)$ is a constant that depends only on $r$. It can be verified easily by the dominated convergence theorem that $h_5(\alpha,\xi)$ is a continuous function of $\alpha$. Then for $\theta$ in a neighborhood of $\theta_0\in \{\theta_i^0\}_{i=1}^{k_0}$ such that $\alpha_1,\alpha_2\geq \alpha_{1i}^0 - \gamma$,
\begin{align*}
\| \phi_T(\zeta|\theta) \|_{L^r(\R^3)} 
\leq \pi_1\left\| h(\zeta|\alpha_1,\xi)\right\|_{L^r} +\pi_2 \left\| h(\zeta|\alpha_2,\xi)\right\|_{L^r} 
\leq  \pi_1 h_5(\alpha_1,\xi) + \pi_2 h_5(\alpha_2,\xi),
\end{align*}
which is a finite continuous function of $\theta = (\pi_1,\alpha_1,\alpha_2)$. We have thus verified that $T$ with the specified $B$ is admissible w.r.t. $\{\theta_i^0\}_{i=1}^{k_0}$.

Moreover, it can also be verified that $\lambda(\theta)$ for $T$ is injective on $\Theta$ provided that $\xi\neq \frac{1}{3},\frac{1}{2},\frac{2}{3}$. By calculation, the Jacobian matrix $J_{\lambda}(\theta)$ of $\lambda(\theta)$ satisfies
$$\text{det}(J_{\lambda})(\theta)= -\frac{6(\xi-1)^3\xi^3(2\xi-1)(3\xi-1)(3\xi-2)
\pi_1\pi_2(\alpha_1-\alpha_2)^4}{\prod_{i=1}^2\left((1+\alpha_i)^2(2+\alpha_i)^2(3+\alpha_i)^2\right)} \not= 0 $$ on $\Theta$ provided that $\xi\neq \frac{1}{3},\frac{1}{2},\frac{2}{3}$; so $J_{\lambda}(\theta)$ is of full rank for each $\theta\in \Theta$ provided that $\xi\neq \frac{1}{3},\frac{1}{2},\frac{2}{3}$. In summary, for $G_0=\sum_{i=1}^{k_0} p_i^0\delta_{\theta_i^0}$ with $\theta_i^0=(\pi_{1i}^0,\alpha_{1i}^0,\alpha_{2i}^0)\in \Theta$, $Tx=((x(B))^2, (x(B))^3, (x(B))^4)^\top $ with $B$ such that 
$$
\xi=H(B)\in \left(\frac{1}{\min_{i\in [k_0]} \alpha_{1i}^0},1-\frac{1}{\min_{i\in [k_0]} \alpha_{1i}^0}\right)\biggr\backslash \left\{\frac{1}{3},\frac{1}{2},\frac{2}{3}\right\}
$$ 
satisfies all the conditions in Corollary \ref{cor:genthm2ndform} and thus \eqref{eqn:genthmcon} and \eqref{eqn:curvatureprodbound} hold.


\section{Sharpness of bounds and minimax theorem}
	\label{sec:minimax}
	
	\subsection{Sharpness of inverse bounds}
	\label{sec:sharp}  



In this subsection 
we consider reverse upper bounds for \eqref{eqn:genthmcon}, which are also reverse upper bounds for \eqref{eqn:curvatureprodbound} by \eqref{eqn:reltwoinversebounds}. Inverse bounds of the form~\eqref{eqn:genthmcon} hold only under some identifiability conditions, while the following upper bound holds generally and is much easier to show.
	
\begin{lem} \label{lem:VD1liminfuppbou}
	Let $k_0\geq 2$  and fix $G_0 =\sum_{i=1}^{k_0}p_i^0\delta_{\theta_i^0} \in \Ec_{k_0}(\Theta)$. Then for any $\m\geq 1$
$$
\liminf_{\substack{G\overset{W_1}{\to} G_0\\ G\in \Ec_{k_0}(\Theta)}} \frac{V(P_{G,\m},P_{G_0,\m})}{D_{\m}(G,G_0)} \leq \liminf_{\substack{G\overset{W_1}{\to} G_0\\ G\in \Ec_{k_0}(\Theta)}} \frac{V(P_{G,\m},P_{G_0,\m})}{D_{1}(G,G_0)}
\leq \frac{1}{2}.
$$
\end{lem}
\begin{proof}
 Consider $G_\ell=\sum_{i=1}^{k_0}p_i^{\ell}\delta_{\theta_i^{0}}$ with $p_i^{\ell}=p_i^0$ for $3\leq \ell\leq k_0$ and $p_1^{\ell}=p_1^0+\frac{1}{\ell}$, $p_2^{\ell}=p_2^0-\frac{1}{\ell}$. Then for sufficiently large $\ell$, $p_1^\ell,p_2^\ell\in(0,1)$ and hence $G_\ell\in \Ec_{k_0}(\Theta)\backslash\{G_0\}$ and satisfies $D_{\m}(G_\ell,G_0)=D_{1}(G_\ell,G_0)=2/\ell$. Thus for sufficiently large $\ell$,
    $$
    \frac{V(P_{G,\m},P_{G_0,\m})}{D_{1}(G,G_0)}=\frac{\ell}{2}\sup_{A \in \mathcal{A}^N }\left|\frac{1}{\ell}\bigotimes^{\m}P_{\theta_1^0}(A)-\frac{1}{\ell}\bigotimes^{\m}P_{\theta_2^0}(A)\right|    = \frac{1}{2}V\left(\bigotimes^{\m}P_{\theta_1^0},\bigotimes^{\m}P_{\theta_2^0}\right)\leq \frac{1}{2}. 
    $$
\end{proof}

The next lemma establishes an upper bound for Hellinger distance of two mixture of product measures by Hellinger distance of individual components. It is an improvement of \cite[Lemma 3.2 (a)]{nguyen2016borrowing}. Such a result is useful in Lemma \ref{lem:optimalsquaretootN}
. A similar result on variation distance is Lemma \ref{lem:Vupperbou}.

\begin{lem}\label{lem:hellingeruppbou}
		 For any $G = \sum_{i=1}^{k_0} p_i \delta_{\theta_i}$ and $G' = \sum_{i=1}^{k_0} p'_i \delta_{\theta'_i}$,
		$$h(\P_{G,\m},\P_{G',\m}) \leq \min_{\tau} \left( \sqrt{\m} \max_{1\leq i \leq k_0} h\left(\P_{\theta_i}, \P_{\theta'_{\tau(i)}}\right) + \sqrt{\frac{1}{2}\sum_{i=1}^{k_0}\left|p_i-p_{\tau(i)}'\right|} \right),$$
		where the minimum is taken over all $\tau$ in the permutation group $S_{k_0}$. 
	\end{lem}

	The inverse bounds expressed by Eq.~\eqref{eqn:genthmcon} are optimal as far as the role of $\m$ in $D_\m$ is concerned. This is made precise by the following result.
	\begin{lem}[Optimality of $\sqrt{\m}$ for atoms] \label{lem:optimalsquaretootN}
		Fix $G_0 = \sum_{i=1}^{k_0} p_i \delta_{\theta_i^0} \in \Ec_{k_0}(\Theta^\circ)$. Suppose there exists $j \in [k_0]$ such that $\liminf\limits_{\theta\to \theta_j^0}\frac{h(\P_\theta,\P_{\theta^0_{j}})}{\|\theta - \theta_{j}^0\|_2}<\infty$ . Then for $\psi(\m)$ such that $\frac{\psi(\m)}{\m}\to \infty$, 
		\begin{equation*}
		\limsup_{\m\to \infty}\liminf_{\substack{G\overset{W_1}{\to} G_0\\ G\in \Ecal_{k_0}(\Theta) }} \frac{h(\P_{G,\m},\P_{G_0,\m})}{\myD_{\psi(\m)}(G,G_0)} = 0. 
		\end{equation*}
	\end{lem}

	Lemma \ref{lem:optimalsquaretootN} establishes that $\sqrt{N}$ is optimal for the coefficients of the component parameters $\theta_i$ in $D_N$. The next lemma establishes that the constant coefficients of the mixing propositions $p_i$ in $D_N$ are also optimal. For $G=\sum_{i=1}^{k_0}p_i\delta_{\theta_i}$ and $G'
	=\sum_{i=1}^{k_0}p'_i\delta_{\theta'_i}$,
	define $$\bar{D}_r(G,G')= \min_{\tau\in S_{k}}\left( \|\theta_{\tau(i)}-\theta'_i\|_2 + r|p_{\tau(i)}-p'_i|  \right).$$ It states that the vanishing of $V(P_{G,N},P_{G_0,N})$ may not induce a faster convergence rate for the mixing proportions $p_i$ in terms of $N$ as the exchangeable length $N$ increases.
	
	\begin{lem}[Optimality of constant coefficient for mixing proportions] \label{lem:optimalmixingproportions}
		Fix $G_0 = \sum_{i=1}^{k_0} p_i \delta_{\theta_i^0} \in \Ec_{k_0}(\Theta^\circ)$. Suppose that the map $\theta\to P_\theta$ is injective. Then for $\psi(\m)$ such that $\psi(\m)\to \infty$, 
		\begin{equation*}
		\limsup_{\m\to \infty}\liminf_{\substack{G\overset{W_1}{\to} G_0\\ G\in \Ecal_{k_0}(\Theta) }} \frac{V(\P_{G,\m},\P_{G_0,\m})}{\bar{\myD}_{\psi(\m)}(G,G_0)} = 0. 
		\end{equation*}
	\end{lem}
	
	\begin{proof}
	 Consider $G_\ell=\sum_{i=1}^{k_0}p_i^\ell\delta_{\theta_i^\ell}\in \Ec_{k_0}(\Theta)$ with $\theta_i^\ell = \theta_i^0$ for any $i$ and $p_i^\ell = p_i^0$ for $i\geq 3$, $p_1^\ell=p_1^0+1/\ell$, $p_2^\ell=p_2^0-1/\ell$. Then for large $\ell$, $\bar{D}_{\psi(N)}(G_\ell,G_0)= \psi(N)(|p_1^\ell-p_1^0|+|p_2^\ell-p_2^0|)= 2\psi(N)/\ell$. Note that $V(P_{G_\ell,N},P_{G_0,N})= V(P_{\theta_1^0},P_{\theta_2^0})/\ell$ and hence 
	$$
	\liminf_{\substack{G\overset{W_1}{\to} G_0\\ G\in \Ecal_{k_0}(\Theta) }} \frac{V(\P_{G,\m},\P_{G_0,\m})}{\bar{\myD}_{\psi(\m)}(G,G_0)} \leq \frac{V(\P_{G_\ell,\m},\P_{G_0,\m})}{\bar{\myD}_{\psi(\m)}(G_\ell,G_0)} = \frac{V(P_{\theta_1^0},P_{\theta_2^0})}{2\psi(N)},
	$$
	which completes the proof.
	\end{proof}
	
	A slightly curious and pedantic way to gauge the meaning of the double infimum limiting arguments in the inverse bound~\eqref{eqn:genthmcon}, is to express its claim as follows:  
	$$
	0< \liminf_{\m\to \infty}\liminf_{\substack{G\overset{W_1}{\to} G_0\\ G\in \Ecal_{k_0}(\Xi) }} \frac{V(\P_{G,\m},\P_{G_0,\m})}{\myD_{\m}(G,G_0)} = \lim_{ k \to \infty}\ \inf_{\m\geq k}\ \lim_{\epsilon \to 0}\ \inf_{G\in B_{W_1}(G_0,\epsilon)\backslash \{G_0\}} \frac{V(\P_{G,\m},\P_{G_0,\m})}{\myD_{\m}(G,G_0)},
	$$
	where $B_{W_1}(G_0,R)\subset \Ec_{k_0}(\Theta)$ is defined in \eqref{eqn:BW1def}. It is possible to alter the order of the four operations and consider the resulting outcome. The following lemma shows the last display is the only order to possibly obtain a  positive outcome. 
	\begin{lem}
	    \begin{enumerate}[label=\alph*)]
	    \item \label{item:unnamea}
	    \begin{multline*}
	    \lim_{ k \to \infty}\ \lim_{\epsilon \to 0}\ \inf_{\m\geq k}\ \inf_{G\in B_{W_1}(G_0,\epsilon)\backslash \{G_0\}} \frac{V(\P_{G,\m},\P_{G_0,\m})}{\myD_{\m}(G,G_0)} \\ = \lim_{ k \to \infty}\ \lim_{\epsilon \to 0}\ \inf_{G\in B_{W_1}(G_0,\epsilon)\backslash \{G_0\}}\ \inf_{\m\geq k}\  \frac{V(\P_{G,\m},\P_{G_0,\m})}{\myD_{\m}(G,G_0)} =  0
	    \end{multline*}
	    \item
	    \begin{multline*}
	    \lim_{\epsilon \to 0}\ \lim_{ k \to \infty}\ \inf_{\m\geq k}\ \inf_{G\in B_{W_1}(G_0,\epsilon)\backslash \{G_0\}} \frac{V(\P_{G,\m},\P_{G_0,\m})}{\myD_{\m}(G,G_0)} \\ =\lim_{\epsilon \to 0}\  \lim_{ k \to \infty}\ \inf_{G\in B_{W_1}(G_0,\epsilon)\backslash \{G_0\}}\ \inf_{\m\geq k}\  \frac{V(\P_{G,\m},\P_{G_0,\m})}{\myD_{\m}(G,G_0)} =  0
	    \end{multline*}
	    \item 
	    $$
	    \lim_{\epsilon \to 0}\ \inf_{G\in B_{W_1}(G_0,\epsilon)\backslash \{G_0\}} \lim_{ k \to \infty}\ \inf_{\m\geq k}\ \frac{V(\P_{G,\m},\P_{G_0,\m})}{\myD_{\m}(G,G_0)} =0.
	    $$
	    \end{enumerate}
	\end{lem}
	\begin{proof}
	   The claims follow from
	    $$
	    \inf_{\m\geq k}\ \inf_{G\in B_{W_1}(G_0,\epsilon)\backslash \{G_0\}} \frac{V(\P_{G,\m},\P_{G_0,\m})}{\myD_{\m}(G,G_0)} = \inf_{G\in B_{W_1}(G_0,\epsilon)\backslash \{G_0\}} \inf_{\m\geq k}\  \frac{V(\P_{G,\m},\P_{G_0,\m})}{\myD_{\m}(G,G_0)} 
	    $$
	    
	    and 
	    $$
	    \inf_{\m\geq k}\  \frac{V(\P_{G,\m},\P_{G_0,\m})}{\myD_{\m}(G,G_0)} \leq \inf_{\m\geq k}\ \frac{1}{\myD_{\m}(G,G_0)} =0.
	    $$
	\end{proof}

	\subsection{Minimax lower bounds}
	\label{sec:minimaxlowbou}

	Given $G=\sum_{i=1}^{k_0} p_i \delta_{\theta_i}\in \Ec_{k_0}(\Theta)$ and $G'= \sum_{i=1}^{k_0} p'_i \delta_{\theta'_i} \in \Ec_{k_0}(\Theta)$, define additional notions of distances
	\begin{eqnarray}
	d_{\bm{\Theta}}(G',G):= \min_{\tau\in S_{k_0}}\sum_{i=1}^{k_0}\|\theta'_{\tau(i)}-\theta_i\|_2 \\
	d_{\bm{p}}(G',G):=\min_{\tau\in S_{k_0}}\sum_{i=1}^{k_0}|p'_{\tau(i)}-p_i|. 
	\end{eqnarray}
	Notice that we denote $d_{\Theta}$ to be a distance on $\Theta$ in Section \ref{sec:prelim}. Here the $d_{\bm{\Theta}}$ with bold subscript is on $\Ec_{k_0}(\Theta)$. These two notions of distance are pseudometrics on the space of measures $\Ec_{k_0}(\Theta)$, i.e., they share the same properties as a metric except that allow the distance between two different points be zero. $d_{\bm{\Theta}}(G',G)$ focuses on the distance between atoms of two mixing measure; while $d_{\bm{p}}(G',G)$ focuses on the mixing probabilities of the two mixing measures. It is clear that
	\begin{equation}
	D_N(G,G') \geq \sqrt{N} d_{\bm{\Theta}}(G,G') + d_{\bm{p}}(G,G').
	\label{eqn:d1dthetadp}
	\end{equation}

	We proceed to present minimax lower bounds for any sequence of estimators $\hat{G}$, which are measurable functions of $X^1_{[\m]},\ldots, X_{[\m]}^{\n}$, where the sequence length are assumed to be equal for simplicity. The minimax bounds are stated in terms of the aforementioned (pseudo-)metrics $d_{\bm{p}}$ and $d_{\bm{\Theta}}$, as well as the usual metric $D_N$ studied.
	
	\begin{thm}[Minimax Lower Bound] \label{thm:minimaxlower} 
	In the following three bounds the infimum is taken for all $\hat{G} $ measurable functions of $X^1_{[\m]},\ldots, X_{[\m]}^{\n}$.
	\begin{enumerate}[label=\alph*)] 
	\item \label{item:minimaxlowera}	
	Suppose there exists $\theta_0\in\Theta^\circ$ and $\beta_0>0$ such that $\limsup\limits_{\theta \to\theta_0 }\frac{h\left(\P_{\theta},\P_{\theta_0}\right)}{\|{\theta}-{\theta_0}\|_2^{\beta_0}}<\infty$. Moreover, suppose there exists a set of distinct $k_0-1$ points $\{\theta_i\}_{i=1}^{k_0-1}\subset \Theta\backslash \{\theta_0\}$ satisfying 
	$\min_{0\leq i<j\leq k_0-1}h(\P_{\theta_i},\P_{\theta_j})>0$. Then 
	$$
	\inf_{\hat{G}\in \Ec_{k_0}(\Theta)}\sup_{G\in \Ec_{k_0}(\Theta)}\E_{\bigotimes^\n\P_{G,\m}}
	d_{\bm{\Theta}}(\hat{G},G)
	\geq C(\beta_0,k_0) \left(\frac{1}{\sqrt{\n}\sqrt{\m}}\right)^{\frac{1}{\beta_0}},
	$$
	where $C(\beta_0,k_0)$ is a constant depending on $\beta_0$, $k_0$ and the probability family $P_\theta$. 
	\item \label{item:minimaxlowerb} 
	Let $k_0\geq 2$.
	$$\inf_{\hat{G}\in \Ec_{k_0}(\Theta)}\sup_{G\in \Ec_{k_0}(\Theta)}\E_{\bigotimes^\n\P_{G,\m}}
	d_{\bm{p}}(\hat{G},G) \geq C(k_0)\frac{1}{\n},$$
	where $C(k_0)$ is a constant depending on $k_0$ and the probability family $P_\theta$.
    \item 
    Let $k_0\geq 2$. Suppose the conditions of part (a) hold. Then,
    $$\inf_{\hat{G}\in \Ec_{k_0}(\Theta)}\sup_{G\in \Ec_{k_0}(\Theta)}\E_{\bigotimes^\n\P_{G,\m}} D_N(\hat{G},G)\geq C(\beta_0,k_0) \sqrt{N} \left(\frac{1}{\sqrt{\n}\sqrt{\m}}\right)^{\frac{1}{\beta_0}} + C(k_0)\frac{1}{\n}.$$
	\end{enumerate}
	\end{thm}

	\begin{rem} \label{rem:minimax}
	\begin{enumerate}[label=\alph*)]
	 \item 
	The condition that there exists a set of distinct $k_0-1$ points $\{\theta_i\}_{i=1}^{k_0-1}\subset \Theta\backslash \{\theta_0\}$ satisfying 
	$\min_{0\leq i<j\leq k_0-1}h(\P_{\theta_i},\P_{\theta_j})>0$ immediately follows from the injectivity of the map $\theta \mapsto P_\theta$ (recall that this 
	condition is assumed throughout the paper). 
	
	\item The condition that there exists $\theta_0\in\Theta^\circ$ and $\beta_0>0$ such that $\limsup\limits_{\theta \to\theta_0 }\frac{h\left(\P_{\theta},\P_{\theta_0}\right)}{\|{\theta}-{\theta_0}\|_2^{\beta_0}}<\infty$ holds for most probability kernels considered in practice. For example, it is satisfied with $\beta_0=1$ for all full rank exponential families of distribution in their canonical form as shown by Lemma \ref{lem:exphellinger}. It can then be shown that this condition with $\beta_0=1$ is also satisfied by full rank exponential families in general form specified in Corollary \ref{cor:expfamtheta}. Notice that the same remark applies to the condition in Lemma \ref{lem:optimalsquaretootN}.

	\item If conditions of Theorem \ref{thm:minimaxlower} \ref{item:minimaxlowera} hold with $\beta_0=1$, then 
	$$
	\inf_{\hat{G}\in \Ec_{k_0}(\Theta)}\sup_{G\in \Ec_{k_0}(\Theta)}\E_{\bigotimes^\n\P_{G,\m}}
	d_{\bm{\Theta}}(\hat{G},G)
	\geq  \frac{C}{\sqrt{\n}\sqrt{\m}}. 
	$$
	That is, the convergence rate of the best possible estimator for the worst scenario is at least  $\frac{1}{\sqrt{\n}\sqrt{\m}}$. Recall that Theorem  \ref{thm:posconnotid} implied that the convergence rate of the atoms is $O_P(\sqrt{\frac{\ln(\n N)}{\n N}})$, which is obtained by replacing $\bar{N}_m$ with $N$ in \eqref{eqn:atomconvergencerate}. It is worth noting that while the minimax rate seems to match the posterior contraction rate of the atoms except for a logarithmic factor, such a comparison is not very meaningful as pointwise posterior contraction bounds and minimax lower bounds are generally not considered to be compatible. In particular, in the posterior contraction Theorem \ref{thm:posconnotid}, the truth $G_0$ is fixed and the hidden constant $O_P(\sqrt{\frac{\ln(\n N)}{\n N}})$ depends on $G_0$, which is clearly not the case in the above results obtained under the minimax framework. In short, we do not claim that the Bayesian procedure described in Theorem \ref{thm:posconnotid} is optimal in the minimax sense; nor do we claim that the bounds given in Theorem~\ref{thm:minimaxlower} are sharp (i.e., achievable by some statistical procedure). 
	\myeoe
	\end{enumerate}
	\end{rem}



\section{Extensions and discussions}  
\label{sec:extensions}
\subsection{On compactness assumption} 
\label{sec:compactness}

    In Theorems~\ref{thm:posconnotid} we impose that the parameter subset $\Theta_1$ is compact. This appears to be a strong assumption, although it is a rather standard one for most theoretical investigations of parameter estimation in finite mixture models (see~\cite{chen1995optimal,nguyen2013convergence,ho2016convergence,ho2016strong,ghosal2017fundamentals,wu2020optimal}). We surmise that in the context of mixture models, it might not be possible to achieve the global parameter estimation rate without a suitable constraint on the parameter space, such as compactness. In this subsection we clarify the roles of the compactness condition within our approach and discuss possible alternatives to relax compactness to boundedness.
    
    
    The proof of Theorem \ref{thm:posconnotid} follows the basic structure of Lemma \ref{lem:convergencerate}. To obtain the posterior contraction rate to mixture densities and the posterior consistency w.r.t. $W_1$ for general probability kernel $f(x|\theta)$, a global inverse bound Lemma \ref{cor:VlowbouW1} is applied (as an example, it follows that the posterior contraction rate to mixture densities and Lemma \ref{cor:VlowbouW1} together imply the posterior consistency w.r.t. $W_1$). The compactness of $\Theta_1$ is only used to establish Lemma \ref{cor:VlowbouW1}. It might be possible to have a posterior contraction result for the population density estimation and the posterior consistency result \emph{without} Lemma \ref{cor:VlowbouW1} (e.g. by an existence of test argument), but such an approach would require additionally stronger and perhaps explicit knowledge of the kernel $f(x|\theta)$, and thus is beyond the scope of this paper.
    
    The compactness of $\Theta_1$ is only used to guarantee Lemma \ref{cor:VlowbouW1}, which is used in the posterior contraction and consistency results mentioned above.  
     It is possible to have a posterior contraction result for the population density estimation and the posterior consistency result \emph{without} Lemma \ref{cor:VlowbouW1} (e.g. by existence of test), but such an approach would require additionally stronger and perhaps explicit knowledge of the kernel $f(x|\theta)$.

In this subsection we provide a substitute to the compactness assumption in Lemma \ref{cor:VlowbouW1}, which removes the compactness assumption in Theorem \ref{thm:posconnotid}. It is clear that $\Theta$ is required to be a bounded set. The compactness assumption may be relaxed by the necessary boundedness assumption, provided that an identifiability condition additionally holds. This can be seen by the following claim.

\begin{lem} \label{lem:VlowbouW1new}
Fix $G_0\in\Ec_{k_0}(\Theta^\circ)$. 
    Suppose $\Theta$ is bounded. Let $n_1(G_0)$ be given by~\eqref{eqn:defn0n1n2}. Suppose there exists $n_0\geq 1$ such that for $\epsilon>0$
\begin{equation}
\inf_{G\in \cup_{k\leq k_0}\Ec_k(\Theta): W_1(G,G_0)>\epsilon} h(p_{G,n_0},p_{G_0,n_0})>0. \label{eqn:separationatG0}
\end{equation}
Then
    $$
     h(P_{G,\m},P_{G_0,\m}) \geq  C(G_0,\Theta) W_1(G,G_0), \quad \forall G\in \bigcup_{k=1}^{k_0}\Ec_{k}(\Theta_1),\; \forall \m\geq n_1(G_0)\vee n_0,
     $$
     provided $ n_1(G_0)\vee n_0 <\infty$, where $C(G_0,\Theta) >0$ is a constant that depends on $G_0$ and $\Theta$.
\end{lem}
\begin{proof}
    In this proof we write  $n_1$ for  $n_1(G_0)$ . By the definition of $n_1$, for any $\m \geq n_1$ 
        \begin{equation}
        \liminf_{\substack{G\overset{W_1}{\to}G_0\\ G\in \Ec_{k_0}(\Theta)}  } \frac{V(P_{G,\m},P_{G_0,\m})}{D_{1}(G,G_0)}>0. \label{eqn:VmD1bounew}
        \end{equation}
        By Lemma \ref{lem:relW1D1} \ref{item:relW1D1b} one may replace the $D_1(G,G_0)$ in the preceding display by $W_1(G,G_0)$. Fix $\m_1 = n_1 \vee n_0$. Then there exists $R>0$ depending on $G_0$ such that 
        \begin{equation}
       \inf_{G\in B_{W_1}(G_0,R)\backslash \{G_0\} } \frac{V(P_{G,N_1},P_{G_0,N_1})}{W_1(G,G_0)}>0,  \label{eqn:positiveneighborhoodnew}
        \end{equation}
          where $B_{W_1}(G_0,R)$ is the open ball in metric space $(\bigcup_{k=1}^{k_0} \Ec_{k}(\Theta), W_1)$ with center at $G_0$ and radius $R$. Here we used the fact that any sufficiently small open ball in $(\bigcup_{k=1}^{k_0} \Ec_{k}(\Theta), W_1)$ with center in $\Ec_{k_0}(\Theta)$ is in $\Ec_{k_0}(\Theta)$.
        By assumption $\m\geq n_0$
        $$
        \inf_{G\in \bigcup_{k=1}^{k_0} \Ec_{k}(\Theta_1) \backslash B_{W_1}(G_0,R)} \frac{h(P_{G,\m},P_{G_0,\m})}{W_1(G,G_0)}>0. 
        $$
        Combining the last display with $N=\m_1 $ and \eqref{eqn:positiveneighborhoodnew} yields 
        $ 
        h(P_{G,\m_1},P_{G,\m_1}) \geq C(G_0,\Theta)W_1(G,G_0). 
        $
        Observing $h(P_{G,\m},P_{G_0,\m})$ increases with $\m$, the proof is then complete.
        \end{proof}
        
        
        Despite the above possibilities for relaxing the compactness assumption on $\Theta_1$, we want to point out that other assumptions may still implicitly require the compactness. For example, suppose that kernel $f$ takes the explicit form $f_{\Nc}(x|\mu,\sigma)$, the density of univariate normal distribution with mean $\mu$ and standard deviation $\sigma$. Then $h(f_{\Nc}(x|\mu,\sigma_1),f_{\Nc}(x|\mu,\sigma_2))=1-\sqrt{\frac{2\sigma_1\sigma_2}{\sigma_1^2+\sigma_2^2}}$. With $\sigma_2=2\sigma_1$, $h(f_{\Nc}(x|\mu,\sigma_1),f_{\Nc}(x|\mu,\sigma_2))=1-\sqrt{\frac{4}{5}}$ which can not be upper bounded $L_2|\sigma_2-\sigma_1|^{\beta_0}=L_2\sigma_1^{\beta_0}$, a quantity convergences to $0$ when $\sigma_1$ converges to $0$. That is, the assumption \ref{item:kernel} cannot hold if $\sigma$ is not bounded away from $0$, which excludes bounded intervals of the form $(0,a)$. 

\subsection{Kernel $P_\theta$ is a location-scale mixture of Gaussian distributions} 
\label{sec:locationscalmixturegaussian}
In Section \ref{sec:mixofgaussian} we demonstrated an application of Theorem~\ref{thm:genthm} to obtain inverse bound \eqref{eqn:curvatureprodbound} when kernel $P_\theta$ is the location mixture of Gaussian distributions. It is of interest to extend Theorem~\ref{thm:genthm} to richer kernels often employed in practice. 
The local-scale mixture of Gaussian distributions represent a salient example. Here, we shall discuss several technical difficulties that arise in such a pursuit. The first difficulty is that in Theorem~\ref{thm:genthm}, the parameter space $\Theta$ is assumed to be a subset of an Euclidean space obtained via a suitable (i.e., homeomorphic) parametrization. For the "location-scale Gaussian mixture" kernel, such a parametrization is elusive. 

Recall that the parameter set of a $k$-component location mixture of Gaussian distributions given by \eqref{eqn:locationmixture} is 
$
\bar{\Theta} := \{\{(\pi_i,\mu_i)\}_{i=1}^k: \sum_{i=1}^k \pi_i =1, \mu_i\neq \mu_j \},
$
or $\tilde{\Theta}:=\{\sum_{i=1}^{k} \pi_i\delta_{\mu_i}: \sum_{i=1}^k \pi_i =1, \mu_i\neq \mu_j\}=\Ec_k(\R)$.
To apply the result in Theorem \ref{thm:genthm} we parametrize the kernel and index it by parameters in a subset of a suitable Euclidean space. In Section \ref{sec:mixofgaussian} we identify $\bar{\Theta}$ or $\tilde{\Theta}$ by a subset $\Theta$ of Euclidean space as in \eqref{eqn:parameterspaceforlocationmixturegaussian} by ranking $\mu_i$ in increasing order. This identification is a bijection and moreover a homeomorphism. The properties of bijection and homeomorphism are necessary for the reparametrization since we need convergence of the parameters in the reparametrization space is equivalent to the convergence in the original parameter space $\bar{\Theta}$. So the parametrization is suitable for the application of Theorem~\ref{thm:genthm}. However this scheme is not straightforward to be generalized to the case the atom space (space of $\mu$ in this particular example) is more than one dimension as discussed below.

For the case of $k$-component location-scale mixture of Gaussian distributions, the parameter set is $
\tilde{\Theta} := \{\sum_{i=1}^k \pi_i \delta_{(\mu_i,\delta_i)}: \sum_{i=1}^k \pi_i =1, (\mu_i,\sigma_i) \neq (\mu_j,\sigma_j) \text{ and } \sigma_i>0 \} = \Ec_k(\R\times \R_+).
$
Similar to the location mixture, one may attempt to reparametrize $\tilde{\Theta}$ by ranking $(\mu_i,\sigma_i)$ in the lexicographically increasing order. While this reparametrization is a bijection, it is not a homeomorphism. To see this, consider $k=2$ and $F_0=\frac{1}{2} \delta_{ (1, 3)}+ \frac{1}{2}\delta_{(1,2)}\in \bar{\Theta}$. The reparametrization of $F_0$ is $(\frac{1}{2},1,2,1,3)$ since $(1,2)<(1,3)$ in the lexicographically order. Consider $F_n=\frac{1}{2} \delta_{ (1-\frac{1}{n}, 3)}+ \frac{1}{2}\delta_{(1+\frac{1}{n},2)}$ and its reparametrization is $(\frac{1}{2},1-\frac{1}{n},3,1+\frac{1}{n},2)$. It is clear that $W_1(F_n,F_0)\to 0$ as $n\to\infty$ but the Euclidean distance of the corresponding reparametrized parameters does not. 
This issue underscore one among many challenges that arise as we look at increasingly richer models that have already been widely applied in numerous application domains.


	
\subsection{Other extensions} A direction of interest is the study of overfitted mixture models, i.e., the true number of mixture components $k_0$ may be unknown and $k_0 \leq k$. As previous studies suggest, a stronger notion of identifiability such as second-order identifiability may play a fundamental role (see~\cite{ho2019singularity}). Observing that \eqref{eqn:curvatureprodbound} can also be viewed as uniform versions of \eqref{eqn:genthmcon} since they holds for any fixed $G$ in a neighborhood of $G_0$ and any $H \overset{W_1}{\to} G$, it would also be interesting to generalize Theorem \ref{thm:posconnotid} to a uniform result beyond a fixed $G_0$. In addition, if $P_\theta$ is taken to a mixture distribution, what happens if this mixture is also overfitted? We can expect a much richer range of parameter estimation behavior and more complex roles $m$ and $N$ play in the rates of convergence.

\addcontentsline{toc}{section}{References}

	\bibliography{identifiability}{}
	\bibliographystyle{plain}

\appendix

\section{Examples and Proofs for Section 3}
\label{sec:proofsindifferentsections}

\begin{exa}
\label{exa:wassersteinandDN}
    Consider $G_1=p_1^1\delta_{\theta_1}+p_2^1\delta_{\theta_2}$ and $G_2=p_1^2\delta_{\theta_1}+p_2^2\delta_{\theta_2}\in \Ec_2(\Theta)$ with $p_1^1\not = p_1^2$. When $\m$ is sufficiently large, $D_\m(G_1,G_2)=|p_1^1-p_1^2|+|p_2^1-p_2^2|$, a constant independent of $\m$. But with $\dtheta$ being Euclidean distance multiplied by $\sqrt{\m}$ 
    \begin{align*}
    W_p^p(G_1,G_2;\dtheta) = & \min_{\bm{q}}(q_{12}+q_{21})\left(\sqrt{\m}\|\theta_1-\theta_2\|_2\right)^p\\
    =&\left(\sqrt{\m}\|\theta_1-\theta_2\|_2\right)^p \frac{1}{2}\left(|p_1^1-p_1^2|+|p_2^1-p_2^2|\right),
    \end{align*}
    where $\bm{q}$ is a coupling as in \eqref{eqn:Wpdef}. So 
    $$
    W_p(G_1,G_2;\dtheta)=\sqrt{\m}\|\theta_1-\theta_2\|_2\left(\frac{1}{2}(|p_1^1-p_1^2|+|p_2^1-p_2^2|)\right)^{1/p},
    $$ which increases to $\infty$ when $\m\to \infty$. Even though $G_1$ and $G_2$ share the set of atoms, $W_p(G_1,G_2;\dtheta)$ is still of order $\sqrt{\m}$. Thus, $W_p(G_1,G_2;\dtheta)$ couples atoms and probabilities; in other words it does not  separate them in the way $D_{\m}$ does. 
    \myeoe
\end{exa}

\begin{proof}[Proof of Lemma \ref{lem:relW1D1}]
\label{proof:relW1D1}
a) 
        The proof is trivial and is therefore omitted.
        
b) 
		Let $G=\sum_{i=1}^k p_i \delta_{\theta_i}$ and $G'=\sum_{i=1}^k p'_i \delta_{\theta'_i}$. 
		Let $\tau$ be the optimal permutation that achieves $D_1(G,G')=\sum_{i=1}^k \left(\|\theta_{\tau(i)}-\theta'_i\|_2+|p_{\tau(i)}-p'_i|\right)$. Let $\bm{q}$ be a coupling of the mixing probabilities $\bm{p}=(p_1,\ldots,p_k)$ and $\bm{p}'=(p'_1,\ldots,p'_k)$ such that $q_{\tau(i),i}= \min\{p_{\tau(i)},p_i\}$ and then the remaining mass to be assigned is $\sum_{i=1}^k(p_{\tau(i)}-q_{\tau(i),i})= \frac{1}{2}\sum_{i=1}^k|p_{\tau(i)}-p_i| $. Thus,
		\begin{align*}
		W_1(G,G')\leq & \sum_{i=1}^k q_{\tau(i),i} \|\theta_{\tau(i)}-\theta'_i\|_2 + \frac{1}{2}\sum_{i=1}^k|p_{\tau(i)}-p_i| \text{diam}(\Theta)\\
		\leq & \max\left\{1,\frac{\text{diam}(\Theta)}{2}\right\} D_1(G,G'). 
		\end{align*}
		The proof for the case $p\neq 1$ proceeds in the same procedure.

c) 
		Consider any $G_n\in \Ec_{k}(\Theta)$ and $G_n\overset{W_1}{\to} G_0$, and one may write $G_n = \sum_{i=1}^k p_i^n \delta_{\theta_i^n}$ for $n\geq 0$ such that $p_i^n\to p_i^0$ and $\theta_i^n \to \theta_i^0$. 
		Then when $n$ is sufficiently large, $G_n\in \Ec_{k}(\Theta_1)$ for $\Theta_1=\bigcup_{i=1}^{k_0}B(\theta_i^0,\frac{1}{2})$, where $B(\theta_i^0,\rho)\subset \R^q$ is the open ball with center at $\theta_i^0$ of radius $\rho$. Then by \ref{item:relW1D1b} for large $n$, $ W_1(G_n,G_0)\leq C(G_0) D_1(G_n,G_0)$, which entails $\liminf\limits_{\substack{G \overset{W_1}{\to} G_0\\G\in \Ec_k(\Theta)}} \frac{D_1(G,G_0)}{W_1(G,G_0)}>0$.

		Denote $\bm{p}^n=(p_1^n,\ldots,p_k^n)$ for $n\geq 0$. Let $\bm{q}_n=(q^n_{ij})_{i,j\in[k]}$ be a coupling between 
		$\bm{p}^n$ and $\bm{p}^0$ such that $W_1(G_n,G_0)=\sum_{ij} q^n_{ij} \|\theta^n_i-\theta^0_j\|_2$. Since $\theta^n_j\to\theta^0_j$, when $n$ is large, 
		\begin{equation}
		\label{eqn:W1lowerboutheta}
		W_1(G_n,G_0) = \sum_{ij} q^n_{ij} \|\theta^n_i-\theta^0_j\|_2 \geq \sum_{ij} q^n_{ij} \|\theta^n_j-\theta^0_j\|_2 = \sum_{j=1}^k p_j^0 \|\theta^n_j-\theta^0_j\|_2.
		\end{equation}
		Moreover, when $n$ is large, $\|\theta_\alpha^n-\theta_\beta^0\|_2\geq \frac{1}{2}\min_{1\leq i<\ell\leq k} \|\theta_i^0 -\theta_\ell^0\|_2:=\frac{1}{2}\rho$ for any $\alpha\neq \beta$. Thus when $n$ is large,
		\begin{multline}
		W_1(G_n,G_0) 
		\geq \sum_{\alpha\neq \beta} q^n_{\alpha\beta} \|\theta^n_\alpha-\theta^0_\beta\|_2 
		\geq \frac{1}{2}\rho \sum_{\alpha\neq \beta} q^n_{\alpha\beta}
		\geq \frac{1}{4}\rho \sum_{j=1}^k|p_j^n-p_j^0|,
		\label{eqn:W1lowerboup}
		\end{multline}
		where the last inequality follows from 
		$$
		\frac{1}{2} \sum_{j=1}^k|p_j^n-p_j^0| = V(\bm{p}^n,\bm{p}^0) = \inf_{ \bm{\pi} \text{ coupling of } \bm{p}^n \text{ and } \bm{p}^0} \  \sum_{\alpha\neq \beta} \pi_{\alpha\beta} \leq \sum_{\alpha\neq \beta} q^n_{\alpha\beta}.
		$$

		Combining \eqref{eqn:W1lowerboutheta} and \eqref{eqn:W1lowerboup}, for sufficiently large $n$, 
		\begin{align*}
		W_1(G_n,G_0) \geq & \frac{1}{2} \sum_{j=1}^k p_j^0 \|\theta_j^n - \theta_j^0\|_2 + \frac{1}{8}\rho \sum_{j=1}^k|p_j^n-p_j^0| \\
		\geq & 
		\frac{1}{2} \min\left\{ \min_{\ell}p_\ell^0,    \frac{1}{4}\rho \right\} \sum_{j=1}^k(\|\theta_j^n - \theta_j^0\|_2+|p_j^n-p_j^0|) \\
		= & 
		\frac{1}{2} \min\left\{ \min_{\ell}p_\ell^0,    \frac{1}{4}\rho \right\} D_1(G_n,G_0),
		\end{align*}
		which entails $\liminf\limits_{\substack{G \overset{W_1}{\to} G_0\\G\in \Ec_k(\Theta)}} \frac{W_1(G,G_0)}{D_1(G,G_0)}>0$.
		
d) 
		Based on \ref{item:relW1D1d}, there exists $c(G_0)>0$ such that for $G\in \Ec_{k_0}(\Theta)$ satisfying $W_1(G,G_0)<c(G_0)$:
		$
		W_1(G,G_0)\geq C_1(G_0) D_1(G,G_0). 
		$
		For $G\in \Ec_{k_0}(\Theta)$ satisfying $W_1(G,G_0)\geq c(G_0)$:
		$$
		\frac{W_1(G,G_0)}{D_1(G,G_0)}\geq \frac{c(G_0)}{k_0 \text{diam}(\Theta)+1}. 
		$$
\end{proof} 

\section{Additional examples and proofs for Section 4}


\subsection{Additional examples and proofs for Section \ref{subsec:basictheory}}
\label{sec:proofsfirstorderidentifiable}

\begin{exa}[Location gamma kernel] \label{exa:locgamma}
For gamma distribution with fixed $\alpha\in (0,1)\bigcup (1,2)$ and $\beta>0$, consider its location family with density  
$$
f(x|\theta) = \frac{\beta^{\alpha}(x-\theta)^{\alpha-1}e^{-\beta(x-\theta)}}{\Gamma(\alpha)}\1ve_{(\theta,\infty)}(x)
$$ w.r.t. Lebesgue measure $\mu$ on $\Xfrak = \R$. The parameter space is $\Theta = \R$. Observe that
$$
\lim_{a\to 0^+} \frac{f(\theta_0|\theta_0+a)-f(\theta_0|\theta_0)}{a}= 0
$$
and 
$$
\lim_{a\to 0^+} \frac{f(\theta_0|\theta_0-a)-f(\theta_0|\theta_0)}{a} = \frac{\beta^\alpha}{\Gamma(\alpha)}\lim_{a\to 0^+} a^{\alpha-2}e^{-\beta a} = \infty,
$$
since $\alpha<2$. Then for any $x$, $f(x|\theta)$ as a function of $\theta$ is not differentiable at $\theta =x$. So it is not identifiable in the first order as defined in \cite{ho2016strong}.
However, this family does satisfy the $(\{\theta_i\}_{i=1}^k,\Nc)$ first-order identifiable definition with $\Nc=\bigcup_{i=1}^{k}(\theta_i-\rho,\theta_i+\rho)$ where $\rho=\frac{1}{4}\min_{1\leq i<j\leq k}|\theta_i-\theta_j|$. Indeed, observing that
$$
\frac{\partial}{\partial \theta} f(x|\theta) = \left(\beta-\frac{\alpha-1}{x-\theta}\right) f(x|\theta), \quad \forall \theta \not = x,
$$
then \eqref{eqn:conlininda} become
$$
0
= \sum_{i=1}^k \left( a_i \beta-a_i \frac{\alpha-1}{x-\theta_i}  + b_i  \right)f(x|\theta_i)\quad   \text{for }\ \mu-a.e.\ x\in \R\backslash \Nc.
$$
Without loss of generality, assume $\theta_1<\ldots<\theta_k$. Then for $\mu-a.e.\ x\in (\theta_1,\theta_2)\backslash\Nc = [\theta_1+\rho,\theta_2-\rho]$, the above display become 
$$
\left( a_1 \beta-a_1 \frac{\alpha-1}{x-\theta_1}  + b_1  \right)\frac{\beta^{\alpha}(x-\theta_1)^{\alpha-1}e^{-\beta(x-\theta_1)}}{\Gamma(\alpha)} =0 
$$
which implies $a_1=b_1=0$ since $\alpha\neq 1$. Repeating the above argument on interval $(\theta_2,\theta_3),\ldots (\theta_k,\infty)$ shows $a_i=b_i=0$ for any $i\in [k]$. 
So  this family is $(\{\theta_i\}_{i=1}^k,\Nc)$ first-order identifiable. Moreover, for every $x$ in $\R \backslash \Nc$,  $f(x|\theta)$ is continuously differentiable w.r.t. $\theta$ in a neighborhood of  $\theta_i^0$ for $i\in [k_0]$. By Lemma \ref{lem:firstidentifiable} \ref{item:firstidentifiableb} for any $G_0\in \Ec_{k_0}(\Theta)$ \eqref{eqn:curvaturebound} holds. \myeoe
\end{exa}

\begin{proof}[Proof of Lemma \ref{lem:firstidentifiable} \ref{item:firstidentifiableb}] 

 Suppose the equation \eqref{eqn:curvaturebound} is incorrect. Then there exists $G_\ell,H_\ell \in \Ec_{k_0}(\Theta)$  such that         
 $$
	    \begin{cases}
	    G_\ell\neq H_\ell, & \forall \ell \\
	    G_\ell,H_\ell \overset{W_1}{\to} G_0, & \text{ as } \ell \to \infty \\
	    \frac{V(P_{G_\ell},P_{H_\ell})}{D_1(G_\ell,H_\ell)} \to 0, & \text{ as } \ell \to \infty.
	    \end{cases}
	    $$
	    We may relabel the atoms of $G_\ell$ and $H_\ell$ such that $G_\ell= \sum_{i=1}^{k_0}p^\ell_i \delta_{\theta_i^\ell}$, $H_\ell = \sum_{i=1}^{k_0}\pi_i^\ell \delta_{\eta_i^\ell}$ with $\theta_i^\ell,\eta_i^\ell \to \theta_i^0$ and $p_i^\ell,\pi_i^\ell \to p_i^0$ as $\ell\to \infty$ for any $i\in [k_0]$. With subsequences argument if necessary, we may further require
	    \begin{equation}
	    \frac{\theta_i^\ell - \eta_i^\ell}{D_1(G_\ell,H_\ell)} \to a_i \in \R^q, \quad \frac{p_i^\ell - \pi_i^\ell}{D_1(G_\ell,H_\ell)} \to b_i\in \R,  \quad \forall 1\leq i \leq k_0, \label{eqn:thetaDGlHlrel}
	    \end{equation}
	    where $b_i$ and the components of $a_i$ are in $[-1,1]$ and $\sum_{i=1}^{k_0}b_i =0$. Moreover, $D_1(G_\ell,H_\ell)=\sum_{i=1}^{k_0}\left(\|\theta^\ell_i-\eta_i^\ell\|_2 +|p^\ell_i-\pi_i^\ell|\right)$ for sufficiently large $\ell$, which implies 
	    \begin{equation*}
	    \sum_{i=1}^{k_0}\|a_i\|_2+\sum_{i=1}^{k_0}|b_i| = 1. 
	    \end{equation*} 
	    It then follows that at least one of $a_i$ is not $\bm{0}\in \R^q$ or one of $b_i$ is not $0$.   
	    On the other hand,
	    \begin{align*}
	    0 & =\lim_{\ell\to \infty}\frac{2V(P_{G_\ell},P_{H_\ell})}{D_1(G_\ell,H_\ell)} \\ 
	    &\geq \lim_{\ell\to \infty}\int_{\Xfrak\backslash \Nc}\left| \sum_{i=1}^{k_0} p_i^\ell \frac{f(x|\theta_i^\ell) - f(x|\eta_i^\ell)}{D_1(G_\ell,H_\ell)} + \sum_{i=1}^{k_0} f(x|\eta_i^\ell) \frac{p_i^\ell - \pi_i^\ell}{D_1(G_\ell,H_\ell)}   \right| \mu(dx)\\
	    & \geq \int_{\Xfrak\backslash \Nc} \liminf_{\ell\to \infty}\left| \sum_{i=1}^{k_0} p_i^\ell \frac{f(x|\theta_i^\ell) - f(x|\eta_i^\ell)}{D_1(G_\ell,H_\ell)} + \sum_{i=1}^{k_0} f(x|\eta_i^\ell) \frac{p_i^\ell - \pi_i^\ell}{D_1(G_\ell,H_\ell)}   \right| \mu(dx) \\
	    & = \int_{\Xfrak\backslash \Nc} \left| \sum_{i=1}^{k_0} p_i^0 a_i^\top  \nabla_{\theta}f(x|\theta_i^0) + \sum_{i=1}^{k_0} f(x|\theta_i^0) b_i   \right| \mu(dx),
	    \end{align*}
	    where the second inequality follows from Fatou's Lemma, and the last step follows from Lemma \ref{lem:taylortwisted} \ref{item:taylortwisteda}. Then $\sum_{i=1}^{k_0} p_i^0 a_i^\top  \nabla_{\theta}f(x|\theta_i^0) + \sum_{i=1}^{k_0} f(x|\theta_i^0) b_i=0$ for $\mu-a.e.\ x\in \Xfrak \backslash \Nc$. Thus we find a nonzero solution to \eqref{eqn:conlininda}, \eqref{eqn:conlinindb} with $k,\theta_i$ replaced by $k_0,\theta_i^0$. 
	    
	    However, the last statement contradicts with the definition of $(\{\theta_i^0\}_{i=1}^{k_0},\Nc)$ first-order identifiable.
	\end{proof}

	\begin{proof}[Proof of Lemma \ref{lem:nececondition}]
	By Lemma \ref{lem:firstidentifiableelaborate} \ref{item:firstidentifiableelaborateb} $(a_1,b_1,\ldots,a_{k_0},b_{k_0})$ is also a nonzero solution of the system of equations \eqref{eqn:conlininda}, \eqref{eqn:conlinindb}. Let $a'_i = \frac{a_i/p_i^0}{\sum_{i=1}^{k_0}\left(\|a_i/p_i^0\|_2+|b_i|\right)}$ and $b'_i=\frac{b_i}{\sum_{i=1}^{k_0}\left(\|a_i/p_i^0\|_2+|b_i|\right)}$. Then $a'_i$ and $b'_i$ satisfy  $\sum_{i=1}^{k_0}\left(\|a'_i\|_2+|b'_i|\right)=1$ and $(p_1^0 a'_1$, $b'_1,\ldots,p_{k_0}^0 a'_{k_0}$, $b'_{k_0})$
	    is also a nonzero solution of the system of equations \eqref{eqn:conlininda}, \eqref{eqn:conlinindb} with $k,\theta_i$ replaced respectively by $k_0,\theta_i^0$.
	    Let $G_\ell=p_i^\ell \delta_{\theta_i^\ell}$ with $p_i^\ell=p_i^0+b'_i \frac{1}{\ell}$ and $\theta_i^\ell=\theta_i^0+\frac{1}{\ell}a'_i$ for $1\leq i \leq k_0$. When $\ell$ is large, $0<p_i^\ell<1$ and $\theta_i^\ell\in \Theta$ since $0<p_i^0<1$ and $\theta_i^0\in \Theta^\circ$. Moreover, $\sum_{i=1}^{k_0}p_i^\ell=1$ since $\sum_{i=1}^{k_0}b'_i=0$. Then $G_\ell\in \Ec_{k_0}(\Theta)$ and $G_\ell\not = G_0$ since at least one of $a'_i$ or $b'_i$ is nonzero. When $\ell$ is large $D_1(G_\ell,G_0)=\sum_{i=1}^{k_0}\left(\|\theta_i^\ell-\theta_i^0\|_2+|p_i^\ell-p_i^0|\right)=\frac{1}{\ell}$. Thus when $\ell$ is large
	    \begin{align}
	        \frac{2V(P_{G_\ell},P_{G_0})}{D_1(G_\ell,G_0)}  
	    = & \int_{\Xfrak\backslash \Nc}\left| \sum_{i=1}^{k_0} p_i^\ell \frac{f(x|\theta_i^\ell) - f(x|\theta_i^0)}{1/\ell} + \sum_{i=1}^{k_0} b'_i f(x|\theta_i^0)    \right| \mu(dx). \label{eqn:2VD1inteformtemp1}
	    \end{align}
	Since by condition \ref{item:nececonditionc} when $\ell$ is large
	$$
	\left|\frac{f(x|\theta_i^\ell) - f(x|\theta_i^0)}{1/\ell}\right| = \left|\frac{f(x|\theta_i^0+ \frac{1}{\ell} \frac{\|a'_i\|_2}{\|a_i\|_2}  a_i) - f(x|\theta_i^0)}{1/\ell}\right| \leq \frac{\|a'_i\|_2}{\|a_i\|_2} \bar{f}(x|\theta_i^0,a_i),
	$$
	the integrand of \eqref{eqn:2VD1inteformtemp1} is bounded by $\sum\limits_{i=1}^{k_0}\frac{1/p_i^0}{\sum_{i=1}^{k_0}\left(\|a_i/p_i^0\|_2+|b_i|\right)}\bar{f}(x|\theta_i^0,a_i)+\sum\limits_{i=1}^{k_0} |b'_i| f(x|\theta_i^0)$, which is integrable w.r.t. to $\mu$ on $\Xfrak\backslash \Nc$. 
	Then by the dominated convergence theorem 
	\begin{align*}
	        \lim_{\ell\to \infty}\frac{2V(P_{G_\ell},P_{G_0})}{D_1(G_\ell,G_0)}  
	    = & \int_{\Xfrak \backslash \Nc}\left| \sum_{i=1}^{k_0} p_i^0 \langle a'_i, \nabla_{\theta} f(x|\theta_i^0)\rangle + \sum_{i=1}^{k_0} b'_i f(x|\theta_i^0) \right| \mu(dx) =0.
	    \end{align*}
	Thus 
	$$
	\liminf_{\substack{G\overset{W_1}{\to} G_0\\ G\in \Ec_{k_0}(\Theta)}} \frac{V(P_G,P_{G_0})}{D_1(G,G_0)}=0.
	$$
	and the proof is completed by \eqref{eqn:reltwoinversebounds}.
	\end{proof}

\begin{proof}[Proof of Lemma \ref{lem:unifidentifiable}]
    It suffices to prove \eqref{eqn:curvaturebound} since \eqref{eqn:noproductlowbou} is a direct consequence of \eqref{eqn:curvaturebound}.
    
    Without loss of generality, assume $\theta_1^0<\theta_2^0<\ldots <\theta_{k_0}^0$. Let $\Nc=\bigcup_{i=1}^{k_0}(\theta_i^0-\rho,\theta_i^0+\rho)$, where $\rho=\frac{1}{4}\min_{1\leq i<j\leq k_0}|\theta_i^0-\theta_j^0|$. Notice that for $x\in \R\backslash \Nc$, $f(x|\theta)$ as a function of $\theta$ is continuously differentiable on $(\theta_i^0-\rho,\theta_i^0+\rho)$ for each $i\in [k_0]$.
    
    Suppose \eqref{eqn:curvaturebound} is not true. Proceed exactly the same as the proof of Lemma \ref{lem:firstidentifiable} \ref{item:firstidentifiableb} except the last paragraph to obtain a nonzero solution $(p_i^0 a_i, b_i: i\in [k_0])$ of \eqref{eqn:conlininda}, \eqref{eqn:conlinindb} with $k,\theta_i$ replaced by $k_0,\theta_i^0$. For the uniform distribution family, one may argue that the nonzero solution has to satisfy
    \begin{equation}
    -p_i^0 a_i/\theta_i^0 +b_i = 0 \quad \forall i\in [k_0]. \label{eqn:unifaibirel}
    \end{equation}
    Indeed, start from the rightmost interval that intersects with the support from only one mixture component,  for $\mu-a.e.\ x\in (\theta_{k_0-1}^0,\theta_{k_0}^0)\backslash \Nc = [\theta_{k_0-1}^0+\rho,\theta_{k_0}^0-\rho]$ 
    \begin{align*}
    0 = &\sum_{i=1}^{k_0}\left(p_i^0  a_i\frac{\partial}{\partial \theta}f(x|\theta_i^0)  + b_i f(x|\theta_i^0) \right)\\
    = & \sum_{i=1}^{k_0}\left(-p_i^0  a_i/\theta_i^0  + b_i\right) f(x|\theta_i^0) \\
    = & \left(-p_{k_0}^0  a_{k_0}/\theta_{k_0}^0  + b_{k_0}\right)/\theta_{k_0}^0,
    \end{align*}
    which implies $-p_{k_0}^0  a_{k_0}/\theta_{k_0}^0  + b_{k_0}=0$. Repeat the above argument on interval $(\theta_{k_0-2}^0,\theta_{k_0-1}^0)$, $\ldots$, $(\theta_{1}^0,\theta_{2}^0)$, $(0,\theta_{1}^0)$ and \eqref{eqn:unifaibirel} is established.
    
    Combining~\eqref{eqn:unifaibirel} with the fact that some of the $a_i$ or $b_i$ is non-zero, it follows that $|a_\alpha|> 0$ for some $\alpha\in [k_0]$. When $\ell$ is sufficiently large, $\theta_i^\ell,\eta_i^\ell \in (\theta_i^0-\rho,\theta_i^0+\rho)$. For sufficiently large $\ell$
    \begin{align*}
    &\frac{2V(P_{G_\ell},P_{H_\ell})}{D_1(G_\ell,H_\ell)}\\
    \geq &   \frac{1}{D_1(G_\ell,H_\ell)}\int_{\min\{\theta_\alpha^\ell, \eta_\alpha^\ell\}}^{\max\{\theta_\alpha^\ell, \eta_\alpha^\ell\}} \left|p_{G_\ell}(x)-p_{H_\ell}(x)\right| dx\\
    \overset{(*)}{=} &  \frac{1}{D_1(G_\ell,H_\ell)} \int_{\min\{\theta_\alpha^\ell, \eta_\alpha^\ell\}}^{\max\{\theta_\alpha^\ell, \eta_\alpha^\ell\}} \left|{ \frac{\pi_\alpha^\ell \bm{1}(\theta_\alpha^\ell<\eta_\alpha^\ell)+p_\alpha^\ell \bm{1}(\theta_\alpha^\ell\geq\eta_\alpha^\ell)}{\max\{\theta_\alpha^\ell, \eta_\alpha^\ell\}} + \sum_{i=\alpha+1}^{k_0} \frac{p_i^{\ell}}{\theta_i^\ell}-\sum_{i=\alpha+1}^{k_0}\frac{\pi_i^{\ell}}{\eta_i^\ell} }\right| dx\\
    \overset{(**)}{=} &  \frac{|\theta_\alpha^\ell-\eta_\alpha^\ell|}{D_1(G_\ell,H_\ell)} \left|{ \frac{\pi_\alpha^\ell \bm{1}(\theta_\alpha^\ell<\eta_\alpha^\ell)+p_\alpha^\ell \bm{1}(\theta_\alpha^\ell\geq\eta_\alpha^\ell)}{\max\{\theta_\alpha^\ell, \eta_\alpha^\ell\}}+ \sum_{i=\alpha+1}^{k_0} \frac{p_i^{\ell}}{\theta_i^\ell}-\sum_{i=\alpha+1}^{k_0}\frac{\pi_i^{\ell}}{\eta_i^\ell} }\right|\\
    \to & {|a_\alpha|}\frac{p_\alpha^0}{\theta_\alpha^0}>0,
    \end{align*}
    where the step $(*)$ follows from carefully examining the support of $f(x|\theta)$,
    the step $(**)$ follows from the integrand is a constant, and the last step follows from \eqref{eqn:thetaDGlHlrel}. The last display contradicts with the choice of $G_\ell,H_\ell$, which satisfies $\frac{V(P_{G_\ell},P_{H_\ell})}{D_1(G_\ell,H_\ell)}\to 0$.
\end{proof}

\begin{proof}[Proof of Lemma \ref{lem:exadisidentifiable}]
    
    (a): inverse bound~\eqref{eqn:noproductlowbou} holds \\
    Without loss of generality, assume $\xi_1^0\leq \xi_2^0\leq \ldots \leq \xi_{k_0}^0$. Let $\Nc=\bigcup_{i=1}^{k_0}\{\xi_i^0\}$. Notice that for $x\in \R\backslash \Nc$, $f(x|\theta)$ as a function of $\theta$ is  differentiable at $\theta_i^0=(\xi_i^0,\sigma_i^0)$ for each $i\in [k_0]$.
    
    Suppose \eqref{eqn:noproductlowbou} is not true. Proceed exactly the same as the proof of Lemma \ref{lem:firstidentifiable} \ref{item:firstidentifiablea} except the last paragraph to obtain a nonzero solution $(p_i^0 a_i, b_i: i\in [k_0])$ of \eqref{eqn:conlininda}, \eqref{eqn:conlinindb} with $k,\theta_i$ replaced by $k_0,\theta_i^0$. Write the two-dimensional vector $a_i$ as $a_i=(a_i^{(\xi)},a_i^{(\sigma)})$. For the location-scale exponential distribution, one may argue that the nonzero solution has to satisfy
    \begin{equation}
    a_i^{(\sigma)}=0,\quad p_i^0 a^{(\xi)}_i/\sigma_i^0 +b_i = 0, \quad \forall i\in [k_0]. \label{eqn:expaibirel}
    \end{equation}
    Indeed, let $\bigcup_{i=1}^{k_0}\{\xi_i^0\}=\{\xi'_1,\xi'_2,\ldots,\xi'_{k'}\}$ with $\xi'_1<\xi'_2<\ldots<\xi'_{k'}$ where $k'$ is the number of distinct elements. Define $I'(\xi)=\{i\in[k_0]: \xi_i^0 = \xi\}$.
    Then for $\mu-a.e. \ x\in \R\backslash \Nc$
    \begin{align*}
    0 = &\sum_{i=1}^{k_0}\left(p_i^0 \langle a_i, \nabla_{(\xi,\sigma)}f(x|\xi_i^0,\sigma_i^0) \rangle + b_i f(x|\xi_i^0,\sigma_i^0) \right) \\
    = &\sum_{j=1}^{k'}\sum_{i\in I'(\xi'_j)}\left(p_i^0 \langle a_i, \nabla_{(\xi,\sigma)}f(x|\xi'_j,\sigma_i^0) \rangle + b_i f(x|\xi'_j,\sigma_i^0) \right)\\
    = & \sum_{j=1}^{k'}\ \ \sum_{i\in I'(\xi'_j)}\left( p_i^0 a_i^{(\xi)} \frac{1}{\sigma_i^0} + p_i^0 a_i^{(\sigma)}\frac{x-\xi_i^0-\sigma_i^0}{(\sigma_i^0)^2} +b_i   \right) f(x|\xi'_j,\sigma_i^0).
    \end{align*}
    Start from the leftmost interval that intersects with the support from only one mixture component, for $\mu-a.e.\ x\in (\xi'_1,\xi'_2)\backslash \Nc =(\xi'_1,\xi'_2)$,
    \begin{align*}
    0 = & \sum_{i\in I'(\xi'_1)}\left( p_i^0 a_i^{(\xi)} \frac{1}{\sigma_i^0} + p_i^0 a_i^{(\sigma)}\frac{x-\xi_i^0-\sigma_i^0}{(\sigma_i^0)^2} +b_i  \right) f(x|\xi'_1,\sigma_i^0)\\
    =&\sum_{i\in I'(\xi'_1)}\left( p_i^0 a_i^{(\xi)} \frac{1}{\sigma_i^0} + p_i^0 a_i^{(\sigma)}\frac{x-\xi_i^0-\sigma_i^0}{(\sigma_i^0)^2} +b_i  \right) \exp\left(\frac{\xi'_1}{\sigma_i^0}\right) \exp\left(-\frac{x}{\sigma_i^0}\right).
    \end{align*}
    Since $\sigma_i^0$ for $i\in I'(\xi'_1)$ are all distinct, by Lemma \ref{lem:polexplinind} \ref{item:polexplininda}
    $$ a_i^{(\sigma)}=0,\quad p_i^0 a^{(\xi)}_i/\sigma_i^0 +b_i = 0, \quad \forall i\in I'(\xi'_1). $$
    Repeat the above argument on interval $(\xi'_2,\xi'_3),\ldots, (\xi'_{k'-1},\xi'_{k'}), (\xi'_{k'},\infty)$ and \eqref{eqn:expaibirel} is established.
    
Since at least one of $a_i$ or $b_i$ is not zero, from \eqref{eqn:expaibirel} it is clear that 
at least one of $\{b_i\}_{i=1}^{k_0}$ is not zero. Then by $\sum_{i=1}^{k_0}b_i=0$ at least one of $b_i$ is positive. By \eqref{eqn:expaibirel} at least one of $a_i^{(\xi)}$ is negative. Let $\alpha\in \argmax\limits_{i\in\{j\in [k_0]:a_j^{(\xi)}<0\} } a_i^{(\xi)}$. That is $a_{\alpha}^{(\xi)}$ is a largest negative one among $\{a_i^{(\xi)}\}_{i\in [k_0]}$. Let $\rho =  \frac{1}{2}\min_{1\leq i<j \leq k'} |\xi'_i-\xi'_j|$ to be half of the smallest distance among different $ \{\xi'_i\}_{i=1}^{k'}$. By subsequence argument if necessary, we require for any $i\in [k_0]$, $\xi_i^\ell \in (\xi_i^0-\rho,\xi_i^0+\rho)$.
    
     Let $I(\alpha) = \{i\in [k_0]| \xi_i^0 = \xi_\alpha^0\}$ to be the set of indices for those sharing the same $\xi_i^0$ as $\xi_{\alpha}^0$. We now consider subsequences such that $\xi_i^\ell$ for $i\in I(\alpha)$ satisfies finer properties as follows. Divide the index set $I(\alpha)$ into three subsets, $J(\alpha) := \{i\in I(\alpha)| a_i^{(\xi)}=a_\alpha^{(\xi)}\}$, $J_<(\alpha) := \{i\in I(\alpha)| a_i^{(\xi)}<a_\alpha^{(\xi)}\}$ and $J_>(\alpha) := \{i\in I(\alpha)| a_i^{(\xi)}>a_\alpha^{(\xi)}\}$. Note $J(\alpha)$
     is the index set for those sharing the same $\xi_i^0$ as $\xi_{\alpha}^0$ and sharing the same $a_i^{(\xi)}$ as $a_{\alpha}^{(\xi)}$ (so their $a_i^{(\xi)}$ are also largest negative ones among $\{a_i^{(\xi)}\}_{i\in [k_0]}$), while $J_{>}(\alpha)$ corresponds for indices $i$ for which $\xi_i^0=\xi_\alpha^0$ and $a_i^{(\xi)} \geq 0$, and $J_{<}(\alpha)$ corresponds for indices $i$ for which $\xi_i^0=\xi_\alpha^0$ and $a_i^{(\xi)} < a_{\alpha}^{(\xi)}$. 
     To be clear, the two subsets $J_{<\alpha}$ and $J_{>\alpha}$ may be empty, but $J_\alpha$ is non-empty by our definition.
     
    For any $i\in J_<(\alpha)$, $j\in J(\alpha)$
    $$
    \frac{\xi_i^\ell - \xi_{\alpha}^0}{D_1(G_\ell,G_0)}\to a_i^{(\xi)}< a_{
    \alpha}^{(\xi)}\leftarrow \frac{\xi_{j}^\ell - \xi_{\alpha}^0}{D_1(G_\ell,G_0)}.
    $$
    Then for large $\ell$, $\xi_i^\ell<\xi_j^{\ell}$ for any $i\in J_<(\alpha)$ and $j\in J(\alpha)$. Similarly for large $\ell$, $\xi_j^{\ell}<\xi_k^\ell$ for any $j\in J(\alpha)$ and $k\in J_>(\alpha)$. Thus by subsequence argument if necessary, we require $\xi_i^\ell$ additionally satisfy the conditions specified in the last two sentences for all $\ell$.
    
    Consider $\max_{j\in J(\alpha)}\{\xi_j^{\ell}\}$ and there exists $\bar{\alpha}\in J(\alpha)$ such that $\xi_{\bar{\alpha}}^{\ell}= \max_{j\in J(\alpha)}\{\xi_j^{\ell}\}$ for infinitely many $\ell$ since $J(\alpha)$ has finite cardinality.
    By subsequence argument if necessary, we require $\xi_{\bar{\alpha}}^{\ell} = \max_{j\in J(\alpha)}\{\xi_j^{\ell}\}$ for all $\ell$. Moreover, since $a_{\bar{\alpha}}^{\xi}=a_{\alpha}^{\xi} <0$ we may further require $\xi_{\bar{\alpha}}^\ell<\xi_{\alpha}^0$ for all $\ell$. Finally, for each $k\in J_>(\alpha)$ such that $a_k^{(\xi)}>0$, we may further require $\xi_{k}^{\ell}>\xi_{\alpha}^0$ for all $\ell$ by subsequences.
    To sum up, $\{\xi_i^\ell\}$ for $i\in I(\alpha)$ satisfies:
    \begin{equation}
    \begin{cases}
    \xi_i^\ell \leq \xi_{\bar{\alpha}}^\ell< \xi_{\alpha}^0  ,  & \forall \ell, \quad \forall\ i\in J_<(\alpha)\bigcup J(\alpha) \\
    \xi_i^\ell > \xi_{\bar{\alpha}}^\ell  ,  &  \forall \ell, \quad \forall \ i\in J_>(\alpha) \\
    \xi_i^\ell >\xi_{\alpha}^0 , & \forall \ell,  \forall i\in J_>(\alpha) \text{ and } a_i^{(\xi)}>0
    \end{cases}. \label{eqn:alphaproperty}
    \end{equation}
 
     Let $\bar{\xi}^\ell = \min\left\{\min\limits_{i\in \{j\in I(\alpha):a_j^{(\xi)}=0 \} } \xi_i^\ell, \xi_{\alpha}^0\right\}$ with the convention that the minimum over an empty set is $\infty$. Then $\bar{\xi}^\ell \leq \xi_{\alpha}^0$ and $\bar{\xi}^\ell \to \xi_{\alpha}^0$. Moreover, by property \eqref{eqn:alphaproperty}, $\bar{\xi}^\ell> \xi_{\bar{\alpha}}^{\ell}$.
     Thus on $(\xi_{\bar{\alpha}}^{\ell},  \bar{\xi}^\ell)$, 1) for any $i> \max I(\alpha)$, $f(x|\xi_i^\ell,\sigma_i^\ell)=0=f(x|\xi_i^0,\sigma_i^0)$  since $\xi_i^\ell,\xi_i^0\geq \xi_{\alpha}^0\geq \bar{\xi}^\ell$; 2) for $i\in J_>(\alpha)$, $f(x|\xi_i^\ell,\sigma_i^\ell)=0$ due to $\xi_i^\ell\geq \bar{\xi}^\ell$ due to \eqref{eqn:alphaproperty}; 3) for $i\in I(\alpha)$, $f(x|\xi_i^0,\sigma_i^0)=0$ since $\xi_i^0=\xi_{\alpha}^0\geq \bar{\xi}^\ell$. Then
     \begin{align*}
    &\frac{2V(P_{G_\ell},P_{G_0})}{D_1(G_\ell,G_0)}\\
    \geq &   \frac{1}{D_1(G_\ell,G_0)}\int_{\xi_{\bar{\alpha}}^\ell}^{\bar{\xi}^\ell} \left|p_{G_\ell}(x)-p_{G_0}(x)\right| dx\\
    = &  \frac{1}{D_1(G_\ell,G_0)} \int_{\xi_{\bar{\alpha}}^\ell}^{\bar{\xi}^\ell} 
    \left| \sum_{i\in J_<(\alpha)\bigcup J(\alpha)  } p_i^\ell  \frac{1}{\sigma_i^\ell  } \exp\left(-\frac{x-\xi_{\alpha}^\ell}{\sigma_i^\ell} \right) \right.\\
    &\quad \left. + \sum_{i < \min I(\alpha)} \left(p_i^\ell  \frac{1}{\sigma_i^\ell  } \exp\left(-\frac{x-\xi_i^\ell}{\sigma_i^\ell} \right) - p_i^0  \frac{1}{\sigma_i^0 }\exp\left(-\frac{x-\xi_i^0}{\sigma_i^0 }\right) \right)
     \right| dx. \numberthis \label{eqn:VD1lowboulocexpdis}
    \end{align*}
    
    Denote the integrand (including the absolute value) in the preceding display by $A_\ell(x)$. Then  as a function on $ [\xi_{\alpha}^0-\rho, \xi_{\alpha}^0]$, $A_\ell(x)$ converges uniformly to $$
    \sum\limits_{i\in J_<(\alpha)\bigcup J(\alpha) } p_i^0  \frac{1}{\sigma_i^0  } \exp\left(-\frac{x-\xi_{\alpha}^0}{\sigma_i^0} \right):=B(x).$$ Since $B(x)$ is positive and continuous on compact interval $[\xi_{\alpha}^0-\rho, \xi_{\alpha}^0]$, for large $\ell$
    $$
    |A_\ell(x)-B(x)|\leq \frac{1}{\ell} \leq \frac{1}{2} \min B(x) \leq \frac{1}{2} B(x), \quad \forall x\in [\xi_{\alpha}^0-\rho, \xi_{\alpha}^0],
    $$
    which yields 
    $$
    A_\ell(x) \geq \frac{1}{2} B(x)\geq \frac{1}{2} p_{\bar{\alpha}}^0  \frac{1}{\sigma_{\bar{\alpha}}^0  } \exp\left(-\frac{x-\xi_{\alpha}^0}{\sigma_{\bar{\alpha}}^0} \right)\geq \frac{1}{2} p_{\bar{\alpha}}^0  \frac{1}{\sigma_{\bar{\alpha}}^0  }, \quad \forall x\in [\xi_{\alpha}^0-\rho, \xi_{\alpha}^0].
    $$
    
    Plug the preceding display into \eqref{eqn:VD1lowboulocexpdis}, one obtains for large $\ell$,
    \begin{align*}
    \frac{2V(P_{G_\ell},P_{G_0})}{D_1(G_\ell,G_0)}
    \geq &  \frac{1}{D_1(G_\ell,G_0)} \int_{\xi_{\bar{\alpha}}^\ell}^{\bar{\xi}^\ell}
    \frac{1}{2} p_{\bar{\alpha}}^0  \frac{1}{\sigma_{\bar{\alpha}}^0  }  dx\\
    = &  \left(\frac{\xi_{\alpha}^0-\xi_{\bar{\alpha}}^\ell}{D_1(G_\ell,G_0)} -  \frac{\xi_{\alpha}^0-\bar{\xi}^\ell}{D_1(G_\ell,G_0)}\right)\frac{1}{2} p_{\bar{\alpha}}^0  \frac{1}{\sigma_{\bar{\alpha}}^0  } \\
    \to & (-a_{\bar{\alpha}}^{(\xi)}-0)\frac{1}{2} p_{\bar{\alpha}}^0\frac{1}{\sigma_{\bar{\alpha}}^0  } >0
     \numberthis \label{eqn:locscaexplowerbou}
    \end{align*}
    where the convergence in the last step is due to \eqref{eqn:thetaGlG0rel}. \eqref{eqn:locscaexplowerbou} contradicts with the choice of $G_\ell$, which satisfies $\frac{V(P_{G_\ell},P_{G_0})}{D_1(G_\ell,G_0)}\to 0$. 

\noindent
(b): inverse bound \eqref{eqn:curvaturebound} does not hold\\
Recall $\Theta=\R\times (0,\infty)$. Consider $G_0=\sum_{i=1}^{k_0}p_i^0\delta_{(\xi_i^0,\sigma_i^0)}\in \Ec_{k_0}(\Theta)$ with $\xi_1^0=\xi_2^0$, $\sigma_1^0\neq \sigma_2^0$ and $p_1^0/\sigma_1^0=p_2^0/\sigma_2^0$. Denote $\psi=p_1^0/\sigma_1^0>0$. Consider $G_\ell = p_1^\ell \delta_{(\xi_1^\ell,\sigma_1^0)} +p_2^\ell\delta_{(\xi_2^0,\sigma_2^0)}+\sum_{i=3}^{k_0}p_i^0\delta_{(\xi_i^0,\sigma_i^0)}$ with $p_1^\ell = p_1^0 + \frac{\psi}{2+2\psi}\frac{1}{\ell}$, $\xi_1^\ell=\xi_1^0-\frac{1}{2+2\psi}\frac{1}{\ell}$ and $p_2^\ell = p_2^0 - \frac{\psi}{2+2\psi}\frac{1}{\ell} $. Consider $H_\ell = p_2^0 \delta_{(\tilde{\xi}_2^\ell,\sigma_2^0)} + \sum_{i=1,i\neq 2}^{k_0}p_i^0\delta_{(\xi_i^0,\sigma_i^0)}$ with $\tilde{\xi}_2^\ell = \xi_1^\ell$. It is clear that when $\ell$ is large, $G_\ell, H_\ell\in \Ec_{k_0}(\Theta) $ and $G_\ell,H_\ell \to G_0$. Moreover, when $\ell$ is large, $D_1(G_\ell,H_\ell) = 1/\ell$ since $\xi_1^0=\xi_2^0$. Then 
\begin{equation}
    \frac{V(G_\ell,H_\ell)}{D_1(G_\ell,H_\ell)} = \ell \int_{\R}\left| p_1^\ell f(x|\xi_1^\ell,\sigma_1^0) +p_2^\ell f(x|\xi_1^0,\sigma_2^0) - p_1^0 f(x|\xi_1^0,\sigma_2^0)- p_2^0 f(x|\xi_1^\ell,\sigma_2^0) \right| dx \label{eqn:VD1first}
\end{equation}
since $\xi_1^0=\xi_2^0$ and $\tilde{\xi}_2^\ell = \xi_1^\ell$. Denote the integrand in \eqref{eqn:VD1first} (including the absolute value) by $A_\ell(x)$. By definition of $f(x|\xi,\sigma)$ for location-scale exponential distribution, $A_\ell(x)=0$ on $(-\infty,\xi_1^\ell)$. 

On $(\xi_1^\ell,\xi_1^0)$, $A_\ell(x)=\left|\frac{p_1^\ell}{\sigma_1^0} e^{-\frac{x-\xi_1^\ell}{\sigma_1^0}}-\frac{p_2^0}{\sigma_2^0} e^{-\frac{x-\xi_1^\ell}{\sigma_2^0}}\right|$ converges to $B(x):= \left|\frac{p_1^0}{\sigma_1^0} e^{-\frac{x-\xi_1^0}{\sigma_1^0}}-\frac{p_2^0}{\sigma_2^0} e^{-\frac{x-\xi_1^0}{\sigma_2^0}}\right|$ uniformly. Then 
\begin{align*}
    &\limsup_\ell\ \ell\int_{\xi_1^\ell}^{\xi_1^0} A_\ell(x)dx\\
    \leq & \limsup_\ell\ \ell\int_{\xi_1^\ell}^{\xi_1^0} B(x)dx + \limsup_\ell\ \ell (\xi_1^0-\xi_1^\ell) \sup_{x\in (\xi_1^\ell,\xi_1^0)}{|A_\ell(x)-B(x)|}  \\
    = & \limsup_\ell\ \frac{1}{2+2\psi} \frac{1}{\xi_1^0-\xi_1^\ell}\int_{\xi_1^\ell}^{\xi_1^0} B(x)dx \\
     = & \frac{1}{2+2\psi} B(\xi_1^0)\\
     = & 0, \numberthis \label{eqn:VD1second}
\end{align*}
where the first equality follows by second fundamental theorem of calculus, and the last step follows by $p_1^0/\sigma_1^0=p_2^0/\sigma_2^0$.

On $(\xi_1^0,\infty)$, $A_\ell(x) = \left|\frac{p_1^\ell}{\sigma_1^0} e^{-\frac{x-\xi_1^\ell}{\sigma_1^0}}+\frac{p_2^\ell}{\sigma_2^0}e^{-\frac{x-\xi_1^0}{\sigma_2^0}} -\frac{p_1^0}{\sigma_1^0}e^{-\frac{x-\xi_1^0}{\sigma_1^0}} -\frac{p_2^0}{\sigma_2^0} e^{-\frac{x-\xi_1^\ell}{\sigma_2^0}} \right|$. Then on $(\xi_1^0,\infty)$,
\begin{align*} 
&\ell A_\ell(x) \\
= & 
\ell \left|\frac{p_1^\ell}{\sigma_1^0} \left(e^{-\frac{x-\xi_1^\ell}{\sigma_1^0}}-e^{-\frac{x-\xi_1^0}{\sigma_1^0}}\right) +\frac{p_1^\ell-p_1^0}{\sigma_1^0}e^{-\frac{x-\xi_1^0}{\sigma_1^0}}
+\frac{p_2^\ell-p_2^0}{\sigma_2^0}e^{-\frac{x-\xi_1^0}{\sigma_2^0}}  +\frac{p_2^0}{\sigma_2^0} \left(e^{-\frac{x-\xi_1^0}{\sigma_2^0}} - e^{-\frac{x-\xi_1^\ell}{\sigma_2^0}}\right) \right| \\
\to & 
\left|-\frac{p_1^0}{\sigma_1^0} e^{-\frac{x-\xi_1^0}{\sigma_1^0}} \frac{1}{\sigma_1^0} \frac{1}{2+2\psi} +\frac{\psi}{2+2\psi}\frac{1}{\sigma_1^0}e^{-\frac{x-\xi_1^0}{\sigma_1^0}}
-\frac{\psi}{2+2\psi} \frac{1}{\sigma_2^0}e^{-\frac{x-\xi_1^0}{\sigma_2^0}}  +\frac{p_2^0}{\sigma_2^0} e^{-\frac{x-\xi_1^0}{\sigma_2^0}} \frac{1}{\sigma_2^0} \frac{1}{2+2\psi} \right|\\
=& 0
\end{align*}
where the last step follows by $\frac{p_1^0}{\sigma_1^0}=\frac{p_2^0}{\sigma_2^0}=\psi$. It is easy to find an envelope function of $\ell A_\ell(x)$ that is integrable on $(\xi_1^0,\infty)$ and thus by the dominated convergence theorem,
\begin{equation}
\lim_{\ell} \ell \int_{\xi_1^0}^\infty A_\ell(x)dx = 0.  \label{eqn:VD1third}
\end{equation}
The conclusion then follows from \eqref{eqn:VD1first}, \eqref{eqn:VD1second} and \eqref{eqn:VD1third}.
\end{proof}

\begin{proof}[Proof of Lemma \ref{lem:nececonditionh}]
	    Take $\tilde{f}(x)=\max_{i\in[k_0]}\bar{f}(x)\sqrt{f(x|\theta_i^0)}$. Then $\tilde{f}(x)$ is $\mu$-integrable by Cauchy-Schwarz inequality. Moreover for any $i\in [k_0]$ and any $0<\Delta\leq \gamma_0$
	    $$
	    \left| \frac{f(x|\theta_i^0+a_i\Delta)-f(x|\theta_i^0)}{\Delta} \right|\leq \tilde{f}(x) \quad \mu-a.e.\ x\in \Xfrak.
	    $$
	    Then by Lemma \ref{lem:firstidentifiableelaborate} \ref{item:firstidentifiableelaborateb} $(a_1,b_1,\ldots,a_{k_0},b_{k_0})$ is a nonzero solution of the system of equations \eqref{eqn:conlininda}, \eqref{eqn:conlinindb}. 
	    
	    Let $a'_i = \frac{a_i/p_i^0}{\sum_{i=1}^{k_0}\left(\|a_i/p_i^0\|_2+|b_i|\right)}$ and $b'_i=\frac{b_i}{\sum_{i=1}^{k_0}\left(\|a_i/p_i^0\|_2+|b_i|\right)}$. Then $a'_i$, $b'_i$ satisfy $\sum_{i=1}^{k_0}(\|a'_i\|_2+|b'_i|)=1$ and $(p_1^0 a'_1$, $b'_1,\ldots,p_{k_0}^0 a'_{k_0}$, $b'_{k_0})$
	    is also a nonzero solution of \eqref{eqn:conlininda}, \eqref{eqn:conlinindb} with $k,\theta_i$ replaced respectively by $k_0,\theta_i^0$.
	    Let $G_\ell=p_i^\ell \delta_{\theta_i^\ell}$ with $p_i^\ell=p_i^0+b'_i \frac{1}{\ell}$ and $\theta_i^\ell=\theta_i^0+\frac{1}{\ell}a'_i$ for $1\leq i \leq k_0$. When $\ell$ is large, $0<p_i^\ell<1$ and $\theta_i^\ell\in \Theta$ since $0<p_i^0<1$ and $\theta_i^0\in \Theta^\circ$. Moreover, $\sum_{i=1}^{k_0}p_i^\ell=1$ since $\sum_{i=1}^{k_0}b'_i=0$. Then $G_\ell\in \Ec_{k_0}(\Theta)$ and $G_\ell\not = G_0$ since at least one of $a'_i$ or $b'_i$ is nonzero. When $\ell$ is large $D_1(G_\ell,G_0)=\sum_{i=1}^{k_0}\left(\|\theta_i^\ell-\theta_i^0\|_2+|p_i^\ell-p_i^0|\right)=\frac{1}{\ell}$. Thus when $\ell$ is large
	    \begin{align*}
	        &\frac{2h^2(P_{G_\ell},P_{G_0})}{D_1^2(G_\ell,G_0)}\\
	        = &\int_{ S}\left|\frac{p_{G_\ell}(x)-p_{G_0}(x)}{D_1(G_\ell,G_0)} \frac{1}{\sqrt{p_{G_\ell}(x)}+\sqrt{p_{G_0}(x)}}\right|^2  \mu(dx) \\
	    = & \int_{S \backslash \Nc}\left| \left(\sum_{i=1}^{k_0} p_i^\ell \frac{f(x|\theta_i^\ell) - f(x|\theta_i^0)}{1/\ell} + \sum_{i=1}^{k_0} b'_i f(x|\theta_i^0)\right) \frac{1}{\sqrt{p_{G_\ell}(x)}+\sqrt{p_{G_0}(x)}}    \right|^2 \mu(dx).
	    \end{align*}
	The integrand of the last integral is bounded by 
	\begin{align*}
	&\left| \sum_{i=1}^{k_0} \frac{p_i^\ell}{\sqrt{p_i^0}} \frac{f(x|\theta_i^\ell) - f(x|\theta_i^0)}{1/\ell \times \sqrt{f(x|\theta_i^0)}} + \sum_{i=1}^{k_0} \frac{b'_i}{\sqrt{p_i^0}} \sqrt{f(x|\theta_i^0)}\right|^2 \\
	\leq & 2k_0 \sum_{i=1}^{k_0} \frac{(p_i^\ell)^2}{p_i^0} \left|\frac{f(x|\theta_i^\ell) - f(x|\theta_i^0)}{1/\ell \times \sqrt{f(x|\theta_i^0)}}\right|^2 + 2k_0\sum_{i=1}^{k_0} \frac{(b'_i)^2}{p_i^0} f(x|\theta_i^0)  \\
	\leq &
	    2k_0 \sum_{i=1}^{k_0} \frac{1}{p_i^0} \left(\frac{1/p_i^0}{\sum_{i=1}^{k_0}\left(\|a_i/p_i^0\|_2+|b_i|\right)}\right)^2 \bar{f}^2(x) + 2k_0\sum_{i=1}^{k_0} \frac{(b'_i)^2}{p_i^0} f(x|\theta_i^0),
	\end{align*} 
	which is integrable w.r.t. to $\mu$ on $S \backslash \Nc $. 
	Here the last inequalities follows from 
	$$
	\left|\frac{f(x|\theta_i^\ell) - f(x|\theta_i^0)}{1/\ell \times \sqrt{f(x|\theta_i^0)}}\right| = \left|\frac{f(x|\theta_i^0+\frac{\|a'_i\|_2}{\|a_i\|_2}a_i \Delta)-f(x|\theta_i^0)}{\Delta \sqrt{f(x|\theta_i^0)}}\right|\leq \frac{\|a'_i\|_2}{\|a_i\|_2}\bar{f}(x).
	$$
	Then by the dominated convergence theorem 
	\begin{align*}
	        &\lim_{\ell\to \infty}\frac{2h^2(P_{G_\ell}(x),P_{G_0}(x))}{D^2_1(G_\ell,G_0)}\\  
	    = & \int_{S\backslash \Nc}\left| \left(\sum_{i=1}^{k_0} p_i^0 \langle a'_i, \nabla_{\theta} f(x|\theta_i^0)\rangle + \sum_{i=1}^{k_0} b'_i f(x|\theta_i^0) \right) \frac{1}{2\sqrt{p_{G_0}(x)}}\right|^2 \mu(dx) \\
	    =&0.
	    \end{align*}
	The proof is completed by \eqref{eqn:reltwoinversebounds}.
	\end{proof}

\subsection{Proofs for Section \ref{sec:appendixfinercharacterizations}}

	\begin{proof}[Proof of Lemma \ref{lem:firstidentifiableelaborate}]\
a) 
 For $x\in \Xfrak\backslash \Nc$, $\nabla_\theta \tilde{f}(x|\theta_i)= g(\theta_i) \nabla_\theta f(x|\theta_i) + f(x|\theta_i)\nabla_\theta g(x|\theta_i)$. Then $(\tilde{a}_1,\tilde{b}_1,\ldots,\tilde{a}_k,\tilde{b}_k)$ is a solution of \eqref{eqn:conlininda} with $f$ replaced by $\tilde{f}$ if and only if $(a_1,b_1,\ldots,a_k,b_k)$ with $a_i=g(\theta_i)\tilde{a}_i$ and $b_i=\langle \tilde{a}_i, \nabla_{\theta}g(\theta_i)\rangle + \tilde{b}_i g(\theta_i)$ is a solution of \eqref{eqn:conlininda}. We can write $\tilde{a}_i=a_i/g(\theta_i)$ and $\tilde{b}_i = (b_i -\langle a_i, \nabla_{\theta}g(\theta_i)\rangle/g(\theta_i) )/g(\theta_i)$. Thus $(\tilde{a}_1,\tilde{b}_1,\ldots,\tilde{a}_k,\tilde{b}_k)$ is zero if and only if $(a_1,b_1,\ldots,a_k,b_k)$ is zero. 
 
b) 
 Under the conditions, by dominated convergence theorem
 $$
 \int_{\Xfrak\backslash \Nc}\left\langle a_i , \nabla_\theta  f(x|\theta_i) \right\rangle d\mu = \left\langle a_i, \nabla_\theta   \left. \int_{\Xfrak\backslash \Nc} f(x|\theta) d\mu  \right\rangle \right|_{\theta=\theta_i}=0. 
 $$
where the last step follows from $\mu(\Nc)=0$ and the fact that $f(x|\theta)$ is a density with respect to $\mu$. Thus for $(a_1,b_1,\ldots,a_k,b_k)$ any solution of \eqref{eqn:conlininda}, 
$$
\sum_{i=1}^{k}b_i = \int_{\Xfrak\backslash \Nc} \sum_{i=1}^{k} \left(\langle a_i , \nabla_\theta  f(x|\theta_i) \rangle    + b_i f(x|\theta_i)\right) d\mu =0. 
$$
So $(a_1,b_1,\ldots,a_k,b_k)$ is also a solution of the system \eqref{eqn:conlininda},\eqref{eqn:conlinindb}. 

It remains to show \eqref{eqn:dctthetaiai2} is equivalent to the same conditions on $\tilde{f}$. Suppose \eqref{eqn:dctthetaiai2} is true. Then there exists small enough  $\tilde{\gamma}(\theta_i,a_i)<\gamma(\theta_i,a_i)$ such that for $0<\Delta\leq \tilde{\gamma}(\theta_i,a_i)$
\begin{align*}
&\left|\frac{\tilde{f}(x|\theta_i+a_i\Delta)-\tilde{f}(x|\theta_i)}{\Delta}\right|\\ \leq & g(\theta_i+ a_i \Delta)\left|\frac{f(x|\theta_i+a_i\Delta)-f(x|\theta_i)}{\Delta}\right| + \left|\frac{g(\theta_i+a_i\Delta)-g(\theta_i)}{\Delta}\right|f(x|\theta_i) \\
\leq & C(g,\theta_i,a_i,\tilde{\gamma}(\theta_i,a_i)) (\bar{f}(x|\theta_i,a_i) + f(x|\theta_i)) \quad \mu-a.e.\ \Xfrak 
\end{align*}
and thus one can take $\mu$-integrable $\bar{f}_1(x|\theta_i,a_i)=C(g,\theta_i,a_i,\tilde{\gamma}(\theta_i,a_i)) (\bar{f}(x|\theta_i,a_i) + f(x|\theta_i))$. The reverse direction follows similarly.

c) 
It is a direct consequence from parts \ref{item:firstidentifiableelaboratea} and \ref{item:firstidentifiableelaborateb}.
\end{proof}

\begin{proof}[Proof of Lemma \ref{cor:expD1equivalent}] Notice that $f(x|\theta)$ is continuously differentiable at every $\theta\in \Theta^\circ$ when fixing any $x\in \Xfrak$. By Lemma \ref{lem:expdct} and Lemma \ref{lem:firstidentifiableelaborate} \ref{item:firstidentifiableelaboratec}, \eqref{eqn:conlininda} has the same solutions as the system \eqref{eqn:conlininda},\eqref{eqn:conlinindb}.

It is obvious that \ref{item:expD1equivalenta} implies \ref{item:expD1equivalentb} and that \ref{item:expD1equivalentc} implies \ref{item:expD1equivalentd}.
That \ref{item:expD1equivalenta} implies \ref{item:expD1equivalentc} and that \ref{item:expD1equivalentb} implies \ref{item:expD1equivalentd} follow from $V(p_G,p_{G_0}) \leq \sqrt{2} h(p_G,p_{G_0})$. \ref{item:expD1equivalente} implies \ref{item:expD1equivalenta} follows from Lemma \ref{lem:firstidentifiable} \ref{item:firstidentifiableb}. It remains to prove \ref{item:expD1equivalentd} implies \ref{item:expD1equivalente}.

Suppose \ref{item:expD1equivalentd} holds and the system of equations \eqref{eqn:conlininda}, \eqref{eqn:conlinindb} with $k,\theta_i$ replaced respectively by $k_0,\thetaeta_i^0$ has a nonzero solution $(a_1, b_1,\ldots,a_{k_0},b_{k_0})$. By Lemma \ref{lem:expdct}, the condition \ref{item:nececonditionhd} of Lemma \ref{lem:nececonditionh} is satisfied with $\gamma_0=\min_{i\in [k_0]} \gamma(\theta_i^0,a_i)$ and $\bar{f}(x)=\max_{i\in [k_0]} \bar{f}(x|\theta_i^0,a_i)$. Thus by Lemma \ref{lem:nececonditionh}, \ref{item:expD1equivalentd} does not hold. This is a contradiction and thus \ref{item:expD1equivalentd} implies \ref{item:expD1equivalente}. 
\end{proof}

\begin{proof}[Proof of Lemma \ref{cor:expD1equivalenttheta}]

Consider  $\tilde{f}(x|\eta):=f(x|\theta)$ to be the same kernel but under the new parameter $\eta=\eta(\theta)$. 
Note $\{\tilde{f}(x|\eta)\}_{\eta\in\Theta}$ with $\Xi:=\eta(\Theta)$ is the canonical parametrization of the same exponential family. Write $\eta_i^0=\eta(\theta_i^0)$. Since $J_\eta(\theta)=(\frac{\partial \eta^{(i)}}{\partial \theta^{(j)}}(\theta))_{ij}$ exists  at $\theta_i^0$ and  at those points,
$$
 \nabla_{\theta}f(x|\theta_i^0) =  \left(J_\eta(\theta_i^0)\right)^\top \nabla_{\eta}\tilde{f}(x|\eta_i^0),\quad \forall i \in [k_0]
$$
and thus
\begin{equation}
\sum_{i=1}^{k_0}\left(\langle a_i, \nabla_{\theta}f(x|\theta_i^0) \rangle + b_i f(x|\theta_i^0)\right) = \sum_{i=1}^{k_0}\left(\langle J_{\eta}(\theta_i^0) a_i, \nabla_{\eta}\tilde{f}(x|\eta_i^0) \rangle + b_i \tilde{f}(x|\eta_i^0)\right). \label{eqn:linindthetaeta}
\end{equation}
Then \eqref{eqn:conlininda}, \eqref{eqn:conlinindb} with $k,\theta_i$ replaced respectively by $k_0,\theta_i^0$ has only the zero solution if and only if \eqref{eqn:conlininda}, \eqref{eqn:conlinindb} with $k,\theta_i,f$ replaced respectively by $k_0,\eta_i^0,\tilde{f}$ has only the zero solution.

Suppose that $(a_1,b_1,\ldots,a_{k_0},b_{k_0})$ is a solution of \eqref{eqn:conlininda} with $k,\theta_i$ replaced respectively by $k_0,\theta_i^0$. Then by \eqref{eqn:linindthetaeta} $(\tilde{a}_1,\tilde{b}_1,\ldots,\tilde{a}_{k_0},\tilde{b}_{k_0})$ with $\tilde{a}_i = J_{\eta}(\theta_i^0) a_i $, $\tilde{b}_i=b_i$ is a solution of \eqref{eqn:conlininda} with $k,\theta_i,f$ replaced respectively by $k_0,\eta_i^0,\tilde{f}$. Then by Lemma \ref{cor:expD1equivalent}, it necessarily has $\sum_{i=1}^{k_0}b_i=\sum_{i=1}^{k_0}\tilde{b}_i=0$. That is, $(a_1,b_1,\ldots,a_{k_0},b_{k_0})$ is a solution of the system of equations \eqref{eqn:conlininda}, \eqref{eqn:conlinindb} with $k,\theta_i$ replaced respectively by $k_0,\theta_i^0$. As a result, with $k,\theta_i$ replaced respectively by $k_0,\theta_i^0$,  \eqref{eqn:conlininda} has the same solutions as the system \eqref{eqn:conlininda},\eqref{eqn:conlinindb}.

The rest of the proof is completed by appealing to Lemma \ref{lem:invariance} and Lemma \ref{cor:expD1equivalent}.
\end{proof}

\subsection{Auxiliary lemmas for Section \ref{sec:proofsfirstorderidentifiable}}
\label{sec:auxiliarylemmas4}

\begin{lem} \label{lem:taylortwisted}
Consider $g(x)$ on $\R^d$ is a function with its gradient $\nabla g(x)$ existing in a neighborhood of $x_0$ and with $\nabla g(x)$ continuous at $x_0$. 
\begin{enumerate}[label=\alph*)]
	\item \label{item:taylortwisteda}
	 Then when $x\to x_0$ and $y\to x_0$
	 $$
	 |g(x)-g(y)-\langle\nabla g(x_0),x-y\rangle | = o(\|x-y\|_2).
	 $$
	\item \label{item:taylortwistedb}
	If in addition, the Hessian $\nabla^2 g(x)$ is continuous in a neighborhood of $x_0$. Then for any $x,y$ in a closed ball $B$ of $x_0$ contained in that neighborhood,
	\begin{align*}
	&|g(x)-g(y)-\langle\nabla g(x_0),x-y\rangle |\\ 
	\leq & \int_{0}^1\int_0^1  \|\nabla^2 g(x_0+s(y+t(x-y)-x_0))\|_2 ds dt\ \|x-y\|_2 \max\{\|x-x_0\|_2, \|y-x_0\|_2 \}\\
	\leq & d\sum_{1\leq i,j\leq d} \int_{0}^1\int_0^1 \left|\frac{\partial^2 g}{\partial x^{(i)}x^{(j)}}(x_0+s(y+t(x-y)-x_0))\right|  ds dt \times \\
	&\quad  \ \|x-y\|_2 \max\{\|x-x_0\|_2, \|y-x_0\|_2 \}.
	\end{align*}
	Moreover 
	$$
	|g(x)-g(y)-\langle\nabla g(x_0),x-y\rangle | \leq  L \|x-y\|_2 \max\{\|x-x_0\|_2,\|y-x_0\|_2\}. 
	$$
	where $L=\sup_{x\in B} \|\nabla^2 g(x)\|_2<\infty$.
\end{enumerate}
\end{lem}

\begin{lem}\label{lem:polexplinind}
Let $k$ be a positive integer, $b_1<\ldots<b_k$ be a sequence of real numbers and let $\mu$ be the Lebesgue measure on $\R$. 
\begin{enumerate}[label=\alph*)]
\item \label{item:polexplininda}
Let $\{h_i(x)\}_{i=1}^k$ be a sequence of polynomials. Consider any nonempty interval $I$. Then
$$
\sum_{i=1}^k h_i(x)e^{b_i x} =0 \quad \mu-a.e. \ x\in I
$$
implies $h_i(x)\equiv 0$ for any $i\in [k]$. 
\item \label{item:polexplinindb}
 Let $\{h_i(x)\}_{i=1}^k$ be a sequence of functions, where each is of the form $\sum_{j=1}^{m_i} a_jx^{\gamma_j}$, i.e. a finite linear combination of power functions. Let $\{g_i(x)\}_{i=1}^k$ be another sequence of such functions. Consider any nonempty interval $I\subset (0,\infty)$. Then 
$$
\sum_{i=1}^k (h_i(x)+g_i(x)\ln(x)) e^{b_i x} =0 \quad \mu-a.e. \ x\in I
$$
implies when $x\neq 0$ $h_i(x)\equiv 0$ and $g_i(x)\equiv 0$ for any $i\in [k]$.
\end{enumerate}
\end{lem}

\begin{lem}\label{lem:expdct} 
Let $f(x|\theta)$ be the density of a full rank exponential family in canonical form specified as in Lemma \ref{cor:expD1equivalent}. Then for any $\theta\in \Theta^\circ$ and $a\in \R^q$ there exists $\gamma(\theta,a)>0$ such that for any $0<\Delta\leq \gamma(\theta,a)$,
$$
\left|\frac{f(x|\thetaeta+a\Delta)-f(x|\thetaeta)}{\Delta  \sqrt{f(x|\thetaeta)}} \right|\leq \bar{f}(x|\theta,a) \quad \forall x\in S=\{x|f(x|\theta)>0\}
$$
with $\int_{\Xfrak}\bar{f}^2(x|\theta,a)d\mu<\infty$ and
$$
\left|\frac{f(x|\thetaeta+a\Delta)-f(x|\thetaeta)}{\Delta} \right|\leq \tilde{f}(x|\theta,a) \quad \forall x\in \Xfrak
$$
with $\int_{\Xfrak}\tilde{f}(x|\theta,a)d\mu<\infty$. Here $\gamma(\theta,a)$, $\bar{f}(x|\theta,a)$ and $\tilde{f}(x|\theta,a)$ depend on $\theta$ and $a$.
\end{lem}


\section{Proofs, additional lemmas and calculation details for Section 5}

This section contains all proofs for Section \ref{sec:inversebounds} except that of Theorem~\ref{thm:expfam}, Theorem~\ref{thm:genthm}. The proofs of Theorem~\ref{thm:expfam} and Theorem~\ref{thm:genthm} occupy the bulk of the paper and will be presented in Section~\ref{sec:proofprodinvbou}. 
This section also contains additional lemmas on the invariance of different parametrizations and on determinant of a type of generalized Vandermonde matrices, and contains calculation details for Examples \ref{exa:bernoulli} and \ref{exa:gamma2}.

\subsection{Proofs for Section \ref{sec:implication} and Corollary \ref{cor:expfamtheta}}
\label{sec:proofsimplicationofinversebound}

\begin{proof}[Proof of Lemma \ref{cor:identifiabilityallm}]
        In this proof we write $n_1$ and $\underbar{\m}_1$ for $n_1(G_0)$ and $\underbar{\m}_1(G_0)$ respectively.  By Lemma~\ref{lem:n0nbar} \ref{item:n1Nbar}, $n_1 = \underbar{\m}_1 <\infty$.
        For each $\m\geq 1$, there exists $R_{\m}(G_0)>0$ such that for any $G\in \Ec_{k_0}(\Theta)\backslash \{G_0\}$ and $W_1(G,G_0)<R_{\m}(G_0)$
        \begin{equation}
         \frac{V(P_{G,\m},P_{G_0,\m})}{\myD_\m(G,G_0)} \geq \frac{1}{2} \liminf_{\substack{G\overset{W_1}{\to} G_0\\ G\in \Ecal_{k_0}(\Theta) }} \frac{V(P_{G,\m},P_{G_0,\m})}{\myD_\m(G,G_0)}.
         \label{eqn:VNDNlowbouliminfG}
        \end{equation}
        Take $c(G_0,\m_0)=\min\limits_{1\leq i\leq \m_0}R_i(G_0)>0$. 
        Moreover, by the definition~\eqref{eqn:Nbar} for any $N \geq \underbar{\m}_1$,
        $$
        \liminf_{\substack{G\overset{W_1}{\to} G_0\\ G\in \Ecal_{k_0}(\Theta) }} \frac{V(P_{G,\m},P_{G_0,\m})}{\myD_\m(G,G_0)} \geq \inf_{ \m\geq \underbar{\m}_1}\liminf_{\substack{G\overset{W_1}{\to} G_0\\ G\in \Ecal_{k_0}(\Theta) }} \frac{V(P_{G,\m},P_{G_0,\m})}{\myD_\m(G,G_0)}>0.
        $$
        Combining the last two displays completes the proof with 
        $$
        C(G_0)=\frac{1}{2} \inf_{ \m\geq \underbar{\m}_1}\liminf_{\substack{G\overset{W_1}{\to} G_0\\ G\in \Ecal_{k_0}(\Theta) }} \frac{V(P_{G,\m},P_{G_0,\m})}{\myD_\m(G,G_0)}.
        $$
\end{proof}

 \begin{proof}[Proof of Lemma \ref{cor:VlowbouW1}] 
    In this proof we write  $n_1,n_0$ for $n_1(G_0), n_0(G_0,\cup_{k\leq k_0} \Ec_k(\Theta_1))$ respectively.
    By the definition of $n_1$, for any $\m \geq n_1$ 
        \begin{equation}
        \liminf_{\substack{G\overset{W_1}{\to}G_0\\ G\in \Ec_{k_0}(\Theta_1)}  } \frac{V(P_{G,\m},P_{G_0,\m})}{D_{1}(G,G_0)}>0. \label{eqn:VmD1bou}
        \end{equation}
        
        By Lemma \ref{lem:relW1D1} \ref{item:relW1D1b} one may replace the $D_1(G,G_0)$ in the preceding display by $W_1(G,G_0)$. Fix $\m_1 = n_1 \vee n_0$. Then there exists $R>0$ depending on $G_0$ such that 
        \begin{equation}
       \inf_{G\in B_{W_1}(G_0,R)\backslash \{G_0\} } \frac{V(P_{G,N_1},P_{G_0,N_1})}{W_1(G,G_0)}>0,  \label{eqn:positiveneighborhood}
        \end{equation}
          where $B_{W_1}(G_0,R)$ is the open ball in metric space $(\bigcup_{k=1}^{k_0} \Ec_{k}(\Theta_1), W_1)$ with center at $G_0$ and radius $R$. Here we used the fact that any sufficiently small open ball in $(\bigcup_{k=1}^{k_0} \Ec_{k}(\Theta_1), W_1)$ with center in $\Ec_{k_0}(\Theta_1)$ is in $\Ec_{k_0}(\Theta_1)$. 
          
          Notice that $\bigcup_{k=1}^{k_0} \Ec_{k}(\Theta_1) $ is compact under the $W_1$ metric if $\Theta_1$ is compact. By the assumption that the map $\theta\mapsto P_\theta$ is continuous, Lemma \ref{lem:Vupperbou} and the triangle inequality of total variation distance, $V(P_{G,\m},P_{G_0,\m})$ with domain $(\bigcup_{k=1}^{k_0} \Ec_{k}(\Theta_1), W_1)$  is a continuous function of $G$ for each $\m$.
        Then
        $G\mapsto \frac{V(P_{G,\m},P_{G_0,\m})}{W_1(G,G_0)}$
        is a continuous map on $\bigcup_{k=1}^{k_0} \Ec_{k}(\Theta_1)\setminus \{G_0\}$ for each $\m$. Moreover $\frac{V(P_{G,\m},P_{G_0,\m})}{W_1(G,G_0)}$ is positive on the compact set $\bigcup_{k=1}^{k_0} \Ec_{k}(\Theta_1) \backslash B_{W_1}(G_0,R)$ provided $\m\geq n_0$.
        As a result for each $\m\geq n_0$
        $$
        \min_{G\in \bigcup_{k=1}^{k_0} \Ec_{k}(\Theta_1) \backslash B_{W_1}(G_0,R)} \frac{V(P_{G,\m},P_{G_0,\m})}{W_1(G,G_0)}>0. 
        $$
        Combining the last display with $\m_1 = n_1\vee n_0$ and \eqref{eqn:positiveneighborhood} yields 
        \begin{equation} 
        V(P_{G,\m_1},P_{G,\m_1}) \geq C(G_0,\Theta_1) W_1(G,G_0), \label{eqn:VmW1loucon} 
        \end{equation}
        where $C(G_0,\Theta_1)$ is a constant depending on $G_0$ and $\Theta_1$. Observing that $V(P_{G,\m},P_{G_0,\m})$ increases with $\m$, the proof is then complete.
        \end{proof}

\begin{proof}[Proof of Lemma \ref{lem:invariance}]
 It's easy to see when $\theta $ is a sufficiently small neighborhood of $\theta_i^0$,
 $$
(2\|(J_{g}(\theta_i^0))^{-1}\|_2)^{-1} \|\theta-\theta_i^0\|_2  \leq \|g(\theta) -g(\theta_i^0) \|_2 \leq  2\|J_g(\theta_i^0)\|_2 \|\theta-\theta_i^0\|_2.
 $$
 Then when $G$ is in a small neighborhood of $G_0$ under $W_1$
 \begin{align*}
 (2 \max_{1\leq i \leq k_0}\|(J_g(\theta_i^0))^{-1}\|_2+1)^{-1} D_N(G,G_0) \leq & D_N(G^{\eta},G_0^\eta) \\ 
 \leq &  (2 \max_{1\leq i \leq k_0}\|J_g(\theta_i^0)\|_2+1) D_N(G,G_0).
  \end{align*} 
  Moreover
 $
 V(\tilde{\P}_{G^{\eta},\m},\tilde{\P}_{G_0^\eta,\m})=V(\P_{G,\m},\P_{G_0,\m}).
 $ 
 Denote the left side and right side of \eqref{eqn:invarianceprod} respectively by $L$ and $R$. Then $L\leq C(G_0) R$ and $L\geq c(G_0) R$ with  
 $$
 C(G_0)=2 \max_{1\leq i \leq k_0}\|(J_g(\theta_i^0))^{-1}\|_2+1, \quad
 c(G_0)=(2 \max_{1\leq i \leq k_0}\|J_g(\theta_i^0)\|_2+1)^{-1}.
 $$   
 The other equation in the statement follows similarly.
\end{proof}

\begin{proof}[Proof of Corollary \ref{cor:expfamtheta}]
Consider  $\tilde{f}(x|\eta):=f(x|\theta)$ be the same kernel but under the new parameter $\eta=\eta(\theta)$. 
Note $\{\tilde{f}(x|\eta)\}_{\eta\in\Xi}$ with $\Xi:=\eta(\Theta)$ is the canonical parametrization of the same exponential family. Write $\eta_i^0=\eta(\theta_i^0)$. The proof is then completed by applying Lemma \ref{thm:expfam} to $\tilde{f}(x|\eta)$ and then by applying Lemma \ref{lem:invariance}.
\end{proof}

\subsection{Auxiliary lemmas for Section \ref{sec:proofsimplicationofinversebound}}
\label{sec:auxiliarylemmas5}

 \begin{lem} [Lack of identifiability] \label{lem:nececondition2}
Fix $G_0=\sum_{i=1}^{k_0}p_i^0\delta_{\theta_i^0}\in \Ec_{k_0}(\Theta)$. Suppose 
	    $
	    \sum_{i=1}^{k_0}b_i P_{\theta_i^0}=0 
	    $
	    has a nonzero solution $(b_1,\ldots,b_{k_0})$, where the $0$ is the zero measure on $\Xfrak$. Then 
	    \begin{equation}
	    \liminf_{\substack{G\overset{W_1}{\to} G_0\\ G\in \Ec_{k_0}(\Theta)}} \frac{V(P_G,P_{G_0})}{D_1(G,G_0)}=0.  \label{eqn:liminfD102}
	    \end{equation}
	\end{lem}

\begin{lem} \label{lem:conditioncprod}
Suppose the same conditions in Corollary \ref{cor:curvaturenoncurvequiv} hold. Then for any $a\in \R^q$,  for each $i\in [k_0]$, and for any $0<\Delta\leq\gamma(\theta_i^0,a)$,
    $$
	    \left|\frac{\prod_{j=1}^{\m}f(x_j|\theta_i^0+a \Delta)-\prod_{j=1}^{\m}f(x_j|\theta_i^0)}{\Delta}\right| \leq \tilde{f}_\Delta(\bar{x}|\theta_i^0,a,N),  \quad \bigotimes^N\mu-a.e.\bar{x} (x_1,\ldots,x_\m)
    $$
    where $\tilde{f}_\Delta(\bar{x}|\theta_i^0,a,N)$ satisfies 
    $$
    \lim_{\Delta \to 0^+}\int_{\Xfrak^\m}\tilde{f}_\Delta(\bar{x}|\theta_i^0,a,N) d \bigotimes^\m \mu  = \int_{\Xfrak^\m} \lim_{\Delta\to 0^+} \tilde{f}_{\Delta}(\bar{x}|\theta_i^0,a,N) d \bigotimes^\m \mu.
    $$
\end{lem}

 \begin{lem} \label{lem:Vupperbou}
        For any $G = \sum_{i=1}^{k_0} p_i \delta_{\theta_i}$ and $G' = \sum_{i=1}^{k_0} p'_i \delta_{\theta'_i}$,
		\begin{align*}
		V(\P_{G,\m},\P_{G',\m}) \leq & \min_{\tau} \left( \sqrt{2\m} \max_{1\leq i \leq k_0} h\left(\P_{\theta_i}, \P_{\theta'_{\tau(i)}}\right) + \frac{1}{2}\sum_{i=1}^{k_0}\left|p_i-p_{\tau(i)}'\right| \right), \\
		V(\P_{G,\m},\P_{G',\m}) \leq & \min_{\tau} \left( N \max_{1\leq i \leq k_0} V\left(\P_{\theta_i}, \P_{\theta'_{\tau(i)}}\right) + \frac{1}{2}\sum_{i=1}^{k_0}\left|p_i-p_{\tau(i)}'\right| \right),
		\end{align*}
		where the minimum is taken over all $\tau$ in the permutation group $S_{k_0}$.
    \end{lem}
    \begin{proof}
        The proof is similar as that of Lemma \ref{lem:hellingeruppbou}.
    \end{proof}

\subsection{Identifiability of Bernoulli kernel in Example \ref{exa:bernoulli}}
\label{sec:identifiabilityBernoulli}
In this section we prove $n_0(G,\cup_{\ell\leq k}\Ec_\ell (\Theta) ) = 2k-1$ for any $G\in \Ec_k(\Theta)$ for the Bernoulli kernels in Example~\ref{exa:bernoulli}. (Note: The authors find a proof of this result in the technical report \cite[Lemma 3.1 and Theorem 3.1]{elmore2003identifiability} after preparation of this manuscript. Since technical report \cite{{elmore2003identifiability}} is difficult to be found online, the proof given below is different from the technical report and will be presented for completeness.)

For any $G\in \Ec_k(\Theta)$, there are $2k-1$ parameters to determine it.  $f_n(x_1,\ldots,x_n)$ has effective $n$ equations for different value $(x_1,\ldots,x_n)$ since 
$\sum_{j=1}^n x_j$ can takes $n+1$ values and $f_n$ is a probability density. Thus to have $P_{G,n}$ strictly identifiable for $G\in \Ec_k(\Theta)$, a necessary condition is that $n\geq 2k-1$ for almost all $G$ under Lebesgue. In fact, in Lemma~\ref{lem:bernoun0lowbou} part~\ref{item:bernoun0lowboue} it is established that for any $n\leq 2k-2$ for any $G\in \Ec_k(\Theta)$ there exist infinitely many $G'\in\Ec_k(\Theta)\backslash \{G_0\}$ such that $P_{G',n}=P_{G,n}$, which implies $n_0(G,\cup_{\ell \leq k} \Ec_\ell(\Theta))\geq 2k-1$ for \emph{all} $G\in\Ec_k(\Theta)$.

Let us now verify that $n_0(G,\cup_{\ell \leq k} \Ec_\ell(\Theta))=2k-1$ for any $G\in \Ec_k(\Theta)$. In the following $n=2k-1$. For any $G=\sum_{i=1}^{k}p_i\delta_{\theta_i}$ and consider $G'=\sum_{i=1}^{k}p'_i\delta_{\theta'_i}\in \bigcup_{i=1}^{k}\Ec_{k}(\Theta)$ 
such that $p_{G',n}=p_{G,n}$. Notice that $G'\in \bigcup_{i=1}^{k}\Ec_{k}(\Theta)$ means that it is possible some of $p'_i$ is zero. $p_{G',n}=p_{G,n}$ implies
\begin{equation}
\label{eqn:bernoulliidentifiablesystem}
\sum_{i=1}^k p'_i (\theta'_i)^j (1-\theta'_i)^{n-j} -  \sum_{i=1}^k p_i (\theta_i)^j (1-\theta_i)^{n-j} = 0\quad  \forall j=0,1,\cdots,n.
\end{equation}
Notice that the above system of equations does not include the constraint $\sum_{i=1}^kp'_i=1$ since it is redundant: by multiplying both sides of the equation $j$ by $\binom{n}{j}$ and summing up, we obtain $\sum_{i=1}^kp'_i=\sum_{i=1}^kp_i=1$. (In fact, in the above system of equations the equation with $j=n$ (or arbitrary $j$) can be replaced by $\sum_{i=1}^k{p'_i}=1.$) 

We now show that the only solution is $G'=G$, beginning with the following simple observation.
Notice that for a set $\{\xi_i\}_{i=1}^{2k}$ of $2k$ distinct elements in $(0,1)$, the system of linear equations of $y=(y_1,\ldots,y_{k'})$ with $k'\leq 2k$:
$$
\sum_{i=1}^{k'} y_i (\xi_i)^j(1-\xi_i)^{n-j}=0 \quad \forall j=0,1,\ldots,n=2k-1
$$
has only the zero solution since by setting $\tilde{y}_i=(1-\xi_i)^n y_i$ the system of equations of $\tilde{y}$:
$$
\sum_{i=1}^{k'} \tilde{y}_i \left(\frac{\xi_i}{1-\xi_i}\right)^j=0 \quad \forall j=0,1,\ldots,n
$$
has its coefficients of the first $k'$ equations forming a non-singular Vandermonde matrix. 

If some $\theta_i$ is not in $\{\theta'_i\}_{i=1}^k$, then by the observation in last paragraph $p_i=0$, which contradicts with $G\in \Ec_k(\Theta)$. As a result, $\{\theta'_i\}_{i=1}^k=\{\theta_i\}_{i=1}^k$. Suppose $\theta'_{l_i}=\theta_i$ for $i\in [k]$. Then the system of equations
\eqref{eqn:bernoulliidentifiablesystem} become
$$
\sum_{i=1}^k (p'_{l_i}-p_i) (\theta_i)^j (1-\theta_i)^{n-j} = 0 \quad  \forall j=0,1,\ldots,n.
$$
Applying the observation from last paragraph again  
yields $p'_{l_i}=p_i$ for $i\in[k]$. That is, the only solution of \eqref{eqn:bernoulliidentifiablesystem} is $G'=G$. Thus $n_0(G,\cup_{\ell \leq k} \Ec_\ell(\Theta))\leq 2k-1$, which together with the fact that $n_0(G,\cup_{\ell \leq k} \Ec_\ell(\Theta))\geq 2k-1$ yield $n_0(G,\cup_{\ell \leq k} \Ec_\ell(\Theta))=2k-1$ for any $G\in \Ec_k(\Theta)$.


Part \ref{item:bernoun0lowboue} of the first lemma and \ref{item:determinantc} of the second lemma below are used in the preceding analysis of example on Bernoulli kernel. 

\begin{lem} \label{lem:bernoun0lowbou}
    \begin{enumerate}[label=\alph*)]
    \item \label{item:bernoun0lowboua}
    Let $\eta_1,\eta_2,\ldots,\eta_{2k}$ be $2k$ distinct real numbers. Let $n\leq 2k-2$. Then the system of $n+1$ linear equations of $(y_1,y_2,\ldots,y_{2k})$    
    \begin{equation}
    \sum_{i=1}^{2k} y_i \eta_i^j = 0  \quad \forall j\in [n]\cup\{0\}  \label{eqn:ylinsys}
    \end{equation}
    has all the solutions given by 
    \begin{equation}
    y_i = - \sum_{q=n+2}^{2k} y_q \prod_{\substack{\ell\neq i\\ \ell = 1}}^{n+1} \frac{(\eta_q-\eta_{\ell}) }{(\eta_i-\eta_\ell) } \quad \forall i\in [n+1]   \label{eqn:ysol}
    \end{equation}
    for any $y_{n+2},\ldots,y_{2k} \in \R$. 
    \item \label{item:bernoun0lowboub}
    For any $0<\eta_{k+1}<\eta_{k+2}<\ldots<\eta_{2k}$ and for any positive $y_{k+1},y_{k+2},\ldots,y_{2k}$, there exists infinitely many $\eta_1,\eta_2,\dots,\eta_k$ satisfying
     \begin{align*}
     \eta_{k+i-1}<\eta_i <\eta_{k+i},\quad \text{for }2\leq i \leq k, \text{ and } 0<\eta_1<\eta_{k+1} \text{ and }\\
     y_i = -  y_{2k} \prod_{\substack{\ell\neq i\\ \ell = 1}}^{2k-1} \frac{(\eta_{2k}-\eta_{\ell}) }{(\eta_i-\eta_\ell) } \quad \forall k+1 \leq i\leq 2k-1.
    \end{align*}
    \item \label{item:bernoun0lowbounewc}
    For any $0<\eta_{k+1}<\eta_{k+2}<\ldots<\eta_{2k}$ and for any positive $y_{k+1},y_{k+2},\ldots,y_{2k}$, the system of equations of $(y_1,\ldots,y_k,\eta_1,\ldots,\eta_k)$ 
    \begin{align*}
    \sum_{i=1}^{2k} y_i \eta_i^j &= 0  \quad \forall j\in [2k-2]\cup\{0\}\\
    y_i &<0  \quad \forall i\in [k] \numberthis \label{eqn:ynegative} \\
    \eta_1 \in (0,\eta_{k+1}),\ \eta_i&\in (\eta_{k+i-1},\eta_{k+i}) \quad  \forall 2\leq i \leq k \numberthis \label{eqn:etarange}
    \end{align*}
    has infinitely many solutions.
    \item \label{item:bernoun0lowbouc}
    If $P_{G,n}=P_{G',n}$ for some positive integer $n$, then $P_{G,m}=P_{G',m}$ for any integer $1\leq m\leq n$.
    \item \label{item:bernoun0lowboue}
    Consider the kernel specified in Example \ref{exa:bernoulli}.
    For any $G\in \Ec_k(\Theta)$ and for any $n \leq  2k-2$, there exists infinitely many $G'\in \Ec_k(\Theta)$ such that $P_{G,n}=P_{G',n}$. In particular, this shows $n_0(G,\cup_{\ell \leq k} \Ec_\ell(\Theta))\geq 2k-1$ for any $G\in \Ec_k(\Theta)$. 
    \end{enumerate}
\end{lem}

\begin{proof}[Proof of Lemma \ref{lem:bernoun0lowbou}]
a) 
    By Lagrange interpolation formula over $\eta_1$,$\eta_2$,$\ldots$,$\eta_{n+1}$, 
    $$
    x^j = \sum_{i=1}^{n+1} \eta_{i}^j \prod_{\substack{\ell\neq i\\ \ell = 1}}^{n+1} \frac{(x-\eta_{\ell}) }{(\eta_i-\eta_\ell) }, \quad \forall j\in [n]\cup\{0\}, \ \forall x\in \R.
    $$
    In particular, for any $n+2\leq q\leq 2k$,
    $$
    \eta_{q}^j = \sum_{i=1}^{n+1} \eta_{i}^j \prod_{\substack{\ell\neq i\\ \ell = 1}}^{n+1} \frac{(\eta_q-\eta_{\ell}) }{(\eta_i-\eta_\ell) }, \quad \forall j\in [n]\cup\{0\}.
    $$
    Plugging the above identity into~\eqref{eqn:ylinsys}, it is clear that the $y_i$ specified in \eqref{eqn:ysol} are solutions of \eqref{eqn:ylinsys}. Notice that the coefficient matrix of $\eqref{eqn:ylinsys}$ is $A=(\eta_i^j)_{j\in [n]\cup\{0\},i\in [2k]}\in \R^{(n+1)\times (2k)}$ has rank $n+1$ since the submatrix consisting the first $n+1$ columns form a non-singular Vandermonde matrix. Thus all the solutions of \eqref{eqn:ylinsys} form a subspace of $\R^{2k}$ of dimension $2k-(n+1)$, which implies \eqref{eqn:ysol} are all the solutions.

b) 
    Let $a>0$. Consider a polynomial $g(x)$ such that $g(0)=(-1)^{k+1}a$, $g(\eta_{2k})=-\frac{1}{y_{2k}}$, and for $k+1 \leq i\leq 2k-1$, $g(\eta_i)=\frac{1}{y_i}\prod\limits_{\substack{\ell\neq i\\ \ell = k+1}}^{2k-1} \frac{(\eta_{2k}-\eta_{\ell})}{(\eta_i-\eta_\ell) }$. Then 
    this $k+1$ points determines uniquely a polynomial $g(x)$ with degree at most $k$. By our construction, $g(x)$ satisfies
    \begin{equation}
    y_i g(\eta_i) = -y_{2k}g(\eta_{2k})\prod\limits_{\substack{\ell\neq i\\ \ell = k+1}}^{2k-1} \frac{(\eta_{2k}-\eta_{\ell})}{(\eta_i-\eta_\ell) }, \quad \forall \ k+1\leq i \leq 2k-1 \label{eqn:ygeqn}
    \end{equation}
    Moreover, noticing that $g(\eta_i)>0$ for $i$ odd integer between $k+1$ and $2k$, and $g(\eta_i)<0$ for $i$ even integer between $k+1$ and $2k$. Then there must exist $\eta_1\in (0,\eta_{k+1})$ and $\eta_i\in (\eta_{k+i-1},\eta_{k+i})$ for $2\leq i \leq k$ such that $g(\eta_i)=0$. Then $g(x)=b\prod_{i=1}^k(x-\eta_i)$ where $b<0,\eta_1,\eta_2,\ldots,\eta_k$ are  constants that depend on $a,\eta_{k+1},\ldots,\eta_{2k},y_{k+1},\ldots,y_{2k}$. Plug $g(x)=b\prod_{i=1}^k(x-\eta_i)$ into \eqref{eqn:ygeqn} shows that $(\eta_1,\eta_2,\ldots,\eta_k)$ is a solution for the system of equations in the statement. By changing value of $a$, we get infinitely many solutions.
    
c) 
    
    First, we apply part \ref{item:bernoun0lowboua} with $n=2k-2$: for any $2k$ distinct real numbers $\eta_1,\ldots,\eta_{2k}$, the system of linear equations of $(x_1,\ldots,x_{2k})$
    $$
    \sum_{i=1}^{2k}x_i\eta_i^j = 0 \quad \forall j \in [2k-2]\cup\{0\}
    $$
    has a solution 
    $$
    x_i = -  y_{2k} \prod_{\substack{\ell\neq i\\ \ell = 1}}^{2k-1} \frac{(\eta_{2k}-\eta_{\ell}) }{(\eta_i-\eta_\ell) } \quad \forall i\in [2k-1],
    $$
    where we have specified $x_{2k}=y_{2k}$. 
    
    Next, for the $\eta_{k+1},\ldots, \eta_{2k}$ given in the lemma's statement, by part \ref{item:bernoun0lowboub} we can choose $\eta_1,\ldots,\eta_k$ that satisfy the requirements there. Accordingly, $x_i=y_i$ for $k+1\leq i\leq 2k$. Moreover, it follows from the ranking of $\{\eta_i\}_{i=1}^{2k}$ that $x_i<0$ for any $i\in [k]$. Thus $(x_1,\ldots,x_k,\eta_1,\ldots,\eta_k)$ is a solution of the system of equations in the statement. The infinite many solutions conclusion follows since there are infinitely many $(\eta_1,\ldots,\eta_k)$ by part \ref{item:bernoun0lowboub}.
    
d) 
    $P_{G,n-1}=P_{G',n-1}$ follows immediately from for any $A\in \Ac^{n-1}$, the product sigma-algebra on $\Xfrak^{n-1}$,
    $$
    P_{G,n-1}(A)=P_{G,n}(A \times \Xfrak) = P_{G',n}(A \times \Xfrak) = P_{G',n-1}(A) .  
    $$
    Repeating this procedure inductively and the conclusion follows.
    \item 
    By part \ref{item:bernoun0lowbouc} it suffices to prove that $n=2k-2$. Write $G=\sum_{i=1}^kp_i\delta_{\theta_i}$ with $\theta_1<\theta_2<\ldots<\theta_k$. Consider any $G'=\sum_{i=1}^{k}p'_i\delta_{\theta'_i}\in \Ec_k(\Theta)$ with $\theta'_1<\theta'_2<\ldots<\theta'_k$ such that $P_{G,n}=P_{G',n}$.  $P_{G,n}=P_{G',n}$ for $n=2k-2$ is 
    \begin{align}
    \sum_{i=1}^k p'_i (\theta'_i)^j (1-\theta'_i)^{2k-2-j} &=  \sum_{i=1}^k p_i (\theta_i)^j (1-\theta_i)^{2k-2-j}\quad  \forall j=0,1,\cdots,2k-2. \label{eqn:identifiable2k-2}\\
     0<\theta'_1<\ldots<\theta'_k<1, &\ p'_i>0,\ \forall i \in [k] \label{eqn:p'theta'constraint}
    \end{align}
    Note the system of equations \eqref{eqn:identifiable2k-2} automatically implies $\sum_{i=1}^k p'_i = \sum_{i=1}^k p_i=1$. Let $y_i=-p'_i(1-\theta'_i)^{2k-2}$, $\eta_i=\theta'_i/(1-\theta'_i)$ for $i\in [k]$ and let $y_{k+i}=p_i(1-\theta_i)^{2k-2}$, $\eta_{k+i}=\theta_i/(1-\theta_i)$. Then  $\eta_{k+1}<\eta_{k+2}<\ldots<\eta_{2k}$ and  $y_i>0$ for $k+1 \leq i\leq 2k$. 
    Then $(p'_1,\ldots,p'_k,\theta'_1,\ldots,\theta'_k)$ is a solution of \eqref{eqn:identifiable2k-2}, \eqref{eqn:p'theta'constraint} if and only if the corresponding $(y_1,\ldots,y_k,\eta_1,\ldots,\eta_k)$ is the solution of 
    \begin{align*}
    \sum_{i=1}^{2k}y_i\eta_i^j &=0, \quad \forall j\in [2k-2]\cup \{0\}. \\
    0<\eta_1<\ldots<\eta_k,  &\  y_i<0, \ \forall i\in [k]. 
    \end{align*}
    By part \ref{item:bernoun0lowbounewc}, the system of equations in last display has infinitely many solutions additionally satisfying  \eqref{eqn:etarange}. For each such solution, the corresponding $(p'_1,\ldots,p'_k,\theta'_1,\ldots,\theta'_k)$ is a solution of system of equations \eqref{eqn:identifiable2k-2} \eqref{eqn:p'theta'constraint} additionally satisfying  $0<\theta'_1<\theta_1$ and $\theta_{i-1}<\theta'_i<\theta_i$ for $2\leq i\leq k$. By the comments after \eqref{eqn:identifiable2k-2},\eqref{eqn:p'theta'constraint} we also have $\sum_{i=1}^{k}p'_i=\sum_{i=1}^{k}p_i=1$. Thus, such $(p'_1,\ldots,p'_k,\theta'_1,\ldots,\theta'_k)$ gives $G'\in \Ec_k(\Theta)$ such that $P_{G',2k-2}=P_{G,2k-2}$. The existence of infinitely many such $G'$ follows from the existence of infinitely many solutions $(y_1,\ldots,y_k,\eta_1,\ldots,\eta_k)$ by part \ref{item:bernoun0lowbounewc}.
\end{proof}

\subsection{Proof of Lemma \ref{lem:determinant}}
\begin{proof}[Proof of Lemma \ref{lem:determinant}]
a) 
It's obvious that $q^{(1)}(x,y),q^{(2)}(x,y)$ are multivariate polynomials and that 
 \begin{align*}
 q^{(1)}(y,y)=&\lim_{x\to y}q^{(1)}(x,y) = f'(y),\\
 q^{(2)}(y,y)=&\lim_{x\to y}q^{(2)}(x,y) = f''(y).
 \end{align*} 
That means 
$
q^{(1)}(x,y)-f'(y)$ has factor $x-y$ and thus $\bar{q}^{(2)}(x,y)$ is a multivariate polynomial and
$$
\bar{q}^{(2)}(y,y)=\lim_{x\to y} \frac{q^{(1)}(x,y)-f'(y)}{x-y} = \lim_{x \to y}  \frac{f(x)-f(y)-f'(y)(x-y)}{(x-y)^2}
=\frac{1}{2}q^{(2)}(y,y).
$$
Then $\bar{q}^{(2)}(x,y)-\frac{1}{2}q^{(2)}(x,y)$ has factor $x-y$ and thus $\bar{q}^{(3)}(x,y)$ is a multivariate polynomial.

b) 
Write $A^{(k)}$ for $A^{(k)}(x_1,\ldots,x_k)$ in this proof. Denote $\underbar{A}\in \R^{(2k-2)\times (2k)}$ the bottom $(2k-2)\times 2k$ matrix of $A^{(k)}$. Let $q_j^{(1)}(x,y)$, $q_j^{(2)}(x,y)$, $\bar{q}_j^{(2)}(x,y)$ and $\bar{q}_j^{(3)}(x,y)$ be defined in part \ref{item:determinanta} with $f$ replace by $f_j$. Then by subtracting the third row from the first row, the fourth row from the second row and then factor the common factor $(x_1-x_2)$ out of the resulting first two rows
\begin{align*}
\text{det}(A^{(k)})
=&
(x_1-x_2)^2 \text{det} 
\begin{pmatrix} 
&q_1^{(1)}(x_1,x_2),&\ldots,& q_{2k}^{(1)}(x_1,x_2) \\
& q_1^{(2)}(x_1,x_2),&\ldots,& q_{2k}^{(2)}(x_1,x_2)\\
&\ &\underbar{A}&\\
\end{pmatrix}\\
=& 
(x_1-x_2)^3 \text{det} 
\begin{pmatrix} 
&\bar{q}_1^{(2)}(x_1,x_2),&\ldots,& \bar{q}_{2k}^{(2)}(x_1,x_2) \\
& q_1^{(2)}(x_1,x_2),&\ldots,& q_{2k}^{(2)}(x_1,x_2)\\
&\ &\underbar{A}&\\
\end{pmatrix}
\\
=& 
(x_1-x_2)^4 \text{det} \begin{pmatrix} &\bar{q}_1^{(3)}(x_1,x_2),&\ldots,& \bar{q}_{2k}^{(3)}(x_1,x_2) \\
& q_1^{(2)}(x_1,x_2),&\ldots,& q_{2k}^{(2)}(x_1,x_2)\\
&\ &\underbar{A}&\\
\end{pmatrix}
\end{align*}
where the second equality follows by subtracting the fourth row from first row and then factor the common factor $(x_1-x_2)$ out of the resulting row. The last step of the preceding display follows by subtracting 1/2 times the second row from the first row and then extract the common factor $(x_1-x_2)$ out of the resulting row. Thus $(x_1-x_2)^4$ is a factor of $\text{det}(A^{(k)})$, which is a multivariate polynomial in $x_1,\ldots,x_k$. By symmetry, $\prod_{1\leq \alpha<\beta\leq k}(x_\alpha-x_\beta)^4$ is a factor of $\text{det}(A^{(k)})$.

c) 
We prove the statement by induction. It's easy to verify the statement when $k=1$. Suppose the statement for $k$ holds. By \ref{item:determinantb}, 
$$
\text{det}(A^{(k+1)}(x_1,\dots,x_{k+1}))=g_{k+1}(x_1,\ldots,x_{k+1})\prod_{1\leq \alpha<\beta\leq k+1}(x_{\alpha}-x_{\beta})^4
$$ 
for some multivariate polynomial $g_{k+1}(x_1,\ldots,x_{k+1})$. By the Leibniz formula of determinant, in $\text{det}(A^{(k+1)}(x_1,\dots,x_k,x_{k+1}))$ the term of highest degree of $x_\alpha$ is $f_{2(k+1)}(x_\alpha) f'_{2k+1}(x_\alpha)$ or $f'_{2(k+1)}(x_\alpha) f_{2k+1}(x_\alpha)$,
which both have degree $4k$ since $f_j(x)$ has degree $j-1$ and $f'_j(x)$ has degree $j-2$. Moreover, in $\prod_{1\leq \alpha<\beta\leq k+1}(x_{\alpha}-x_{\beta})^4$ the degree of $x_\alpha$ is $4k$ and the corresponding term is $x_\alpha^{4k}$, which implies in $g_{k+1}(x_1,\ldots,x_{k+1})$ the degree of $x_\alpha$ is no more than $0$ for any $\alpha\in [k+1]$. As a result,  $g_{k+1}(x_1,\ldots,x_{k+1})=q_{k+1}$ is a constant. Thus 
\begin{align}
    \text{det}(A^{(k+1)}(x_1,\dots,x_k,0)) = & q_{k+1} \left(\prod_{1\leq \alpha<\beta\leq k} (x_\alpha-x_\beta)^4 \right) \prod_{\alpha=1}^k x_\alpha^4,  \label{eqn:k+1x0vandermonde}
\end{align}
On the other hand, 
\begin{align*}
\text{det}(A^{(k+1)}(x_1,\dots,x_k,0))= & \text{det} 
\begin{pmatrix} 
f_1(x_1|k+1),& f_2(x_1|k+1) , &  \ldots, & f_{2(k+1)}(x_1|k+1) \\ 
f'_1(x_1|k+1), & f'_2(x_1|k+1),  & \ldots, & f'_{2(k+1)}(x_1|k+1) \\
\vdots & \vdots & \vdots \\
f_1(x_k|k+1),  & f_2(x_k|k+1), & \ldots, & f_{2(k+1)}(x_k|k+1) \\ 
f'_1(x_k|k+1),& f'_2(x_k|k+1), & \ldots, & f'_{2(k+1)}(x_k|k+1) \\
1, &0 , &\ldots, &0\\
0, & 1, & 0, \ldots,& 0  
\end{pmatrix}\\
= & \text{det} 
\begin{pmatrix} 
f_3(x_1|k+1),& f_3(x_1|k+1) , &  \ldots, & f_{2(k+1)}(x_1|k+1) \\ 
f'_3(x_1|k+1), & f'_3(x_1|k+1),  & \ldots, & f'_{2(k+1)}(x_1|k+1) \\
\vdots & \vdots & \vdots \\
f_3(x_k|k+1),  & f_3(x_k|k+1), & \ldots, & f_{2(k+1)}(x_k|k+1) \\ 
f'_3(x_k|k+1),& f'_3(x_k|k+1), & \ldots, & f'_{2(k+1)}(x_k|k+1) \\
\end{pmatrix} \numberthis \label{eqn:k+1tokvandermonde}
\end{align*}
where the second equality follows by Laplace expansion along the last row. Observing that $f_j(x)=x^2f_{j-2}(x)$ and $f'_j(x)= x^2f'_{j-2}(x) + 2xf_{j-2}(x) $, plug these two equations into \eqref{eqn:k+1tokvandermonde} and simplify the resulting determinant, and one has
\begin{equation}
\text{det}(A^{(k+1)}(x_1,\dots,x_k,0))=\text{det}(A^{(k)}(x_1,\dots,x_k))\prod_{ \alpha =1 }^k x_\alpha^4. \label{eqn:k+1tok0vandermondenew}
\end{equation}
Comparing \eqref{eqn:k+1tok0vandermondenew} to \eqref{eqn:k+1x0vandermonde}, together with the induction assumption that statement for $k$ holds, 
$$
q_{k+1}=1.
$$
That is, we proved the statement for $k+1$.

d) 
We prove $\text{det}(A^{(k)}(x_1,\dots,x_k))=\prod_{1\leq \alpha<\beta\leq k}(x_{\alpha}-x_{\beta})^4$ by induction. Write $f_j(x|k)$ for $f_j(x)$ in the following induction to emphasize its dependence on $k$. It is easy to verify the case holds when $k=1$. Suppose the statement for $k$ holds. By \ref{item:determinantb}, $\text{det}(A^{(k+1)}(x_1,\dots,x_{k+1}))=g_{k+1}(x_1,\ldots,x_{k+1})\prod_{1\leq \alpha<\beta\leq k+1}(x_{\alpha}-x_{\beta})^4$ for some multivariate polynomial $g_{k+1}$. Since $f_j(x|k+1)$ has degree $n=2(k+1)-1$ and $f'_j(x|k+1)$ has degree $2k$, by the Leibniz formula of determinant $\text{det}(A^{(k+1)}(x_1,\dots,x_k,x_{k+1}))$ has degree no more than $2k+(2k+1)=4k+1$ for any $x_\alpha$ for $\alpha\in [k+1]$. Moreover, in $\prod_{1\leq \alpha<\beta\leq k+1}(x_{\alpha}-x_{\beta})^4$ the degree of $x_\alpha$ is $4k$, which implies in $g_{k+1}(x_1,\ldots,x_{k+1})$ the degree of $x_\alpha$ is no more than $1$. As a result, it is eligible to write $g_{k+1}(x_1,\ldots,x_{k+1})=h_1(x_1,\dots,x_k)x_{k+1}+h_2(x_1,\ldots,x_k)$ where $h_1,h_2$ are multivariate polynomials of $x_1,\ldots,x_k$. Thus 
\begin{align}
    \text{det}(A^{(k+1)}(x_1,\dots,x_k,0)) = & h_2(x_1,\ldots,x_k) \left(\prod_{1\leq \alpha<\beta\leq k} (x_\alpha-x_\beta)^4 \right) \prod_{\alpha=1}^k x_\alpha^4,  \label{eqn:k+1x0}
\end{align}
and
    \begin{align*}
    &\text{det}(A^{(k+1)}(x_1,\dots,x_k,1)) \\
    = & (h_1(x_1,\ldots,x_k)+h_2(x_1,\ldots,x_k)) \left(\prod_{1\leq \alpha<\beta\leq k} (x_\alpha-x_\beta)^4 \right) \prod_{\alpha=1}^k (x_\alpha-1)^4. 
    \numberthis \label{eqn:k+1x1}
\end{align*}
On the other hand, 
\begin{align*}
&\text{det}(A^{(k+1)}(x_1,\dots,x_k,0))\\
= & \text{det} 
\begin{pmatrix} 
f_1(x_1|k+1),& f_2(x_1|k+1) , &  \ldots, & f_{2(k+1)}(x_1|k+1) \\ 
f'_1(x_1|k+1), & f'_2(x_1|k+1),  & \ldots, & f'_{2(k+1)}(x_1|k+1) \\
\vdots & \vdots & \vdots \\
f_1(x_k|k+1),  & f_2(x_k|k+1), & \ldots, & f_{2(k+1)}(x_k|k+1) \\ 
f'_1(x_k|k+1),& f'_2(x_k|k+1), & \ldots, & f'_{2(k+1)}(x_k|k+1) \\
1, &0 , &\ldots, &0\\
-(2(k+1)-1), & 1, & 0, \ldots,& 0  
\end{pmatrix}\\
= & \text{det} 
\begin{pmatrix} 
f_3(x_1|k+1),& f_3(x_1|k+1) , &  \ldots, & f_{2(k+1)}(x_1|k+1) \\ 
f'_3(x_1|k+1), & f'_3(x_1|k+1),  & \ldots, & f'_{2(k+1)}(x_1|k+1) \\
\vdots & \vdots & \vdots \\
f_3(x_k|k+1),  & f_3(x_k|k+1), & \ldots, & f_{2(k+1)}(x_k|k+1) \\ 
f'_3(x_k|k+1),& f'_3(x_k|k+1), & \ldots, & f'_{2(k+1)}(x_k|k+1) \\
\end{pmatrix} \numberthis \label{eqn:k+1tok}
\end{align*}
where the second equality follows by Laplace expansion along the last row. Observing that $f_j(x|k+1)=x^2f_{j-2}(x|k)$ and $f'_j(x|k+1)= x^2f'_{j-2}(x|k) + 2xf_{j-2}(x|k) $, plug these two equations into \eqref{eqn:k+1tok} and simplify the resulting determinant, and one has
\begin{equation}
\label{eqn:k+1tok0}
\text{det}(A^{(k+1)}(x_1,\dots,x_k,0))=\text{det}(A^{(k)}(x_1,\dots,x_k))\prod_{ \alpha =1 }^k x_\alpha^4. 
\end{equation}
Analogous argument produces
\begin{equation}
\label{eqn:k+1tok1}
\text{det}(A^{(k+1)}(x_1,\dots,x_k,1))=\text{det}(A^{(k)}(x_1,\dots,x_k)) \prod_{ \alpha =1 }^k (1-x_\alpha)^4.
\end{equation}
Comparing \eqref{eqn:k+1tok0} to \eqref{eqn:k+1x0}, together with the induction assumption that statement for $k$ holds, 
$$
h_2(x_1,\ldots,x_k)=1, \quad \forall x_1,\ldots,x_k.
$$
Comparing \eqref{eqn:k+1tok1} to \eqref{eqn:k+1x1}, together with the induction assumption that statement for $k$ holds and the preceding display,
$$
h_1(x_1,\ldots,x_k)=0, \quad \forall x_1,\ldots,x_k.
$$
That is, $g_{k+1}(x_1,\ldots,x_{k+1})=1$ for any $x_1,\ldots,x_{k+1}$.
\end{proof}

	
	\subsection{Calculation details in Example \ref{exa:exponentialcontinuerevision}} 
	\label{sec:detailexponentialexa}
	As in Example \ref{exa:exponentialcontinuerevision}, take the $Tx=(x,x^2)^\top $. Then one may check $\lambda(\xi,\sigma)=(\xi+\sigma,\sigma^2+(\sigma+\xi)^2)^\top $. So condition \ref{item:genthmd} is satisfied. The characteristic function is  
	$$
	\phi_T(\zeta_1,\zeta_2|\xi,\sigma)= \int_{\R} e^{\ive (\zeta_1 x+\zeta_2 x^2)}f(x|\xi,\sigma)dx = \frac{1}{\delta} e^{\frac{\xi}{\sigma}} \int_{\xi}^\infty e^{\ive (\zeta_1 x+\zeta_2 x^2)} e^{-\frac{x}{\sigma}} dx. 
	$$ 
	The verification of \ref{item:genthme} and \eqref{eqn:uniformboundednew} are consequences of Leibniz rule for calculating derivatives and the dominated convergence theorem, and are omitted. To verify \eqref{eqn:integrabilitynew}, notice that $|x|f(x|\xi,\sigma)$ is increasing on $(-\infty,-|\xi|)$ and decreasing on $(\sigma \vee |\xi|, \infty)$. That is, the conditions of Lemma \ref{lem:radialupp} is satisfied with $\alpha_1=1$, $b_1= |\xi| $ and $c_1= |\xi| \vee \sigma$. Moreover, it is clear that $\|f(x|\xi,\sigma)\|_{L^\infty(\R)} \leq 1/\sigma$. Then by Lemma \ref{lem:charintegrability}, for any $r> 4$,
	\begin{align*}
	&\|g(\zeta|\xi,\sigma)\|_{L^r(\R^2)} \\
	\leq 
	& C_2 (|\xi|\vee\sigma +2) \biggr (\||x|f(x|\xi,\sigma) \|_{L^1(\R)} + \frac{2}{\sigma} +\left\|(|x|+1)\frac{\partial f(x|\xi,\sigma)}{\partial x}\right\|_{L^1(\R)} +1
	\biggr )\\
	 := & h(\xi,\sigma).
	\end{align*}
	It can be verified easily by the dominated convergence theorem that $h(\mu,\sigma)$ is a continuous function of $\theta=(\xi,\sigma)$ on $\Theta$. Thus \eqref{eqn:integrabilitynew} in \ref{item:genthmg} is verified. We have then verified that $T$ is admissible with respect to $\Theta$.
	
	One can easily check $\lambda:\Theta\to \R^2$ is injective on $\Theta$. Moreover by simple calculations the Jacobi determinant of $\lambda(\theta)$ is $\text{det}(J_{\lambda})=2\sigma>0$, which implies $J_{\lambda}$ is of full rank on $\Theta$. Then by Corollary \ref{cor:genthm2ndform}, \eqref{eqn:genthmcon} and \eqref{eqn:curvatureprodbound} hold for any $G_0\in \Ec_{k_0}(\Theta)$ for any $k_0\geq 1$.



\section{Proofs of inverse bounds for mixtures of product distributions}
\label{sec:proofprodinvbou}

For an overview of our proof techniques, please refer to Section~\ref{sec:background}. The proofs of both Theorem~\ref{thm:expfam} and Theorem~\ref{thm:genthm} follow the same structure. The reader should read the former first before attempting the latter, which is considerably more technical and lengthy. 

\subsection{Proof of Theorem \ref{thm:expfam}}
	\begin{proof}[Proof of Theorem \ref{thm:expfam}] 
	$\ $ \\
	\noindent \textbf{Step 1} (Proof by contradiction with subsequences) \\
	Suppose \eqref{eqn:curvatureprodbound} is not true. Then $\exists \{\m_{\ell}\}_{\ell=1}^\infty$ subsequence of natural numbers tending to infinity such that 
	$$
	\lim_{r\to 0}\ \  \inf_{\substack{G, H\in B_{W_1}(G_0,r)\\ G\not = H }} \frac{V(\P_{G,\m_{\ell}},\P_{H,\m_{\ell}})}{\myD_{\m_{\ell}}(G,H)}\to 0 \quad \text{ as } \m_{\ell} \to \infty.
	$$
	Then $\exists \{G_\ell\}_{\ell=1}^\infty,  \{H_\ell\}_{\ell=1}^\infty  \subset \Ecal_{k_0}(\Theta)$ such that 
	\begin{equation}
	\begin{cases}
	G_{\ell}\not = H_{\ell} & \forall \ell \\
	\myD_{\m_{\ell}}(G_\ell,G_0)\to 0, D_{\m_{\ell}}(H_{\ell},G_0)\to 0 & \text{ as } \ell \to \infty \\
	\frac{V(\P_{G_\ell,\m_{\ell}},\P_{H_{\ell},\m_{\ell}})}{\myD_{\m_{\ell}}(G_\ell,H_{\ell})}\to 0 & \text{ as } {\ell} \to \infty. 
	\end{cases}
	\label{eqn:expratiotozero1}
	\end{equation} 
	To see this, for each fixed $\ell$, and thus fixed $\m_\ell$, $D_{\m_\ell}(G,G_0)\to 0$ if and only if $W_1(G,G_0)\to 0$. Thus, there exists $G_\ell, H_\ell \in \Ec_{k_0}(\Theta)$ such that 
	$G_\ell\not = H_\ell$, $D_{\m_\ell}(G_\ell,G_0) \leq  \frac{1}{\ell}$, $D_{\m_\ell}(H_\ell,G_0) \leq  \frac{1}{\ell}$ and
	$$
	\frac{V(\P_{G_\ell,\m_{\ell}},\P_{H_\ell,\m_{\ell}})}{\myD_{\m_{\ell}}(G_\ell,H_\ell)} \leq \lim_{r\to 0}\ \  \inf_{\substack{G, H\in B_{W_1}(G_0,r)\\ G\not = H }} \frac{V(\P_{G,\m_{\ell}},\P_{H,\m_{\ell}})}{\myD_{\m_{\ell}}(G,H)} +\frac{1}{\ell},
	$$
	thereby ensuring that~\eqref{eqn:expratiotozero1} hold.

	Write $G_0=\sum_{i=1}^{k_0}p_i^0\delta_{\theta_i^0}$. We may relabel the atoms of $G_\ell$ and $H_\ell$ such that $G_\ell = \sum_{i=1}^{k_0}p_i^\ell \delta_{\thetaeta_i^\ell}$, $H_\ell = \sum_{i=1}^{k_0}\pi_i^\ell \delta_{\eta_i^\ell}$ with $\thetaeta^\ell_i, \eta_i^\ell \to \thetaeta_i^0$ and $p_i^\ell, \pi_i^\ell \to p_i^0$ for any $i\in [k_0]$. By subsequence argument if necessary, we may require $\{G_\ell\}_{\ell=1}^\infty$, $\{H_\ell\}_{\ell=1}^\infty$ additionally satisfy:
	\begin{equation}
	\frac{\sqrt{\m_{\ell}}\left(\thetaeta^\ell_i - \eta^\ell_i\right)}{\myD_{\m_{\ell}}(G_\ell,H_\ell)} \to a_i\in \R^q, \quad \frac{p^\ell_i - \pi^\ell_i}{\myD_{\m_{\ell}}(G_\ell,H_\ell)} \to b_i\in \R, \quad \forall 1\leq i \leq k_0, \label{eqn:etaMlratio}
	\end{equation}
	where the components of $a_i$ are in $[-1,1]$ and $\sum_{i=1}^{k_0}b_i=0$. It also follows that at least one of $a_i$ is not $\bm{0}\in \R^s$ or one of $b_i$ is not $0$. Let $\alpha\in \{1\leq i \leq k_0: a_i \not = \bm{0} \text{ or } b_i \not = 0  \}$.\\
	
	\noindent \textbf{Step 2} (Change of measure by index $\alpha$ and application of CLT)\\
	$\P_{\thetaeta,\m}$ has density w.r.t. $\bigotimes^{\m} \mu$ on $\Xfrak^{\m}$: $$\bar{f}(\bar{x}|\thetaeta,\m) = \prod_{j=1}^{\m} f(x_j|\thetaeta) = e^{\thetaeta^\top  \left(\sum_{j=1}^{\m} T(x_j)\right)- \m A(\thetaeta)} \prod_{j=1}^{\m} h(x_j)
	,$$
	where any $\bar{\xve}\in \Xfrak^{\m}$ is partitioned into $\m$ blocks as  $\bar{\xve}=(\xve_1,\xve_2,\ldots,\xve_{\m})$ with $\xve_i\in \Xfrak$.
	Then 
	\begin{align*}
	&\frac{2V(\P_{G_\ell,\m_{\ell}},\P_{H_\ell,\m_{\ell}})}{\myD_{\m_{\ell}}(G_\ell,H_\ell)} \\	
	=& 
	\int_{\Xfrak^{\m_{\ell} }}\left| \sum_{i=1}^{k_0}\frac{p_i^\ell e^{\left\langle \thetaeta_i^{\ell}, \sum_{j=1}^{\m_{\ell}} T(x_j)\right\rangle- \m_{\ell} A(\thetaeta_i^{\ell})} -  \pi_i^\ell e^{\left\langle \eta_i^{\ell}, \sum_{j=1}^{\m_{\ell}} T(x_j)\right\rangle- \m_{\ell} A(\eta_i^{\ell})} }{\myD_{\m_{\ell}}(G_\ell,H_\ell)}  \right|  
	\prod_{j=1}^{\m_{\ell}} h(x_j)   d\bigotimes^{\m_\ell} \mu  \\
	=& 
	\int_{\Xfrak^{\m_{\ell} }}\left| \sum_{i=1}^{k_0}\frac{p_i^\ell e^{\left\langle \thetaeta_i^{\ell}, \sum_{j=1}^{\m_{\ell}} T(x_j)\right\rangle- \m_{\ell} A(\thetaeta_i^{\ell})} - \pi_i^\ell e^{\left\langle \eta_i^{\ell}, \sum_{j=1}^{\m_{\ell}} T(x_j)\right\rangle- \m_{\ell} A(\eta_i^{\ell})}}{\myD_{\m_{\ell}}(G_\ell,H_\ell)e^{\left\langle \thetaeta_\alpha^{0}, \sum_{j=1}^{\m_{\ell}} T(x_j)\right\rangle- \m_{\ell} A(\thetaeta_\alpha^{0})}}   \right|  
	\bar{f}(\bar{x}|\thetaeta_\alpha^0,\m_{\ell})  d\bigotimes^{\m_\ell} \mu \\ 	
	= & \E_{\thetaeta_\alpha^0} \left|F_{\ell}\left(\sum_{j=1}^{\m_{\ell}} T(X_j)\right)\right|, \numberthis \label{eqn:expVMratioasmeanF}
	\end{align*}
	where $X_j$ are i.i.d. random variables having densities $f(\cdot|\thetaeta_\alpha^0)$, and 
	\begin{align*}
	F_{\ell}(y) := &  \sum_{i=1}^{k_0}\frac{p_i^\ell \exp\left(\left\langle \thetaeta_i^{\ell}, y\right\rangle- \m_{\ell} A(\thetaeta_i^{\ell})\right) -  \pi_i^\ell\exp\left(\left\langle \eta_i^{\ell}, y\right\rangle- \m_{\ell} A(\eta_i^{\ell})\right)}{\myD_{\m_{\ell}}(G_\ell,H_\ell)\exp\left(\left\langle \thetaeta_\alpha^{0}, y\right\rangle- \m_{\ell} A(\theta_\alpha^{0})\right)}   .
	\end{align*}

	Let $Z_\ell = \left(\sum_{j=1}^{\m_{\ell}} T(X_j) - \m_{\ell} \E_{\thetaeta_\alpha^0}T(X_j)\right)/\sqrt{\m_{\ell}} $. Then since $\thetaeta_\alpha^0\in \Theta^\circ$, the mean and covariance matrix of $T(X_j)$ are respectively $\nabla_\thetaeta A(\thetaeta_\alpha^0)$ and $\nabla^2_\thetaeta A(\thetaeta_\alpha^0)$, the gradient and Hessian of $A(\thetaeta)$ evaluated at $\thetaeta_\alpha^0$. Then by central limit theorem, $Z_\ell$ converges in distribution to $Z\sim \Nc(\bm{0}, \nabla^2_\thetaeta A(\thetaeta_\alpha^0))$. Moreover, 
	\begin{equation}
	F_{\ell}\left(\sum_{j=1}^{\m_{\ell}} T(X_j)\right) = F_{\ell}\left(\sqrt{\m_{\ell}}Z_\ell+ \m_{\ell} \nabla_{\thetaeta}A(\thetaeta_\alpha^0) \right) = \Hfun_\ell(Z_\ell),\label{eqn:expFH}
	\end{equation}
	where $\Hfun_\ell(z) := F_{\ell}\left(\sqrt{\m_{\ell}}z+ \m_{\ell} \nabla_{\thetaeta}A(\thetaeta_\alpha^0) \right)$. \\
	
	\noindent \textbf{Step 3} (Application of continuous mapping theorem)\\  
	Define $\Hfun(z)=p_\alpha^0\left\langle a_\alpha, z\right\rangle + b_\alpha$. Supposeing that
	\begin{equation}
	\Hfun_\ell(z_\ell)\to \Hfun(z) \text{ for any sequence } z_\ell \to z\in \R^q, \label{eqn:expcontiforsequence}
	\end{equation}
	a property to be verified in the sequel, then by Generalized Continuous Mapping Theorem (\cite{wellner2013weak} Theorem 1.11.1), $\Hfun_\ell(Z_\ell)$ converges in distribution to $\Hfun(Z)$. Applying Theorem 25.11 in \cite{billingsley2008probability},
	\begin{equation}
	\E|\Hfun(Z)|\leq \liminf_{\ell\to \infty} \E_{\thetaeta_\alpha^0}|\Hfun_\ell(Z_\ell)|=0, \label{eqn:meanHZlimit}
	\end{equation}
	where the equality follows by \eqref{eqn:expratiotozero1}, \eqref{eqn:expVMratioasmeanF} and \eqref{eqn:expFH}. Since $\Hfun(z)$ is a non-zero affine transform and the covariance matrix of $Z$ is positive definite due to full rank property of exponential family, $\Hfun(Z)$ is either a nondegenerate gaussian random variable or a non-zero constant, which contradicts with \eqref{eqn:meanHZlimit}. 
	
	It remains in the proof to verify \eqref{eqn:expcontiforsequence}. Consider any sequence $z_\ell \to z$. Write 
	\begin{equation}
	\Hfun_\ell(z_\ell) = \sum_{i=1}^{k_0} I_i, \label{eqn:Idef}
	\end{equation}
	where 
	\begin{align*}
	I_i := \frac{p_i^\ell \exp\left( g_\ell(\theta_i^\ell)  \right) -  \pi_i^\ell \exp\left(g_\ell(\eta_i^\ell)\right)}{\myD_{\m_{\ell}}(G_\ell,H_\ell)\exp\left( g(\theta_{\alpha}^0) \right)}, 
	\end{align*} 
	with
	$$
	g_\ell(\theta) :=\left \langle \theta, \sqrt{\m_{\ell}} z_\ell + \m_{\ell} \nabla_\thetaeta A(\thetaeta_\alpha^0)\right\rangle - \m_{\ell} A(\thetaeta).
	$$ 
	
	For any $i\in [k_0]$, by Taylor expansion of $A(\theta)$ at $\theta_{i}^0$ and the fact that $A(\theta)$ is infinitely differentiable at $\theta_{i}^0\in \Theta^\circ$, for large $\ell$,
	$$
	|A(\eta_i^\ell) -A(\theta_i^0)-\langle \nabla A(\theta_{i}^0),   \eta_i^\ell - \theta_i^0 \rangle | \leq 2\|\nabla^2 A(\theta_i^0) \|_2  \|\eta_i^\ell - \theta_i^0\|_2^2, 
	$$
	which implies 
	\begin{equation}
	\lim_{\ell\to \infty} \m_\ell|A(\eta_i^\ell) -A(\theta_i^0)-\langle \nabla A(\theta_{i}^0),   \eta_i^\ell - \theta_i^0 \rangle | \leq 2\|\nabla^2 A(\theta_i^0) \|_2 \lim_{\ell\to \infty} D_{\m_\ell}^2(H_\ell,G_0) = 0 \label{eqn:Ntylor2nddir0}
	\end{equation}
	where the equality follows from \eqref{eqn:expratiotozero1}, and the inequality follows from that 
	\begin{align}
	D_{\m_\ell}(H_\ell,G_0)= &\sum_{i=1}^{k_0}(\sqrt{\m_\ell}\|\eta_i^\ell-\theta_i^0\|_2 +|\pi_i^\ell-p_i^0|)  \label{eqn:DNHG0}\\
	D_{\m_\ell}(G_\ell,G_0)= &\sum_{i=1}^{k_0}(\sqrt{\m_\ell}\|\theta_i^\ell-\theta_i^0\|_2 +|p_i^\ell-p_i^0|) \label{eqn:DNGG0}
	\end{align}
	for large $\ell$. The same conclusion holds with $\eta_i^\ell$ replaced by $\theta_i^\ell$ in the last two displays. 
	
	For $i\in [k_0]$, by Lemma \ref{lem:taylortwisted} \ref{item:taylortwistedb} and the fact that $A(\theta)$ is infinitely differentiable at $\theta_{i}^0\in \Theta^\circ$, for large $\ell$
	$$
	|A(\theta_i^\ell)-A(\eta_i^\ell) - \langle \nabla A(\theta_{i}^0), \theta_i^\ell - \eta_i^\ell \rangle | \leq 2\|\nabla^2 A(\theta_i^0) \|_2  \|\theta_i^\ell - \eta_i^\ell\|_2 (\|\theta_i^\ell - \theta_i^0\|_2+\|\eta_i^\ell - \eta_i^0\|_2),
	$$
	which implies 
	\begin{align*}
	&\lim_{\ell\to \infty} \frac{\m_\ell|A(\theta_i^\ell) -A(\eta_i^\ell)-\langle \nabla A(\theta_{i}^0),   \theta_i^\ell - \eta_i^\ell \rangle |}{D_{\m_\ell}(G_\ell,H_\ell)}\\
	\leq &  2\|\nabla^2 A(\theta_i^0) \|_2 \lim_{\ell\to \infty}  \frac{\sqrt{\m_\ell}\|\theta_i^\ell - \eta_i^\ell\|_2}{D_{\m_\ell}(G_\ell,H_\ell)} ( D_{\m_\ell}(G_\ell,G_0)+  D_{\m_\ell}(H_\ell,G_0))  \\
	= &  0 \numberthis \label{eqn:Ntylor2nddir02}
	\end{align*}
	where the inequality follows from \eqref{eqn:DNHG0} and \eqref{eqn:DNGG0}, and the equality follows from \eqref{eqn:expratiotozero1} and \eqref{eqn:etaMlratio}.

	\textbf{Case 1:} Calculate $\lim_{\ell\to\infty} I_\alpha$. \\
	When $\ell\to\infty$
	\begin{equation}
	g_\ell(\eta_\alpha^\ell) - g_\ell(\theta_\alpha^0) = \left \langle \eta_\alpha^\ell -\theta_\alpha^0, \sqrt{\m_{\ell}} z_\ell \right\rangle - \m_{\ell}\left(  A(\eta_\alpha^\ell) - A(\theta_\alpha^0) - \left\langle\eta_\alpha^\ell -\theta_\alpha^0,  \nabla_\thetaeta A(\thetaeta_\alpha^0)\right\rangle \right) \to 0  \label{eqn:gdif0}
	\end{equation}
	by \eqref{eqn:expratiotozero1} and \eqref{eqn:Ntylor2nddir0} with $i=\alpha$. Similarly, one has
	\begin{equation}
	\lim_{\ell\to\infty}\left(g_\ell(\theta_\alpha^\ell) - g_\ell(\theta_\alpha^0)\right) = 0  \label{eqn:gdif0theta}
	\end{equation}
	
	Moreover when $\ell\to \infty$
	\begin{align*}
	\frac{g_\ell(\theta_\alpha^\ell) - g_\ell(\eta_\alpha^\ell)}{D_{\m_\ell}(G_\ell,H_\ell)} = &\frac{\left \langle \theta_\alpha^\ell -\eta_\alpha^\ell, \sqrt{\m_{\ell}} z_\ell \right\rangle - \m_{\ell}\left(  A(\theta_\alpha^\ell) - A(\eta_\alpha^\ell) - \left\langle\theta_\alpha^\ell -\eta_\alpha^\ell,  \nabla_\thetaeta A(\thetaeta_\alpha^0)\right\rangle \right)}{{D_{\m_\ell}(G_\ell,H_\ell)}} \\
	\to & \langle a_\alpha, z\rangle \numberthis  \label{eqn:gdifDN0}
	\end{align*}
	by \eqref{eqn:etaMlratio} and \eqref{eqn:Ntylor2nddir02} with $i=\alpha$.

	Thus
	\begin{align*}
	& \lim_{\ell\to \infty} I_\alpha \\
	= & \lim_{\ell\to \infty} \frac{p_\alpha^\ell \exp\left(g_\ell(\theta_\alpha^\ell)- g_\ell(\theta_\alpha^0) \right)  -  \pi_\alpha^\ell \exp\left(g_\ell(\eta_\alpha^\ell)- g_\ell(\theta_\alpha^0) \right)}{\myD_{\m_{\ell}}(G_\ell,H_\ell)} \\
	= & \lim_{\ell\to \infty} p_\alpha^\ell\frac{ e^{g_\ell(\theta_\alpha^\ell)- g_\ell(\theta_\alpha^0) }  -   e^{g_\ell(\eta_\alpha^\ell)- g_\ell(\theta_\alpha^0) } }{\myD_{\m_{\ell}}(G_\ell,H_\ell)} + \lim_{\ell\to\infty}\frac{ (p_\alpha^\ell-  \pi_\alpha^\ell)e^{g_\ell(\eta_\alpha^\ell)- g_\ell(\theta_\alpha^0) }}{\myD_{\m_{\ell}}(G_\ell,H_\ell)}\\
	\overset{(*)}{=} & p_\alpha^0 \lim_{\ell\to \infty} \frac{ e^{\xi_\ell }  (g_\ell(\theta_\alpha^\ell)- g_\ell(\eta_\alpha^\ell) ) }{\myD_{\m_{\ell}}(G_\ell,H_\ell)} + \lim_{\ell\to\infty}\frac{ (p_\alpha^\ell-  \pi_\alpha^\ell)e^{g_\ell(\eta_\alpha^\ell)- g_\ell(\theta_\alpha^0) }}{\myD_{\m_{\ell}}(G_\ell,H_\ell)}\\
	\overset{(**)}{=} & p_\alpha^0 \lim_{\ell\to \infty} \frac{   g_\ell(\theta_\alpha^\ell)- g_\ell(\eta_\alpha^\ell ) }{\myD_{\m_{\ell}}(G_\ell,H_\ell)} + \lim_{\ell\to\infty}\frac{ p_\alpha^\ell-  \pi_\alpha^\ell}{\myD_{\m_{\ell}}(G_\ell,H_\ell)} \\
	\overset{(***)}{=}& p_\alpha^0 \left\langle a_\alpha,  z  \right\rangle  + b_\alpha,  \numberthis \label{eqn:Ialphalowboutemp1}
	\end{align*}
	where step $(*)$ follows from mean value theorem with $\xi_\ell$ on the line segment between $g_\ell(\theta_\alpha^\ell)- g_\ell(\theta_\alpha^0)$ and $g_\ell(\eta_\alpha^\ell)- g_\ell(\theta_\alpha^0)$, step $(**)$ follows from $g_\ell(\theta_\alpha^\ell)- g_\ell(\theta_\alpha^0)$, $g_\ell(\eta_\alpha^\ell)- g_\ell(\theta_\alpha^0)\to 0$ due to \eqref{eqn:gdif0}, \eqref{eqn:gdif0theta} and hence $\xi_\ell\to 0$, and step $(***)$ follows from \eqref{eqn:gdifDN0} and \eqref{eqn:etaMlratio}.

	\textbf{Case 2:} Calculate $\lim_{\ell\to\infty} I_i$ for $i\neq \alpha$. \\
	For $i\not = \alpha$, 
	\begin{align*}
	&\frac{\exp\left(g_\ell(\theta_i^\ell)\right)}{\exp\left(g_\ell(\theta_\alpha^0)\right)} \\
	= & \exp\left(\left\langle \thetaeta_i^{\ell}- \thetaeta_\alpha^{0}, \sqrt{\m_{\ell}} z_\ell + \m_{\ell} \nabla_\thetaeta A(\thetaeta_\alpha^0) \right\rangle- \m_{\ell} \left(A(\thetaeta_i^{\ell})-A(\thetaeta_\alpha^{0})\right)\right)\\
	= & \exp\left( - \m_{\ell} \left(A(\thetaeta_i^{\ell})-A(\thetaeta_\alpha^{0}) - \left\langle \thetaeta_i^{\ell}- \thetaeta_\alpha^{0},    \nabla_\thetaeta A(\thetaeta_\alpha^0) \right\rangle - \frac{1}{\sqrt{\m_{\ell}}}\left\langle \thetaeta_i^{\ell}- \thetaeta_\alpha^{0}, z_\ell  \right\rangle \right)\right) \\
	\leq & \exp\left( - \frac{\m_{\ell}}{2} \left(A(\thetaeta_i^{0})-A(\thetaeta_\alpha^{0}) - \left\langle \thetaeta_i^{0}- \thetaeta_\alpha^{0},    \nabla_\thetaeta A(\thetaeta_\alpha^0) \right\rangle  \right)\right) \quad \text{for sufficiently large } \ell, \numberthis \label{eqn:expratio}
	\end{align*}
	where the last inequality follows from $\lim_{\ell\to\infty}\frac{1}{\sqrt{\m_{\ell}}}\left\langle \thetaeta_i^{\ell}- \thetaeta_\alpha^{0}, z_\ell  \right\rangle =0 $ and 
	\begin{equation}
	A(\thetaeta_i^{0})-A(\thetaeta_\alpha^{0}) - \left\langle \thetaeta_i^{0}- \thetaeta_\alpha^{0},    \nabla_\thetaeta A(\thetaeta_\alpha^0) \right\rangle>0, \label{eqn:Ataylorpositive}
	\end{equation}
	implied by strict convexity of $A(\thetaeta)$ over $\Theta^\circ$ due to full rank property of exponential family.
	Similarly, for sufficiently large $\ell$,
	\begin{equation}
	\frac{\exp\left(g_\ell(\eta_i^\ell)\right)}{\exp\left(g_\ell(\theta_\alpha^0)\right)} \leq \exp\left( - \frac{\m_{\ell}}{2} \left(A(\thetaeta_i^{0})-A(\thetaeta_\alpha^{0}) - \left\langle \thetaeta_i^{0}- \thetaeta_\alpha^{0},    \nabla_\thetaeta A(\thetaeta_\alpha^0) \right\rangle  \right)\right). \label{eqn:expratio2}
	\end{equation}

	It follows that for $i\neq \alpha$	
	\begin{align*}
	&\lim_{\ell\to \infty}|I_i|\\
	\leq & \lim_{\ell\to\infty} p_i^\ell \left|\frac{ \exp\left(g_{\ell}(\theta_i^\ell)\right) -  \exp\left(g_{\ell}(\eta_i^\ell)\right)}{\myD_{\m_{\ell}}(G_\ell,H_\ell)\exp\left(g_\ell(\theta_\alpha^0) \right)}\right|  +  \lim_{\ell\to\infty} \left|\frac{ p_i^\ell - \pi_i^\ell }{D_{\m_\ell}(G_\ell,H_\ell)} \right| \frac{\exp\left(g_\ell(\eta_i^\ell)\right)}{\exp\left(g_\ell(\theta_\alpha^0)\right)} \\
	\leq & p_i^0\lim_{\ell\to \infty}\frac{\max\{\exp\left(g_{\ell}(\theta_i^\ell)\right), \exp\left(g_{\ell}(\eta_i^\ell)\right)\} }{\exp\left(g_\ell(\theta_\alpha^0)\right)}  \left| \frac{g_{\ell}(\theta_i^\ell) - g_\ell(\eta_i^\ell) }{\myD_{\m_{\ell}}(G_\ell,H_\ell)}\right|  +  |b_i| \lim_{\ell\to\infty}   \frac{\exp\left(g_\ell(\eta_i^\ell)\right)}{\exp\left(g_\ell(\theta_\alpha^0)\right)}\\
	\leq &   \lim_{\ell\to\infty}e^{ - \frac{\m_{\ell}}{2} \left(A(\thetaeta_i^{0})-A(\thetaeta_\alpha^{0}) - \left\langle \thetaeta_i^{0}- \thetaeta_\alpha^{0},    \nabla_\thetaeta A(\thetaeta_\alpha^0) \right\rangle  \right)} \left(p_i^0\left| \frac{g_{\ell}(\theta_i^\ell) - g_\ell(\eta_i^\ell) }{\myD_{\m_{\ell}}(G_\ell,H_\ell)}\right| +|b_i|\right), \numberthis \label{eqn:Iiupp}
	\end{align*}
	where the second inequality follows by applying the mean value theorem on the first term and applying \eqref{eqn:etaMlratio} to the second term, while the last inequality follows from \eqref{eqn:expratio} and \eqref{eqn:expratio2}.
	
	Since 
	\begin{align*}
	&\limsup_{\ell\to\infty}\frac{1}{\sqrt{\m_{\ell}}} \left| \frac{g_{\ell}(\theta_i^\ell) - g_\ell(\eta_i^\ell) }{\myD_{\m_{\ell}}(G_\ell,H_\ell)}\right|\\
	=&\limsup_{\ell\to\infty}\frac{1}{\sqrt{\m_{\ell}}}\left| \frac{\left\langle \thetaeta_i^{\ell}-\eta_i^{\ell}, \sqrt{\m_{\ell}} z_\ell + \m_{\ell} \nabla_\thetaeta A(\thetaeta_\alpha^0)\right\rangle- \m_{\ell} \left(A(\thetaeta_i^{\ell}) -A(\eta_i^{\ell}) \right)}{\myD_{\m_{\ell}}(G_\ell,H_\ell)}\right| \\
	\leq&  \limsup_{\ell\to\infty} \left| \frac{- \sqrt{\m_{\ell}} \left(A(\thetaeta_i^{\ell}) - A(\thetaeta_i^{0}) - \left\langle \thetaeta_i^{\ell}-\eta_i^{\ell},   \nabla_\thetaeta A(\thetaeta_i^0) \right\rangle \right)}{\myD_{\m_{\ell}}(G_\ell,H_\ell)}\right| \\
	&  + \limsup_{\ell\to\infty} \frac{1}{\sqrt{\m_{\ell}}} \left| \frac{\left\langle \sqrt{\m_{\ell}} (\thetaeta_i^{\ell}-\eta_i^{\ell}),  z_\ell \right\rangle}{\myD_{\m_{\ell}}(G_\ell,H_\ell)}\right|+ \limsup_{\ell\to\infty}\left| \frac{  \sqrt{\m_{\ell}}\left\langle \thetaeta_i^{\ell}-\eta_i^{\ell},  \nabla_\thetaeta A(\thetaeta_\alpha^0) - \nabla_\thetaeta A(\thetaeta_i^0)\right\rangle}{\myD_{\m_{\ell}}(G_\ell,H_\ell)}\right|\\
	=&  \left|\left\langle a_i,  \nabla_\thetaeta A(\thetaeta_\alpha^0) - \nabla_\thetaeta A(\thetaeta_i^0)\right\rangle\right|,
	\end{align*}
	where the last step follows from \eqref{eqn:etaMlratio} and \eqref{eqn:Ntylor2nddir02}.
	Then for sufficiently large $\ell$
	\begin{equation}
	\left| \frac{g_{\ell}(\theta_i^\ell) - g_\ell(\eta_i^\ell) }{\myD_{\m_{\ell}}(G_\ell,H_\ell)}\right| \leq \left(\left|\left\langle a_i,  \nabla_\thetaeta A(\thetaeta_\alpha^0) - \nabla_\thetaeta A(\thetaeta_i^0)\right\rangle\right|+\frac{1}{\ell}\right)\sqrt{\m_{\ell}}. \label{eqn:expmeanvalueupp}
	\end{equation}
	Plug \eqref{eqn:expmeanvalueupp} into \eqref{eqn:Iiupp}, for any $i\not = \alpha$,
	\begin{align*}
	&\lim_{\ell\to\infty} |I_i|\\
	\leq & \lim_{\ell\to\infty} e^{ - \frac{\m_{\ell}}{2} \left(A(\thetaeta_i^{0})-A(\thetaeta_\alpha^{0}) - \left\langle \thetaeta_i^{0}- \thetaeta_\alpha^{0},    \nabla_\thetaeta A(\thetaeta_\alpha^0) \right\rangle  \right)} \left(\left|\left\langle a_i,  \nabla_\thetaeta A(\thetaeta_\alpha^0) - \nabla_\thetaeta A(\thetaeta_i^0)\right\rangle\right|+\frac{1}{\ell}\right)\sqrt{\m_{\ell}}\\
	= & 0.  \numberthis \label{eqn:Iilim}
	\end{align*}

	Combining \eqref{eqn:Idef}, \eqref{eqn:Ialphalowboutemp1} and \eqref{eqn:Iilim}, we see that \eqref{eqn:expcontiforsequence} is established. This concludes the proof of the theorem. 
\end{proof}

\subsection{Proof of Theorem \ref{thm:genthm}}
\label{sec:proofoftheoremgen}

\begin{proof}[Proof of Theorem \ref{thm:genthm}] \
	\noindent \textbf{Step 1} (Proof by contradiction with subsequences) \\
	This step is similar to the proof of Theorem \ref{thm:expfam}. Suppose that \eqref{eqn:curvatureprodbound} is not true. Then $\exists \{\m_{\ell}\}_{\ell=1}^\infty$ subsequence of natural numbers tending to infinity such that 
	$$
	\lim_{r\to 0}\ \  \inf_{\substack{G, H\in B_{W_1}(G_0,r)\\ G\not = H }} \frac{V(\P_{G,\m_{\ell}},\P_{H,\m_{\ell}})}{\myD_{\m_{\ell}}(G,H)}\to 0 \quad \text{ as } \m_{\ell} \to \infty.
	$$
	Then $\exists \{G_\ell\}_{\ell=1}^\infty,  \{H_\ell\}_{\ell=1}^\infty  \subset \Ecal_{k_0}(\Theta)$ such that 
	\begin{equation}
	\begin{cases}
	G_{\ell}\not = H_{\ell} & \forall \ell \\
	\myD_{\m_{\ell}}(G_\ell,G_0)\to 0, D_{\m_{\ell}}(H_{\ell},G_0)\to 0 & \text{ as } \ell \to \infty \\
	\frac{V(\P_{G_\ell,\m_{\ell}},\P_{H_{\ell},\m_{\ell}})}{\myD_{\m_{\ell}}(G_\ell,H_{\ell})}\to 0 & \text{ as } {\ell} \to \infty. 
	\end{cases}
	\label{eqn:ratiotozero1}
	\end{equation} 
	To see this, for each fixed $\ell$, and thus fixed $\m_\ell$, $D_{\m_\ell}(G,G_0)\to 0$ if and only if $W_1(G,G_0)\to 0$. Thus, there exist $G_\ell, H_\ell \in \Ec_{k_0}(\Theta)$ such that 
	$G_\ell\not = H_\ell$, $D_{\m_\ell}(G_\ell,G_0) \leq  \frac{1}{\ell}$, $D_{\m_\ell}(H_\ell,G_0) \leq  \frac{1}{\ell}$ and
	$$
	\frac{V(\P_{G_\ell,\m_{\ell}},\P_{H_\ell,\m_{\ell}})}{\myD_{\m_{\ell}}(G_\ell,H_\ell)} \leq \lim_{r\to 0}\ \  \inf_{\substack{G, H\in B_{W_1}(G_0,r)\\ G\not = H }} \frac{V(\P_{G,\m_{\ell}},\P_{H,\m_{\ell}})}{\myD_{\m_{\ell}}(G,H)} +\frac{1}{\ell},
	$$
	thereby ensuring that~\eqref{eqn:ratiotozero1} hold.

	Write $G_0=\sum_{i=1}^{k_0}p_i^0\delta_{\theta_i^0}$. We may relabel the atoms of $G_\ell$ and $H_\ell$ such that $G_\ell = \sum_{i=1}^{k_0}p_i^\ell \delta_{\thetaeta_i^\ell}$, $H_\ell = \sum_{i=1}^{k_0}\pi_i^\ell \delta_{\eta_i^\ell}$ with $\thetaeta^\ell_i, \eta_i^\ell \to \thetaeta_i^0$ and $p_i^\ell, \pi_i^\ell \to p_i^0$ for any $i\in [k_0]$. By subsequence argument if necessary, we may require $\{G_\ell\}_{\ell=1}^\infty$, $\{H_\ell\}_{\ell=1}^\infty$ additionally satisfy:
	\begin{align}
	D_{\m_\ell}(H_\ell,G_0)= &\sum_{i=1}^{k_0}(\sqrt{\m_\ell}\|\eta_i^\ell-\theta_i^0\|_2 +|\pi_i^\ell-p_i^0|)  \label{eqn:DNHG0gen}\\
	D_{\m_\ell}(G_\ell,G_0)= &\sum_{i=1}^{k_0}(\sqrt{\m_\ell}\|\theta_i^\ell-\theta_i^0\|_2 +|p_i^\ell-p_i^0|) \label{eqn:DNGG0gen}
	\end{align}
	and
	\begin{equation}
	\frac{\sqrt{\m_{\ell}}\left(\thetaeta^\ell_i - \eta^\ell_i\right)}{\myD_{\m_{\ell}}(G_\ell,H_\ell)} \to a_i\in \R^q, \quad \frac{p^\ell_i - \pi^\ell_i}{\myD_{\m_{\ell}}(G_\ell,H_\ell)} \to b_i\in \R, \quad \forall 1\leq i \leq k_0, \label{eqn:etaMlratiogen}
	\end{equation}
	where the components of $a_i$ are in $[-1,1]$ and $\sum_{i=1}^{k_0}b_i=0$. It also follows that at least one of $a_i$ is not $\bm{0}\in \R^q$ or one of $b_i$ is not $0$. Let $\alpha\in \{1\leq i \leq k_0: a_i \not = \bm{0} \text{ or } b_i \not = 0  \}$.\\

	\noindent
	\textbf{Step 2} (Transform the probability measure to support in $\R^s$) \\ 	
	Let $T_1: (\Xfrak, \mathcal{A}) \to (\R^s,\mathcal{B}(\R^s))$ be an arbitrary measurable map in this step. Extend $T_1$ to product space by $\bar{T}_1:\Xfrak^\m \to \R^{\m s}$ by $\bar{T}_1\bar{x}=((T_1x_1)^\top ,\ldots, (T_1x_{\m})^\top )^\top $ where any $\bar{\xve}\in \Xfrak^{\m}$ is partitioned into $\m$ blocks as  $\bar{\xve}=(\xve_1,\xve_2,\ldots,\xve_{\m})$ with $\xve_i\in \Xfrak$. Then one can easily verify that $(\bigotimes^{\m}\P_\theta) \circ \bar{T}_1^{-1}= \bigotimes^{\m}(\P_\theta \circ T_1^{-1}) $, and hence for any $G\in \Ec_{k_0}(\Theta)$ 
	$$\P_{G,\m}\circ \bar{T}_1^{-1} = \sum_{i=1}^{k_0}p_i (\P_{\thetave_i,\m}\circ \bar{T}_1^{-1}) = \sum_{i=1}^{k_0}p_i \left(\bigotimes^{\m}\left(\P_{\theta_i} \circ T_1^{-1}\right)\right).$$
	Further consider another measurable map $T_0: (\R^{\m s},\mathcal{B}(\R^{\m s})) \to (\R^s,\mathcal{B}(\R^s)) $ defined by $T_0 \bar{t} = \sum_{i=1}^{\m} {t_i}$ where $\bar{t}\in \R^{\m s}$ is partitioned equally into $\m$ blocks $\bar{t}=(t_1^\top ,t_2^\top ,\ldots,t_{\m}^\top )^\top \in \R^{\m s}$. Denote the induced probability measure on $\R^s$ 
	under $T_0\circ \bar{T}_1$ of the $\P_{\theta,\m}$
	by $Q_{\theta,\m} := \left(\bigotimes^{\m}\left(\P_\theta \circ T_1^{-1}\right)\right) \circ T_0^{-1} $. Then the induced probability measure under $T_0\circ \bar{T}_1$ of the mixture $\P_{G,\m}$ is 
	$$
	\P_{G,\m}\circ \bar{T}_1^{-1} \circ T_0^{-1} = \sum_{i=1}^{k_0} p_i Q_{\theta_i,\m} : = Q_{G,\m}.
	$$
	Note the dependences of $\bar{T}_1$ and $T_0$ on $\m$ are both suppressed, so are the dependences on $T_1$ of $\Q_{\theta,\m}$ and $\Q_{G,\m}$.  
	
	Then by definition of total variation distance
	$$
	V(\P_{G,\m},\P_{H,\m}) \geq V(Q_{G,\m},Q_{H,\m}), \quad \forall \m, \forall\ T_1. 
	$$
	The above display and \eqref{eqn:ratiotozero1} yield
	\begin{equation}
	\lim_{\ell\to 0} \frac{V(\Q_{G_\ell,\m_{\ell}},\Q_{H_\ell,\m_{\ell}})}{\myD_{\m_{\ell}}(G_\ell,H_\ell)}=0, \quad \forall T_1. \label{eqn:ratiotozero}
	\end{equation}
	\\

	\noindent \textbf{Step 3} (Application of the central limit theorem)\\
	In the rest of proof specialize $T_1$ in step 2 to be $T_\alpha$. Write $T=T_\alpha$ to simplify the notation in the rest of the proof. Let $\gamma>0$ and $r\geq 1$ be the same as in Definition \ref{def:admissible} of $T=T_\alpha$ with respect to the finite set $\{\theta_i^0\}_{i=1}^{k_0}$ and define $\bar{\Theta}(G_0):=\bigcup_{i=1}^{k_0}B(\theta_i^0,\gamma)$. By subsequences if necessary, we may further require that $G_\ell, H_\ell$ satisfy $\theta_i^\ell,\eta_i^\ell \in B(\theta_i^0,\gamma)$ for all $i\in [k_0]$ and $\m_\ell\geq r$.
	
	Consider $\{X_i\}_{i=1}^\infty\overset{i.i.d.}{\sim} \P_\theta $. Then $Y_{\ell} = \sum_{i=1}^{\m_{\ell}} TX_i $ is distributed by probability measure $Q_{\theta,\m_{\ell}}$, which has characteristic function $(\phi_T(\zeta|\theta))^{\m_{\ell}}$. For $\theta\in\bar{\Theta}(G_0)$, by \eqref{eqn:integrabilitynew} in \ref{item:genthmg} and by Fourier inversion theorem, 
	$Q_{\theta,\m_{\ell}}$ and $Y_\ell$ therefore have density $f_{Y}(y|\theta,\m_{\ell})$ with respect to Lebesgue measure given by
	\begin{equation}
	f_{Y}(y|\theta,\m_{\ell}) = \frac{1}{(2\pi)^s} \int_{\R^s} e^{-i\zeta^\top y}(\phi_T(\zeta|\theta))^{\m_{\ell}} d\zeta. \label{eqn:Ydenfor}
	\end{equation}
	Then $Q_{G_\ell,\m_{\ell}}$ has density with respect to Lebesgue measure given by $\sum_{i=1}^{k_0} p_i^\ell f_{Y}(y|\theta_i^\ell,\m_{\ell})$, and similarly $Q_{H_\ell,\m_{\ell}}$ has density with respect to Lebesgue measure $\sum_{i=1}^{k_0} \pi_i^\ell f_{Y}(y|\eta_i^\ell,\m_{\ell})$. Thus
	\begin{equation}
	2V(Q_{G_\ell,\m_{\ell}},Q_{H_\ell,\m_{\ell}}) = \int_{\R^s}\left|\sum_{i=1}^{k_0} p_i^{\ell} f_{Y}(y|\theta_i^{\ell},\m_{\ell}) - \sum_{i=1}^{k_0} \pi^\ell_i f_{Y}(y|\eta^\ell_i,\m_{\ell})\right|dy. \label{eqn:vdisdensityform}
	\end{equation}
	
	For $Y_\ell$ has density  $f_{Y}(y|\theta,\m_{\ell})$, define $Z_\ell = (Y_\ell -\m_{\ell} \lambda_\theta)/\sqrt{\m_{\ell}} $, where $\lambda_\theta=\E_\theta TX_1$. Note that this transform from $Y_\ell$ to $Z_\ell$ depends on $\theta$ in the density of $Y_\ell$. Then by the change of variable formula, $Z_\ell$ has density $f_Z(z|\theta, \m_{\ell})$ with respect to Lebesgue measure, given by
	$$
	f_Z(z|\theta, \m_{\ell}) = f_{Y}(\sqrt{\m_{\ell}}z+\m_{\ell}\lambda_\theta|\theta,\m_{\ell}) \m_{\ell}^{s/2},
	$$
	or equivalently 
	\begin{equation}
	f_Y(y|\theta, \m_{\ell}) = f_{Z}((y -\m_{\ell} \lambda_\theta)/\sqrt{\m_{\ell}}|\theta,\m_{\ell}) /\m_{\ell}^{s/2}. \label{eqn:densityztoy}
	\end{equation}
	Now, applying the local central limit theorem (Lemma \ref{thm:localcenlim}), $f_Z(z|\theta, \m_{\ell})$ converges uniformly in $z$ to $f_{\mathcal{N}}(\zve|\thetave)$ for every $\theta\in\bar{\Theta}(G_0)$. Here $f_{\mathcal{N}}(\zve|\thetave)$ is the density of $\mathcal{N}(\bm{0},\Lambda_{\thetave})$, the multivariate normal with mean $\bm{0}$ and covariance matrix $\Lambda_{\thetave}$. Next specialize the previous statement to $\theta^0_\alpha$, and define $$
	w_\ell = \sup\left\{ w \geq 0: f_Z(z|\theta^0_\alpha, \m_{\ell}) \geq  \frac{1}{(2\pi)^{s/2}}\frac{1}{2\ell} \text{ for all } \|z\|_2 \leq w \right \}.
	$$
	We use the convention that the supreme of $\emptyset$ is $0$ in the above display. Because of the uniform convergence of $f_Z(z|\theta^0_\alpha, \m_{\ell})$ to $f_{\mathcal{N}}(\zve|\thetave^0_\alpha)$, we have $w_\ell\to \infty$ when $\ell \to \infty$. It follows from \eqref{eqn:densityztoy} that $f_Y(y|\theta^0_\alpha, \m_{\ell})>0$ on $B_\ell:=\{ y \in \R^s| y= \sqrt{\m_{\ell}} z + \m_{\ell}\lambda_{\theta^0_\alpha} \text{ for } \|z\|_2 \leq w_\ell   \}.$
	Then by \eqref{eqn:vdisdensityform}
	\begin{align*}
	& \frac{2V(\Q_{G_\ell,\m_{\ell}},\Q_{H_\ell,\m_{\ell}})}{\myD_{\m_{\ell}}(G_\ell,H_\ell)}\\
	= & \int_{\R^s}\left|\sum_{i=1}^{k_0} \frac{p_i^{\ell}\left(f_{Y}(y|\theta_i^{\ell},\m_{\ell}) - f_{Y}(y|\eta^\ell_i,\m_{\ell}) \right) }{\myD_{\m_{\ell}}(G_\ell,H_\ell)} + \sum_{i=1}^{k_0} \frac{(p_i^{\ell}-\pi^\ell_i)f_{Y}(y|\eta^\ell_i,\m_{\ell})}{\myD_{\m_{\ell}}(G_\ell,H_\ell)} \right|dy\\
	\geq & 
	\int_{B_\ell}\left|\sum_{i=1}^{k_0} \frac{p_i^{\ell}\left(f_{Y}(y|\theta_i^{\ell},\m_{\ell}) - f_{Y}(y|\eta^\ell_i,\m_{\ell})\right)+(p_i^{\ell}-\pi^\ell_i)f_{Y}(y|\eta^\ell_i,\m_{\ell})}{\myD_{\m_{\ell}}(G_\ell,H_\ell)f_{Y}(y|\theta^0_{\alpha},\m_{\ell})} 
	 \right|f_{Y}(y|\theta^0_{\alpha},\m_{\ell})dy \\
	= & \E_{\theta_\alpha^0}|F_\ell(Y_\ell)|\\
	= & \E_{\theta_\alpha^0}|\Hfun_\ell(Z_\ell)|, \numberthis \label{eqn:lowerbouexp}
	\end{align*}
	where 
	$$
	F_\ell(y)=\left(\sum_{i=1}^{k_0} \frac{p_i^{\ell}\left(f_{Y}(y|\theta_i^{\ell},\m_{\ell}) - f_{Y}(y|\eta^\ell_i,\m_{\ell})\right)+(p_i^{\ell}-\pi^\ell_i)f_{Y}(y|\eta^\ell_i,\m_{\ell})}{\myD_{\m_{\ell}}(G_\ell,H_\ell)f_{Y}(y|\theta^0_{\alpha},\m_{\ell})}\right)\1ve_{{B_\ell}}(y),
	$$
	and 
	$$
	\Hfun_\ell(z) = F_\ell(\sqrt{\m_{\ell}} z + \m_{\ell}\lambda_{\theta^0_\alpha}).
	$$
	Observe if $Z_\ell$ has density $f_Z(z|\theta^0_\alpha, \m_{\ell})$, then $Z_\ell$ converges in distribution to $Z\sim \Nc(\bm{0},\Lambda_{\theta_\alpha^0})$.\\
	
	\noindent \textbf{Step 4} (Application of a continuous mapping theorem)\\
	Define $\Hfun(z)=p_\alpha^0\left(J_{\lambda}(\theta_\alpha^0)a_\alpha\right)^\top \Lambda_{\theta_\alpha^0}^{-1}z + b_\alpha$, where $J_{\lambda}(\theta_\alpha^0)\in \R^{s\times q}$ is the Jacobian matrix of $\lambda(\theta)$ evaluated at $\theta_\alpha^0$. Supposing that
	\begin{equation}
	\Hfun_\ell(z_\ell)\to \Hfun(z) \text{ for any sequence } z_\ell \to z\in \R^s, \label{eqn:contiforsequence}
	\end{equation}
	a property to be verified later, then by Generalized Continuous Mapping Theorem (\cite[Theorem 1.11.1]{wellner2013weak}), $\Hfun_\ell(Z_\ell)$ converges in distribution to $\Hfun(Z)$. Applying  \cite[Theorem 25.11]{billingsley2008probability},
	
	$$
	\E|\Hfun(Z)|\leq \liminf_{\ell\to \infty} \E_{\theta_\alpha^0}|\Hfun_\ell(Z_\ell)|=0,
	$$ 
	where the equality follows \eqref{eqn:lowerbouexp} and \eqref{eqn:ratiotozero}. Note that $\Lambda_\theta$ is positive definite (by \ref{item:genthmd}) and $J_\lambda(\theta_\alpha^0)$ is of full column rank. In addition, by our choice of $\alpha$, either $a_\alpha$ or $b_\alpha$ is non-zero. Hence, $\Hfun(z)$ is a non-zero affine function of $z$. For such an $\Hfun(z)$, $\E|\Hfun(Z)|$ cannot be zero, which results in a contradiction. As a result, it remains in the proof to establish \eqref{eqn:contiforsequence}.

	 We will now impose the following lemma and proceed to verify \eqref{eqn:contiforsequence}, while the technical and lengthy proof of the lemma will be given in the Section \ref{sec:proofoftechnicallemma}.
	
	\begin{lem} \label{lem:technical}
	 Suppose all the conditions in Theorem \ref{thm:genthm} hold and let $\gamma,r$ be defined as in the first paragraph in \textbf{Step 3}. For any $1\leq i \leq k_0$, for any pair of sequences $\bar{\theta}_i^\ell, \bar{\eta}_i^\ell \in B(\theta_i^0,\gamma)$ and for any increasing $\bar{\m}_\ell\geq r$ satisfying  $\sqrt{\bar{\m}_\ell}\|\bar{\theta}_i^\ell-\theta_i^0\|_2, \sqrt{\bar{\m}_\ell}\|\bar{\eta}_i^\ell-\theta_i^0\|_2 \to 0$ and $\bar{\m}_\ell\to \infty$: 
	\begin{align*}
	&J(\bar{\theta}_i^{\ell},\bar{\eta}_i^{\ell},\bar{\m}_\ell)\\
	:=
	&\bar{\m}_{\ell}^{s/2}\sup_{y\in \R^{s}}\left|f_{Y}(y|\bar{\theta}_i^{\ell},\bar{\m}_{\ell}) - f_{Y}(y|\bar{\eta}^\ell_i,\bar{\m}_{\ell}) -  \sum_{j=1}^q \frac{\partial f_{\Ncal}(\yve|\thetave_i^0,\bar{\m}_{\ell})}{\partial \theta^{(j)}}\left((\bar{\theta}_i^\ell)^{(j)}-(\bar{\eta}_i^\ell)^{(j)}\right)\right|\\
	= &  o(\sqrt{\bar{\m}_\ell}\|\bar{\theta}_i^\ell-\bar{\eta}_i^\ell\|_2), \quad \text{ as } \ell\to\infty, \numberthis \label{eqn:Jldefto0}
	\end{align*}
	where $f_{\Ncal}(\yve|\thetave,\m)$ is the density with respect to Lebesgue measure of $\Ncal(\m\lambda_{\thetave},\m\Lambda_{\thetave})$ when $\Lambda_{\thetave}$ is positive definite. 
	\end{lem}

	\noindent \textbf{Step 5} (Verification of \eqref{eqn:contiforsequence})\\
	Write $\myD_\ell= \myD_{\m_{\ell}}(G_\ell,H_\ell)$ for abbreviation in the remaining of this proof. Observe by the local central limit theorem (Lemma \ref{thm:localcenlim})
	\begin{align*}
	\left|f_Z(z_\ell|\theta^0_{\alpha},\m_{\ell})-f_{\Nc}(z|\theta^0_{\alpha})\right|\leq & \sup_{z'\in \R^s}|f_Z(z'|\theta^0_{\alpha},\m_{\ell}) - f_{\Ncal}(z'|\theta^0_{\alpha})|  + |f_{\Ncal}(z_\ell|\theta^0_{\alpha})-f_{\Ncal}(z|\theta^0_{\alpha})|\\
	\to & 0,  
	\end{align*}
	as $\ell\to \infty$, which implies
	\begin{equation}
	\lim_{\ell\to\infty}f_Z(z_\ell|\theta^0_{\alpha},\m_{\ell})=f_{\Nc}(z|\theta^0_{\alpha}). \label{eqn:fzzllimit}
	\end{equation}

	Hereafter $\frac{\partial f_{\Yma}(\sqrt{\m_{\ell}} z_\ell + \m_{\ell}\lambda_{\theta^0_\alpha}|\theta_i^{0},\m_{\ell})}{\partial \theta^{(j)}}:=\left.\frac{\partial f_{\Yma}(y|\theta_i^{0},\m_{\ell})}{\partial \theta^{(j)}}\right|_{y= \sqrt{\m_{\ell}} z_\ell + \m_{\ell}\lambda_{\theta^0_\alpha}}$. Similar definition applies to $\frac{\partial f_{\Nc}(\sqrt{\m_{\ell}} z_\ell + \m_{\ell}\lambda_{\theta^0_\alpha}|\theta_i^{0},\m_{\ell})}{\partial \theta^{(j)}}$. Then for each $i\in [k_0]$,
	\begin{align*}
	&\frac{1}{\myD_\ell}\frac{f_{Y}(\sqrt{\m_{\ell}} z_\ell + \m_{\ell}\lambda_{\theta^0_\alpha}|\theta_i^{\ell},\m_{\ell}) - f_{Y}(\sqrt{\m_{\ell}} z_\ell + \m_{\ell}\lambda_{\theta^0_\alpha}|\eta^\ell_i,\m_{\ell}) }{f_{Y}(\sqrt{\m_{\ell}} z_\ell + \m_{\ell}\lambda_{\theta^0_\alpha}|\theta^0_{\alpha},\m_{\ell})} \1ve_{{B_\ell}}(\sqrt{\m_{\ell}} z_\ell + \m_{\ell}\lambda_{\theta^0_\alpha})\\
	=& \frac{\m_{\ell}^{s/2}}{\myD_\ell}\frac{f_{Y}(\sqrt{\m_{\ell}} z_\ell + \m_{\ell}\lambda_{\theta^0_\alpha}|\theta_i^{\ell},\m_{\ell}) - f_{Y}(\sqrt{\m_{\ell}} z_\ell + \m_{\ell}\lambda_{\theta^0_\alpha}|\eta^\ell_i,\m_{\ell}) }{f_{Z}( z_\ell |\theta^0_\alpha,\m_{\ell})} \1ve_{{E_\ell}}( z_\ell )\\ 
	\leq &  \left(\frac{\m_{\ell}^{s/2}}{\myD_\ell}\frac{\sum_{j=1}^q \frac{\partial f_{\Ncal}(\sqrt{\m_{\ell}} z_\ell + \m_{\ell}\lambda_{\theta^0_\alpha}|\theta_i^{0},\m_{\ell})}{\partial \theta^{(j)}}\left((\theta_i^\ell)^{(j)}-(\eta_i^\ell)^{(j)}\right) }{f_{Z}( z_\ell |\theta^0_\alpha,\m_{\ell})} + \frac{J(\theta_i^{\ell},\eta_i^{\ell},\m_\ell)}{D_\ell f_{Z}( z_\ell |\theta^0_\alpha,\m_{\ell})} \right) \1ve_{{E_\ell}}( z_\ell ),\numberthis \label{eqn:firsttermsimplify}
	\end{align*}
	where the first equality follows from \eqref{eqn:densityztoy} and where in the first equality $E_\ell = \{ z\in \R^s| \|z\|_2 \leq w_\ell \}$. Observe that for any $i\in [k_0]$, 
	\begin{align}
	\sqrt{\m_\ell}\|\theta_i^\ell-\theta_i^0\|\to & 0,  \label{eqn:thetatheta0}\\
	\sqrt{\m_\ell}\|\eta_i^\ell-\theta_i^0\|\to & 0, \label{eqn:etatheta0}
	\end{align}
	by \eqref{eqn:DNHG0gen}, \eqref{eqn:DNGG0gen} and \eqref{eqn:ratiotozero1}. 
	Then
	by applying \eqref{eqn:Jldefto0} with $\bar{\theta}_i^\ell,\bar{\eta}_i^\ell,\bar{\m}_\ell$ respectively be $\theta_i^\ell,\eta_i^\ell,\m_\ell$, and by \eqref{eqn:etaMlratiogen},
	$$
	\lim_{\ell\to \infty}\frac{ J(\theta_i^{\ell},\eta_i^{\ell},\m_\ell)}{D_\ell } \to 0,
	$$
	which together with \eqref{eqn:fzzllimit} yield
	\begin{equation}
	\lim_{\ell\to \infty}\frac{ J(\theta_i^{\ell},\eta_i^{\ell},\m_\ell)}{D_\ell f_{Z}( z_\ell |\theta^0_\alpha,\m_{\ell})} \1ve_{{E_\ell}}( z_\ell )\to 0. \label{eqn:JDratio0}
	\end{equation}
	Thus by \eqref{eqn:firsttermsimplify} and \eqref{eqn:JDratio0}
	\begin{align*}
	&\sum_{i=1}^{k_0}\frac{p_i^\ell}{\myD_\ell}\frac{f_{Y}(\sqrt{\m_{\ell}} z_\ell + \m_{\ell}\lambda_{\theta^0_\alpha}|\theta_i^{\ell},\m_{\ell}) - f_{Y}(\sqrt{\m_{\ell}} z_\ell + \m_{\ell}\lambda_{\theta^0_\alpha}|\theta^0_i,\m_{\ell}) }{f_{Y}(\sqrt{\m_{\ell}} z_\ell + \m_{\ell}\lambda_{\theta^0_\alpha}|\theta^0_{\alpha},\m_{\ell})} \1ve_{{B_\ell}}(\sqrt{\m_{\ell}} z_\ell + \m_{\ell}\lambda_{\theta^0_\alpha})\\
	\to& \sum_{i=1}^{k_0}p_i^0 \lim_{\ell\to \infty} \left(\frac{\m_{\ell}^{s/2}}{\myD_\ell}\frac{\sum_{j=1}^q \frac{\partial f_{\Ncal}(\sqrt{\m_{\ell}} z_\ell + \m_{\ell}\lambda_{\theta^0_\alpha}|\theta_i^{0},\m_{\ell})}{\partial \theta^{(j)}}\left((\theta_i^\ell)^{(j)}-(\eta_i^\ell)^{(j)}\right) }{f_{Z}( z_\ell |\theta^0_\alpha,\m_{\ell})} \right) \1ve_{{E_\ell}}( z_\ell ), \numberthis \label{eqn:limitfistpart}
	\end{align*}
	provided the right hand side exists. 
	
	Note that for each $j\in [q]$, and any $\theta\in \bar{\Theta}(G_0)$, by a standard calculation for Gaussian density,
	\begin{align*}
	&\frac{\partial f_{\Ncal}(y|\theta,\m)}{\partial \theta^{(j)}} \\
	= & f_{\Ncal}(y|\theta,\m) 
	 \quad \quad \left(-\frac{1}{2}\text{det}\left(\Lambda_{\theta}\right)^{-1}\frac{\partial \text{det}\left(\Lambda_{\theta}\right)}{\partial \theta^{(j)}} \right. \\
	 & \quad \left.+ 
	\left(\frac{\partial \lambda_{\theta}}{\partial \theta^{(j)}}\right)^\top \Lambda_{\theta}^{-1}(y-\m\lambda_\theta) - \frac{1}{2\m}(y-\m\lambda_\theta)^\top  \left(\frac{\partial \Lambda_{\theta}^{-1}}{\partial \theta^{(j)}}\right)(y-\m\lambda_\theta)\right),
	\end{align*}
	so we have
	\begin{align*}
	&\frac{\partial f_{\Ncal}(\sqrt{\m_{\ell}} z_\ell + \m_{\ell}\lambda_{\theta^0_\alpha}|\theta_i^{0},\m_{\ell})}{\partial \theta^{(j)}}\\
	= & f_{\Ncal}(\sqrt{\m_{\ell}} z_\ell + \m_{\ell}\lambda_{\theta^0_\alpha}|\theta_i^{0},\m_{\ell})\left(  \left(\frac{\partial \lambda_{\theta_i^0}}{\partial \theta^{(j)}}\right)^\top \Lambda_{\theta_i^0}^{-1}(\sqrt{\m_{\ell}}z_\ell+\m_{\ell}(\lambda_{\theta_\alpha^0}-\lambda_{\theta_i^0}))  \right. \\
	&\left. -\frac{1}{2}\frac{1}{\text{det}\left(\Lambda_{\theta_i^0}\right)}\frac{\partial \text{det}\left(\Lambda_{\theta_i^0}\right)}{\partial \theta^{(j)}} - \frac{1}{2}(z_\ell+\sqrt{\m_{\ell}}(\lambda_{\theta_\alpha^0}-\lambda_{\theta_i^0}))^\top  \frac{\partial \Lambda_{\theta_i^0}^{-1}}{\partial \theta^{(j)}}(z_\ell+\sqrt{\m_{\ell}}(\lambda_{\theta_\alpha^0}-\lambda_{\theta_i^0}))\right)\\
	= & \m_{\ell}^{-\frac{s}{2}} f_{\Ncal}( z_\ell + \sqrt{\m_{\ell}}(\lambda_{\theta^0_\alpha}-\lambda_{\theta^0_i})|\theta_i^{0})\left(  \left(\frac{\partial \lambda_{\theta_i^0}}{\partial \theta^{(j)}}\right)^\top \Lambda_{\theta_i^0}^{-1}(\sqrt{\m_{\ell}}z_\ell+\m_{\ell}(\lambda_{\theta_\alpha^0}-\lambda_{\theta_i^0}))  \right. \\
	&\left. -\frac{1}{2}\frac{1}{\text{det}\left(\Lambda_{\theta_i^0}\right)}\frac{\partial \text{det}\left(\Lambda_{\theta_i^0}\right)}{\partial \theta^{(j)}} - \frac{1}{2}(z_\ell+\sqrt{\m_{\ell}}(\lambda_{\theta_\alpha^0}-\lambda_{\theta_i^0}))^\top  \frac{\partial \Lambda_{\theta_i^0}^{-1}}{\partial \theta^{(j)}}(z_\ell+\sqrt{\m_{\ell}}(\lambda_{\theta_\alpha^0}-\lambda_{\theta_i^0}))\right).
	\end{align*}
	
	Thus, when $i\not =\alpha$,
	\begin{align*}
	\m_{\ell}^{\frac{s-1}{2}}\left|\frac{\partial f_{\Ncal}(\sqrt{\m_{\ell}} z_\ell + \m_{\ell}\lambda_{\theta^0_\alpha}|\theta_i^{0},\m_{\ell})}{\partial \theta^{(j)}}\right|
	\leq & \m_{\ell}^{-\frac{1}{2}} f_{\Ncal}( z_\ell + \sqrt{\m_{\ell}}(\lambda_{\theta^0_\alpha}-\lambda_{\theta^0_i})|\theta_i^{0})C(\theta_i^0, z)\m_{\ell}\\
	\to & 0, \numberthis \label{eqn:derlimitinotequalpha}
	\end{align*}
	where the inequality holds for sufficiently large $\ell$, $C(\theta_i^0, z)$ is a constant that only depends on $\theta_i^0$ and $z$, and the last step follows from $\lambda_{\theta^0_\alpha}\not = \lambda_{\theta^0_i}$ by condition 1) in the statement of theorem.
	
	When $i=\alpha$,
	\begin{align*}
	&\m_{\ell}^{\frac{s-1}{2}}\frac{\partial f_{\Ncal}(\sqrt{\m_{\ell}} z_\ell + \m_{\ell}\lambda_{\theta^0_\alpha}|\theta_i^{0},\m_{\ell})}{\partial \theta^{(j)}}\\
	= & \frac{f_{\Ncal}( z_\ell  |\theta_\alpha^{0})}{\sqrt{\m_{\ell}}} \left(-\frac{1}{2}\frac{1}{\text{det}\left(\Lambda_{\theta_i^0}\right)}\frac{\partial \text{det}\left(\Lambda_{\theta_i^0}\right)}{\partial \theta^{(j)}} + \left(\frac{\partial \lambda_{\theta_\alpha^0}}{\partial \theta^{(j)}}\right)^\top \Lambda_{\theta_\alpha^0}^{-1}(\sqrt{\m_{\ell}}z_\ell)  -\frac{1}{2}z_\ell^\top  \left(\frac{\partial \Lambda_{\theta_\alpha^0}^{-1}}{\partial \theta^{(j)}}\right)z_\ell\right)\\
	\to & f_{\Ncal}( z  |\theta_\alpha^{0})\left(\frac{\partial \lambda_{\theta_\alpha^0}}{\partial \theta^{(j)}}\right)^\top \Lambda_{\theta_\alpha^0}^{-1}z. \numberthis \label{eqn:derlimitiequalpha}
	\end{align*}
	Plugging  \eqref{eqn:derlimitinotequalpha} and \eqref{eqn:derlimitiequalpha} into \eqref{eqn:limitfistpart}, and then combining with \eqref{eqn:fzzllimit} and \eqref{eqn:etaMlratiogen},
	\begin{align*}
	&\sum_{i=1}^{k_0}\frac{p_i^\ell}{\myD_\ell}\frac{f_{Y}(\sqrt{\m_{\ell}} z_\ell + \m_{\ell}\lambda_{\theta^0_\alpha}|\theta_i^{\ell},\m_{\ell}) - f_{Y}(\sqrt{\m_{\ell}} z_\ell + \m_{\ell}\lambda_{\theta^0_\alpha}|\eta^\ell_i,\m_{\ell}) }{f_{Y}(\sqrt{\m_{\ell}} z_\ell + \m_{\ell}\lambda_{\theta^0_\alpha}|\theta^0_{\alpha},\m_{\ell})} \1ve_{{B_\ell}}(\sqrt{\m_{\ell}} z_\ell + \m_{\ell}\lambda_{\theta^0_\alpha})\\
	\to &p_\alpha^0\sum_{j=1}^q  a_\alpha^{(j)}\left(\frac{\partial \lambda_{\theta_\alpha^0}}{\partial \theta^{(j)}}\right)^\top \Lambda_{\theta_\alpha^0}^{-1}z\\
	=& p_\alpha^0\left(J_{\lambda}(\theta_\alpha^0)a_\alpha\right)^\top \Lambda_{\theta_\alpha^0}^{-1}z. \numberthis \label{eqn:Hlpart2limit}
	\end{align*}
	
	Next, we turn to the second summation in the definition of $\Hfun_\ell$ in a similar fashion. By \eqref{eqn:densityztoy},
	\begin{align*}
	 & \frac{f_{Y}(\sqrt{\m_{\ell}} z_\ell + \m_{\ell}\lambda_{\theta^0_\alpha}|\eta^\ell_i,\m_{\ell})}{f_{Y}(\sqrt{\m_{\ell}} z_\ell + \m_{\ell}\lambda_{\theta^0_\alpha}|\theta^0_{\alpha},\m_{\ell})} \1ve_{{B_\ell}}(\sqrt{\m_{\ell}} z_\ell + \m_{\ell}\lambda_{\theta^0_\alpha}) \\
	= & \m_\ell^{s/2}\frac{f_{Y}(\sqrt{\m_{\ell}} z_\ell + \m_{\ell}\lambda_{\theta^0_\alpha}|\eta^\ell_i,\m_{\ell})}{f_Z(z_\ell|\theta^0_{\alpha},\m_{\ell})}\1ve_{{E_\ell}}(z_\ell  )\\
	\leq & \m_\ell^{s/2} \left(\frac{\sum_{j=1}^q \frac{\partial f_{\Ncal}(\sqrt{\m_{\ell}} z_\ell + \m_{\ell}\lambda_{\theta^0_\alpha}|\theta_i^{0},\m_{\ell})}{\partial \theta^{(j)}}\left((\theta_i^\ell)^{(j)}-(\theta_i^0)^{(j)}\right)}{f_Z(z_\ell|\theta^0_{\alpha},\m_{\ell})}\right.\\
	&\left.+\frac{f_{Y}(\sqrt{\m_{\ell}} z_\ell + \m_{\ell}\lambda_{\theta^0_\alpha}|\theta^0_i,\m_{\ell}) }{f_Z(z_\ell|\theta^0_{\alpha},\m_{\ell})}\right)\1ve_{{E_\ell}}(z_\ell)+ \frac{J(\eta_i^\ell,\theta_i^0,\m_\ell)}{f_Z(z_\ell|\theta^0_{\alpha},\m_{\ell})}\1ve_{{E_\ell}}(z_\ell). \numberthis \label{eqn:densityratio1}
	\end{align*}
	Due to \eqref{eqn:etatheta0}, by applying \eqref{eqn:Jldefto0} with $\bar{\theta}_i^\ell,\bar{\eta}_i^\ell,\bar{\m}_\ell$ respectively be $\eta_i^\ell,\theta_i^0,\m_\ell$, and by  \eqref{eqn:ratiotozero1},
	$$
	\lim_{\ell\to \infty} J(\eta_i^{\ell},\theta_i^{0},\m_\ell) \to 0,
	$$
	which together with \eqref{eqn:fzzllimit} yield
	\begin{equation}
	\lim_{\ell\to \infty}\frac{ J(\eta_i^{\ell},\theta_i^{0},\m_\ell)}{f_{Z}( z_\ell |\theta^0_\alpha,\m_{\ell})} \1ve_{{E_\ell}}( z_\ell )\to 0. \label{eqn:JDratio02}
	\end{equation} 
	Moreover for any $i\in [k_0]$, 
	\begin{align*}
	&\left|\m_\ell^{s/2}\sum_{j=1}^q \frac{\partial f_{\Ncal}(\sqrt{\m_{\ell}} z_\ell + \m_{\ell}\lambda_{\theta^0_\alpha}|\theta_i^{0},\m_{\ell})}{\partial \theta^{(j)}}\left((\theta_i^\ell)^{(j)}-(\theta_i^0)^{(j)}\right)\right|\\ 
	\leq & \max_{1\leq j\leq q} \m_\ell^{(s-1)/2}\left|\frac{\partial f_{\Ncal}(\sqrt{\m_{\ell}} z_\ell + \m_{\ell}\lambda_{\theta^0_\alpha}|\theta_i^{0},\m_{\ell})}{\partial \theta^{(j)}}\right| 
	\sqrt{q}\sqrt{\m_\ell}\|\theta_i^\ell-\theta_i^0\|_2 \\
	\to & 0. \numberthis \label{eqn:parf0}
	\end{align*}
	by \eqref{eqn:derlimitinotequalpha} and \eqref{eqn:derlimitiequalpha} and \eqref{eqn:thetatheta0}.
	
	Combining \eqref{eqn:densityratio1}, \eqref{eqn:JDratio02}, \eqref{eqn:parf0} and \eqref{eqn:etaMlratiogen},
	\begin{align*}
	&\lim_{\ell\to\infty}\sum_{i=1}^{k_0} \frac{p_i^{\ell}-\pi^\ell_i}{\myD_{{\ell}}} \frac{f_{Y}(\sqrt{\m_{\ell}} z_\ell + \m_{\ell}\lambda_{\theta^0_\alpha}|\theta^0_i,\m_{\ell})}{f_{Y}(\sqrt{\m_{\ell}} z_\ell + \m_{\ell}\lambda_{\theta^0_\alpha}|\theta^0_{\alpha},\m_{\ell})} \1ve_{{B_\ell}}(\sqrt{\m_{\ell}} z_\ell + \m_{\ell}\lambda_{\theta^0_\alpha})\\
	=&\sum_{i=1}^{k_0} b_i \lim_{\ell\to\infty} \m_\ell^{s/2}\frac{f_{Y}(\sqrt{\m_{\ell}} z_\ell + \m_{\ell}\lambda_{\theta^0_\alpha}|\theta^0_i,\m_{\ell})}{f_Z(z_\ell|\theta^0_{\alpha},\m_{\ell})}\1ve_{{E_\ell}}(z_\ell)\\
	=& \sum_{i=1}^{k_0} b_i \lim_{\ell\to\infty} \frac{f_{Z}( z_\ell + \sqrt{\m_{\ell}}(\lambda_{\theta^0_\alpha}-\lambda_{\theta^0_i})|\theta^0_i,\m_{\ell})}{f_Z(z_\ell|\theta^0_{\alpha},\m_{\ell})}\1ve_{{E_\ell}}(z_\ell) \numberthis \label{eqn:2ndtermlimit}
	\end{align*}
	where the last step is due to \eqref{eqn:densityztoy}.
	 
	When $i=\alpha$, the term in the preceding display equals to $\1ve_{{E_\ell}}(z_\ell )$, which converges to $1$ as $\ell \rightarrow \infty$. When $i \not = \alpha$, 
	\begin{align*}
	&|f_Z(\sqrt{\m_{\ell}}(\lambda_{\theta_\alpha^0}-\lambda_{\theta_i^0})+ z_\ell|\theta^0_{i},\m_{\ell})| \\
	\leq & \sup_{z'\in \R^s}|f_Z(z'|\theta^0_{i},\m_{\ell}) - f_{\Ncal}(z'|\theta^0_{i})| + f_{\Ncal}(\sqrt{\m_{\ell}}(\lambda_{\theta_\alpha^0}-\lambda_{\theta_i^0})+ z_\ell|\theta^0_{i})\\
	\to & 0, \numberthis \label{eqn:numtozero}
	\end{align*}
	where the last step follows from Lemma \ref{thm:localcenlim} and $\lambda_{\theta_\alpha^0}\not =\lambda_{\theta_i^0}$ by condition 1) in the statement of the theorem. 
	Plug \eqref{eqn:numtozero} and \eqref{eqn:fzzllimit} into \eqref{eqn:2ndtermlimit}:
	\begin{equation}
	\lim_{\ell\to\infty}\sum_{i=1}^{k_0} \frac{p_i^{\ell}-\pi^\ell_i}{\myD_{{\ell}}} \frac{f_{Y}(\sqrt{\m_{\ell}} z_\ell + \m_{\ell}\lambda_{\theta^0_\alpha}|\theta^0_i,\m_{\ell})}{f_{Y}(\sqrt{\m_{\ell}} z_\ell + \m_{\ell}\lambda_{\theta^0_\alpha}|\theta^0_{\alpha},\m_{\ell})} \1ve_{{B_\ell}}(\sqrt{\m_{\ell}} z_\ell + \m_{\ell}\lambda_{\theta^0_\alpha})=b_\alpha. \label{eqn:Hllimitpart1}
	\end{equation}
	
	Finally, combining \eqref{eqn:Hllimitpart1} and \eqref{eqn:Hlpart2limit} to obtain 
	$$
	\lim_{\ell\to \infty}\Hfun_\ell(z_\ell) = \Hfun(z)=p_\alpha^0\left(J_{\lambda}(\theta_\alpha^0)a_\alpha\right)^\top \Lambda_{\theta_\alpha^0}^{-1}z + b_\alpha.
	$$
	Thus, ~\eqref{eqn:contiforsequence} is established, so we can conclude the proof of the theorem.
\end{proof}

\subsection{Bounds on characteristic functions for distributions with bounded density}
\label{sec:characteristicfunctionboundeddensity}

To prove Lemma \ref{lem:technical}, we need the next lemma, which is a generalization of the corollary to \cite[Lemma 1]{statulyavichus1965limit}. It gives an upper bound on the magnitude of the characteristic function for distributions with bounded density with respect to Lebesgue measure.

\begin{lem} \label{lem:cfuppbou}
	Consider a random vector $X\in \R^d$ with $\phi(\zeta)$ its characteristic function. Suppose $X$ has density $f(x)$ with respect to Lebesgue measure upper bounded by $U$, and has positive definite covariance matrix $\Lambda$. Then for all $\zeta \in \R^d$
	$$
	|\phi(\zeta)|\leq \exp\left(-\frac{C(d)\|\zeta\|_2^2}{(\|\zeta\|_2^2\lambda_{\max}(\Lambda)+1)\lambda_{\max}^{d-1}(\Lambda)U^2}\right),
	$$ 
	where $C(d)$ is some constant that depends only on $d$, and $\lambda_{\max}(\Lambda)$ is the largest eigenvalue.
\end{lem}
\begin{proof}
	It suffices to prove for $\zeta\not =\bm{0}\in \R^d$. 
	
	\noindent \textbf{Step 1} In this step we prove the special case $\zeta = te_1$ for $t> 0$, where $e_1$ is the standard basis in $\R^d$. Define $I(\zeta)=\frac{1}{2}\left(1-|\phi(\zeta)|^2\right)$ and it is easy to verify 
	\begin{equation}
	|\phi(\zeta)|\leq \exp(-I(\zeta)). \label{eqn:modphiuppbou}
	\end{equation}
	Denote by $\tilde{f}$ to be the density w.r.t. Lebesgue measure of symmetrized random vector $X-X'$, where $X'$ is an independent copy of $X$. Then $\tilde{f}$ also has upper bound $U$ and $|\phi(\zeta)|^2$ is the characteristic function of $X-X'$ and 
	\begin{equation}
	|\phi(\zeta)|^2= \int_{\R^d}e^{i\zeta^\top x}\tilde{f}(x)dx = \int_{\R^d} \cos(\zeta^\top x)\tilde{f}(x)dx. \label{eqn:symrvcha}
	\end{equation}
	Write $x=(x^{(1)},\ldots,x^{(d)})$ and let $G_j=\{x\in \R^d| x^{(1)}\in (\frac{j}{t}-\frac{1}{2t},\frac{j}{t}+\frac{1}{2t}]  \}$ be the strip of length $\frac{1}{t}$ centered at $\frac{j}{t}$ across the $x^{(1)}$-axis. Then by \eqref{eqn:symrvcha}
	\begin{align*}
	I(2\pi\zeta)=&\int_{\R^d} \sin^2(\pi\zeta^\top x)\tilde{f}(x)dx\\
	\geq & \int_{B} \sin^2(\pi tx^{(1)})\tilde{f}(x)dx\\
	= & \sum_{j=-\infty}^{\infty} \int_{G_j \bigcap B } \sin^2(\pi tx^{(1)})\tilde{f}(x)dx \\
	= & \sum_{j=-\infty}^{\infty} \int_{G_j \bigcap B } \sin^2(\pi t(x^{(1)}-j/t))\tilde{f}(x)dx \\
	\geq & 4t^2\sum_{j=-\infty}^{\infty} \int_{G_j \bigcap B } (x^{(1)}-j/t)^2\tilde{f}(x)dx, \numberthis \label{eqn:Ilowbou}
	\end{align*}
	where the first inequality follows from $\zeta=te_1$ and $B$ is a subset in $\R^d$ to be determined, and the last inequality follows from $|\sin(\pi x)|\geq 2 |x| $ for $|x|\leq \frac{1}{2}$. 
	
	Let $B=\{z\in \R^d| |z^{(i)}|< 2\sqrt{d\lambda_{\max}(\Lambda)}  \ \forall i\geq 2, \text{ and } |z^{(1)}|< \frac{r}{t}+\frac{1}{2t} \}$ with $r= \min\{b \text{ integer}: \frac{b}{t}+ \frac{1}{2t} \geq 2\sqrt{d\lambda_{\max}(\Lambda)}  \}$.  Then $B\subset \bigcup_{j=-r}^r G_j $ and thus  \eqref{eqn:Ilowbou} become 
	\begin{align*}
	I(2\pi\zeta) \geq &4t^2\sum_{j=-r}^{r} \int_{G_j \bigcap B } (x^{(1)}-j/t)^2\tilde{f}(x)dx \\
	\overset{(*)}{=} & 4t^2\sum_{j=-r}^{r} \int_{G } (x^{(1)}-j/t)^2\tilde{f}(x)\1ve_{G_j\bigcap B}(x)  dx \\
	\overset{(**)}{\geq} & 4t^2\sum_{j=-r}^{r} \frac{Q_j^3}{12U^2(4\sqrt{d\lambda_{\max}(\Lambda)})^{2(d-1)}}  \\
	\overset{(***)}{\geq} & 4t^2 \frac{Q^3}{12(2r+1)^2U^2(4\sqrt{d\lambda_{\max}(\Lambda)})^{2(d-1)}}, \numberthis \label{eqn:Ilowbou2}
	\end{align*}
	where in step $(*)$ $G= \{z\in \R^d||z^{(i)}|< 2\sqrt{d\lambda_{\max}(\Lambda)}  \ \forall i\geq 2 \}$, step $(**)$ with $Q_j= \int_{G_j\bigcap B}\tilde{f}(x)dx$ follows from Lemma \ref{lem:onecorsquintmin} \ref{item:onecorsquintmintb} and step $(***)$ with $Q=\sum_{j=-r}^rQ_j = \int_B\tilde{f}(x)dx$ follows from Jensen's inequality. The inequalities in step $(**)$ and $(***)$ are attained with  $\tilde{f}(x)= U\sum\limits_{j=-r}^r \1ve_{W_j}\left(x\right)\ a.e.\ x\in G$ where $W_j=\{ z| |z^{(i)}|< 2\sqrt{d\lambda_{\max}(\Lambda)}, \forall i\geq 2, \text{ and } |z^{(1)}-j/t|< a \}$ for positive $a$ satisfies $$(2a)(4\sqrt{d\lambda_{\max}(\Lambda)})^{d-1}U(2r+1)=Q.$$ 
	
	Observe $\{z\in \R^d| z^\top  (2\Lambda)^{-1} z < 2d \} \subset B$ and thus
	\begin{equation*}
	Q=\P(X-X'\in B) \geq 1- \P( (X-X')^\top  (2\Lambda)^{-1} (X-X') \geq 2d )\geq \frac{1}{2}, 
	\end{equation*}
	where the last step follows from Markov inequality. Moreover by our choice of $r$, $2r+1\leq 4t\sqrt{d\lambda_{\max}(\Lambda)}+2$. Then \eqref{eqn:Ilowbou2} become
	\begin{align*}
	I(2\pi\zeta) \geq & t^2\frac{1}{24(4t\sqrt{d\lambda_{\max}(\Lambda)}+2)^2(4\sqrt{d\lambda_{\max}(\Lambda)})^{2(d-1)}U^2}\\	 \geq & \frac{C(d)t^2}{(t^2\lambda_{\max}(\Lambda)+1)\lambda_{\max}^{d-1}(\Lambda)U^2},	
	\end{align*}
	where $C(d)$ is a constant that depends only on $d$. The last display replacing $2\pi\zeta=2\pi t e_1$ by $\zeta= te_1$, together with \eqref{eqn:modphiuppbou} yield the desired conclusion.\\

	\noindent {\textbf{Step 2}} For any $\zeta\not = 0$, denote $t=\|\zeta\|_2$ and $u_1=\zeta/\|\zeta\|_2$. Consider an orthogonal matrix $U_{\zeta}$ with its first row $u_1^\top $. Then 
	$\phi(\zeta)=\E e^{itu_1^\top X}=\E e^{ite_1^\top Z}$ where $Z=U_{\zeta}X$. Since $Z$ has density $f_Z(z)=f(U_{\zeta}^\top z)$ with respect to Lebesgue measure, $f_Z(z)$ has the same upper bound $U$ and positive definite covariance matrix $U_{\zeta}\Lambda U_{\zeta}^\top $ with the same largest eigenvalue as $\Lambda$. The result then follows by applying \textbf{Step 1} to $\left|\E e^{ite_1^\top Z} \right|$.    
\end{proof}


\subsection{Auxiliary lemmas for Sections \ref{sec:proofoftheoremgen} and \ref{sec:characteristicfunctionboundeddensity}}
\label{sec:alinversebound}
Consider a family of probabilities $\{\P_{\thetave}\}_{\thetave\in \Theta}$ on $\R^d$, where $\thetave$ is the parameter of the family and $\Theta\subset \R^q$ is the parameter space. $\E_\theta$ denotes the expectation under the probability measure $\P_\theta$. Consider $\{\Xma_i\}_{i=1}^\infty $ a sequence of independent and identically distributed random vectors from $\P_{\thetave_0}$. Suppose  $\E_{\thetave_0}\Xma_1$ exists and define $\Zma_{\m} = \frac{\sum_{i=1}^{\m} {\Xma_i}-\m\E_{\thetave_0}\Xma_1}{\sqrt{\m}}$. The next result establishes that the density of $\Zma_{\m}$ converges uniformly to that of a multivariate normal distribution. 

\begin{lem}[Local Central Limit Theorem] \label{thm:localcenlim}
	Suppose $\{\Xma_i\}_{i=1}^\infty $ a sequence of independent and identically distributed random vectors from $\P_{\thetave_0}$. Suppose  $\E_{\thetave_0}\Xma_1$ and $\Lambda_{\thetave_0}:= \E_{\thetave_0}(\Xma_1-\E_{\thetave_0}\Xma_1)(\Xma_1-\E_{\thetave_0}\Xma_1)^\top $ exist and $\Lambda_{\thetave_0}$ is positive definite. Let the characteristic function of $\P_{\thetave}$ be $\phi(\zeta|\thetave):=\E_{\theta}e^{i\zeta^\top  X_1}$ and suppose there exists $r\geq 1$ such that $|\phi(\zeta|\thetave_0)|^r$ is Lebesgue integrable on $\R^d$. Then when $\m\geq r$, $\Zma_{\m}$ has density with respect to Lebesgue measure on $\R^d$, and its density $f_{\Zma}(\zve|\thetave_0,\m)$ as $\m$ tends to infinity converges uniformly in $\zve$ to $f_{\mathcal{N}}(\zve|\thetave_0)$, the density of $\mathcal{N}(\bm{0},\Lambda_{\thetave_0})$.
\end{lem}	
The special case for $d=1$ of the above lemma is Theorem 2 in Section 5, Chapter  XV of   \cite{feller2008introduction}. That proof can be generalized to $d>1$ without much difficulties and therefore the proof of Lemma \ref{thm:localcenlim} is omitted.

\begin{lem} \label{lem:onecorsquintmin}
	\begin{enumerate}[label=\alph*)]
		\item \label{item:onecorsquintmina}
		Consider a Lebesgue measurable function on $\R$ satisfies $0\leq f(x)\leq U$ and  $\int_{\R} f(x)dx = E \in (0,\infty)$. Then for any $b>0$
		$$
		\int_{\R}(x-b)^2f(x)dx  
		\geq   \frac{E^3}{12U^2},
		$$
		and the equality holds if and only if $f(x)=U\1ve_{[b-\frac{E}{2U},b+\frac{E}{2U}]}(x)\ a.e.$.
		\item \label{item:onecorsquintmintb}
		For $a>0$ define a set $G= \{z\in \R^d| |z^{(i)}|<a \quad \forall i\geq 2 \}$. Consider a Lebesgue measurable function on $\R^d$ satisfies $0\leq f(x)\leq U$ on $G$  and  $\int_{G} f(x)dx = E \in (0,\infty)$. Then for any $b>0$
		$$
		\int_{G} (x^{(1)}-b)^2f(x)dx \geq 
		\frac{E^3}{12U^2(2a)^{2(d-1)}},
		$$
		and the equality holds if and only if  $f(x)=U\1ve_{G_1}(x) \ a.e.\ x\in G $ where $G_1 = [b-\frac{E}{2U(2a)^{d-1}},b+\frac{E}{2U(2a)^{d-1}}] \times (-a,a)^{d-1} $.
	\end{enumerate}
\end{lem}

\subsection{Proof of the technical lemma in the proof of Theorem \ref{thm:genthm}}
\label{sec:proofoftechnicallemma}

Lemma \ref{lem:technical} plays an essential role in the proof of Theorem \ref{thm:genthm} presented in Section \ref{sec:proofoftheoremgen}. 

\begin{proof}[Proof of Lemma \ref{lem:technical}]
    We will write $\theta_i^\ell,\eta_i^\ell,\m_\ell$ respectively for $\bar{\theta}_i^\ell,\bar{\eta}_i^\ell,\bar{\m}_\ell$ in this proof. But $\theta_i^\ell,\eta_i^\ell,\m_\ell$ in this proof are generic variables and might not necessarily be the same as in the proof of Theorem \ref{thm:genthm}. Let $\bar{\Theta}(G_0)$ be the same as in the first paragraph of the Step 3 in the proof of Theorem \ref{thm:genthm}.
    
    For any $\theta\in \bar{\Theta}(G_0)$, by condition \ref{item:genthmd} $\left.\nabla_{\zeta}\ \phi_T(\zeta|\theta)\right|_{\zeta=\bm{0}}=\i\lambda_{\theta}$, and $\left.\textbf{Hess}_{\zeta}\ \phi_T(\zeta|\theta)\right|_{\zeta=\bm{0}}=\i^2\left(\Lambda_{\theta}+\lambda_{\theta}\lambda_{\theta}^\top \right)$ exist, and
	by condition  \ref{item:genthme}
	$\frac{\partial \lambda_{\theta}}{\partial \theta^{(j)}}$ and $\frac{\partial \Lambda_{\theta}}{\partial \theta^{(j)}}$ exist.  Then, with condition  \ref{item:genthmd} 
	it follows from Pratt's Lemma that		
	$\frac{\partial f_{\Ncal}(\yve|\thetave,\m)}{\partial \theta^{(j)}}$ exists and is given by
	\begin{align}
	\frac{\partial f_{\Ncal}(\yve|\thetave,\m)}{\partial \theta^{(j)}} 
	= \frac{1}{(2\pi)^s}\int_{\R^s} e^{-\i\zetave^\top \yve} e^{\i\m\zetave^\top \lambda_{\theta}-\frac{\m}{2}\zetave^\top \Lambda_{\theta}\zetave}\left(\i\m\zetave^\top \frac{\partial \lambda_{\theta}}{\partial \theta^{(j)}}-\frac{\m}{2}\zetave^\top \frac{\partial \Lambda_{\theta}}{\partial \theta^{(j)}}\zetave\right)   d\zetave .  \label{eqn:derintrep}
	\end{align}
	Plugging the Fourier inversion formula \eqref{eqn:Ydenfor} and \eqref{eqn:derintrep} into \eqref{eqn:Jldefto0}, and
	noting $|e^{-\i\zeta^\top  y}| \leq 1$ for all $y \in \R^s$, for sufficiently large $\ell$ we obtain
	\begin{align*}
	&J(\theta_i^\ell,\eta_i^\ell,\m_\ell)\\
	\leq &  \m_{\ell}^{s/2}\frac{1}{(2\pi)^s}\int_{\R^s} \left | (\phi_T(\zeta|\theta_i^\ell))^{\m_{\ell}} - (\phi_T(\zeta|\eta_i^\ell))^{\m_{\ell}} \right. \\
	&\ \ \ \ \left. - \m_{\ell} e^{\i\m_{\ell}\zetave^\top \lambda_{\theta_i^0}-\frac{\m_{\ell}}{2}\zetave^\top \Lambda_{\theta_i^0}\zetave} \sum_{j=1}^q \left((\theta_i^\ell)^{(j)}-(\eta_i^\ell)^{(j)}\right)\left(\i\zetave^\top \frac{\partial \lambda_{\theta_i^0}}{\partial \theta^{(j)}}-\frac{1}{2}\zetave^\top \frac{\partial \Lambda_{\theta_i^0}}{\partial \theta^{(j)}}\zetave\right) \right | d\zetave \\
	\leq  & \check{J}_\ell + \hat{J}_\ell, 
	\end{align*}
	where
	\begin{align*}
	&\check{J}_\ell := \frac{\m_{\ell}^{s/2}}{(2\pi)^s}\int_{\R^s} \left|(\phi_T(\zeta|\theta_i^\ell))^{\m_{\ell}} - (\phi_T(\zeta|\eta_i^\ell))^{\m_{\ell}} \right.\\
	&\left. -  \m_{\ell} \left(\phi_T(\zeta|\theta_i^{0})\right)^{\m_{\ell}-1} \sum_{j=1}^q \left((\theta_i^\ell)^{(j)}-(\eta_i^\ell)^{(j)}\right)\frac{\partial \phi_T(\zeta|\thetave_i^{0})}{\partial \theta^{(j)}} \right| d\zetave,
	\end{align*}
	and 
	\begin{align*}
	\hat{J}_\ell & := \m_{\ell}^{s/2+1}\frac{1}{(2\pi)^s} \sum_{j=1}^q \left|(\theta_i^\ell)^{(j)}-(\eta_i^\ell)^{(j)}\right| \int_{\R^s} \left|  \left(\phi_T(\zeta|\theta_i^{0})\right)^{\m_{\ell}-1} \frac{\partial \phi_T(\zeta|\thetave_i^{0})}{\partial \theta^{(j)}} - \right.\\
	& \ \ \ \ \left. - \exp\left(\i\m_{\ell}\zetave^\top \lambda_{\theta_i^0}-\frac{\m_{\ell}}{2}\zetave^\top \Lambda_{\theta_i^0}\zetave\right) \left(\i\zetave^\top \frac{\partial \lambda_{\theta_i^0}}{\partial \theta^{(j)}}-\frac{1}{2}\zetave^\top \frac{\partial \Lambda_{\theta_i^0}}{\partial \theta^{(j)}}\zetave\right) \right| d\zetave.
	\end{align*}
	We will show in the sequel that $\check{J}_\ell=o(\sqrt{\m_\ell} \|\theta_i^\ell-\eta_i^\ell\|_2 )$ in Step 1 and  $\hat{J}_\ell=o(\sqrt{\m_\ell} \|\theta_i^\ell-\eta_i^\ell\|_2 )$ in Step 2, thereby establishing \eqref{eqn:Jldefto0}.\\
	 
	\noindent \textbf{Step 1} (Prove $\check{J}_\ell=o(\sqrt{\m_\ell} \|\theta_i^\ell-\eta_i^\ell\|_2 )$)\\
	By Condition \ref{item:genthmg} and Lemma \ref{lem:taylortwisted} \ref{item:taylortwistedb},
	\begin{align*}
	&\check{J}_\ell
	\leq  \frac{\m_{\ell}^{s/2}}{(2\pi)^s}\int_{\R^s} \left|q\sum_{1\leq j,\beta\leq q}   \left(\|\theta_i^\ell-\theta_i^0\|_2 +\|\eta_i^\ell-\theta_i^0\|_2 \right) \|\theta_i^\ell-\eta_i^\ell\|_2 R_1(\zeta;\theta_i^0, \theta_i^\ell, \eta_i^\ell,j,\beta)  \right|d\zeta , \numberthis \label{eqn:Jlbou}
	\end{align*}
	where with $\theta_\ell(t_1,t_2) = \theta_i^0+t_2(\eta_i^\ell  +t_1(\theta_i^\ell-\eta_i^\ell)-\theta_i^0)$
	\begin{align*}
	&R_1(\zeta;\theta_i^0, \theta_i^\ell, \eta_i^\ell,j,\beta), \\
	= &
	\int_{0}^1\int_0^1 \left|\m_{\ell}(\m_{\ell}-1) \left(\phi_T(\zeta|\theta_\ell(t_1,t_2))\right)^{\m_{\ell}-2} \frac{\partial \phi_T(\zeta|\theta_\ell(t_1,t_2))}{\partial \theta^{(j)}} \frac{\partial \phi_T(\zeta|\theta_\ell(t_1,t_2))}{\partial \theta^{(\beta)}} + \right.\\
	& \quad \left. + \m_{\ell} \left(\phi_T(\zeta|\theta_\ell(t_1,t_2))\right)^{\m_{\ell}-1} \frac{\partial^2 \phi_T(\zeta|\theta_\ell(t_1,t_2))}{\partial \theta^{(j)}\partial \theta^{(\beta)}}  \right|dt_2 dt_1.
	\end{align*}
	Then 
	\begin{align*}
	&\int_{\R^s}\left|R_1(\zeta;\theta_i^0, \theta_i^\ell, \eta_i^\ell,j,\beta)\right| d\zeta \\
	\leq & \m_{\ell}  \int_{\R^s} \int_{0}^1 \int_{0}^1 \left| \phi_T(\zeta|\theta_\ell(t_1,t_2))\right|^{\m_{\ell}-2}\times  \\
	& \quad  \left( \m_{\ell}\left|\frac{\partial \phi_T(\zeta|\theta_\ell(t_1,t_2))}{\partial \theta^{(j)}} \frac{\partial \phi_T(\zeta|\theta_\ell(t_1,t_2))}{\partial \theta^{(\beta)}}\right|   +     \left|\frac{\partial^2 \phi_T(\zeta|\theta_\ell(t_1,t_2))}{\partial \theta^{(j)}\partial \theta^{(\beta)}}  \right|\right)  dt_2dt_1 d\zetave \\
	= & \m_{\ell} \int_{0}^1 \int_{0}^1 \int_{\R^s}  \left| \phi_T(\zeta|\theta_\ell(t_1,t_2))\right|^{\m_{\ell}-2}\times\\
	&\quad \left( \m_{\ell}\left|\frac{\partial \phi_T(\zeta|\theta_\ell(t_1,t_2))}{\partial \theta^{(j)}} \frac{\partial \phi_T(\zeta|\theta_\ell(t_1,t_2))}{\partial \theta^{(\beta)}}\right|   +     \left|\frac{\partial^2 \phi_T(\zeta|\theta_\ell(t_1,t_2))}{\partial \theta^{(j)}\partial \theta^{(\beta)}}  \right|\right)d\zetave dt_2 dt_1  \\
	=: & \m_\ell R_2(\theta_i^0, \theta_i^\ell, \eta_i^\ell,j,\beta),\numberthis \label{eqn:R1bou}
	\end{align*}
	where the first inequality follows from the fact that $|\phi_T(\zeta|\theta_\ell(t_1,t_2)))| \leq 1$, and the last inequality follows from Condition \ref{item:genthmg}, Tonelli Theorem and the joint Lebesgue measurability of $\phi_T(\zeta|\theta_\ell(t_1,t_2))$, $\frac{\partial \phi_T(\zeta|\theta_\ell(t_1,t_2))}{\partial \theta^{(j)}}$ and $\left|\frac{\partial^2 \phi_T(\zeta|\theta_\ell(t_1,t_2))}{\partial \theta^{(j)}\partial \theta^{(\beta)}}  \right|$, as  functions of $\zeta$, $t_1$ and $t_2$ by \cite[Lemma 4.51]{guide2006infinite}. Then following \eqref{eqn:Jlbou} and \eqref{eqn:R1bou},
	\begin{align*}
	& \check{J}_\ell \\
	\leq & C(q,s) \m_{\ell}^{s/2+1}  \left\|\theta_i^\ell-\eta_i^\ell\right\|_2 \left(\left\|\theta_i^\ell-\theta_i^0\right\|_2 +\left\|\eta_i^\ell-\theta_i^0\right\|_2\right)  \ \max_{1\leq j,\beta\leq q}\ R_2(\theta_i^0, \theta_i^\ell, \eta_i^\ell,j,\beta)\\
	\numberthis \label{eqn:Jltemp2}
	= & 
	C(q,s) \m_{\ell}\|\theta_i^\ell-\eta_i^\ell\|_2  (\|\theta_i^\ell-\theta_i^0\|_2 + \|\eta_i^\ell-\theta_i^0\|_2) \max_{ j,\beta}\   \int_{0}^1 \int_{0}^1   \int \left| \phi_T(\left.\frac{\bar{\zeta}}{\sqrt{\m_{\ell}}} \right|\theta_\ell(t_1,t_2))\right|^{\m_{\ell}-2}   \\
	& \times
	\left(\m_{\ell} \left|\frac{\partial \phi_T(\frac{\bar{\zeta}}{\sqrt{\m_{\ell}}}|\theta_\ell(t_1,t_2))}{\partial \theta^{(j)}} \frac{\partial \phi_T(\frac{\bar{\zeta}}{\sqrt{\m_{\ell}}}|\theta_\ell(t_1,t_2))}{\partial \theta^{(\beta)}}\right|   +     \left|\frac{\partial^2 \phi_T(\frac{\bar{\zeta}}{\sqrt{\m_{\ell}}}|\theta_\ell(t_1,t_2))}{\partial \theta^{(j)}\partial \theta^{(\beta)}}  \right| \right) d\bar{\zeta} dt_2 dt_1, 
	\end{align*}
	where in the first inequality $C(q,s)$ is some constant that depends on $q$ and $s$, and where the second equality follows from
	\eqref{eqn:R1bou} and changing variable with $\bar{\zeta}=\sqrt{\m_{\ell}}\zeta$. Denote the integrand in the last display by $E_{j,\beta}(\bar{\zeta},t_1,t_2)$.
	
	In the rest of the proofs denote the left hand sides of \eqref{eqn:uniformboundednew} and \eqref{eqn:integrabilitynew}  respectively by $U_1(\theta_0)$ and $U_2(\theta_0)$ for every $\theta_0\in \Theta_1=\{\theta_i^0\}_{i=1}^{k_0}$.
	
	Observe that $f_Y(y|\theta_\ell(t_1,t_2), r)$ exists and has upper bound $\frac{1}{(2\pi)^s} \int_{\R^s} |\phi_{T}(\zeta|\theta_\ell(t_1,t_2))|^rd\zeta  \leq C(s)U_2(\theta_i^0)$ by condition \ref{item:genthmg}. Then invoking Lemma \ref{lem:cfuppbou}, for $\|\zeta\|_2\leq 1$,
	\begin{align*}
	|\phi_T(\zeta|\theta_\ell(t_1,t_2))|^r \leq & \exp\left(-\frac{C(s)\|\zeta\|_2^2}{(\lambda_{\max}(\Lambda_{\theta_\ell(t_1,t_2)})+1)\lambda_{\max}^{s-1}(\Lambda_{\theta_\ell(t_1,t_2)}))U_2^2(\theta_i^0)}\right)\\
	\leq & \exp\left(-\frac{C(s)\|\zeta\|_2^2}{U_3(\theta_i^0)U_2^2(\theta_i^0)}\right), \numberthis \label{eqn:phiTuppsmallzeta}
	\end{align*}
	where the last step follows from $(\lambda_{\max}(\Lambda_{\theta_\ell(t_1,t_2)})+1)\lambda_{\max}^{s-1}(\Lambda_{\theta_\ell(t_1,t_2)}) \leq U_3(\theta_i^0)$ by condition \ref{item:genthmd} with $U_3(\theta_i^0)$ being some constant that depends on $\theta_i^0$.

	Moreover, by the mean value theorem and condition \ref{item:genthmg}: $\forall \|\zeta\|_2 < 1 $
	\begin{multline}
	\left|\frac{\partial \phi_T(\zetave|\theta_\ell(t_1,t_2))}{\partial \theta^{(j)}}\right|=\left|\frac{\partial \phi_T(\zetave|\theta_\ell(t_1,t_2))}{\partial \theta^{(j)}} - \frac{\partial \phi_T(0|\theta_\ell(t_1,t_2))}{\partial \theta^{(j)}} \right| \\
	\leq \|\zeta\|_2\ \sup_{\|\zeta\|_2<1} \left\|\nabla_{\zeta} \frac{\partial \phi_T(\zetave|\theta_\ell(t_1,t_2))}{\partial \theta^{(j)}}\right\|_2 \leq \sqrt{s}U_1(\theta_i^0)\|\zeta\|_2. \label{eqn:thetaderuniuppbou}  
	\end{multline}
	Then
	\begin{align*}
	&\int_{\|\bar{\zeta}\|_2<\sqrt{\m_{\ell}}}  E_{j,\beta}(\bar{\zeta},t_1,t_2) d\bar{\zeta}\\
	\leq & \int_{\|\bar{\zeta}\|_2<\sqrt{\m_{\ell}}}  \exp\left(-\frac{C(s)\|\bar{\zeta}\|_2^2}{rU_3(\theta_i^0)U_2^2(\theta_i^0)}\frac{\m_{\ell}-2}{\m_{\ell}}\right) \left( \left(\sqrt{s}U_1(\theta_i^0)\right)^2\|\bar{\zeta}\|_2^2   +     U_1(\theta_i^0)\right) d\bar{\zeta}\\
	\leq & \int_{\R^s}  \exp\left(-\frac{C(s)\|\bar{\zeta}\|_2^2}{2rU_3(\theta_i^0)U_2^2(\theta_i^0)}\right) \left( \left(\sqrt{s}U_1(\theta_i^0)\right)^2\|\bar{\zeta}\|_2^2   +     U_1(\theta_i^0)\right) d\bar{\zeta}\\
	= & C(s,r,\theta_i^0) \numberthis \label{eqn:intupperbou11},
	\end{align*}
	where the first inequality follows from \eqref{eqn:phiTuppsmallzeta} and \eqref{eqn:thetaderuniuppbou}. 
	
	Let $\eta:=\sup_{\|\zetave\|_2\geq 1} |\phi_T(\zeta|\thetave_i^0)|$. Since the density $f_Y(y|\theta_i^0,r)$ w.r.t. Lebesgue exists and has characteristic function $\phi_T^r(\zeta|\thetave_i^0)$, $\phi_T^r(\zeta|\thetave_i^0)\to 0$ as $\|\zetave\|_2\to \infty$ by Riemann–Lebesgue lemma. It follows that $\eta$ is actually a maximum. Moreover, the existence of the density $f_Y(y|\theta_i^0,r)$ w.r.t. Lebesgue, together with Lemma 4 in Section 1, Chapter XV of \cite{feller2008introduction}, yield $|\phi_T(\zeta|\thetave_i^0)|^r<1$ when $\zetave\not = \bm{0}$.  It follows that $\eta<1$. By the mean value theorem and \ref{item:genthmg}
	$$
	\sup_{\zeta\in \R^s}|\phi_T(\zeta|\theta_\ell(t_1,t_2))-\phi_T(\zeta|\theta_i^0)|\leq \sqrt{q}U_1(\theta_i^0) (\| \theta_i^\ell-\theta_i^0\|_2+\| \eta_i^\ell-\theta_i^0\|_2),
	$$
	which further implies $\sup_{t_1,t_2\in[0,1]}\sup_{\|\zeta\|_2\geq 1}|\phi_T(\zeta|\theta_\ell(t_1,t_2))|<\eta+\frac{1-\eta}{2}:=\eta'<1$ for sufficiently large $\ell$. Then for sufficiently large $\ell$,
	\begin{align*}
	&\int_{\|\bar{\zeta}\|_2\geq\sqrt{\m_{\ell}}}  E_{j,\beta}(\bar{\zeta},t_1,t_2) d\bar{\zeta}\\
	\leq & \left(\eta'\right)^{\m_{\ell}-2-r} \int_{\R^s}  \left| \phi_T\left(\left.\frac{\bar{\zeta}}{\sqrt{\m_{\ell}}} \right|\theta_\ell(t_1,t_2)\right)\right|^{r} \left(\m_{\ell} U_1^2(\theta_i^0)   +     \left|\frac{\partial^2 \phi_T(\frac{\bar{\zeta}}{\sqrt{\m_{\ell}}}|\theta_\ell(t_1,t_2))}{\partial \theta^{(j)}\partial \theta^{(\beta)}}  \right|\right) d\bar{\zeta}\\
	\leq&\left(\eta'\right)^{\m_{\ell}-2-r} \m_{\ell}^{s/2} (\m_{\ell} U_1^2(\theta_i^0)+1)  \int_{\R^s}  \left| \phi_T\left(\left.\zeta \right|\theta_\ell(t_1,t_2)\right)\right|^{r} \left(1   +     \left|\frac{\partial^2 \phi_T(\zeta|\theta_\ell(t_1,t_2))}{\partial \theta^{(j)}\partial \theta^{(\beta)}}  \right|\right)  d\zeta\\
	\leq &  \left(\eta'\right)^{\m_{\ell}-2-r}\m_{\ell}^{s/2}\left(\m_{\ell} U_1^2(\theta_i^0)   +     1\right)U_2(\theta_i^{0}) \numberthis \label{eqn:intupperbou12},
	\end{align*}
	where the first inequality follows from  the definition of $\eta'$ and condition \ref{item:genthmg}, and the last inequality follows from condition \ref{item:genthmg}. \eqref{eqn:intupperbou11} and \eqref{eqn:intupperbou12} immediately imply for any $j$, $\beta$:
	\begin{equation}
	\limsup_{\ell \to \infty}   \int_{0}^1 \int_{0}^1   \int_{\R^s}  E_{j,\beta}(\bar{\zeta},t_1,t_2) d\bar{\zeta} dt_2 dt_1<\infty. \label{eqn:Ejbetalimit} 
	\end{equation}
	The above display together with \eqref{eqn:Jltemp2} and the conditions $\sqrt{N_\ell}\|\theta_i^\ell-\theta_i^0\|_2,\sqrt{N_\ell}\|\eta_i^\ell-\theta_i^0\|_2\to 0$ yield $\check{J}_\ell=o(\sqrt{N_\ell}\|\theta_i^\ell-\eta_i^\ell\|_2)$.\\

	\noindent \textbf{Step 2} (Prove $\hat{J}_\ell=o(\sqrt{\m_\ell} \|\theta_i^\ell-\eta_i^\ell\|_2 )$). 
	A large portion of the proof borrows ideas from Theorem 2 in Chapter XV, Section 5 of \cite{feller2008introduction}. Observe 
	\begin{equation}
	\hat{J}_\ell \leq \sqrt{\m_{\ell}}\|\theta_i^\ell-\eta_i^\ell\|_2\frac{\sqrt{q}}{(2\pi)^s} \max_{1\leq j\leq q} K_\ell(j) \label{eqn:Jhatuppbou}
	\end{equation}
	where as before by a change of variable, $\bar{\zeta}=\sqrt{\m_{\ell}}\zeta$,
	\begin{align*}
	K_\ell(j) := & \m_{\ell}^{\frac{s+1}{2}}\int_{\R^s} \left|  \left(\phi_T(\zeta|\theta_i^{0})\right)^{\m_{\ell}-1} \frac{\partial \phi_T(\zeta|\thetave_i^{0})}{\partial \theta^{(j)}} - \right.\\
	& \ \ \ \ \left. - \exp\left(\i\m_{\ell}\zetave^\top \lambda_{\theta_i^0}-\frac{\m_{\ell}}{2}\zetave^\top \Lambda_{\theta_i^0}\zetave\right) \left(\i\zetave^\top \frac{\partial \lambda_{\theta_i^0}}{\partial \theta^{(j)}}-\frac{1}{2}\zetave^\top \frac{\partial \Lambda_{\theta_i^0}}{\partial \theta^{(j)}}\zetave\right) \right| d\zetave \\
	= & \m_{\ell}^{\frac{s+1}{2}}\int_{\R^s} \left|  \left(e^{-\ive\zeta^\top \lambda_{\theta_i^0}}\phi_T(\zeta|\theta_i^{0})\right)^{\m_{\ell}-1} \frac{\partial \phi_T(\zeta|\thetave_i^{0})}{\partial \theta^{(j)}} - \right.\\
	& \ \ \ \ \left. - \exp\left(\i\zetave^\top \lambda_{\theta_i^0}-\frac{\m_{\ell}}{2}\zetave^\top \Lambda_{\theta_i^0}\zetave\right) \left(\i\zetave^\top \frac{\partial \lambda_{\theta_i^0}}{\partial \theta^{(j)}}-\frac{1}{2}\zetave^\top \frac{\partial \Lambda_{\theta_i^0}}{\partial \theta^{(j)}}\zetave\right) \right| d\zetave \\
	= & \int_{\R^s} \sqrt{\m_{\ell}} \left|\left(e^{-\frac{\i}{\sqrt{\m_{\ell}}}\bar{\zetave}^\top \lambda_{\theta_i^0}}\phi_T\left(\frac{\bar{\zetave}}{\sqrt{\m_{\ell}}}|\theta_i^0\right)\right)^{\m_{\ell}-1} \frac{\partial \phi_T(\frac{\bar{\zetave}}{\sqrt{\m_{\ell}}}|\thetave_i^{0})}{\partial \theta^{(j)}} - \right.\\
	& \ \ \ \ \left. - \exp\left(\frac{\i}{\sqrt{\m_{\ell}}}\bar{\zetave}^\top \lambda_{\theta_i^0}-
	\frac{1}{2}\bar{\zetave}^\top \Lambda_{\theta_i^0}\bar{\zetave}\right) \left(\frac{\i}{\sqrt{\m_{\ell}}}\bar{\zetave}^\top \frac{\partial \lambda_{\theta_i^0}}{\partial \theta^{(j)}}-\frac{1}{2\m_{\ell}}\bar{\zetave}^\top \frac{\partial \Lambda_{\theta_i^0}}{\partial \theta^{(j)}}\bar{\zetave}\right) \right| d\bar{\zetave}.	 \numberthis \label{eqn:intlong}
	\end{align*}
	Denote the integrand in the above display by $A$. Since $\lambda_{\theta_i^0}$ and $\Lambda_{\theta_i^0}$ exist, $e^{-\i\zeta^\top \lambda_{\theta_i^0}}\phi_T(\zeta|\theta_i^0)$ is twice continuously differentiable on $\R^s$, with gradient being $\bm{0}$ and Hessian being $\i^2\Lambda_{\theta_i^0}$ at $\zetave=0$. Then by Taylor Theorem,  
	\begin{align}
	\left|e^{-i\zetave^\top \lambda_{\theta_i^0}}\phi_T(\zeta|\theta_i^0)\right| < \exp\left(-\frac{1}{4}\zetave^\top \Lambda_{\theta_i^0}\zetave\right) \quad \text{ if } 0<\|\zetave\|_2 < \gamma_1, \label{eqn:uppbousmallzeta}
	\end{align}
	for sufficient small $0<\gamma_1<1$, and
	\begin{equation}
	\left(e^{-\frac{\ive}{\sqrt{\m_{\ell}}}\bar{\zetave}^\top \lambda_{\thetave_i^0}}\phi_T\left(\frac{\bar{\zetave}}{\sqrt{\m_{\ell}}}|\thetave_i^0\right)\right)^{\m_{\ell}-1} \to \exp\left(-\frac{1}{2}\bar{\zetave}^\top \Lambda_{\thetave_i^0}\bar{\zetave}\right). \label{eqn:limit1}
	\end{equation}
	
	Let $\eta'':=\sup_{\|\zetave\|_2\geq \gamma_1} |\phi(\zeta|\thetave_0)|$. By the same reasoning of $\eta<1$ in Step 1, $\eta''<1$. Then for any $a>0$,
	\begin{align}
	\int_{\R^s}A d\bar{\zetave} = \int_{\|\bar{\zetave}\|_2\leq a}A d\bar{\zetave} + \int_{a<\|\bar{\zetave}\|_2 < \gamma_1 \sqrt{\m_{\ell}}}A d\bar{\zetave} + \int_{\|\bar{\zetave}\|_2\geq \gamma_1 \sqrt{\m_{\ell}}}A d\bar{\zetave}. \label{eqn:intdivide3}
	\end{align}
	Then, as $\ell \rightarrow \infty$
	\begin{align*}
	&\int_{\|\bar{\zetave}\|_2\geq \gamma_1 \sqrt{\m_{\ell}}}A d\bar{\zetave}  \\
	\leq & \left(\eta''\right) ^{\m_{\ell}-1-r}\sqrt{\m_{\ell}} \int_{\R^s}\left|\phi_T\left(\left.\frac{\bar{\zetave}}{\sqrt{\m_{\ell}}}\right|\thetave_i^0\right)\right|^{r} U_1(\thetave_i^0)d \bar{\zetave} \\
	& +  \sqrt{\m_{\ell}}   \int_{\|\bar{\zetave}\|_2\geq \gamma_1 \sqrt{\m_{\ell}}}\exp\left(-\frac{1}{2}\bar{\zetave}^\top \Lambda_{\thetave_i^0}\bar{\zetave}\right)\left(\frac{1}{\sqrt{\m_{\ell}}}\left|\bar{\zetave}^\top \frac{\partial \lambda_{\thetave_i^0}}{\partial \theta^{(j)}}\right|+\frac{1}{2\m_{\ell}}\left|\bar{\zetave}^\top \frac{\partial \Lambda_{\thetave_i^0}}{\partial \theta^{(j)}}\bar{\zetave}\right|\right) d\bar{\zetave}\\
	= & \left(\eta''\right) ^{\m_{\ell}-1-r}\m_{\ell}^{\frac{s+1}{2}}U_1(\thetave_i^0) \int_{\R^s}\left|\phi_T\left(\left.\zeta\right|\thetave_i^0\right)\right|^{r} d\zeta \\
	& +     \int_{\|\bar{\zetave}\|_2\geq \gamma_1 \sqrt{\m_{\ell}}}\exp\left(-\frac{1}{2}\bar{\zetave}^\top \Lambda_{\thetave_i^0}\bar{\zetave}\right)\left(\left|\bar{\zetave}^\top \frac{\partial \lambda_{\thetave_i^0}}{\partial \theta^{(j)}}\right|+\frac{1}{2\sqrt{\m_{\ell}}}\left|\bar{\zetave}^\top \frac{\partial \Lambda_{\thetave_i^0}}{\partial \theta^{(j)}}\bar{\zetave}\right|\right) d\bar{\zetave}\\
	\rightarrow &  0, \numberthis \label{eqn:int3}
	\end{align*}
	where the first inequality follows from condition \ref{item:genthmg} and the definition of $\eta''$, and the last step follows from $\eta''<1$ and condition \ref{item:genthmg}.
	
	By condition \ref{item:genthme}, $\frac{\partial \phi_T(\zeta|\thetave_i^0)}{\partial \theta^{(j)}}$ as a function of $\zeta$ has gradient at $0$: $\ive\frac{\partial \lambda_{\thetave_i^0}}{\partial \theta^{(j)}}$. Then by Taylor Theorem:
	\begin{equation}
	\sqrt{\m_{\ell}}\frac{\partial \phi_T(\frac{\bar{\zetave}}{\sqrt{\m_{\ell}}}|\thetave_i^0)}{\partial \theta^{(j)}} \to \ive\bar{\zetave}^\top \frac{\partial \lambda_{\thetave_i^0}}{\partial \theta^{(j)}}. \label{eqn:limit2} 
	\end{equation}
	Moreover, specialize $t=0$ in \eqref{eqn:thetaderuniuppbou} and one has: $\forall \|\zeta\|_2<1$  
	\begin{equation}
	\left|\frac{\partial \phi_T(\zetave|\thetave_i^0)}{\partial \theta^{(j)}}\right|\leq \sqrt{s}U_1(\theta_i^0)\|\zeta\|_2. \label{eqn:partialbou}
	\end{equation}
	By combining \eqref{eqn:uppbousmallzeta} and \eqref{eqn:partialbou}, we obtain as $\ell \rightarrow \infty$
	\begin{align*}
	& \int_{a<\|\bar{\zetave}\|_2 < \gamma_1 \sqrt{\m_{\ell}}}A d\bar{\zetave} \\
	\leq & \sqrt{\m_{\ell}}\int_{a<\|\bar{\zetave}\|_2 < \gamma_1 \sqrt{\m_{\ell}}}  \exp\left(-\frac{\m_{\ell}-1}{4\m_{\ell}}\bar{\zetave}^\top \Lambda_{\thetave_i^0}\bar{\zetave}\right)  \sqrt{s}U_1(\theta_i^0)\left( \frac{\|\bar{\zeta}\|_2}{\sqrt{\m_{\ell}}}\right)\\
	&  +    \exp\left(-\frac{1}{2}\bar{\zetave}^\top \Lambda_{\thetave_i^0}\bar{\zetave}\right)\left(\left|\frac{1}{\sqrt{\m_{\ell}}}\bar{\zetave}^\top \frac{\partial \lambda_{\thetave_i^0}}{\partial \theta^{(j)}}\right|+\left|\frac{1}{2\m_{\ell}}\bar{\zetave}^\top \frac{\partial \Lambda_{\thetave_i^0}}{\partial \theta^{(j)}}\bar{\zetave}\right|\right)   d\bar{\zetave} \\
	\leq & \int_{a<\|\bar{\zetave}\|_2 < \gamma_1 \sqrt{\m_{\ell}}}  2\exp\left(-\frac{1}{8}\bar{\zetave}^\top \Lambda_{\thetave_i^0}\bar{\zetave}\right)  C(\thetave_i^0,s)\left(\|\bar{\zetave}\|_2+\|\bar{\zetave}\|_2^2\right)    d\bar{\zetave}\\
	\to & C(\thetave_i^0,s) \int_{\|\bar{\zetave}\|_2>a }  2\exp\left(-\frac{1}{8}\bar{\zetave}^\top \Lambda_{\thetave_i^0}\bar{\zetave}\right)  \left(\|\bar{\zetave}\|_2+\|\bar{\zetave}\|_2^2\right)    d\bar{\zetave}, \numberthis \label{eqn:int2}
	\end{align*}	
	where in the second inequality we impose $\m_{\ell} \geq 2$ since it's the limit that is of interest, and $C(\thetave_i^0,s)$ is a constant that depends on $\thetave_i^0$ and $s$.
	
	Finally by \eqref{eqn:limit1} and \eqref{eqn:limit2}, when $\|\bar{\zeta}\|_2\leq a$
	$$
	\sqrt{\m_{\ell}} \left(e^{-\frac{\i}{\sqrt{\m_{\ell}}}\bar{\zetave}^\top \lambda_{\theta_i^0}}\phi_T\left(\frac{\bar{\zetave}}{\sqrt{\m_{\ell}}}|\theta_i^0\right)\right)^{\m_{\ell}-1} \frac{\partial \phi_T(\frac{\bar{\zetave}}{\sqrt{\m_{\ell}}}|\thetave_i^{0})}{\partial \theta^{(j)}} \to \exp\left(-\frac{1}{2}\bar{\zetave}^\top \Lambda_{\thetave_i^0}\bar{\zetave}\right)\ive\bar{\zetave}^\top \frac{\partial \lambda_{\thetave_i^0}}{\partial \theta^{(j)}}.
	$$
	Moreover 
	\begin{align*}
	&\sqrt{\m}_\ell\exp\left(\frac{\i}{\sqrt{\m_{\ell}}}\bar{\zetave}^\top \lambda_{\theta_i^0}-
	\frac{1}{2}\bar{\zetave}^\top \Lambda_{\theta_i^0}\bar{\zetave}\right) \left(\frac{\i}{\sqrt{\m_{\ell}}}\bar{\zetave}^\top \frac{\partial \lambda_{\theta_i^0}}{\partial \theta^{(j)}}-\frac{1}{2\m_{\ell}}\bar{\zetave}^\top \frac{\partial \Lambda_{\theta_i^0}}{\partial \theta^{(j)}}\bar{\zetave}\right) \\
	\to &\exp\left(-\frac{1}{2}\bar{\zetave}^\top \Lambda_{\thetave_i^0}\bar{\zetave}\right)\ive\bar{\zetave}^\top \frac{\partial \lambda_{\thetave_i^0}}{\partial \theta^{(j)}}
	\end{align*}
	and hence $\lim_{\ell\to \infty}A=0 $ when $\|\bar{\zeta}\|_2\leq a$. One can find an integrable envelope function for $A$ when $\|\bar{\zeta}\|_2\leq a$ in similar steps as \eqref{eqn:int2}, and then by the dominated convergence theorem,
	\begin{equation}
	\int_{\|\bar{\zetave}\|_2\leq a}A d\bar{\zetave} \to 0. \label{eqn:int1}
	\end{equation}
	
	Plug \eqref{eqn:int1}, \eqref{eqn:int2} and \eqref{eqn:int3} into \eqref{eqn:intdivide3} and \eqref{eqn:intlong}, and one has
	\begin{align*}
	\limsup_{\ell \to \infty}\ \  K_\ell(j) 
	\leq  C(\thetave_i^0,s) \int_{\|\bar{\zetave}\|_2>a }  2\exp\left(-\frac{1}{8}\bar{\zetave}^\top \Lambda_{\thetave_i^0}\bar{\zetave}\right)  \left(\|\bar{\zetave}\|_2+\|\bar{\zetave}\|_2^2\right)    d\bar{\zetave}.
	\end{align*}
	Letting $a\to \infty$ in the above display yields $K_\ell(j)\to 0$, which together with \eqref{eqn:Jhatuppbou} imply $\hat{J}_\ell = o(\sqrt{\m_\ell}\|\theta_i^\ell-\eta_i^\ell\|_2)$. 
\end{proof}


\section{Proofs for Section 6}

\subsection{Proofs of Theorem \ref{thm:posconnotid} and Corollary \ref{cor:posconexp}}
\label{sec:proofsposteriorcontraction}
For $B$ a subset of metric space with metric $D$, the minimal number of balls with centers in $B$ and of radius $\epsilon$ needed to cover $B$ is known as the $\epsilon$-covering number of $B$ and is denoted by $\Nf (\epsilon,B,D)$. Define the root average square Hellinger metric:
	$$
	d_{\n,h}(G,G_0) = \sqrt{\frac{1}{\n}\sum_{i=1}^{\n}h^2(p_{G,\m_i},p_{G_0,\m_i})  }.
	$$

\begin{proof}[Proof of Theorem \ref{thm:posconnotid}]

a) 
The proof structure is the same as Lemma \ref{lem:convergencerate} except that to take the varied sequence lengths into account, the distance $d_{m,h}$ is used in the place of total variation $V$ for the mixture densities.
We verify conditions (i) and (ii) of \cite[Theorem 8.23]{ghosal2017fundamentals}, respectively, in Step 1 and Step 2 below to obtain a posterior contraction bound on the mixture density. In Step 3 we prove a posterior consistency result and then apply Lemma \ref{cor:identifiabilityallm} to transfer posterior contraction result on density estimation to parameter estimation. \\

\noindent {\bf Step 1} (Verification of condition (i) of \cite[Theorem 8.23]{ghosal2017fundamentals}) \\
   Write $n_1$ and $n_0$ respectively for $n_1(G_0)$ and $n_0(G_0,\cup_{k\leq k_0} \Ec_k(\Theta_1))$ in the proof for clean presentation. 
   Note that \ref{item:kernel} implies that $\theta\mapsto P_\theta$ from $(\Theta,\|\cdot\|_2)$ to $(\{P_\theta\}_{\theta\in \Theta}, h)$ is continuous. Then due to Lemma \ref{cor:VlowbouW1} and Lemma \ref{lem:relW1D1} \ref{item:relW1D1e}, for any $\m \geq n_1\vee n_0$, and any $ G\in\Ec_{k_0}(\Theta_1)$,
	\begin{equation}
	\label{eqn:lowboundhwdi}
\sqrt{2}	h(p_{G,\m},p_{G_0,\m}) \geq V(p_{G,\m},p_{G_0,\m})
	\geq C(G_0,\Theta_1)W_1(G,G_0) \geq  C(G_0,\Theta_1)D_1(G,G_0). 
	\end{equation}

In the remainder of the proof $\m \geq n_1\vee n_0$ is implicitly imposed.	By~\eqref{eqn:lowboundhwdi} it holds that, for all $G \in \Ec_{k_0}(\Theta_1)$
	\begin{equation}
	d_{\n,h}(G,G_0)\geq C(G_0,\Theta_1)W_1(G,G_0)\geq C(G_0,\Theta_1)D_1(G,G_0).\label{eqn:dnhD1lowbou}
	\end{equation}
	Then 
	\begin{equation}
	\{G\in \Ec_{k_0}(\Theta_1): d_{\n,h}(G,G_0) \leq \epsilon \} \subset \left\{G\in \Ec_{k_0}(\Theta_1): D_1(G,G_0)\leq \frac{\epsilon}{C(G_0,\Theta_1)}\right\} \label{eqn:setinclusionhD1diflen}
	\end{equation}
	and thus for any $j\in \mathbb{N}$,
	\begin{multline}
	\Pi\left(d_{\n,h}(G,G_0) \leq 2j\epsilon\right) \leq \Pi\left(D_1(G,G_0)\leq \frac{2j\epsilon}{C(G_0,\Theta_1)}\right) \\
	\lesssim  k_0! \left( \frac{2j\epsilon}{C(G_0,\Theta_1)}\right)^{k_0-1} \left( \frac{2j\epsilon}{C(G_0,\Theta_1)}\right)^{qk_0}, \label{eqn:pihdiflen}
	\end{multline}
	where the last inequality follows from \ref{item:prior}.
	
	By an argument similar to  \cite[Lemma 3.2 (a)]{nguyen2016borrowing}, for any $G=\sum_{i=1}^{k_0}p_i\delta_{\theta_i}\in \Ec_{k_0}(\Theta_1)$
	\begin{align*}
	K(p_{G_0,\m_i},p_{G,\m_i}) \leq & \m_i L_1 W_{\alpha_0}^{\alpha_0}(G,G_0)\\
	\leq & \m_i C(\text{diam}(\Theta_1),\alpha_0, L_1) \min_{\tau\in S_{k_0}} \sum_{i=1}^{k_0}\left(\|\theta_{\tau(i)}-\theta^0_{i}\|_2^{\alpha_0} + |p_{\tau(i)}-p^0_i|  \right),
	\end{align*}
	where the second inequality follows from Lemma \ref{lem:relW1D1} \ref{item:relW1D1b} and \ref{item:kernel}. 
	Then 
	\begin{align*}
	\frac{1}{\n}\sum_{i=1}^{\n}K(p_{G_0,\m_i}, p_{G,\m_i})  \leq \bar{\m}_{\n} C(\text{diam}(\Theta_1),\alpha_0) \min_{\tau\in S_{k_0}} \sum_{i=1}^{k_0}\left(\|\theta_{\tau(i)}-\theta^0_{i}\|_2^{\alpha_0} + |p_{\tau(i)}-p^0_i|  \right),
	\end{align*}
	and
	\begin{align*}
	&\Pi\left( \frac{1}{\n}\sum_{i=1}^{\n}K(p_{G_0,\m_i}, p_{G,\m_i}) \leq \epsilon^2     \right) \\
	\gtrsim &\left( \frac{\epsilon^2}{\bar{\m}_{\n} C(\text{diam}(\Theta_1),\alpha_0, L_1)}\right)^{qk_0/\alpha_0} \left( \frac{\epsilon^2}{\bar{\m}_{\n} C(\text{diam}(\Theta_1),\alpha_0, L_1)}\right)^{k_0-1}. \numberthis \label{eqn:piKlowdiflen}
	\end{align*}
	
	Combining \eqref{eqn:pihdiflen} and \eqref{eqn:piKlowdiflen},
	\begin{align*}
	& \frac{\Pi\left(d_{\n,h}(G,G_0) \leq 2j\epsilon\right)}{\Pi\left( \frac{1}{\n}\sum_{i=1}^{\n}K(p_{G_0,\m_i}, p_{G,\m_i}) \leq \epsilon^2     \right)} \\ 
	\leq & C(G_0,\Theta_1,q,\alpha_0,k_0, L_1) j^{qk_0+k_0-1} \bar{\m}_{\n}^{qk_0/\alpha_0+k_0-1}\epsilon^{-qk_0(2/\alpha_0-1)-(k_0-1)}.
	\end{align*}
	By Remark \ref{rem:Lippowerbou} $\alpha_0\leq 2$. Then based on the last display one may verify with $$\epsilon_{\n,\bar{\m}_{\n}}=C(G_0,\Theta,q,k_0,\alpha_0,\beta_0, L_1, L_2) \sqrt{\frac{\ln(\n\bar{\m}_{\n})}{\n}} $$ for some large enough constant $C(G_0,\Theta_1,q,k_0,\alpha_0,\beta_0)$,
	$$ 
	\frac{\Pi\left(d_{\n,h}(G,G_0) \leq 2j\epsilon_{\n,\bar{\m}_{\n}}\right)}{\Pi\left( \frac{1}{\n} \sum_{i=1}^{\n} K(p_{G_0,\m_i}, p_{G,\m_i}) \leq \epsilon^2_{\n,\bar{\m}_{\n}} \right)  } \leq \exp\left(\frac{1}{4}j\n\epsilon_{\n,\bar{\m}_{\n}}^2\right).
	$$
\\

\noindent {\bf Step 2} (Verification of condition (ii) of in \cite[Theorem 8.23]{ghosal2017fundamentals})\\	By \eqref{eqn:setinclusionhD1diflen},
	\begin{align*}
	&\sup_{\epsilon\geq \epsilon_{\n,\bar{\m}_{\n}}} \ln \Nf \left(\frac{1}{36}\epsilon,\{ G\in \Ec_{k_0}(\Theta_1): d_{\n,h}(G,G_0)\leq 2\epsilon \}, d_{\n,h}\right)\\ 
	\leq & 
	\sup_{\epsilon\geq \epsilon_{\n,\bar{\m}_{\n}}} \ln \Nf\left(\frac{1}{36}\epsilon,\left\{ G\in \Ec_{k_0}(\Theta_1): D_1(G,G_0)\leq \frac{2\epsilon}{C(G_0,\Theta_1)}\right \}, d_{\n,h} \right) \\
	\leq &  qk_0\ln \left(1+\frac{4\times (144L_2)^{\frac{1}{\beta_0}}}{C(G_0,\Theta_1)}  \bar{\m}_{\n}^{\frac{1}{2\beta_0}}\epsilon_{\n,\bar{\m}_{\n}}^{-(\frac{1}{\beta_0}-1)} \right) +(k_0-1)\ln \left(1+10\times 72^2 \epsilon_{\n,\bar{\m}_{\n}}^{-2}\right),
	\end{align*}
	where the last inequality follows from Lemma \ref{lem:Nuppboudiflen}. 
	By Remark \ref{rem:Lippowerbou} $\beta_0\leq 1$. Then based on the last display one may verify with $\epsilon_{\n,\bar{\m}_{\n}}=C(G_0,\Theta_1,q,k_0,\alpha_0,\beta_0, L_1, L_2) \sqrt{\frac{\ln(\n\bar{\m}_{\n})}{\n}} $ for some large enough constant $C(G_0,\Theta_1,q,k_0,\alpha_0,\beta_0)$,
	\begin{equation}
	\sup_{\epsilon\geq \epsilon_{\n,\bar{\m}_{\n}}} \ln \Nf\left(\frac{1}{36}\epsilon,\{ G\in \Ec_{k_0}(\Theta_1): d_{\n,h}(G,G_0)\leq 2\epsilon \}, d_{\n,h}\right) \leq \n\epsilon^2_{\n,\bar{\m}_{\n}}. \label{eqn:entropycondition}
	\end{equation}
	
	Now we invoke \cite[Theorem 8.23]{ghosal2017fundamentals} (the Hellinger distance defined in \cite{ghosal2017fundamentals} differs from our definition by a factor of constant. But this constant factor only affect the coefficients of $\epsilon_{\n,\bar{\m}_{\n}}$ but not the conclusion of convergence),
we have for every $\bar{M}_\n\to \infty$,  
	\begin{equation}
	\Pi(G\in \Ec_{k_0}(\Theta_1): d_{\n,h}(G,G_0) \geq \bar{M}_\n \epsilon_{\n,\bar{\m}_{\n}} |X_{[\m_1]}^1,\ldots,X_{[\m_{\n}]}^{\n} ) \to 0  \label{eqn:vanposcontractiondiflen}
	\end{equation}
	in $\P_{G_0,\m_1}\bigotimes \cdots \bigotimes \P_{G_0,\m_{\n}}$-probability as $\n\to\infty$. \\
	
\noindent {\bf Step 3} (From convergence of densities to that of parameters)
Since $n_1 \leq \m_i\leq \m_0:= \sup_{i}\m_i$, by Lemma~\ref{cor:identifiabilityallm} for $G\in \Ec_{k_0}(\Theta)$ satisfying $W_1(G,G_0)<c(G_0,\m_0)$
	\begin{align}
	d_{\n,h}(G,G_0) \geq  C(G_0)\sqrt{\frac{1}{\n}\sum_{i=1}^{\n} D^2_{\m_i}(G,G_0)}. \label{eqn:dnhlowbou1}
	\end{align}    
	By Lemma \ref{lem:optimalpermutation} for $G=\sum_{j=1}^{k_0}p_j\delta_{\theta_j}\in \Ec_{k_0}(\Theta_1)$ satisfying $D_1(G,G_0)<\frac{1}{2}\rho$ where $\rho:=\min_{1\leq i<j\leq k_0} \|\theta_i^0-\theta_j^0\|_2$, there exists a $\tau\in S_{k_0}$ such that
	\begin{align*}
	\sqrt{\frac{1}{\n}\sum_{i=1}^{\n} D^2_{\m_i}(G,G_0)} = & \sqrt{\frac{1}{\n}\sum_{i=1}^{\n} \left(\sqrt{\m_i}\sum_{j=1}^{k_0} \|\theta_{\tau(j)} -\theta_j^0\|_2 + \sum_{j=1}^{k_0}|p_{\tau(j)}-p_j^0|  \right)^2} \\ 
	\geq & \sqrt{\frac{1}{\n}\sum_{i=1}^{\n} \left(\m_i\left(\sum_{j=1}^{k_0} \|\theta_{\tau(j)} -\theta_j^0\|_2\right)^2 + \left(\sum_{j=1}^{k_0}|p_{\tau(j)}-p_j^0|  \right)^2\right)}\\
	= & \sqrt{ \bar{\m}_{\n}\left(\sum_{j=1}^{k_0} \|\theta_{\tau(j)} -\theta_j^0\|_2\right)^2 + \left(\sum_{j=1}^{k_0}|p_{\tau(j)}-p_j^0|  \right)^2}\\
	\geq & \frac{1}{\sqrt{2}} \left(\sqrt{\bar{\m}_{\n}}\sum_{j=1}^{k_0} \|\theta_{\tau(j)} -\theta_j^0\|_2 + \sum_{j=1}^{k_0}|p_{\tau(j)}-p_j^0|\right) \\
	=& \frac{1}{\sqrt{2}} D_{\bar{\m}_{\n}}(G,G_0). \numberthis \label{eqn:rootaveDsqu}
	\end{align*}
	
	Let $\mathcal{G}= \{G\in \Ec_{k_0}(\Theta_1)|W_1(G,G_0)<c(G_0,\m_0), D_1(G,G_0)<\frac{1}{2}\rho \} $. Combining \eqref{eqn:dnhlowbou1} and \eqref{eqn:rootaveDsqu}, for any $G\in \mathcal{G}$
	$$
	d_{\n,h}(G,G_0) \geq  \frac{C(G_0)}{\sqrt{2}} \bar{D}_{\bar{\m}_{\n}}(G,G_0).
	$$
By the union bound,	
	\begin{align*}
	&\Pi(G\in \Ec_{k_0}(\Theta_1): D_{\bar{\m}_{\n}}(G,G_0) \geq \frac{\sqrt{2}\bar{M}_\n}{C(G_0)} \epsilon_{\n,{\bar{\m}_{\n}} } |X_{[\m_1]}^1, \ldots,X_{[\m_{\n}]}^{\n} ) \\
	\leq  &\Pi(G\in \Ec_{k_0}(\Theta_1): d_{\n,h}(G,G_0) \geq \bar{M}_\n \epsilon_{\n,\bar{\m}_{\n}} |X_{[\m_1]}^1,\ldots,X_{[\m_{\n}]}^{\n} ) + \Pi(\mathcal{G}^c |X_{[\m_1]}^1,\ldots,X_{[\m_{\n}]}^{\n} )\\
	\to & 0
	\end{align*}
	in $\bigotimes_{i=1}^{\n}\P_{G_0,\m_i}$-probability as $\n\to\infty$ by applying \eqref{eqn:vanposcontractiondiflen} to the first term. The reason the second term vanishes is as follows. Note that the second term converges to $0$ essentially is a posterior consistency result with respect to $W_1$ (or $D_1$) metric. Here we prove it by \eqref{eqn:dnhD1lowbou} and \eqref{eqn:vanposcontractiondiflen}. 
	By \eqref{eqn:dnhD1lowbou},
	$$
	\mathcal{G}^c \subset \{G\in\Ec_{k_0}(\Theta): d_{\n,h}(G,G_0)> C(G_0,\rho,\m_0,\Theta_1) \} 
	$$
	for some constant $C(G_0,\rho,\m_0,\Theta_1) >0$. For some slow-increasing $\bar{M}'_{\n}$ such that $\bar{M}'_{\n}\epsilon_{\n,\bar{\m}_{\n}}\to 0$ as $\n\to \infty$, 
	\begin{align*}
	&\{G\in\Ec_{k_0}(\Theta_1): d_{\n,h}(G,G_0)> C(G_0,\rho,\m_0,\Theta_1) \} \\
	\subset &\{G\in\Ec_{k_0}(\Theta_1): d_{\n,h}(G,G_0)>\bar{M}'_{\n}\epsilon_{\n,\bar{\m}_{\n}}\}
	\end{align*}
	holds for large $\n$. Combining the last two displays and \eqref{eqn:vanposcontractiondiflen} yields 
	$$
	\Pi(\mathcal{G}^c |X_{[\m_1]}^1,\ldots,X_{[\m_{\n}]}^{\n} )\\
	\to  0.
	$$
	The proof is concluded.

b) 
	If the additional condition of part \ref{item:posconnotidb} is satisfied, then by Remark~\ref{rem:N1Ninfcon} , $n_1(G_0)=1$. That is, the claim of part \ref{item:posconnotida} holds for $n_1(G_0) = 1$. 
\end{proof}

\begin{proof}[Proof of Corollary \ref{cor:posconexp}]
Recall $f(x|\theta)=\exp\left(\langle \eta(\theta),T(x)\rangle -B(\theta) \right)h(x)$.
By easy calculations
    $$
    |K(f(x|\theta_1),f(x|\theta_2))| = |\langle \theta_1-\theta_2, \E_{\theta_1} Tx\rangle - (B(\theta_1)-B(\theta_2))| \leq L_1(\Theta_1) \|\theta_1-\theta_2\|_2.
    $$
    By changing to its canonical parametrization and appeal to Lemma \ref{lem:exphellinger} \ref{item:exphellingerb},
    $$
    |h(f(x|\theta_1),f(x|\theta_2))| \leq L_2(\Theta_1)\|\theta_1-\theta_2\|_2. 
    $$
    Here $L_1(\Theta_1)$ and $L_2(\Theta_1)$ are constants that depend on $\Theta_1$. In summary \ref{item:kernel} is satisfied. Then the conclusions are obtained by applying Theorem \ref{thm:posconnotid}.
\end{proof}

\subsection{Auxiliary lemmas for Section \ref{sec:proofsposteriorcontraction}}
\label{sec:auxiliarylemmas6}

\begin{lem} \label{lem:Nuppboudiflen}
	Fix $G_0=\sum_{i=1}^{k_0}p_i^0\delta_{\theta_i^0}\in\Ec_{k_0}(\Theta_1)$. Suppose $h(f(x|\theta_1),f(x|\theta_2))\leq L_2\|\theta_1-\theta_2\|_2^{\beta_0}$ 
	for some $0<\beta_0\leq 1$ and some $L_2>0$ where $\theta_1, \theta_2$ are any two distinct elements in $\Theta_1$.
	\begin{align*}
	& \Nf \left(\frac{1}{36}\epsilon,\left\{ G\in \Ec_{k_0}(\Theta_1): D_1(G,G_0)\leq \frac{2\epsilon}{C(G_0,\text{diam}(\Theta_1))}\right \}, d_{\n,h} \right)\\
	\leq & \left(1+\frac{4\times (144L_2)^{\frac{1}{\beta_0}}}{C(G_0,\text{diam}(\Theta_1))}  \bar{\m}_{\n}^{\frac{1}{2\beta_0}}\epsilon^{-(\frac{1}{\beta_0}-1)} \right)^{qk_0} \left(1+10\times 72^2 \epsilon^{-2}\right)^{k_0-1}.
	\end{align*}
\end{lem}

	\begin{lem} \label{lem:optimalpermutation}
        For $G_0 = \sum_{i=1}^{k_0}p_i\delta_{\theta_i^0}\in \Ec_{k_0}(\Theta)$ with $\rho=\min_{1\leq i<j\leq k_0} \|\theta_i^0-\theta_j^0\|_2$. If $G =\sum_{i=1}^{k_0}p_i\delta_{\theta_i}  \in \Ec_{k_0}(\Theta)$ satisfying $D_1(G,G_0)< \frac{1}{2}\rho  $, then there exists a unique $\tau\in S_{k_0}$ such that for all real number $ r\geq 1$
        $$
        D_r(G,G_0) = \sum_{i=1}^{k_0} \left(\sqrt{r}\|\theta_{\tau(i)} -\theta_i^0\|_2 + |p_{\tau(i)}-p_i^0|  \right).
        $$
    \end{lem}	

\begin{lem}  \label{lem:exphellinger}
Consider a full rank exponential family's density function $f(x|\theta)$ with respect to a dominating measure $\mu$ on $\Xfrak$, which takes the form
	\begin{align*}
f(x|\thetaeta) = \exp\left(\thetaeta^\top  T(x)- A(\thetaeta)\right)h(x), 
\end{align*}
where $\Theta =\{\thetaeta|A(\thetaeta)<\infty\} \subset \R^s$ is the parameter space of $\thetaeta$.
	\begin{enumerate}[label=\alph*)]
		\item \label{item:exphellingera}
		For any $\thetaeta_0\in \Theta^\circ$
		$$\limsup_{\thetaeta\to\thetaeta_0} \frac{h(f(x|\thetaeta),f(x|\thetaeta_0))}{\|\thetaeta-\thetaeta_0\|_2} \leq \sqrt{\lambda_\text{max}(\nabla_\thetaeta^2A(\thetaeta_0))/8},$$
		where $\lambda_\text{max}(\cdot)$ is the maximum eigenvalue of a symmetric matrix.
		\item \label{item:exphellingerb}  
		For any compact subset $\Theta' \subset \Theta^\circ$, there exists $L_2>0$ such that 
		$$
		h(f(x|\thetaeta_1),f(x|\thetaeta_2))\leq L_2 \|\thetaeta_1-\thetaeta_2\|_2 \quad \forall \ \thetaeta_1,\thetaeta_2\in \operatorname{conv}(\Theta'), 
		$$
		where $\operatorname{conv}(\Theta')$ is the convex hull of $\Theta'$.
	\end{enumerate}
\end{lem}

\subsection{Calculation Details in Example \ref{exa:posnonexponential}}
\label{sec:detailsinposconratenonexp}
\textit{Details for the uniform probability kernel in Example \ref{exa:uniformcontinue}}.
Consider the uniform kernel in Example \ref{exa:uniform} and Example \ref{exa:uniformcontinue}. Write $G_0=\sum_{i=1  }^{k_0}p_i^0\delta_{\theta_i^0}$ with $\theta_1^0<\theta_2^0<\ldots<\theta_{k_0}^0$. Let $\Theta_1$ be a compact subset of $\Theta=(0,\infty)$ such that the condition \ref{item:prior} holds, and additionally satisfies $\max\Theta_1>\theta_{k_0}^0$. The reason for the additional condition will be discussed in the next paragraph. It is easy to establish that for any $\theta_1,\theta_2\in \Theta_1$
$$
h^2(f(x|\theta_1),f(x|\theta_2)) = 1- \sqrt{\frac{\min\{\theta_1,\theta_2\}}{\max\{\theta_1,\theta_2\}}} \leq 1-\frac{\min\{\theta_1,\theta_2\}}{\max\{\theta_1,\theta_2\}} \leq \frac{1}{\min\Theta_1}|\theta_2-\theta_1|,
$$
and thus \eqref{eqn:hellingerlipschitz} holds with $\beta_0=\frac{1}{2}$.

Additional care is needed for this example since the support of $f(x|\theta)$ depends on $\theta$ and $K(f(x|\theta_1),f(x|\theta_2))=\infty$ for $\theta_1>\theta_2$. In particular, the condition \eqref{eqn:KLlipschitz} does not hold for the uniform kernel. In view of that the condition \eqref{eqn:KLlipschitz} is only used to guarantee \eqref{eqn:piKlowdiflen}, we may directly verify \eqref{eqn:piKlowdiflen} for the uniform kernel so that the conclusions of Theorem \ref{thm:posconnotid} hold. Note that the additional condition $\max\Theta_1>\theta_{k_0}^0$ is necessary for \eqref{eqn:piKlowdiflen}, since $\max\Theta_1=\theta_{k_0}^0$ implies $\Pi\left( \frac{1}{\n}\sum_{i=1}^{\n}K(p_{G_0,\m_i}, p_{G,\m_i}) \leq \epsilon^2     \right)=0$. (Actually the condition $\max\Theta_1>\theta_{k_0}^0$ is necessary for a common condition called Kullback-Leibler property \cite[Definition 6.15]{ghosal2017fundamentals}.)  

 We now verify \eqref{eqn:piKlowdiflen}. Denote $\theta_{k_0+1}^0:=\max\Theta_1$ and
$\rho:=\frac{1}{2}\min_{1\leq i\leq k_0}(\theta_{i+1}^0-\theta_i^0)$. In what follows for this example we always write $G=\sum_{i=1}^{k_0}p_i\delta_{\theta_i}\in \Ec_{k_0}(\Theta_1)$ in its increasing representation w.r.t. $\theta$, i.e. $\theta_1<\theta_2<\ldots<\theta_{k_0}$. Consider the following set 
$$
A(G_0) := \left\{ \left.G=\sum_{i=1}^{k_0}p_i\delta_{\theta_i}\in \Ec_{k_0}(\Theta_1)\right| \theta_i\in [\theta_i^0,\ \theta_i^0+\rho], p_j\geq p_j^0, \ \forall i\in [k_0], j\geq 2   \right\}.
$$
For any $G = \sum_{i=1}^{k_0}p_i\delta_{\theta_i} \in A(G_0)$, let $Q$ be a coupling between $G_0$ and $G$ specified as below:
$$
Q:=\sum_{(\alpha,\beta)\in I} q_{\alpha\beta} \delta_{(\theta_\alpha^0,\theta_\beta)}, 
$$
where
\begin{align*}
\quad I:=&I_1\bigcup I_2, \quad 
I_1:=\bigcup_{i= 2}^{k_0} \{(i,i)\}, \quad I_2:= \bigcup_{\beta=1}^{k_0} \{(1,\beta)\}, \\
\quad q_{\alpha\beta} :=& \begin{cases} p_\beta^0, &  (\alpha,\beta)\in I_1\\  
p_\beta-p_\beta^0, &  (\alpha,\beta)\in I_2, \beta\geq 2 \\
p_1, & (\alpha,\beta)=(1,1)
\end{cases}.
\end{align*}
Then for any $N\geq 1$,
\begin{align*}
K(p_{G_0,N},p_{G,N})
= & 
K\left(\sum_{(\alpha,\beta)\in I} q_{\alpha\beta} \prod_{j=1}^N f(x_j|\theta_\alpha^0)  ,\sum_{(\alpha,\beta)\in I} q_{\alpha\beta} \prod_{j=1}^N f(x_j|\theta_\beta) \right)\\
\leq & 
\sum_{(\alpha,\beta)\in I} q_{\alpha\beta} K\left(  \prod_{j=1}^N f(x_j|\theta_\alpha^0)  ,  \prod_{j=1}^N f(x_j|\theta_\beta) \right) \\
= & 
\sum_{(\alpha,\beta)\in I} q_{\alpha\beta}NK\left(f(x_1|\theta_\alpha^0),  f(x_1|\theta_\beta) \right), 
\numberthis
\label{eqn:KLuppermixture} 
\end{align*}
where the first inequality follows by the joint convexity of the Kullback-Leibler divergence. For any $\theta_1\leq\theta_2\in \Theta_1 $,
\begin{equation}
K(f(x|\theta_1),f(x|\theta_2)) = \ln \left(\frac{\theta_2}{\theta_1}\right) \leq \frac{\theta_2-\theta_1}{\theta_1} \leq  \frac{\theta_2-\theta_1}{\min \Theta_1}. \label{eqn:KLdivergencelipuniform}
\end{equation}
By our choice of $I$, $\theta_\alpha^0\leq \theta_\beta$ for any $(\alpha,\beta)\in I$. Then plug \eqref{eqn:KLdivergencelipuniform} into \eqref{eqn:KLuppermixture},
\begin{align*}
K(p_{G_0,N},p_{G,N}) \leq & \frac{N}{\min \Theta_1} \sum_{(\alpha,\beta)\in I} q_{\alpha\beta} (\theta_\beta-\theta_\alpha^0) \\ 
\leq & \frac{N}{\min \Theta_1} \min\{1,diam(\Theta_1)\} \left( \sum_{\beta=1}^{k_0} \left(\theta_\beta-\theta_\beta^0\right) + \sum_{\beta=2}^{k_0} \left(p_\beta-p_\beta^0\right) \right). \numberthis \label{eqn:KLuppbou2}
\end{align*}
In fact, one can show $\sum_{(\alpha,\beta)\in I} q_{\alpha\beta} (\theta_\beta-\theta_\alpha^0)=W_1(G_0,G)$ but we do not have to use this fact here. Now by \eqref{eqn:KLuppbou2}, for any $G\in A(G_0)$,
$$
\frac{1}{\n}\sum_{i=1}^{\n}K(p_{G_0,\m_i}, p_{G,\m_i}) \leq C(\Theta_1) \bar{N}_m \left( \sum_{\beta=1}^{k_0} \left(\theta_\beta-\theta_\beta^0\right) + \sum_{\beta=2}^{k_0} \left(p_\beta-p_\beta^0\right) \right).
$$
Thus
\begin{align*}
	&\Pi\left( \frac{1}{\n}\sum_{i=1}^{\n}K(p_{G_0,\m_i}, p_{G,\m_i}) \leq \epsilon^2     \right) \\
	\geq & \Pi\left(A(G_0) \bigcap  \left\{C(\Theta_1) \bar{N}_m \left( \sum_{\beta=1}^{k_0} \left(\theta_\beta-\theta_\beta^0\right) + \sum_{\beta=2}^{k_0} \left(p_\beta-p_\beta^0\right) \right) \leq \epsilon^2 \right\} \right) \\
	\gtrsim &\left( \frac{\epsilon^2}{\bar{\m}_{\n} C(\Theta_1)}\right)^{k_0} \left( \frac{\epsilon^2}{\bar{\m}_{\n} C(\Theta_1)}\right)^{k_0-1},
	\end{align*}
which is \eqref{eqn:piKlowdiflen} for the example of $f(x|\theta)$ being uniform kernels.

As a result, the conclusions of Theorem \ref{thm:posconnotid} hold. Moreover by Example \ref{exa:uniformcontinue} $n_1(G_0)=1$ and one can directly verify that $n_0(G_0,\cup_{k\leq k_0}\Ec_k(\Theta_1))=1$.\\

\textit{Details for the location-scale exponential kernel in Example \ref{exa:exponentialcontinuerevision}}. By similar calculations as above, one may show that the conclusion of Theorem \ref{thm:posconnotid} holds even though the KL-divergence is not Lipschitz as in assumption \ref{item:kernel}. By Example \ref{exa:exponentialcontinuerevision} $n_1(G_0)=1$ and one can directly verify that $n_0(G_0,\cup_{k\leq k_0}\Ec_k(\Theta_1))=1$.

\textit{Details for the case that kernel is location-mixture Gaussian in Example \ref{exa:kernelmixture}}. It suffices to verify assumption \ref{item:kernel} such that Theorem \ref{thm:posconnotid} can be applied. 

Note that 
\begin{align*}
h(f(x|\theta_1),f(x|\theta_2))=& h\left(\sum_{i=1}^k\pi_{i1} f_{\Nc}(x|\mu_{i1},\sigma^2),\sum_{i=1}^k\pi_{i2} f_{\Nc}(x|\mu_{i2},\sigma^2)\right) \\
 \overset{(*)}{\leq} & 
 \max_{1\leq i \leq k} h\left(f_{\Nc}(x|\mu_{i1},\sigma^2), f_{\Nc}(x|\mu_{i2},\sigma^2)\right) + \sqrt{\frac{1}{2}\sum_{i=1}^{k}\left|\pi_{i1}-\pi_{i2}\right|} \\
 \overset{(**)}{\leq} &  \frac{1}{2\sqrt{2}\sigma}\max_{1\leq i \leq k} |\mu_{i1}-\mu_{i2}| + \sqrt{\frac{1}{2}\sum_{i=1}^{k}\left|\pi_{i1}-\pi_{i2}\right|}\\
 \leq & C(\sigma,k,\Theta_1) \sqrt{\|\theta_1-\theta_2\|_2},
\end{align*}
where step $(*)$ follows from Lemma \ref{lem:hellingeruppbou}, and step $(**)$ follows from the formula of Hellinger distance between two Gaussian distributions. We also have
\begin{align*}
K(f(x|\theta_1),f(x|\theta_2))=& K\left(\sum_{i=1}^k\pi_{i1} f_{\Nc}(x|\mu_{i1},\sigma^2),\sum_{j=1}^k\pi_{j2} f_{\Nc}(x|\mu_{j2},\sigma^2)\right) \\
\overset{(*)}{=}& \min_{q} K\left(\sum_{i,j=1}^k q_{ij} f_{\Nc}(x|\mu_{i1},\sigma^2),\sum_{i,j=1}^k q_{ij} f_{\Nc}(x|\mu_{j2},\sigma^2)\right) \\
 \overset{(**)}{\leq} & \min_{q} \sum_{i,j=1}^k q_{ij} K\left( f_{\Nc}(x|\mu_{i1},\sigma^2), f_{\Nc}(x|\mu_{j2},\sigma^2)\right)\\
 \overset{(***)}{=} &  \min_{q} \sum_{i,j=1}^k q_{ij} \frac{|\mu_{i1}-\mu_{j2}|^2}{2\sigma^2}  \\
 \leq & C(\sigma,k,\Theta_1) \|\theta_1-\theta_2\|_2,
\end{align*}
where in step $(*)$ $(q_{ij})_{i,j\in [k]}$ is any coupling between $(\pi_{i1})_{i\in [k]}$ and $(\pi_{j2})_{j\in [k]}$ and the minimization is taken among all such couplings, step $(**)$ follows from the joint convexity of KL-divergence, 
Lemma \ref{lem:hellingeruppbou}, and step $(***)$ follows from the formula of KL-divergence between two Gaussian distributions. 

Thus assumption \ref{item:kernel} is satisfied. Moreover, by Appendix \ref{sec:locationscalmixturegaussian} $n_1(G_0)<\infty$. Hence Theorem \ref{thm:posconnotid} holds. The calculations of $n_0(G_0,\cup_{k\leq k_0}\Ec_k(\Theta_1))$ and $n_1(G_0)$ are left as exercises to interested readers. 

\section{Proofs and calculation details for Section 7}	
\subsection{Proofs for Section \ref{sec:fourieranalysis}}
\begin{proof} [Proof of Lemma \ref{eqn:smallsigularvalueA(x)}]
Let $c(\cdot)$ be a positive constant that depends on its parameters in this proof.\\
\textbf{Claim 1}: There exists $c_6>0$ that depends only on $d,j_1,\ldots,j_d$ such that $S_{\text{min}}(A(x)) \geq c_6|x|^{-(j_d-j_1)(d-1)}$ for any $|x|>1$. Suppose this is not true, then there is $\{x_{\n}\}_{\n=1}^\infty$ such that $|x_{\n}| > 1$, and as $m\to\infty$,
	\begin{equation}
	|x_{\n}|^{(j_d-j_1)(d-1)}S_{\text{min}}(A(x_{\n})) \to 0. \label{eqn:xnsmato0} 
	\end{equation}
	Let $B(x,t)=|x|^t A(x)$ with $t$ being some positive number to be specified. The characteristic polynomial of $B(x,t)B^\top (x,t)$ is 
	$$\text{det}\left(\lambda I - B(x,t)B^\top (x,t)\right) = \lambda^d + \sum_{i=0}^{d-1}\gamma_i(x,t)\lambda^i.$$
	When $|x|>1$,  since $|A_{\alpha\beta}(x)|\leq c_4(d,j_1,\cdots,j_d) |x|^{j_d-j_1}$ for any $\alpha,\beta \in [d]$, the entries of $B(x,t)B^\top (x,t)$ are bounded by $d\left(c_4(d,j_1,\cdots,j_d) |x|^{(j_d-j_1+t)}\right)^2$. Thus $|\gamma_i(x,t)|\leq c_8(d,j_1,\cdots,j_d)\left( |x|^{(j_d-j_1+t)}\right)^{2(d-i)}$ for $1\leq i \leq d-1$. Moreover, 
	$$
	|\gamma_0(x,t)|=\left|x^{dt}\text{det}(A(x))\right|^2=\left(\prod_{i=1}^d j_i!\right)^2 |x|^{2dt} =c_5(d,j_1,\cdots,j_d)|x|^{2dt},
	$$ 
	with $c_5(d,j_1,\cdots,j_d) = (\prod_{i=1}^d j_i!)^2>0 $. Let $\lambda_{\text{min}}(x,t)\geq 0$ be the smallest eigenvalue of $B(x,t)B^\top (x,t)$. Then 
	$$
	\lambda_{\text{min}}^d(x,t) + \sum_{i=0}^{d-1}\gamma_i(x,t)\lambda_{\text{min}}^i(x,t) =0.
	$$    
	When $x\not = 0$, $\lambda_{\text{min}}(x,t) > 0$ since $\gamma_0(x,t)\not = 0$. Thus when $x\not = 0$,
	\begin{equation}
	\frac{1}{\lambda_{\text{min}}(x,t)} = -\frac{1}{\gamma_0(x,t)} \lambda_{\text{min}}^{d-1}(x,t) - \sum_{i=1}^{d-1}\frac{\gamma_i(x,t)}{\gamma_0(x,t)}\lambda_{\text{min}}^{i-1}(x,t). \label{eqn:lambdaminequ} 
	\end{equation} 
	Moreover, when $|x|>1$, $\left|\frac{\gamma_i(x,t)}{\gamma_0(x,t)}\right|\leq \frac{c_8(d,j_1,\cdots,j_d)}{c_5(d,j_1,\cdots,j_d)} \frac{|x|^{2(j_d-j_1)(d-i)}}{|x|^{2t i}} \leq \frac{c_8(d,j_1,\cdots,j_d)}{c_5(d,j_1,\cdots,j_d)} \frac{|x|^{2(j_d-j_1)(d-1)}}{|x|^{2t}} $ for any $1\leq i \leq d-1$. Then by \eqref{eqn:xnsmato0}, $\lambda_{\text{min}}(x_{\n},t_0) = \left(|x_{\n}|^{(j_d-j_1)(d-1)}S_{\text{min}}(A(x_{\n}))\right)^2 \to 0 $, where $t_0=(j_d-j_1)(d-1)$, so $\frac{1}{\lambda_{\text{min}}(x,t_0)} \to \infty$. On the other hand, since $|\frac{1}{\gamma_0(x_{\n},t_0)}|$ and $\frac{\gamma_i(x_{\n},t_0)}{\gamma_0(x_{\n},t_0)}$ are bounded and $\lambda_{\text{min}}(x_{\n},t_0)\to 0$,
	\begin{multline*}
	\limsup_{\n\to \infty}\left|-\frac{1}{\gamma_0(x_{\n},t_0)} \lambda_{\text{min}}^{d-1}(x_{\n},t_0) - \sum_{i=1}^{d-1}\frac{\gamma_i(x_{\n},t_0)}{\gamma_0(x_{\n},t_0)}\lambda_{\text{min}}^{i-1}(x_{\n},t_0)\right| \\ =\limsup_{\n\to \infty} \left|\frac{\gamma_1(x_{\n},t_0)}{\gamma_0(x_{\n},t_0)}\right| \leq  \frac{c_8(d,j_1,\cdots,j_d)}{c_5(d,j_1,\cdots,j_d)}.
	\end{multline*}
	These contradict with \eqref{eqn:lambdaminequ} and hence the claim at the beginning of this paragraph is established.
	
	 Since $S_\text{min}(A(x))>0$ on $|x|\leq 1$ and $S_\text{min}(A(x))$ is continuous, 
	 $$
	 \min_{x\in[-1,1]}S_\text{min}(A(x))\geq c_7>0.
	 $$
	 Then take $c_3=\min\{c_6,c_7\}$ and the proof is complete.
\end{proof}

\begin{proof}[Proofs of Lemma \ref{lem:radialupp}]
Let $\psi_w(x)=w^\top Tx$. Then 
	$$
	\left(\frac{d^{j_1} \psi_w(x)}{d x^{j_1}},\frac{d^{j_2} \psi_w(x)}{d x^{j_2}},\ldots,\frac{d^{j_d} \psi_w(x)}{d x^{j_d}}\right)^\top  = A(x)w
	$$
	where $A(x)\in \R^{d\times d}$ with entries $A_{\alpha\beta}(x)=0$ for $\alpha>\beta$ and $A_{\alpha\beta}(x)=\frac{j_\beta!}{(j_\beta-j_\alpha)! }x^{j_\beta-j_\alpha}$ for $\alpha \leq \beta$. Then for any $w \in S^{d-1}$, 
	\begin{equation}
	\max_{1\leq i \leq d}\left|\frac{d^{j_i} \psi_w(x)}{d x^{j_i}}\right|= \|A(x)w\|_\infty \geq \frac{1}{\sqrt{d}}\|A(x)w\|_2\geq \frac{1}{\sqrt{d}}S_{\text{min}}(A(x)) \geq \frac{1}{\sqrt{d}} c_3 \max\{1,|x|\}^{-\alpha_0} ,  \label{eqn:derlowbou}
	\end{equation}
	where $\alpha_0=(j_d-j_1)(d-1)$ and the last inequality follows from Lemma \ref{eqn:smallsigularvalueA(x)}.
		
 	\noindent
	\textbf{Case 1}: ($j_1>1$). Partition the real line according to the increasing sequence $\{a_t\}_{t=-\infty}^{\infty}$ where $$a_t= \begin{cases} 
 	2 a_{t+1} & t \leq -1 \\
 	\lfloor -c_1 \rfloor -1 & t = 0  \\ 
 	 b_t & 1\leq t \leq \ell \\
 	\lceil c_1 \rceil +1 & t = \ell+1 \\
 	2 a_{t-1} & t \geq \ell +2 \end{cases}.$$
 	For $t\leq -1$, by  \eqref{eqn:derlowbou} we know $\max\limits_{1\leq i \leq d}\left|\frac{d^{j_i} \psi_w(x)}{d x^{j_i}}\right| \geq \frac{1}{\sqrt{d}}c_3 |a_t|^{-\alpha_0} $ for all $x\in [a_t,a_{t+1}]$. In order to appeal to Lemma~\ref{lem:corput}, we need to specify the points $\{t_{\beta}\}_{\beta=0}^{\beta_0}$ with $t_0=a_t<t_1<\ldots<t_{\beta_0} = a_{t+1}$, where $\{t_{\beta}\}_{\beta=1}^{\beta_0-1}$ is defined as the set of roots in $(a_t,a_{t+1})$ of any of the following $d-1$ equations,
 	$$
 	\left|\frac{d^{j_i} \psi_w(x)}{d x^{j_i}}\right| = \frac{1}{\sqrt{d}}c_3 |a_t|^{-\alpha_0}, \quad i\in [d-1].
 	$$
 	Thus $\{t_{\beta}\}_{\beta=0}^{\beta_0}$ is a partition of $[a_t,a_{t+1}]$ such that for each $0\leq \beta \leq \beta_0-1$, $\left|\frac{d^{j_{k_\beta}} \psi_w(x)}{d x^{j_{k_\beta}}}\right| \geq \frac{1}{\sqrt{d}}c_3 |a_i|^{-\alpha_0}$ holds for some index $k_\beta \in [d]$ and for all $x\in [t_\beta,t_{\beta+1}]$. Since $\frac{d^{j_{m}} \psi_w(x)}{d x^{j_{m}}}$ is polynomial of degree $j_d - j_{m}$, it follows that $\beta_0 -1 \leq 2 \sum_{m=1}^{d}(j_d-j_m)$.
 	Let $\tilde{c}_0$ be the maximum of $\{\tilde{c}_{j_{m}}\}_{m=1}^d$, where $\tilde{c}_{j_{m}}$ are the coefficients $c_k$ corresponding to $k=j_{m}$ in Lemma \ref{lem:corput}. Then by Lemma \ref{lem:corput}, for $\lambda>1$
 	\begin{align*}	
 	&\left|\int_{[t_\beta,t_{\beta+1}]}e^{\i\lambda\psi_w(x)} f(x) dx \right| \\
 	 \leq & \tilde{c}_0 \left(\frac{c_3|a_t|^{-\alpha_0}\lambda}{\sqrt{d}}\right)^{-\frac{1}{j_{k_\beta}}}\left(|f(t_{\beta+1})|+\int_{[t_\beta,t_{\beta+1}]} |f'(x)|dx\right) \\
 	\leq & \tilde{c}_0 \max\left\{c_3^{-\frac{1}{j_{1}}},c_3^{-\frac{1}{j_{d}}}\right\} (\sqrt{d})^{\frac{1}{j_1}}  \lambda^{-\frac{1}{j_d}} (|a_t|^{\alpha_0})^{\frac{1}{j_1}}  \left(f(a_{t+1})+\int_{[t_\beta,t_{\beta+1}]} |f'(x)|dx\right), \numberthis \label{eqn:rescorputbbeta}
 	\end{align*} 
 	where the last step follows from $f(x)$ being increasing on $(-\infty,-c_1)$. Then for $\lambda>1$
 	\begin{align*}
 	&\left|\int_{[a_t,a_{t+1}]}e^{\i\lambda\psi_w(x)} f(x) dx \right|\\
 	 \leq & \sum_{\beta=0}^{\beta_0-1}\left|\int_{[t_\beta,t_{\beta+1}]}e^{\i\lambda\psi_w(x)} f(x) dx \right|\\
 	 \overset{(*)}{\leq} & \tilde{c}_0 \max\left\{c_3^{-\frac{1}{j_{1}}},c_3^{-\frac{1}{j_{d}}}\right\} (\sqrt{d})^{\frac{1}{j_1}}  \lambda^{-\frac{1}{j_d}} (|a_t|^{\alpha_0})^{\frac{1}{j_1}} \left(\beta_0f(a_{t+1}) + \int_{[a_t,a_{t+1}]} |f'(x)|dx \right)\\
 	 \overset{(**)}{\leq} & \tilde{c}_0 \left(c_3^{-\frac{1}{j_{1}}}+c_3^{-\frac{1}{j_{d}}}\right) d^{\frac{1}{2j_1}}  \lambda^{-\frac{1}{j_d}}  2^{\frac{\alpha_0}{j_1}}\left( \beta_1  |a_{t+1}|^{\frac{\alpha_0}{j_1}}f(a_{t+1}) + |a_{t+1}|^{\frac{\alpha_0}{j_1}} \int_{[a_t,a_{t+1}]}  |f'(x)|dx \right) \\ 
 	 \leq & C(d,j_1,\cdots,j_d) \lambda^{-\frac{1}{j_d}} \left(  |a_{t+1}|^{\frac{\alpha_0}{j_1}}f(a_{t+1}) + \int_{[a_t,a_{t+1}]} |x|^{\frac{\alpha_0}{j_1}} |f'(x)|dx \right), \numberthis \label{eqn:rescorputai}
 	\end{align*}
 	where the step $(*)$ follows from $\eqref{eqn:rescorputbbeta}$, the step $(**)$ follows from $a_{t}=2a_{t+1}$ and $\beta_0\leq \beta_1 := 2 \sum_{m=1}^d (j_d - j_m)+1$, and the last step follows from $\beta_1 \geq 1$, $|a_t| \geq |x| \geq |a_{t+1}|$ for all $x\in [a_t,a_{t+1}]$ and $C(d,j_1,\cdots,j_d) = \tilde{c}_0 \max\left\{c_3^{-\frac{1}{j_{1}}},c_3^{-\frac{1}{j_{d}}}\right\} (\sqrt{d})^{\frac{1}{j_1}}    2^{\frac{\alpha_0}{j_1}} \beta_1$.
 	
 	For $t\geq \ell+1$, following similar steps as the case $t\leq -1$, one obtain
 	\begin{align*}
 	&\left|\int_{[a_t,a_{t+1}]}e^{\i\lambda\phi_w(x)} f(x) dx \right| \\
 	\leq & C(d,j_1,\cdots,j_d) \lambda^{-\frac{1}{j_d}} \left(  |a_{t}|^{\frac{\alpha_0}{j_1}}f(a_{t}) + \int_{[a_t,a_{t+1}]} |x|^{\frac{\alpha_0}{j_1}} |f'(x)|dx \right), \numberthis \label{eqn:rescorputaip}
 	\end{align*}
 	where $C(d,j_1,\cdots,j_d)$ is the same as in \eqref{eqn:rescorputai}. 	
	
	For $0\leq t \leq \ell$, since $f' $ is continuous on $(a_t,a_{t+1})$ and $f'$ is Lebesgue integrable on $[a_t,a_{t+1}]$, $\lim_{x\to a_{t+1}^-}f(x)$ and $\lim_{x\to a_{t}^+}f(x)$ exist. Define $\tilde{f}(x)=f(x)\bm{1}_{(a_t,a_{t+1})}(x)+\bm{1}_{\{a_{t+1}\}}(x)\lim_{x\to a_{t+1}^-}f(x) + \bm{1}_{\{a_{t}\}}(x)\lim_{x\to a_{t}^+}f(x)$. Then $\tilde{f}(x)$ is absolute continuous on $[a_t,a_{t+1}]$. Moreover, by  \eqref{eqn:derlowbou} we know $\max\limits_{1\leq i \leq d}\left|\frac{d^{j_i} \psi_w(x)}{d x^{j_i}}\right| \geq \frac{1}{\sqrt{d}}c_3  (c_1+2)^{-\alpha_0} $ for all $x\in [a_t,a_{t+1}]$. Following the same argument as in the case $t\leq -1$, let $\{\tilde{t}_{\beta}\}_{\beta=0}^{\tilde{\beta}_0}$ with $\tilde{t}_0=a_t<\tilde{t}_1<\ldots<\tilde{t}_{\beta_0} = a_{t+1}$, where $\{\tilde{t}_{\beta}\}_{\beta=1}^{\tilde{\beta}_0-1}$
		is the set of roots in $(a_t,a_{t+1})$ of the following $d-1$ equations
	$$
	\left|\frac{d^{j_i} \psi_w(x)}{d x^{j_i}}\right| = \frac{1}{\sqrt{d}}c_3 (c_1+2) ^{-\alpha_0}, \quad i\in [d-1].
	$$
	Then $\{\tilde{t}_{\beta}\}_{\beta=0}^{\tilde{\beta}_0}$ is a partition of $[a_t,a_{t+1}]$ such that for each $0\leq \beta \leq \tilde{\beta}_0-1$, $\left|\frac{d^{j_{k_\beta}} \psi_w(x)}{d x^{j_{k_\beta}}}\right| \geq \frac{1}{\sqrt{d}}c_3 (c_1+2)^{-\alpha_0}$ for some $k_\beta \in [d]$ and for all $x\in [\tilde{t}_\beta,\tilde{t}_{\beta+1}]$. Since $\frac{d^{j_{m}} \psi_w(x)}{d x^{j_{m}}}$ are polynomial of degree $j_d - j_{m}$, we have $\tilde{\beta}_0 -1 \leq 2 \sum_{m=1}^d (j_d - j_m)$. Thus by Lemma \ref{lem:corput}, for any $\lambda>1$
	 	\begin{align*}	
	&\left|\int_{[\tilde{t}_\beta,\tilde{t}_{\beta+1}]}e^{\i\lambda\psi_w(x)} f(x) dx \right| \\
	= & \left|\int_{[\tilde{t}_\beta,\tilde{t}_{\beta+1}]}e^{\i\lambda\psi_w(x)} \tilde{f}(x) dx \right|\\
	\leq & \tilde{c}_0 \left(\frac{c_3(c_1+2) ^{-\alpha_0}\lambda}{\sqrt{d}}\right)^{-\frac{1}{j_{k_\beta}}}\left(|\tilde{f}(\tilde{t}_{\beta+1})|+\int_{[\tilde{t}_\beta,\tilde{t}_{\beta+1}]} |f'(x)|dx\right) \\
	\leq & \tilde{c}_0 \max\left\{c_3^{-\frac{1}{j_{1}}},c_3^{-\frac{1}{j_{d}}}\right\} (\sqrt{d})^{\frac{1}{j_1}}  \lambda^{-\frac{1}{j_d}} ((c_1+2) ^{\alpha_0})^{\frac{1}{j_1}}  \left(\|f\|_{L^\infty}+\int_{[\tilde{t}_\beta,\tilde{t}_{\beta+1}]} |f'(x)|dx\right), \numberthis \label{eqn:rescorputbbeta2}
	\end{align*} 
	where the last step follows from $|\tilde{f}(\tilde{t}_{\beta+1})| \leq \|f\|_{L^\infty} $. Then for any $\lambda>1$
	\begin{align*}
	 &\left|\int_{[a_t,a_{t+1}]}e^{\i\lambda\psi_w(x)} f(x) dx \right|\\
	  \leq & \sum_{\beta=0}^{\tilde{\beta}_0-1} \left|\int_{[\tilde{t}_\beta,\tilde{t}_{\beta+1}]}e^{\i\lambda\psi_w(x)} f(x) dx \right| \\
	 \leq & \tilde{c}_0 \max\left\{c_3^{-\frac{1}{j_{1}}},c_3^{-\frac{1}{j_{d}}}\right\} (\sqrt{d})^{\frac{1}{j_1}}  \lambda^{-\frac{1}{j_d}} ((c_1+2)^{\alpha_0})^{\frac{1}{j_1}}  \left( \tilde{\beta}_0\|f\|_{L^\infty}+\int_{[a_t,a_{t+1}]} |f'(x)|dx\right)\\
	  \leq & C(d,j_1,\ldots,j_d) \lambda^{-\frac{1}{j_d}}  (c_1+2)^{\frac{\alpha_0}{j_1}} \left(\|f\|_{L^\infty}+\int_{[a_t,a_{t+1}]} |f'(x)|dx\right), \numberthis \label{eqn:rescorput3}
	\end{align*}
	where $C(d,j_1,\cdots,j_d)$ is the same as in \eqref{eqn:rescorputai}. 
	
Hence,
	\begin{align*}
	&\left|\int_{\R}e^{\i\lambda\psi_w(x)} f(x) dx \right|  \\
	 =& \left|\sum_{t=-\infty}^\infty \int_{[a_t,a_{t+1}]}e^{\i\lambda\psi_w(x)} f(x) dx \right|\\ \leq& \sum_{t=-\infty}^{\infty}\left|\int_{[a_t,a_{t+1}]}e^{\i\lambda\psi_w(x)} f(x) dx \right|  \\
	 \overset{(*)}{\leq} & C(d,j_1,\ldots,j_d) \lambda^{-\frac{1}{j_d}} (c_1+2)^{\frac{\alpha_0}{j_1}}  \left( \sum_{t\leq -1}|a_{t+1}|^{\frac{\alpha_0}{j_1}}f(a_{t+1}) + \right.\\
	 &\left. \sum_{t\geq \ell+1}|a_{t}|^{\frac{\alpha_0}{j_1}}f(a_{t}) + (\ell+1)\|f\|_{L^\infty}+ \left\|\left(|x|^{\frac{\alpha_0}{j_1}}+1\right)f'(x)\right\|_{L^1}\right) \\
	 \overset{(**)}{\leq}& C(d,j_1,\ldots,j_d) \lambda^{-\frac{1}{j_d}} (c_1+2)^{\frac{\alpha_0}{j_1}} \left( \int_{(\infty,-c_1]}|x|^{\frac{\alpha_0}{j_1}}f(x)dx +\right. \\  &\left. \int_{[c_1,\infty)}|x|^{\frac{\alpha_0}{j_1}}f(x)dx+(\ell+1)\|f\|_{L^\infty} +\left\|\left(|x|^{\frac{\alpha_0}{j_1}}+1\right)f'(x)\right\|_{L^1}\right)\\
	\leq & C(d,j_1,\ldots,j_d) \lambda^{-\frac{1}{j_d}} (c_1+2)^{\frac{\alpha_0}{j_1}}\times \\
	&\left( \left\||x|^{\frac{\alpha_0}{j_1}}f(x)\right\|_{L^1} +(\ell+1)\|f\|_{L^\infty} +\left\|\left(|x|^{\frac{\alpha_0}{j_1}}+1\right)f'(x)\right\|_{L^1}\right) \numberthis \label{eqn:case1last}
	\end{align*}
	where the first equality follows from the dominated convergence theorem, the step $(*)$ follows from \eqref{eqn:rescorputai}, \eqref{eqn:rescorputaip}, \eqref{eqn:rescorput3}, and the step $(**)$ follows from the monotonicity of $|x|^{\frac{\alpha_0}{j_1}}f$ when $x<-c_1$, $x>c_1$. \\

\noindent	
\textbf{Case 2}: ($j_1=1$). Fix $\forall w\in S^{d-1}$, $\exists x_1<x_2<\ldots<x_s$ partition $\R$ into $s+1$ disjoint open intervals such that $ \frac{d \psi_w(x)}{dx} $ is monotone on each of those interval. Notice $s\leq j_d-2$ since $ \frac{d \psi_w(x)}{dx} $ is a polynomial of degree $j_d-1$, and $x_1,x_2,\ldots,x_s$ depend on $w$. For $t\leq -1$, on $[a_t, a_{t+1}]$ when we subdivide the interval, besides the partition points $\{t_\beta\}_{\beta=0}^{\beta_0}$, $\{x_1,x_2,\ldots,x_s\} \cap [a_t,a_{t+1}] $ should also be added into the partition points. The new partition points set has at most $\beta_0+1+s \leq \beta_1 + j_d$ points and hence subdivide $[a_t,a_{t+1}]$ into at most $\beta_1 + j_d-1$ intervals such that on each subinterval $\max\limits_{1\leq i \leq d}\left|\frac{d^{j_i} \psi_w(x)}{d x^{j_i}}\right| \geq \frac{1}{\sqrt{d}}c_3 |a_t|^{-\alpha_0} $ and $\frac{d \psi_w(x)}{d x}$ is monotone. Hence Lemma \ref{lem:corput} (part ii)) can be applied on each subinterval. The rest of steps proceed similarly as in Case 1, and one will obtain
	 \begin{align*}
	 &\left|\int_{[a_t,a_{t+1}]}e^{\i\lambda\psi_w(x)} f(x) dx \right|\\
	 \leq & \tilde{C}(d,j_1,\cdots,j_d) \lambda^{-\frac{1}{j_d}} \left(  |a_{t+1}|^{\frac{\alpha_0}{j_1}}f(a_{t+1}) + \int_{[a_t,a_{t+1}]} |x|^{\frac{\alpha_0}{j_1}} |f'(x)|dx \right), \numberthis \label{eqn:rescorputai21}
	 \end{align*}
	 where $\tilde{C}(d,j_1,\cdots,j_d) = \tilde{c}_0 \max\left\{c_3^{-\frac{1}{j_{1}}},c_3^{-\frac{1}{j_{d}}}\right\} (\sqrt{d})^{\frac{1}{j_1}}  2^{\frac{\alpha_0}{j_1}} (\beta_1 + j_d-1)$, a constant that depends only on $d,j_1,\ldots,j_d$. Following the same reasoning one can obtain \eqref{eqn:rescorputaip} for $t\geq \ell+1$ and \eqref{eqn:rescorput3} for $0\leq t \leq \ell$, both with $C(d,j_1,\cdots,j_d)$ replaced by $\tilde{C}(d,j_1,\cdots,j_d)$. As a result, one obtains \eqref{eqn:case1last} with $C(d,j_1,\cdots,j_d)$ replaced by $\tilde{C}(d,j_1,\cdots,j_d)$. 
\end{proof}
	
\begin{proof}[Proof of Lemma \ref{lem:charintegrability}]
By Lemma \ref{lem:radialupp}, when $\|\zeta\|_2 >1$,
	$$
	|g(\zeta)|^r \leq C(f,d,j_1,\ldots,j_d) \|\zeta\|_2^{-\frac{r}{j_d}}.
	$$
	where 
	\begin{multline*}
	    C(f,r,d,j_1,\ldots,j_d)= \\
	    C^r(d,j_1,\ldots,j_d) (c_1+2)^{\alpha_1 r}\left( \left\||x|^{\alpha_1}f(x)\right\|_{L^1} +(\ell+1)\|f\|_{L^\infty} +\left\|\left(|x|^{\alpha_1}+1\right)f'(x)\right\|_{L^1}\right)^r.
	 \end{multline*} 
	 Let $|S^{d-1}|$ denote the area of $S^{d-1}$. Then 
	\begin{align*}
	& \int_{\|\zeta\|_2 > 1} |g(\zeta)|^r d\zeta \\
	 \leq &  C(f,r,d,j_1,\ldots,j_d) \int_{\|\zeta\|_2 > 1}  \|\zeta\|_2^{-\frac{r}{j_d}} d\zeta\\ 
	 \leq & C(f,r,d,j_1,\ldots,j_d)|S^{d-1}| \int_{\left( 1, \infty\right)}  \lambda^{-\frac{r}{j_d}} \lambda^{d-1} d\lambda\\	
	 = & C(r,d,j_1,\ldots,j_d) (c_1+2)^{{\alpha_1 r}}\left( \left\||x|^{\alpha_1}f(x)\right\|_{L^1} +(\ell+1)\|f\|_{L^\infty} +\left\|\left(|x|^{\alpha_1}+1\right)f'(x)\right\|_{L^1}\right)^r, \numberthis \label{eqn:charuppintlarge}
	\end{align*}
	where the last inequality follows from that $\int_{\left( 1, \infty\right)}  \lambda^{-\frac{r}{j_d}} \lambda^{d-1} d\lambda$ is a finite constant that depends on $d$ and $j_d$ for $r> dj_d$ and $C(r,d,j_1,\ldots,j_d) = C^r(d,j_1,\ldots,j_d)|S^{d-1}|\int_{\left( 1, \infty\right)}  \lambda^{-\frac{r}{j_d}} \lambda^{d-1} d\lambda $.
	
    In addition,	
	\begin{equation}
	\int_{\|\zeta\|_2 \leq 1} |g(\zeta)|^r d\zeta \leq \int_{\|\zeta\|_2 \leq 1} \|f\|_{L^1}^r d\zeta = C(d)\|f\|_{L^1}^r, \label{eqn:charuppintsmall}
	\end{equation}
	where $C(d)$ is a constant that depends on $d$.
	
	The proof is then completed by combining \eqref{eqn:charuppintlarge} and \eqref{eqn:charuppintsmall} and $(a^r+b^r)\leq (a+b)^r$ for any $a,b>0,r\geq 1$.
\end{proof}	


	\subsection{Calculation details for Section \ref{sec:mixofgaussian}}
	\label{sec:detailsgaussian}
	In this subsection we verify parts of \ref{item:genthmg} for the $T$ specified in Section \ref{sec:mixofgaussian}. It is easy to verify by the dominated convergence theorem or Pratt's Lemma:
\begin{align*}
\frac{\partial h(\zeta|\mu,\sigma)}{\partial \mu}& = \int_{\R} \exp\left(\bm{i}\sum_{i=1}^k \zeta^{(i)}x^i\right)\frac{\partial f_{\Nc}(x|\mu,\sigma)}{\partial \mu} dx , \\ 
\frac{\partial^2 h(\zeta|\mu,\sigma)}{\partial \mu^2} & = \int_{\R} \exp\left(\bm{i}\sum_{i=1}^k\zeta^{(i)}x^i\right)\frac{\partial^2 f_{\Nc}(x|\mu,\sigma)}{\partial \mu^2} dx,\\
\frac{\partial h(\zeta|\mu,\sigma)}{\partial \zeta^{(j)}}& = \int_{\R} \bm{i}x^j\exp\left(\bm{i}\sum_{i=1}^k\zeta^{(i)}x^i\right) f_{\Nc}(x|\mu,\sigma) dx , \quad j\in [k]
\end{align*}
and 
$$
\frac{\partial^2 h(\zeta|\mu,\sigma)}{\partial \zeta^{(j)} \partial \mu}  = \int_{\R} \bm{i}x^j\exp\left(\bm{i}\sum_{i=1}^k\zeta^{(i)}x^i\right)\frac{\partial f_{\Nc}(x|\mu,\sigma)}{\partial \mu} dx, \quad j\in [k].
$$

Then 
\begin{align}
\left|\frac{\partial h(\zeta|\mu,\sigma)}{\partial \mu}\right| & \leq \int_{\R} \left|\frac{\partial f_{\Nc}(x|\mu,\sigma)}{\partial \mu}\right|dx = \sqrt{\frac{2}{\pi }}\frac{1}{\sigma}, \label{eqn:1stderupp} \\
\left|\frac{\partial^2 h(\zeta|\mu,\sigma)}{\partial \mu^2}\right| & \leq \int_\R \left|\frac{\partial^2 f_{\Nc}(x|\mu,\sigma)}{\partial \mu^2}\right|dx \leq \frac{2}{\sigma^2}, \label{eqn:2ndderupp} \\
\max_{j\in [k]}\left|\frac{\partial h(\zeta|\mu,\sigma)}{\partial \zeta^{(j)}}\right|& \leq \max_{j\in [k]} \int_{\R} \left|x^j f_{\Nc}(x|\mu,\sigma)\right|dx :=h_1(\mu), \label{eqn:1stzetader}  \\
\max_{j\in [k]}\left|\frac{\partial^2 h(\zeta|\mu,\sigma)}{\partial \zeta^{(j)} \partial \mu }\right| & \leq \max_{j\in [k]}\int_{\R} \left|x^j\frac{\partial f_{\Nc}(x|\mu,\sigma)}{\partial \mu}\right|dx :=h_2(\mu), \label{eqn:2ndzetamuder}
\end{align}
where $h_1(\mu)$ and $h_2(\mu)$ are continuous functions of $\mu$ by the dominated convergence theorem, with their dependence on the constant $\sigma$ suppressed.

 It follows that the gradient of $\phi_T(\zeta|\theta)$ with respect to $\theta$ is
\begin{align*}
\nabla_{\theta} \phi_T(\zeta|\theta) =&\left(h(\zeta|\mu_1,\sigma) - h(\zeta|\mu_k,\sigma),\ldots,\right.\\
& h(\zeta|\mu_{k-1},\sigma) - h(\zeta|\mu_k,\sigma), \pi_1 \frac{\partial h(\zeta|\mu_1,\sigma)}{\partial \mu},\ldots, \pi_k \frac{\partial h(\zeta|\mu_k,\sigma)}{\partial \mu})^\top  \numberthis \label{eqn:gradienttheta}
\end{align*}
and Hessian with respect to $\theta$ with $(i,j)$ entry for $j\geq i$ given by
\begin{equation}
\frac{\partial^2}{\partial \theta^{(j)}\partial \theta^{(i)}} \phi_T(\zeta|\theta)  =  
\begin{cases}  
\frac{\partial h(\zeta|\mu_i,\sigma)}{\partial \mu} & i\in [k-1], j= k-1 + i \\
-\frac{\partial h(\zeta|\mu_k,\sigma)}{\partial \mu} & i\in [k-1], j= 2k-1 \\
 \pi_{i-(k-1)}\frac{\partial^2 h(\zeta|\mu_{i-(k-1)},\sigma)}{\partial \mu^2}   & k\leq  i\leq 2k-1, j=i \\
 0 & \text{otherwise}
  \end{cases}
\label{eqn:hessiantheta}
\end{equation}
and the lower part is symmetric to the upper part.

Then by \eqref{eqn:1stderupp}, \eqref{eqn:2ndderupp}, \eqref{eqn:1stzetader}, \eqref{eqn:2ndzetamuder}, \eqref{eqn:gradienttheta} and \eqref{eqn:hessiantheta}, for any $i,j\in [k]$:
 \begin{align*}
  \left|\frac{\partial \phi_T(\zeta|\theta)}{\partial \theta^{(i)}}\right| & \leq 2 +\sqrt{\frac{2}{\pi}}\frac{1}{\sigma},\\ 
\left|\frac{\partial^2 \phi_T(\zeta|\theta) }{\partial \theta^{(i)} \partial \theta^{(j)}}\right|& \leq \sqrt{\frac{2}{\pi}}\frac{1}{\sigma}+\frac{2}{\sigma^2},  \\ \left|\frac{\partial^2 \phi_T(\zeta|\theta)}{ \partial \zeta^{(j)} \partial \theta^{(i)}}\right|& \leq \sum_{i=1}^k \left(h_1(\mu_i)+h_2(\mu_i)\right),
\end{align*}
where the right hand side of the last display is a continuous function of $\theta$ since $h_1$ and $h_2$ are continuous. Hence to verify the condition \ref{item:genthmg} it remains to establish that there exists some $r\geq 1$ such that
$\int_{\R^{2k-1}}\left| \phi_T(\zeta|\theta)\right|^{r}    d\zetave$ on $\Theta$ is upper bounded by a finite continuous function of $\theta$. 


	



\subsection{Calculation details for Section \ref{sec:dirichlet}}
\label{sec:detailsdirichlet}

In this subsection we verify parts of \ref{item:genthmg} for the $T$ specified in Section \ref{sec:dirichlet}. This subsection is similar to the Appendix \ref{sec:detailsgaussian}.

It is easy to verify by the dominated convergence theorem or Pratt's Lemma:
\begin{align*}
\frac{\partial h(\zeta|\alpha,\xi)}{\partial \alpha}& = \int_{\R} \exp\left(\bm{i}\sum_{i=1}^3\zeta^{(i)}z^{i+1}\right)\frac{\partial g(z|\alpha,\xi)}{\partial \alpha} dz , \\ 
\frac{\partial^2 h(\zeta|\alpha,\xi)}{\partial \alpha^2} & = \int_{\R} \exp\left(\bm{i}\sum_{i=1}^3\zeta^{(i)}z^{i+1}\right)\frac{\partial^2 g(z|\alpha,\xi)}{\partial \alpha^2} dz,\\
\frac{\partial h(\zeta|\alpha,\xi)}{\partial \zeta^{(j)}}& = \int_{\R} \bm{i}z^{j+1}\exp\left(\bm{i}\sum_{i=1}^3\zeta^{(i)}z^{i+1}\right) g(z|\alpha,\xi) dz, \quad j=1,2,3
\end{align*}
and 
$$
\frac{\partial^2 h(\zeta|\alpha,\xi)}{\partial \zeta^{(j)}  \partial \alpha} =\frac{\partial^2 h(\zeta|\alpha,\xi)}{  \partial \alpha \partial \zeta^{(j)}}  = \int_{\R} \bm{i}z^{j+1}\exp\left(\bm{i}\sum_{i=1}^3\zeta^{(i)}z^{i+1}\right)\frac{\partial g(z|\alpha,\xi)}{\partial \alpha} dz, \quad j=1,2,3.
$$

From the preceding four displays,
\begin{align}
\left|\frac{\partial h(\zeta|\alpha,\xi)}{\partial \alpha}\right| & \leq \int_{\R} \left|\frac{\partial g(z|\alpha,\xi)}{\partial \alpha}\right|dz :=h_1(\alpha) \label{eqn:1stderuppdirichlet} \\
\left|\frac{\partial^2 h(\zeta|\alpha,\xi)}{\partial \alpha^2}\right| & \leq \int_\R \left|\frac{\partial^2 g(z|\alpha,\xi)}{\partial \alpha^2}\right|dz:=h_2(\alpha), \label{eqn:2ndderuppdirichlet} \\
\max_{j=1,2,3}\left|\frac{\partial h(\zeta|\alpha,\xi)}{\partial \zeta^{(j)}}\right|& \leq \max_{j=1,2,3} \int_{\R} \left|z^{j+1} g(z|\alpha,\xi)\right|dz :=h_3(\alpha) , \label{eqn:1stzetaderdirichlet}  \\
\max_{j=1,2,3}\left|\frac{\partial^2 h(\zeta|\alpha,\xi)}{\partial \zeta^{(j)} \partial \alpha }\right| & \leq \max_{j=1,2,3}\int_{\R} \left|z^{j+1}\frac{\partial g(z|\alpha,\xi)}{\partial \alpha}\right|dz :=h_4(\alpha), \label{eqn:2ndzetamuderdirichlet}
\end{align}
where $h_1(\alpha)$, $h_2(\alpha)$, $h_3(\alpha)$ and $h_4(\alpha)$ are continuous functions of $\alpha$ by the dominated convergence theorem, with their dependence on the constant $\xi$ suppressed.

 It follows that the gradient of $\phi_T(\zeta|\theta)$ with respect to $\theta$ is
\begin{equation}
\nabla_{\theta} \phi_T(\zeta|\theta) = \left(h(\zeta|\alpha_1,\xi) - h(\zeta|\alpha_2,\xi), \pi_1 \frac{\partial h(\zeta|\alpha_1,\xi)}{\partial \alpha}, \pi_2 \frac{\partial h(\zeta|\alpha_2,\xi)}{\partial \alpha}\right)^\top , \label{eqn:gradientthetadirichlet}
\end{equation}
and Hessian with respect to $\theta$ is
\begin{equation}
\textbf{Hess}_{\theta} \phi_T(\zeta|\theta)  = \begin{pmatrix}
0 & \frac{\partial h(\zeta|\alpha_1,\xi)}{\partial \alpha} & -\frac{\partial h(\zeta|\alpha_2,\xi)}{\partial \alpha} \\
\frac{\partial h(\zeta|\alpha_1,\xi)}{\partial \alpha} & \pi_1 \frac{\partial^2 h(\zeta|\alpha_1,\xi)}{\partial \alpha^2} & 0 \\
-\frac{\partial h(\zeta|\alpha_2,\xi)}{\partial \alpha} & 0 & \pi_2 \frac{\partial^2 h(\zeta|\alpha_2,\xi)}{\partial \alpha^2}
\end{pmatrix}. \label{eqn:hessianthetadirichlet}
\end{equation}
Then by \eqref{eqn:1stderuppdirichlet}, \eqref{eqn:2ndderuppdirichlet}, \eqref{eqn:1stzetaderdirichlet}, \eqref{eqn:2ndzetamuderdirichlet}, \eqref{eqn:gradientthetadirichlet} and \eqref{eqn:hessianthetadirichlet}, for any $i,j\in \{1,2,3\}$:
 \begin{align*}
  \left|\frac{\partial \phi_T(\zeta|\theta)}{\partial \theta^{(i)}}\right| & \leq 2 + h_1(\alpha_1)+ h_1(\alpha_2),\\ 
\left|\frac{\partial^2 \phi_T(\zeta|\theta) }{\partial \theta^{(i)} \partial \theta^{(j)}}\right|& \leq  \sum_{i=1}^2\left(h_1(\alpha_i)+h_2(\alpha_i)\right),  \\ \left|\frac{\partial^2 \phi_T(\zeta|\theta)}{ \partial \zeta^{(j)} \partial \theta^{(i)}}\right|& \leq \sum_{i=1}^2\left(h_3(\alpha_i)+h_4(\alpha_i)\right),
\end{align*}
where the right hand side of the preceding $3$ displays are continuous functions of $\theta$ since $h_1$, $h_2$, $h_3$ and $h_4$ are continuous. 

\section{Proofs for Section \ref{sec:minimax}}

	\begin{proof}[Proof of Lemma \ref{lem:hellingeruppbou}]
		\textbf{Step 1}: Suppose $p'_i = p_i$ for any $i\in [k_0]$. In this case, 
		\begin{align*}
		h^2(\P_{G,\m},\P_{G',\m}) =& h^2\left(\sum_{i=1}^{k_0}p_i\P_{\theta_i,\m}, \sum_{i=1}^{k_0}p_i\P_{\theta'_i,\m}\right)\\
		\leq & \sum_{i=1}^{k_0}p_i h^2\left(\P_{\theta_i,\m}, \P_{\theta'_i,\m}\right)\\
		\leq & \m\sum_{i=1}^{k_0} p_i h^2\left(\P_{\theta_i}, \P_{\theta'_i}\right)\\
		\leq & \m \max_{1\leq i \leq k_0} h^2\left(\P_{\theta_i}, \P_{\theta'_i}\right),
		\end{align*}
		where the first inequality follows from the joint convexity of any $f$-divergences (of which squared Hellinger distance is a member), and the second inequality follows from 
		$$
		h^2\left(\P_{\theta_i,\m}, \P_{\theta'_i,\m}\right) = 1-\left(1-h^2\left(\P_{\theta_i}, \P_{\theta'_i}\right) \right)^{\m} \leq \m h^2\left(\P_{\theta_i}, \P_{\theta'_i}\right).
		$$
		
		\noindent \textbf{Step 2}: Suppose $\theta'_i = \theta_i$ for any $i\in [k_0]$. Let $\bm{p}=(p_1,p_2,\ldots,p_{k_0})$ be the discrete probability distribution associated to the weights of $G$ and define $\bm{p}'$ similarly. Consider any $Q=(q_{ij})_{i,j=1}^{k_0}$ to be a coupling of $\bm{p}$ and $\bm{p'}$. 
		Then
		\begin{align*}
		h^2(\P_{G,\m},\P_{G',\m}) = &  h^2\left(\sum_{i=1}^{k_0}\sum_{j=1}^{k_0} q_{ij}\P_{\theta_i,\m},\sum_{i=1}^{k_0}\sum_{j=1}^{k_0} q_{ij} \P_{\theta_j,\m}\right)\\
		\leq & \sum_{i=1}^{k_0}\sum_{j=1}^{k_0} q_{ij} h^2\left(\P_{\theta_i,\m}, \P_{\theta_j,\m}\right)\\
		\leq & \sum_{i=1}^{k_0}\sum_{j=1}^{k_0} q_{ij} \bm{1}(\theta_i \not = \theta_j), \numberthis \label{eqn:hupperbouq}
		\end{align*}
		where the first inequality follows from the joint convexity of any $f$-divergence, and the second inequality follow from the Hellinger distance is upper bounded by $1$. Since \eqref{eqn:hupperbouq} holds for any coupling $Q$ of $\bm{p}$ and $\bm{p}'$, 
		$$
		h^2(\P_{G,\m},\P_{G',\m}) \leq \inf_{Q}\sum_{i=1}^{k_0}\sum_{j=1}^{k_0} q_{ij} \bm{1}(\theta_i \not = \theta_j) = V(\bm{p},\bm{p}') = \frac{1}{2}\sum_{i=1}^{k_0}|p_i-p_i'|.
		$$
		\textbf{Step 3}: General case. Let $G''= \sum_{i=1}^{k_0}{p_i}\delta_{\theta'_i}$. Then by triangular inequality, Step 1 and Step 2,
		\begin{align*}
		h(\P_{G,\m},\P_{G',\m}) \leq &h(\P_{G,\m},\P_{G'',\m}) + h(\P_{G'',\m},\P_{G',\m}) \\
		\leq &\sqrt{\m} \max_{1\leq i \leq k_0} h\left(\P_{\theta_i}, \P_{\theta'_i}\right) + \sqrt{\frac{1}{2}\sum_{i=1}^{k_0}|p_i-p_i'|}.
		\end{align*}
		Finally, notice that the above procedure does not depend on the specific order of atoms of $G$ and $G'$, and thus the proof is complete.
	\end{proof}	
	
	\begin{proof}[Proof of Lemma \ref{lem:optimalsquaretootN}]
		Since $\liminf\limits_{\theta \to\theta^0_j }\frac{h(\P_{\theta},\P_{\theta^0_j})}{\|{\theta}-{\theta^0_j}\|_2}<\infty$, there exists a sequences $\{\theta_j^k\}_{k=1}^\infty\subset \Theta\backslash \cup_{i=1}^{k_0}\{\theta_i^0\}$ such that $\theta_j^k \to \theta_j^0 $ and 
	\begin{equation}
		h(\P_{\theta_j^k},\P_{\theta^0_j}) \leq \gamma \|{\theta_j^k}-{\theta^0_j}\|_2
		\label{eqn:hthetaratio}
		\end{equation}
		 for some $\gamma\in (0,\infty)$.
		Supposing that 
		\begin{equation*}
		\limsup_{\m\to \infty}\liminf_{\substack{G\overset{W_1}{\to} G_0\\ G\in \Ecal_{k_0}(\Theta) }} \frac{h(\P_{G,\m},\P_{G_0,\m})}{\myD_{\psi(\m)}(G,G_0)} = \beta \in (0,\infty], 
		\end{equation*}
		then there exists subsequences $\m_{\ell} \rightarrow \infty$ such that for any $\ell$
		$$ \liminf_{\substack{G\overset{W_1}{\to} G_0\\ G\in \Ecal_{k_0}(\Theta) }} \frac{h(\P_{G,\m_{\ell}},\P_{G_0,\m_{\ell}})}{\myD_{\psi(\m_{\ell})}(G,G_0)} \geq \frac{3}{4}\beta.
		$$		
		Thus for each $\ell$, there exists $\theta_j^{k_\ell}$ such that
		 $G_\ell  =  p_j^0\delta_{\theta_j^{k_\ell} } + \sum\limits_{i =1,i\not=j}^{k_0}{p^0_i} \delta_{\theta^0_i}  \in \Ec_{k_0}(\Theta)\backslash\{G_0\} $,  and
		$$
		\frac{h(\P_{G_\ell,\m_{\ell}},\P_{G_0,\m_{\ell}})}{\myD_{\psi(\m_{\ell})}(G_\ell,G_0)} \geq \frac{\beta}{2}. 
		$$
		By our choice of $G_\ell$, for sufficiently large $\ell$
		$$
		h(\P_{G_\ell,\m_{\ell}},\P_{G_0,\m_{\ell}}) \geq \frac{\beta}{2}\myD_{\psi(\m_{\ell})}(G_\ell,G_0) = \frac{\beta}{2} \sqrt{\psi(\m_{\ell})}  \|\theta_j^{k_\ell} -\theta_j^0\|_2. 
		$$
		On the other hand, by Lemma \ref{lem:hellingeruppbou},
		$$
		 h(\P_{G_\ell,\m_{\ell}},\P_{G_0,\m_{\ell}}) \leq \sqrt{\m_{\ell}}  h(\P_{\theta^{k_\ell}_j},\P_{\theta^0_j}).
		$$
		Combining the last two displays, 
		$$
		\frac{\beta}{2} \leq  \sqrt{\frac{\m_{\ell}}{\psi(\m_{\ell})}} \frac{h(\P_{\theta^{k_\ell}_j},\P_{\theta^0_j})}{\|{\theta^{k_\ell}_j}-{\theta^0_j}\|_2}\leq \gamma  \sqrt{\frac{\m_{\ell}}{\psi(\m_{\ell})}}  \to 0, \quad \text{as  } \ell\to\infty,
		$$
		where the second inequality follows from \eqref{eqn:hthetaratio}. The last display contradicts with $\beta>0$.
	\end{proof}

	\begin{proof}[Proof of Theorem \ref{thm:minimaxlower}]
a) 
	Choose a set of distinct $k_0-1$ points $\{\theta_i\}_{i=1}^{k_0-1}\subset \Theta\backslash \{\theta_0\}$ satisfying 
	$$\rho_1:=\min_{0\leq i<j\leq k_0-1}h(\P_{\theta_i},\P_{\theta_j})>0.$$ 
	Let $\rho := \min_{0\leq i<j \leq k_0-1}\|\theta_i-\theta_j\|_2$. Since $\limsup\limits_{\theta \to\theta_0 }\frac{h\left(\P_{\theta},\P_{\theta_0}\right)}{\|{\theta}-{\theta_0}\|_2^{\beta_0}}<\infty$, there exist $\gamma\in(0,\infty)$ and $r_0\in(0,\min\{\rho,(\rho_1/\gamma)^{1/\beta_0}\})$ such that 
		\begin{equation}
		\frac{h\left(\P_{\theta},\P_{\theta_0}\right)}{\|{\theta}-{\theta_0}\|_2^{\beta_0}} < \gamma,\quad \forall 0<\|\theta-\theta_0\|_2<r_0. \label{eqn:h2normratioupp}
		\end{equation}
		  Consider $G_1= \sum_{i=1}^{k_0} \frac{1}{k_0} \delta_{\theta_i^1} \in \Ec_{k_0}(\Theta) $ and $G_2= \sum_{i=1}^{k_0} \frac{1}{k_0} \delta_{\theta_i^2} \in \Ec_{k_0}(\Theta)$ with $\theta_i^1=\theta_i^2=\theta_i \in \Theta\backslash\{\theta_0\}$ for $ i\in [k_0-1]$ and $\theta_{k_0}^1=\theta_0$, $\theta_{k_0}^2=\theta$ satisfying $\|\theta-\theta_0\|_2 = 2\epsilon<r_0$. Here $\epsilon\in (0,r_0/2) $ is a constant to be determined. Then $d_{\bm{\Theta}}(G_1,G_2) =2\epsilon$. Moreover, $h(\P_\theta,\P_{\theta_0})\leq \gamma\left(2\epsilon\right)^{\beta_0}<\rho_1$.

		  By two-point Le Cam bound (see \cite[(15.14)]{wainwright2019high})
		\begin{equation}
		\inf_{\hat{G}\in \Ec_{k_0}(\Theta)}\sup_{G\in \Ec_{k_0}(\Theta)}\E_{\bigotimes^\n\P_{G,\m}} d_{\bm{\Theta}}(G,\hat{G})\geq \frac{\epsilon}{2} \left(1- V\left(\bigotimes^\n\P_{G_1,\m}, \bigotimes^\n\P_{G_2,\m}\right) \right).  \label{eqn: minimaxtheta}
		\end{equation}
		Notice 
		$$
		V\left(\bigotimes^\n\P_{G_1,\m}, \bigotimes^\n\P_{G_2,\m}\right)  \leq h\left(\bigotimes^\n\P_{G_1,\m}, \bigotimes^\n\P_{G_2,\m}\right)  \leq \sqrt{\n} h\left(\P_{G_1,\m}, \P_{G_2,\m}\right). 
		$$
		With our choice of $G_1$ and $G_2$, by Lemma \ref{lem:hellingeruppbou}, the last display becomes
		\begin{align*}
		 V\left(\bigotimes^\n\P_{G_1,\m}, \bigotimes^\n\P_{G_2,\m}\right) 
		\leq & \sqrt{\n} \sqrt{\m} \min_{\tau\in S_{k_0}} \max_{1\leq i\leq k_0}   h\left(\P_{\theta^1_i}, \P_{\theta^2_{\tau(i)}}\right) \\
		= & \sqrt{\n} \sqrt{\m} h\left(\P_{\theta_0}, \P_{\theta}\right)\\
		\leq & \sqrt{\n} \sqrt{\m} \gamma \left(2\epsilon\right)^{\beta_0}, \numberthis \label{eqn:minimaxVuppbou}
		\end{align*}
		where the equality follows from $$
		 \min_{\tau\in S_{k_0}} \max_{1\leq i\leq k_0}   h(\P_{\theta^1_i}, \P_{\theta^2_{\tau(i)}})=h(\P_{\theta^1_{k_0}}, \P_{\theta^2_{k_0}}) =  h\left(\P_{\theta_0}, \P_{\theta}\right)
		$$
		due to $h\left(\P_{\theta_0}, \P_{\theta}\right)<\rho_1$. 
		Plug \eqref{eqn:minimaxVuppbou} into \eqref{eqn: minimaxtheta},
		\begin{equation}
		\inf_{\hat{G}\in \Ec_{k_0}(\Theta)}\sup_{G\in \Ec_{k_0}(\Theta)}\E_{\bigotimes^\n\P_{G,\m}} d_{\bm{\Theta}}(G,\hat{G}) \geq \frac{\epsilon}{2} \left(1-\gamma \sqrt{\n} \sqrt{\m} (2\epsilon)^{\beta_0}\right). \label{eqn:minimaxlowboueps}
		\end{equation}
		Consider any $a\in(0,1)$ satisfying $a>1-\gamma r_0^{\beta_0}$ and let $2\epsilon= \left(\frac{1-a}{\gamma\sqrt{\n}\sqrt{\m}}\right)^{\frac{1}{\beta_0}} $. Then $2\epsilon \in (0,r_0)$.  Plug the specified $\epsilon$ into \eqref{eqn:minimaxlowboueps}, then the right hand side in the above display becomes
		$$
	 \frac{a}{4}\left(\frac{1-a}{\gamma\sqrt{\n}\sqrt{\m}}\right)^{\frac{1}{\beta_0}} = C(\beta_0) \left(\frac{1}{\sqrt{\n}\sqrt{\m}}\right)^{\frac{1}{\beta_0}},
		$$
	   where $C(\beta_0)$ depends on $\beta_0$. Notice $a,\gamma,r_0$ are  constants that depends on the probability family $\{\P_\theta\}_{\theta\in \Theta}$ and $k_0$.
		
b) 
		Consider $k_0>3$. Let $0<\epsilon<(\frac{1}{3}-\frac{1}{3(k_0-2)})/2$ . Consider $G_1= \sum\limits_{i=1}^{2} \frac{1}{3} \delta_{\theta_i} + \sum\limits_{i=3}^{k_0} \frac{1}{3(k_0-2)} \delta_{\theta_i} \in \Ec_{k_0}(\Theta)$ and $G_2=  (\frac{1}{3}-\epsilon)\delta_{\theta_1} + (\frac{1}{3}+\epsilon)\delta_{\theta_2}  + \sum\limits_{i=3}^{k_0} \frac{1}{3(k_0-2)} \delta_{\theta_i} \in \Ec_{k_0}(\Theta)$. By the range of $\epsilon$, $G_2\in \Ec_{k_0}(\Theta)$ and
		 $d_{\bm{p}}(G_1,G_2) =2\epsilon$. Similar to the proof of \ref{item:minimaxlowera}, 
		$$
		\inf_{\hat{G}\in \Ec_{k_0}(\Theta)}\sup_{G\in \Ec_{k_0}(\Theta)}\E_{\bigotimes^\n\P_{G,\m}} d_{\bm{p}}(\hat{G},G) \geq \frac{\epsilon}{2} \left(1- \sqrt{\n} h\left(\P_{G_1,\m}, \P_{G_2,\m}\right) \right).
		$$
		With our choice of $G_1$ and $G_2$, by Lemma \ref{lem:hellingeruppbou}, 
		$$
		 h\left(\P_{G_1,\m}, \P_{G_2,\m}\right) \leq \sqrt{\frac{1}{2}\times 2\epsilon}=\sqrt{\epsilon}.
		$$
		Combining the last two displays,
		$$\inf_{\hat{G}\in \Ec_{k_0}(\Theta)}\sup_{G\in \Ec_{k_0}(\Theta)}\E_{\bigotimes^\n\P_{G,\m}} d_{\bm{p}}(\hat{G},G)\geq \frac{\epsilon}{2} \left(1- \sqrt{\n} \sqrt{\epsilon} \right).$$
		The proof is complete by specifying $\epsilon =  \frac{1}{\n} (\frac{1}{3}-\frac{1}{3(k_0-2)})/4 <(\frac{1}{3}-\frac{1}{3(k_0-2)})/2$. The case for $k_0=2$ and $k_0=3$ follow similarly. 
		
c) 
		The conclusion follows immediately from \ref{item:minimaxlowera},  \ref{item:minimaxlowerb} and ~\eqref{eqn:d1dthetadp}.
	\end{proof} 

\section{Proofs of Auxiliary Lemmas}

\subsection{Proofs for Section \ref{sec:auxiliarylemmas4}}

	\begin{proof}[Proof for Lemma \ref{lem:taylortwisted}]
a) 
	    \begin{align*}
	    &\lim_{x\neq y, x\to x_0,y\to x_0} \frac{|g(x)-g(y)-\langle\nabla g(x_0),x-y\rangle |}{\|x-y\|_2}\\
	    =& \lim_{x\neq y, x\to x_0,y\to x_0} \frac{|\langle\nabla g(\xi),x-y\rangle-\langle\nabla g(x_0),x-y\rangle |}{\|x-y\|_2} \\
	    \leq & \lim_{x\neq y, x\to x_0,y\to x_0} \|\nabla g(\xi)-\nabla g(x_0) \|_2\\
	    =&0,
	    \end{align*}
	    where the first step follows from mean value theorem with $\xi$ lie on the line segment connecting $x$ and $y$, the second step follows from Cauchy-Schwarz inequality, and the last step follows from the continuity of $\nabla g(x)$ at $x_0$ and $\xi\to x_0$ when $x,y\to x_0$. 
 
b) 
	For $x\not = y$ in $B$ specified in the statement,
	    \begin{align*}
	        &\frac{|g(x)-g(y)-\langle\nabla g(x_0),x-y\rangle |}{\|x-y\|_2} \\
	        = & \frac{|\int_{0}^1\langle\nabla g(y+t(x-y)),x-y\rangle dt -\langle\nabla g(x_0),x-y\rangle |}{\|x-y\|_2} \\
	        =& \frac{|\int_{0}^1\int_0^1\langle \langle  \nabla^2 g(x_0+s(y+t(x-y)-x_0)) ,y+t(x-y)-x_0\rangle,x-y\rangle ds dt  |}{\|x-y\|_2} \\
	        \leq &\frac{\int_{0}^1\int_0^1|\langle \langle  \nabla^2 g(x_0+s(y+t(x-y)-x_0)) ,y+t(x-y)-x_0\rangle,x-y\rangle   |ds dt }{\|x-y\|_2} \\
	        \leq & \int_{0}^1\int_0^1  \|\nabla^2 g(x_0+s(y+t(x-y)-x_0))\|_2 \|y+t(x-y)-x_0\|_2  ds dt \\
	         \leq & \int_{0}^1\int_0^1  \|\nabla^2 g(x_0+s(y+t(x-y)-x_0))\|_2 ds dt\ \max\{\|x-x_0\|_2, \|y-x_0\|_2 \} , \numberthis \label{eqn:twisttaylor1}
	    \end{align*}
	    where the first two equalities follow respectively form fundamental theorem of calculus for $\R$-valued functions and $\R^d$-valued functions. Observe that for any matrix $A\in \R^{d\times d}$, 
	    $$
	    \|A\|_2\leq \|A\|_F \leq d \max_{1\leq i,j\leq d}|A_{ij}|\leq d\sum_{1\leq i,j\leq d}|A_{ij}|
	    $$
	   where $\|\cdot\|_F$ is the Frobenius norm. Applying the preceding display to \eqref{eqn:twisttaylor1},
	   \begin{align*}
	   &\int_{0}^1\int_0^1  \|\nabla^2 g(x_0+s(y+t(x-y)-x_0))\|_2 ds dt \\
	   \leq & d\sum_{1\leq i,j\leq d} \int_{0}^1\int_0^1 \left|\frac{\partial^2 g}{\partial x^{(i)}x^{(j)}}(x_0+s(y+t(x-y)-x_0))\right|  ds dt
	   \end{align*}
	
	Following \eqref{eqn:twisttaylor1},
	$$
	\frac{|g(x)-g(y)-\langle\nabla g(x_0),x-y\rangle |}{\|x-y\|_2} \leq L \max\{\|x-x_0\|_2, \|y-x_0\|_2 \}.
	$$
	\end{proof}

\begin{proof}[Proof of Lemma \ref{lem:polexplinind}] 
a) 
    Define $F(x)=\sum_{i=1}^k h_i(x)e^{b_i x}$.  From the condition $F(x)=0$ on a dense subset of $I$. Then $F(x)=0$ on the closure of that subset, which contains $I$, since it is a continuous function on $\R$. Let $a\in I^\circ$ and consider its Taylor expansion $F(x)=\sum_{i=0}^\infty \frac{F^{(i)}(a)}{i!}(x-a)^i$ for any $x\in \R$. It follows from $F(x)=0$ on $I$ that $F^{(i)}(a)=0$ for any $i\geq 0$. Thus $F(x)\equiv 0$ on $\R$. Then $$0=\lim_{x\to\infty}e^{-b_k x}F(x) = \lim_{x\to\infty}h_k(x).$$ 
    This happen only when $h_k(x)\equiv 0$. Proceed in the same manner to show $h_i(x)\equiv 0$ for $i$ from $k-1$ to $1$.
    
b) 
    Define $H(x)=\sum_{i=1}^k (h_i(x)+g_i(x)\ln(x)) e^{b_i x}$. From the condition $H(x)=0$ on a dense subset of $I$. Then $H(x)=0$ on the closure of that subset excluding $0$, which contains $I$, since it is a continuous function on $(0,\infty)$. Let $a_1\in I^\circ$ and consider its Taylor expansion at $a_1$: $H(x)=\sum_{i=0}^\infty \frac{H^{(i)}(a_1)}{i!}(x-a_1)^i$ for $x\in (0,2a_1)$, since the Taylor series of $\ln(x)$, $x^{\gamma}$ at $a_1$ converges respectively to $\ln(x)$, $x^{\gamma}$ on $(0,2a_1)$ for any $\gamma$. It follows from $H(x)=0$ on $I$  that $H^{(i)}(a_1)=0$ for any $i\geq 0$.
    Thus $H(x)=0$ on $(0,2a_1)$. Now take $a_2=\frac{3}{2}a_1$ and repeat the above analysis with $a_1$ replaced by $a_2$, resulting in $H(x)=0$ on $(0,2a_2)=(0,3a_1)$. Then take $a_3=\frac{3}{2}a_2$ and keep repeating the process, and one obtains $H(x)=0$ on $(0,\infty)$ since $a_1>0$. Let $\gamma_0$ be the smallest power of all power functions that appear in $\{g_i(x)\}_{i=1}^k$, $\{h_i(x)\}_{i=1}^k$, and define $\tilde{H}(x)=x^{-\gamma_0}H(x)$. Then $\tilde{H}(x)=0$ on $(0,\infty)$. Then 
    $$
    0=\lim_{x\to\infty}e^{-b_k x}\tilde{H}(x) = \lim_{x\to\infty}(x^{-\gamma_0}h_k(x)+x^{-\gamma_0}g_k(x)\ln(x)),
    $$
    which happens only when $x^{-\gamma_0}h_k(x)\equiv 0$ and $x^{-\gamma_0}g_k(x)\equiv 0$. That is, when $x\neq 0$, $h_k(x)\equiv 0$ and $g_k(x)\equiv 0$. Proceed in the same manner to show when $x\neq 0$, $h_i(x)\equiv 0$ and $g_i(x)\equiv 0$ for $i$ from $k-1$ to $1$.
\end{proof}

\begin{proof}[Proof of Lemma \ref{lem:expdct}]
Let $\gamma>0$ be such that the line segment between $\theta-a\gamma$ and $\theta+a\gamma$ lie in $\Theta$ and $\int_\Xfrak e^{4\gamma^\top T(x)} f(x|\theta) d\mu<\infty$, $\int_\Xfrak e^{-4\gamma^\top T(x)}f(x|\theta) d\mu<\infty$ due to the fact that the moment generating function exists in a neighborhood of origin for any given $\theta\in \Theta^\circ$. Then for $\Delta\in (0,\gamma]$ and for any $x\in S$
 \begin{align*}
    &\left|\frac{f(x|\thetaeta+a\Delta)-f(x|\thetaeta)}{\Delta  \sqrt{f(x|\thetaeta)}} \right| \\
    = & \sqrt{f(x|\thetaeta)} \left| \frac{\exp(\langle  a \Delta , T(x)\rangle - (A(\thetaeta+a\Delta)-A(\thetaeta)) )-1}{\Delta}  \right| \\
    \overset{(*)}{\leq} & \sqrt{f(x|\thetaeta)} \left| \langle  a , T(x)\rangle- \frac{  A(\thetaeta+a\Delta)-A(\thetaeta)}{\Delta}  \right| e^{\langle  a \Delta , T(x)\rangle - (A(\thetaeta+a\Delta)-A(\thetaeta)) }\\
    \leq & \sqrt{f(x|\thetaeta)} \left( \left| \langle  a , T(x)\rangle \right|+ \|a\|_2 \max_{\Delta\in[0,\gamma]}\|\nabla_\thetaeta A(\thetaeta+a\Delta)\|_2  \right)\times \\
    &\quad e^{ \Delta | \langle  a  , T(x)\rangle|} 
    \max_{\Delta\in[0,\gamma]}e^{- (A(\thetaeta+a\Delta)-A(\thetaeta)) } \\
    \leq & \sqrt{f(x|\thetaeta)} \frac{1}{\gamma}e^{ \gamma\left| \langle  a , T(x)\rangle \right|+ \gamma\|a\|_2 \max_{\Delta\in[0,\gamma]}\|\nabla_\thetaeta A(\thetaeta+a\Delta)\|_2  } \ e^{ \gamma | \langle  a  , T(x)\rangle|} 
    \max_{\Delta\in[0,\gamma]}e^{- (A(\thetaeta+a\Delta)-A(\thetaeta)) }    \\
    = & C(\gamma,a,\thetaeta)\sqrt{f(x|\thetaeta)} e^{ 2\gamma | \langle  a  , T(x)\rangle|} \\
    \leq & \sqrt{ C^2(\gamma,a,\thetaeta)f(x|\thetaeta)  \left(e^{ 4\gamma  \langle  a  , T(x)\rangle} + e^{ -4\gamma  \langle  a  , T(x)\rangle}\right)}, \numberthis \label{eqn:hdctexp}
    \end{align*}
    where step $(*)$ follows from $|e^t-1|\leq |t|e^t$. Then the the first conclusion holds with $$\bar{f}=\sqrt{C^2(\gamma,a,\thetaeta)f(x|\thetaeta)  \left(e^{ 4\gamma  \langle  a  , T(x)\rangle} + e^{ -4\gamma  \langle  a  , T(x)\rangle}\right)}.$$
    Take $\tilde{f}(x)=\bar{f}(x)\sqrt{f(x|\theta)}$ and by Cauchy–Schwarz inequality $\int_\Xfrak \tilde{f}(x) d\mu \leq  \int_\Xfrak \bar{f}^2(x)d\mu <\infty$. Moreover by \eqref{eqn:hdctexp}
    $$
    \left| \frac{f(x|\theta_i^0+\Delta a_i)-f(x|\theta_i^0)}{\Delta} \right| \leq \tilde{f}(x) \quad \forall x\in \Xfrak.
    $$
\end{proof}

\subsection{Proofs for Section \ref{sec:auxiliarylemmas5}}

\begin{proof}[Proof of Lemma \ref{lem:nececondition2}]
    Note that $\sum_{i=1}^{k_0}b_i = \sum_{i=1}^{k_0}b_i P_{\theta_i^0}(\Xfrak)=0$. Construct $G_\ell=\sum_{i=1}^{k_0}p_i^\ell \delta_{\theta_i^0}$ with $p_i^\ell = p_i^0 + b_i/\ell$ for $i\in [k_0]$. For large $\ell$, $p_i^\ell\in (0,1)$ and $\sum_{i=1}^{k_0}p_i^\ell=1$. Then for large $\ell$, $G_\ell\in \Ec_{k_0}(\Theta)$ and $G_\ell\overset{W_1}{\to} G_0$. Then the proof is completed by observing that for large $\ell$
	 $$
	 V(P_{G_\ell},P_{G_0})=\sup_{A\in \Ac}|P_{G_\ell}(A)-P_{G_0}(A)|=\sup_{A\in \Ac}|1/\ell\sum_{i=1}^{k_0}b_iP_{\theta_i^0}(A)|=0, 
	 $$
	 and $ D_1(G_\ell,G_0)= \frac{1}{\ell} \sum_{i=1}^{k_0}|b_i| \neq 0.$
	\end{proof}

\begin{proof}[Proof of Lemma \ref{lem:conditioncprod}]
 By decomposing the difference as a telescoping sum,
    \begin{align*}
   & \left|\frac{\prod_{j=1}^{\m}f(x_j|\theta_i^0+a \Delta)-\prod_{j=1}^{\m}f(x_j|\theta_i^0)}{\Delta}\right| \\
   \leq &\sum_{\ell=1}^{\m} \left(\prod_{j=1}^{\ell-1} f(x_j|\theta_i^0+a \Delta)\right) \left|\frac{f(x_\ell|\theta_i^0+a \Delta)-f(x_\ell|\theta_i^0)}{\Delta}\right| \left(\prod_{j=\ell+1}^{\m}f(x_j|\theta_i^0)\right).
    \end{align*}
     Then the right hand side of the preceding display is upper bounded $\bigotimes^N\mu-a.e.\ \Xfrak^\m$ by 
    $$
    \tilde{f}_\Delta(\bar{x}|\theta_i^0,a,N):= \sum_{\ell=1}^{\m} \left(\prod_{j=1}^{\ell-1} f(x_j|\theta_i^0+a \Delta)\right) \bar{f}_\Delta(x_\ell|\theta_i^0,a) \left(\prod_{j=\ell+1}^{\m}f(x_j|\theta_i^0)\right).
    $$
    For clean presentation we write $\tilde{f}_\Delta(\bar{x}|\theta_i^0,a)$ for $\tilde{f}_\Delta(\bar{x}|\theta_i^0,a,N)$ in the remainder of the proof. Notice that $\tilde{f}_\Delta(\bar{x}|\theta_i^0,a)$ satisfies 
    $$
    \int_{\Xfrak^\m}\tilde{f}_\Delta(\bar{x}|\theta_i^0,a) d \bigotimes^\m \mu = \sum_{\ell=1}^{\m} \int_{\Xfrak} \bar{f}_{\Delta}(x_\ell|\theta_i^0,a) d\mu 
    \to \m \int_{\Xfrak} \lim_{\Delta\to 0^+} \bar{f}_{\Delta}(x|\theta_i^0,a) d\mu.
    $$
    Moreover, for $\bigotimes^N\mu-a.e.\ \bar{x}\in \Xfrak ^\m $
    $$\lim_{\Delta\to 0^+}\tilde{f}_{\Delta}(\bar{x}|\theta_i^0,a) = \sum_{\ell=1}^{\m} \left(\prod_{j=1}^{\ell-1} f(x_j|\theta_i^0)\right) \lim_{\Delta\to 0^+}\bar{f}_\Delta(x_\ell|\theta_i^0,a) \left(\prod_{j=\ell+1}^{\m}f(x_j|\theta_i^0)\right), $$
    and thus 
    $$
    \int_{\Xfrak^\m} \lim_{\Delta\to 0^+}\tilde{f}_{\Delta}(\bar{x}|\theta_i^0,a) d \bigotimes^\m \mu = \sum_{\ell=1}^{\m} \int_{\Xfrak}  \lim_{\Delta\to 0^+}\bar{f}_\Delta(x_\ell|\theta_i^0,a) d\mu = \m \int_{\Xfrak}  \lim_{\Delta\to 0^+}\bar{f}_\Delta(x|\theta_i^0,a) d\mu.
    $$
    \end{proof}

\subsection{Proofs for Section \ref{sec:alinversebound}}
\begin{proof}[lem:onecorsquintmin]
a)		
        It suffices to prove $b=0$ since one can  do the translation $x'=x-b$ to reduce the general case $b$ to the special case $b=0$. Let $f_1(x)=f(x)\1ve_{[-\frac{E}{2U},\frac{E}{2U}]}(x)$, $f_2(x)=f(x)\1ve_{[-\frac{E}{2U},\frac{E}{2U}]^c}(x)$ and $f_U(x)=U\1ve_{[-\frac{E}{2U},\frac{E}{2U}]}(x)-f_1(x)$. Then $$\int_{[-\frac{E}{2U},\frac{E}{2U}]}f_U(x)dx = E - \int_{[-\frac{E}{2U},\frac{E}{2U}]}f_1(x)dx = \int_{[-\frac{E}{2U},\frac{E}{2U}]^c}f_2(x)dx $$ and hence
		\begin{align*}
		\int_{\R}x^2f(x)dx  = & \int_{[-\frac{E}{2U},\frac{E}{2U}]}x^2f_1(x)dx + \int_{[-\frac{E}{2U},\frac{E}{2U}]^c}x^2f_2(x)dx\\
		\geq & \int_{[-\frac{E}{2U},\frac{E}{2U}]}x^2f_1(x)dx + \left(\frac{E}{2U}\right)^2 \int_{[-\frac{E}{2U},\frac{E}{2U}]^c}f_2(x)dx\\
		= & \int_{[-\frac{E}{2U},\frac{E}{2U}]}x^2f_1(x)dx + \left(\frac{E}{2U}\right)^2 \int_{[-\frac{E}{2U},\frac{E}{2U}]}f_U(x)dx\\
		\geq & \int_{[-\frac{E}{2U},\frac{E}{2U}]}x^2f_1(x)dx + \int_{[-\frac{E}{2U},\frac{E}{2U}]}x^2 f_U(x)dx\\
		=& \int_{[-\frac{E}{2U},\frac{E}{2U}]}x^2Udx\\
		= &\frac{E^3}{12U^2}.
		\end{align*}
		The equality holds if and only if the last two inequalities are attained, if and only if $f(x)=U\1ve_{[-\frac{E}{2U},\frac{E}{2U}]}(x)\ a.e.$.
		
b) 
        It suffices to prove $b=0$ since one can always do the translation $y^{(1)}=x^{(1)}-b$ and $y^{(i)}=x^{(i)}$ for all $2\leq i\leq d$ to reduce the general case $b$ to the special case $b=0$. By Tonelli's Theorem, $h(x^{(1)})= \int_{(-a,a)^{d-1}} f(x)dx^{(2)} \ldots dx^{(d)}$ exists for  $a.e.\ x^{(1)}$ and $\int_{\R}h(x^{(1)})dx^{(1)} = E$. Moreover $0\leq h(x^{(1)})\leq U(2a)^{d-1}\ a.e.$ . Then by Tonelli's Theorem and \ref{item:onecorsquintmina}
		$$
		\int_{G} (x^{(1)})^2f(x)dx = \int_{\R}(x^{(1)})^2h(x^{(1)})dx^{(1)}\geq \frac{E^3}{12U^2(2a)^{2(d-1)}}.
		$$
		The equality holds if and only if $h(x^{(1)})=U(2a)^{d-1}\1ve_{[-\frac{E}{2U(2a)^{d-1}},\frac{E}{2U(2a)^{d-1}}]}(x^{(1)})\ a.e.$, if and only if $f(x)=U \ a.e. x\in [-\frac{E}{2U(2a)^{d-1}},\frac{E}{2U(2a)^{d-1}}] \times (-a,a)^{d-1} $. 
\end{proof}

\subsection{Proofs for Section \ref{sec:auxiliarylemmas6}}

\begin{proof}[Proof of Lemma \ref{lem:exphellinger}]
a) It is easy to calculate 
		\begin{equation}
		1-h^2(f(x|\thetaeta_1),f(x|\thetaeta_2)) = \exp\left(A\left(\frac{\thetaeta_1+\thetaeta_2}{2}\right) -\frac{A(\thetaeta_1)+A(\thetaeta_2)}{2}  \right). \label{eqn:hclosed}
		\end{equation}
		Let $g(\thetaeta)=\exp\left(A\left(\frac{\thetaeta_0+\thetaeta}{2}\right) -\frac{A(\thetaeta_0)+A(\thetaeta)}{2}  \right)$. It is easy to verify that $g(\thetaeta_0)=1$, $\nabla g(\thetaeta_0)=0$ and $\nabla^2 g(\thetaeta_0) = -\frac{1}{4}\nabla^2 A(\thetaeta_0)$. Then by \eqref{eqn:hclosed}
		\begin{align*}
		\limsup_{\thetaeta\to\thetaeta_0}\frac{h^2(f(x|\thetaeta),f(x|\thetaeta_0))}{\|\thetaeta-\thetaeta_0\|_2^2} =& \limsup_{\thetaeta\to\thetaeta_0}-\frac{g(\thetaeta)-g(\thetaeta_0)-\langle \nabla g(\thetaeta_0), \thetaeta-\thetaeta_0\rangle }{\|\thetaeta-\thetaeta_0\|_2^2} \numberthis \label{eqn:expdifratio} \\
		=& \limsup_{\thetaeta\to\thetaeta_0}\frac{\frac{1}{8}(\thetaeta-\thetaeta_0)^\top \nabla^2A(\thetaeta_0) (\thetaeta-\thetaeta_0)  +o(\|\thetaeta-\thetaeta_0\|_2^2)}{\|\thetaeta-\thetaeta_0\|_2^2} \\
		\leq & \limsup_{\thetaeta\to\thetaeta_0} \left( \frac{1}{8}\lambda_{\text{max}}(\nabla^2A(\thetaeta_0)) +o(1)\right)\\
		=&\frac{1}{8}\lambda_{\text{max}}(\nabla^2A(\thetaeta_0)).
		\end{align*}
b) 
		First assume that $\Theta'$ is compact and convex.
		For each $\thetaeta,\thetaeta_0\in \Theta'$, by \eqref{eqn:expdifratio},    
		\begin{align*}
		\frac{h^2(f(x|\thetaeta),f(x|\thetaeta_0))}{\|\thetaeta-\thetaeta_0\|_2^2} =& -\frac{g(\thetaeta)-g(\thetaeta_0)-\langle \nabla g(\thetaeta_0), \thetaeta-\thetaeta_0\rangle }{\|\thetaeta-\thetaeta_0\|_2^2} \\
		= & -\frac{\frac{1}{8}(\thetaeta-\thetaeta_0)^\top \nabla^2g(\xi)(\thetaeta-\thetaeta_0) }{\|\thetaeta-\thetaeta_0\|_2^2}\\
		\leq & \frac{1}{8}\sup_{\thetaeta\in \Theta'}\lambda_{\text{max}}(-\nabla^2g(\thetaeta)),
		\end{align*}
		where the second equality follows from Taylor's theorem with $\xi$ in the line joining $\thetaeta$ and $\thetaeta_0$ due to the convexity of $\Theta'$ and Taylor's theorem. The result then follows with $L_2=\sqrt{\frac{1}{8}\sup_{\thetaeta\in \Theta'}\lambda_{\text{max}}(-\nabla^2g(\thetaeta))}$, which is finite since $\nabla^2g(\thetaeta)$, as a function of $A(\theta)$ and its gradient and hessian, is continuous on $\Theta^\circ$.
		
		If $\Theta'$ is compact but not necessarily convex, consider $\operatorname{conv}(\Theta')$. Note that $\operatorname{conv}(\Theta')$, as the convex hull of a compact set, is convex and compact. Moreover, $\operatorname{conv}(\Theta')\subset \Theta^\circ$ since $\Theta^\circ$, as the interior of a convex set, is convex. The proof is then complete by simply repeating the above proof for $\operatorname{conv}(\Theta')$.
\end{proof}

\begin{proof}[Proof of Lemma \ref{lem:Nuppboudiflen}]
	Consider an $\eta_1$-net $\Lambda_i$ with minimum cardinality of $\{\theta: \|\theta-\theta^0_i\|_2 \leq \frac{2\epsilon}{C(G_0,\text{diam}(\Theta_1))} \}$ and an $\eta_2$-net $\bar{\Lambda}$ with minimum cardinality  of $k_0$-probability simplex $\{p\in \R^{k_0}: \sum_{i=1}^{k_0}p_i =1, p_i\geq 0  \}$ under the $l_1$ distance. Construct a set $\tilde{\Lambda}=\{\tilde{G}=\sum_{i=1}^{k_0}p_i\delta_{\theta_i}: (p_1,\ldots,p_{k_0})\in \bar{\Lambda}, \theta_i\in \Lambda_i  \}$. Then for any  $G\in \Ec_{k_0}(\Theta)$ satisfying $ D_1(G,G_0)\leq \frac{2\epsilon}{C(G_0,\text{diam}(\Theta_1))}$, there exists some $\tilde{G}\in \tilde{\Lambda}$, such that by Lemma \ref{lem:hellingeruppbou}
	$$
	h^2(p_{G,\m_i},p_{\tilde{G},\m_i}) \leq \left(\sqrt{\m_i}L_2 \eta_1^{\beta_0} + \frac{1}{\sqrt{2}}\sqrt{\eta_2}\right)^2 \leq 2\left(\m_i L_2^2 \eta_1^{2\beta_0} + \frac{1}{2}\eta_2\right).
	$$
	Thus
	$
	d_{\n,h}(G,\tilde{G}) \leq \sqrt{ 2 L_2^2 \bar{\m}_{\n}\eta_1^{2\beta_0} +\eta_2}.
	$
	
	As a result $ \tilde{\Lambda} $ is a $\sqrt{ 2L_2^2 \bar{\m}_{\n}\eta_1^{2\beta_0} +\eta_2}$-net of $\left\{ G\in \Ec_{k_0}(\Theta): D_1(G,G_0)\leq \frac{2\epsilon}{C(G_0,\text{diam}(\Theta_1))}\right \}$. Since $\tilde{\Lambda}$ is not necessarily subset of $\Ec_{k_0}(\Theta)$,
	\begin{align*}
	\Nf\left(2\sqrt{ 2L_2^2 \bar{\m}_{\n}\eta_1^{2\beta_0} +\eta_2},\left\{ G\in \Ec_{k_0}(\Theta_1): D_1(G,G_0)\leq  \frac{2\epsilon}{C(G_0,\text{diam}(\Theta_1))}\right \}, d_{\n,h} \right)\leq& |\tilde{\Lambda}| \\
	=& |\bar{\Lambda}|\prod_{i=1}^{k_0}|\Lambda_i|.\numberthis \label{eqn:Nuppbou1diflen}
	\end{align*}
	Now specify $\eta_1=\left(\frac{\epsilon}{144L_2\sqrt{\bar{\m}_{\n}}}\right)^{\frac{1}{\beta_0}}$ and thus 
	$$
	|\Lambda_i|\leq \left(1+2\frac{2\epsilon}{C(G_0,\text{diam}(\Theta_1))}/\eta_1\right)^q = \left(1+\frac{4\times (144L_2)^{\frac{1}{\beta_0}}}{C(G_0,\text{diam}(\Theta_1))}  \bar{\m}_{\n}^{\frac{1}{2\beta_0}}\epsilon^{-(\frac{1}{\beta_0}-1)} \right)^q. 
	$$
	Moreover, specify $\eta_2 = \frac{1}{2}\left(\frac{\epsilon}{72}\right)^2$ and by \cite[Lemma A.4]{ghosal2001entropies}, $|\bar{\Lambda}| \leq \left(1+\frac{5}{\eta_2}\right)^{k_0-1}=\left(1+10\times 72^2 \epsilon^{-2}\right)^{k_0-1} $. Plug $\eta_1$ and $\eta_2$ into \eqref{eqn:Nuppbou1diflen} and the proof is complete.
\end{proof}

    \begin{proof}[Proof of Lemma \ref{lem:optimalpermutation}]
        Let $\tau$ be any one in $S_{k_0}$  such that 
        $$
        D_1(G,G_0)=\sum_{i=1}^{k_0}\left(\|\theta_{\tau(i)}-\theta_i^0\|_2+|p_{\tau(i)}-p_i^0|\right).
        $$
        For any $j\not =\tau(i)$, $\|\theta_j - \theta_i^0\|_2 \geq \|\theta_{\tau^{-1}(j)}^0 - \theta_i^0\|_2 - \|\theta_j - \theta_{\tau^{-1}(j)}^0\|_2  > \rho -\rho/2 = \frac{\rho}{2} $. Then
        for any $\tau'\in S_{k_0}$ that is not $\tau$ and for any real number $r\geq 1$
        \begin{multline*}
        \sum_{i=1}^{k_0} \left(\sqrt{r}\|\theta_{\tau'(i)} -\theta_i^0\|_2 + |p_{\tau'(i)}-p_i^0|  \right) > \sqrt{r}\frac{\rho}{2}\\ > \sqrt{r} D_1(G,G_0) \geq  \sum_{i=1}^{k_0} \left(\sqrt{r}\|\theta_{\tau(i)} -\theta_i^0\|_2 + |p_{\tau(i)}-p_i^0|  \right),
        \end{multline*}
        which with $r=1$ shows our choice of $\tau$ is unique and with $r\geq 1$ shows $\tau$ is the optimal permutation for $D_r(G,G_0)$. 
    \end{proof} 

\end{document}